\documentclass[11pt]{amsart}
\usepackage{amssymb,amsmath}
\usepackage{color,comment}
\usepackage{subcaption}
\usepackage{tikz}
\usepackage{float}
\usetikzlibrary{arrows}
\usepackage{graphicx}
\graphicspath{{pictures/}}
\DeclareGraphicsExtensions{.pdf,.png,.jpg, .pdf}
\usepackage[nobysame]{amsrefs}
\AtBeginDocument{
}

\newcommand{\dist}{{\rm dist}}

\def\Bbb{\mathbb}
\def\frak{\mathfrak}
\def\Cal{\mathcal}
\def\Dt{\partial_t}

\def\eb{\varepsilon}
\def\divv{\operatorname{div}}

\def\R {\mathbb{R}}

\def\ind{\operatorname{ind}}

\def\sgn{\operatorname{sgn}}

\def\<{\left<}
\def\>{\right>}
\def\Ree{\operatorname{Re}}
\def\Imm{\operatorname{Im}}
\def\Nx{\nabla_x}
\def\Dx{\Delta_x}
\def\({\left(}
\def\){\right)}
\def\Cal{\mathcal}
\def\Bbb{\mathbb}
\newtheorem{proposition}{Proposition}[section]
\newtheorem{theorem}[proposition]{Theorem}
\newtheorem{corollary}[proposition]{Corollary}
\newtheorem{lemma}[proposition]{Lemma}

\theoremstyle{definition}
\newtheorem{definition}[proposition]{Definition}

\newtheorem{remark}[proposition]{Remark}
\newtheorem{example}[proposition]{Example}

\numberwithin{equation}{section}

\def \no#1#2#3 {{\bf #1} (#3), #2.}
\def \eds#1#2#3 {#1, #2, #3.}

\title[Attractors. Then and now]{Attractors. Then and now}
\author[S. Zelik]{Sergey Zelik${}^{1,2,3}$}
\address{${}^1$ \phantom{e}School of Mathematics and Statistics, Lanzhou University,
Lanzhou  \\ 730000,
P.R. China}
\address{${}^2$
University of Surrey, Department of Mathematics,
Guildford, GU2 7XH, United Kingdom, s.zelik@surrey.ac.uk.}
 \address{${}^3$ Keldysh Institute of Applied Mathematics, Moscow, Russia}
\subjclass[2020]{35B40, 35B41, 37B02, 37B55, 37D10, 37H05, 37L25, 37L30}
\keywords{Dissipative PDEs, Attractors, Inertial manifolds, Finite-dimensional reduction, Determining functionals}
\begin{document}
\begin{abstract} This survey is dedicated to the 100th anniversary of Mark Iosifovich Vishik and is based on a number of mini-courses taught by the author at University of Surrey (UK) and Lanzhou University (China). It discusses the classical and modern results of the theory of attractors for dissipative PDEs including  attractors for autonomous and non-autonomous equations, dynamical systems in general topological spaces, various types of trajectory, pullback and random attractors, exponential attractors, determining functionals and inertial manifolds as well as the dimension theory for the above mentioned classes of attractors. The theoretical results are illustrated by a number of clarifying examples and counter-examples.
\end{abstract}
\thanks{This work is partially supported by  the grants  14-41-00044 and 14-21-00025 of RSF as well as  grants 14-01-00346  and 15-01-03587 of RFBR. The author also thanks V. Chepyzhov, A. Ilyin, V. Kalantarov, A. Kostianko and D. Turaev for many stimulating discussions.}
\maketitle
\scalebox{0.9}{\vbox{\tableofcontents}}
\section{Introduction}

One of the most surprising  lessons of dynamical systems theory is that even relatively simple and deterministic ordinary differential equations (ODEs) may demonstrate very complicated behaviour with strong random features (the so-called deterministic chaos). This phenomenon is intensively studied starting from the second part of 20th century and a lot of prominent results are obtained in this direction as well as a lot of methods for investigating it are developed (such as hyperbolic theory, the Lyapunov exponents, homoclinic bifurcation theory, strange attractors, etc., see \cites{ER85,GH83,KH95,Rue81,Shi01} and references therein). However, this phenomenon occurred to be much more complicated then it was thought from the very beginning, so,  despite many efforts made, the existing theory is mainly restricted to low dimensional model examples (even in these examples, it is established that, in some cases, the problem of a full description of the dynamics is not algorithmically solvable) and we still do not have really effective methods of studying the dynamical chaos in higher dimensional systems of ODEs.
\par
The situation becomes even more complicated when we  deal with dynamical systems generated by partial differential equations (PDEs). For such systems, the initial phase space is infinite-dimensional (e.g. $L^2(\Omega)$, where $\Omega$ is a domain in $\R^d$), which causes a lot of extra difficulties. In addition, together with a temporal variable $t$ and related temporal chaos, we now have spatial variables $x\in\Omega$, so the so-called spatial chaos may naturally appear. Also, we may have an interaction between spatially and temporal chaotic modes which forms the so-called spatio-temporal chaos. As a result, a new type of purely infinite-dimensional dynamics with principally new level of complexity (which is unreachable  in systems of ODEs studied in classical dynamics) may appear, see \cites{MZ07,TZ10,Zel04} and references therein.
\par
Nevertheless, there exists a wide class of PDEs, namely, the class of {\it dissipative} PDEs in bounded domains, where, despite the infinite-dimensionality of the initial phase space, the effective limit dynamics is in some sense finite-dimensional. In particular, this dynamics  can be described by finitely many parameters (the so-called "order parameters" in the terminology of I. Progogine, see \cite{Pri77}), whose evolution obeys a system of ODEs (which is called an inertial form (IF) of the initial dissipative system).
\par
As hinted from the denomination, dissipative systems consume energy (in contrast to
 conservative systems where the total energy is usually preserved), so, in order to
have the non-trivial dynamics, the energy income should be taken into  account (in other
words, the considered system should be {\it open} and should interact with the external world). On
the physical level of rigor, the rich and complicated dynamical structures (often referred as
{\it dissipative structures} following I. Prigogine) arise as the result of interaction of the
following three mechanisms:
\par
1) Energy decay, usually more essential in higher ”Fourier modes”;
\par
2) Energy income, usually through the lower ”Fourier modes”;
\par
3) Energy flow from lower to higher ”modes” provided by the nonlinearities.
\par
Moreover, it is typical  for the dissipative PDEs in {\it bounded} domains that the number
of lower ”modes” where the  energy income is possible is finite. So, it is natural to expect
that these modes are, in a sense, dominating  and the higher modes are slaved by the lower ones. This supports the hypothesis
that the effective dynamics in such systems is finite-dimensional up to some transient behavior
(where the unstable lower "Fourier modes" are treated as the order parameters) and somehow explains why the ideas and techniques of  classical finite-dimensional theory of dynamical systems (DS)
are also effective for describing the dissipative dynamics in such PDEs.
\par
However, the above arguments are very non-rigorous from the mathematical point of view and it is extremely difficult/impossible to make them meaningful. It is even unclear  what are the "modes" in the above statements. Indeed, they  are rigorously defined as Fourier/spectral modes in the {\it linear} theory and a priori make sense only for systems somehow close to the linear ones. Moreover,   such modes are natural for the linearization near  {\it equilibria} only and may not exist at all when the linear equations with time {\it periodic} coefficients are considered. Thus, despite the common usage of modes and related length scales for highly nonlinear systems in physical literature (e.g., in  turbulence, see \cite{firsh} and references therein), the precise meaning of them is usually unclear.
\par
That is the reason why the mathematical theory of dissipative PDEs is  based on related but different concepts which at least can be rigorously defined. Namely, let a dissipative system be given by a semigroup $S(t)$, $t\ge0$, acting in a Hilbert, Banach (or sometimes even Hausdorff topological) space $\Phi$. Usually, in applications, $\Phi$ is some infinite-dimensional functional space where the initial data lives (say, $\Phi=L^2(\Omega)$) and $S(t)$ is a solution operator of the PDE considered. Then the dissipativity is often understood as the validity of the following {\it dissipative} estimate:
\begin{equation}\label{0.dis}
\|S(t)u_0\|_H\le Q(\|u_0\|_H)e^{-\alpha t}+C_*
\end{equation}
for the appropriate positive constants $C_*$ and $\alpha$ and monotone increasing function $Q$ which are independent of $u_0$ and $t$. Roughly speaking, this estimate shows that there is a balance between the energy injection and dissipation, so the energy cannot grow to infinity. Moreover, for very high energy levels, the dissipation is dominating and the energy  decays. As we have already mentioned, the class of dissipative PDEs is rather wide and includes many physically relevant examples, such as Navier-Stokes equations, reaction-diffusion equations, damped wave equations, various pattern formation equations, etc. (see e.g. \cite{tem} and references therein for more examples).
\par
The central concept of the theory is a (global) {\it attractor} $\Cal A$. By definition, it is a {\it compact} subset of the phase space $\Phi$, which is invariant with respect to the semigroup $S(t)$ and attracts the images (under the map $S(t)$) of all bounded sets when time tends to infinity. Thus, on the one hand, the attractor $\Cal A$ consists of the full trajectories of the DS considered and contains all its non-trivial dynamics. On the other hand, it is essentially smaller than the initial phase space. In particular, the compactness assumption shows that the higher modes are indeed suppressed, so the existence of a global attractor already somehow supports the hypothesis on the finite-dimensionality.
\par
The ultimate goal of the present survey is to develop a unified approach to the theory of attractors for dissipative DS generated by PDEs and to clarify its role in rigorous justification of the above mentioned finite-dimensionality conjecture. We emphasize from the very beginning that, although the study of such systems is also strongly based on the Analysis of PDEs (for instance, the well-posedness of 3D Navier-Stokes problem is one of the most challenging open problems in the theory of PDEs and the absence of a reasonable answer whether and in what sense this system is well-posed clearly affects the attractors theory for Navier-Stokes equations), more or less detailed exposition of the modern methods of the Analysis of PDEs is far beyond the scope of this survey, where we will mainly concentrate on various dynamical concepts of the theory of attractors.
\par
The systematic study of the dynamical properties of dissipative PDEs started, to the best of our knowledge, from the pioneering papers of Foias and Prody  \cite{FP67} and Ladyzhenskaya \cite{Lad72} and has been motivated by the dream to understand the turbulence using the methods of DS theory for ODEs. In particular, the existence of an invariant measures for the Navier-Stokes system as well as finitely many determining modes was established in \cite{FP67} and the object which coincides with the modern attractor for the Navier-Stokes problem was implicitly introduced in \cite{Lad72}. We mention also \cite{BiLa71} where the global attractors were introduced with applications to delay differential equations. The theory of attractors has been intensively developing during the last 50 years and many interesting and promising results have been obtained, see \cites{BV92,ChVi02,7,Hal88,Lad22,MZ07,Rob,tem,Zel14} and references therein. In particular, one of the main results of the attractors theory claims that, under the mild assumptions on the dissipative system generated by an evolutionary  PDE in a bounded domain, the corresponding attractor $\Cal A$ has a finite fractal dimension:
$$
\dim_f(\Cal A,\Phi)<\infty.
$$
Combining this fact with the Man\'e projection theorem, we get a one-to-one homeomorphic projection of the attractor $\Cal A$ to the finite-dimensional set $\bar{\Cal A}\subset\R^N$. Moreover, the reduced dynamics on $\bar {\Cal A}$ is governed by a system of ODEs. Thus, the existence of a finite-dimensional attractor allows us to build up an IF of the considered dissipative PDE and somehow justify the finite-dimensionality conjecture, see \cites{Rob11,Zel14} and references therein for more details.
\par
However, this justification is not entirely satisfactory since the obtained IF is only H\"older continuous, no matter how smooth is the initial PDE, so, starting from the smooth DS, we end up with non-smooth equations where even the uniqueness of a solution may be lost. Any attempts to improve the regularity of this reduction lead to the extra restrictive assumptions on the initial PDE which is hard to check and which are violated in many interesting examples. Moreover, as recent examples show, this problem is far from being technical since, despite the finiteness of the fractal dimension of the global attractor and the existence of a H\"older continuous IF, the associated limit dynamics may remain, in a sense, infinite-dimensional and demonstrate features which are not observable in classical dynamics generated by smooth ODEs (like limit cycles with super-exponential rate of attraction, decaying traveling waves in Fourier space, etc.), see \cites{EKZ13, Zel14,KZ18} and references therein. These examples allow us to guess that the sharp borderline between the finite and infinite-dimensional limit dynamics is more related with {\it Lipschitz} continuous IFs and  inertial manifolds rather than with H\"older continuous IFs and the fractal dimension of the attractor.
\par
Thus, despite many efforts done, the rigorous interpretation and justification of the finite-di\-men\-sio\-nal reduction in dissipative systems remains a mystery and is one of the most challenging problems of the modern theory of attractors. We will discuss below in more details some of the currently known approaches  to tackle this and related problems.
\par
The survey is organized as follows.
\par
In section \ref{s1}, we give a flavor of the attractors theory by considering the simplest low dimensional examples, where the attractor can be found more or less explicitly and demonstrate the difference between various types of attractors, their dependence on parameters, their dimensions, etc. We hope that this section will help the reader with understanding more general and abstract theory presented in next sections.
\par
Section \ref{s2} is one of the central sections of the survey, where we are developing the attractors theory in general Hausdorff topological spaces, which, in turn, allows us to build up a unified approach to different types of attractors. Namely, to define an attractor, we first need to specify what objects will be attracted by this attractor. We refer to the collection of these objects as a {\it bornology} $\Bbb B$ on the phase space of the problem  keeping in mind that in the standard theory a (global) attractor usually attracts bounded sets of the phase space. Next, we should  specify in what sense (in what topology) the attraction will hold. In other words, we need to specify the topology on the phase space $\Phi$. Surprisingly, that the considerable attractors theory, which is similar to the standard one, can be constructed in a general situation where $\Phi$ is just a Hausdorff topological space under the minimal assumptions on the bornology $\Bbb B$.
\par
Although attractors in general topological spaces were studied before, see \cites{ChVi02, GMRS05,MarNe02} and references therein, many results of section \ref{s2} are hard/impossible  to find in the literature in the necessary generality, so we present more or less detailed proofs for most of them.
\par
In section \ref{s3}, we apply the unified approach presented above to the case where the considered PDE does not possess a unique solvability property or this uniqueness of a solution is not known yet. One of the most natural approaches to tackle this type of problems (suggested by Chepyzhov and Vishik in \cite{ChVi95}, see also \cite{ChVi02}, and Sell \cite{sell96}) is related with constructing the trajectory dynamical system associated with the  PDE under consideration and studying its attractors. Roughly speaking, we replace the initial phase space $\Phi$, where the problem is ill-posed and the corresponding solution semigroup can be defined as a multi-valued semigroup only, by the new phase space $\Cal K_+$ which consists of all positive semi-trajectories of our PDE satisfying some nice properties. Then, if $\Cal K_+$ is chosen in a proper way, the semigroup $T_h$, $h\in\R_+$, of temporal shifts ($(T_hu)(t):=u(t+h)$) acts on $\Cal K_+$ and $(T_h,\Cal K_+)$ is exactly the trajectory dynamical system associated with the considered problem.
\par
Note that, in the case where the uniqueness property holds, this semigroup is conjugated with the usual solution semigroup $S(t)$ acting in the initial phase space $\Phi$, so the theory is consistent. The advantage of this approach is that the construction of the trajectory dynamical system does not require the uniqueness property to be satisfied, so we may study its attractors  in a usual way avoiding the usage of multi-valued maps. The only problem is to define the bornology and topology on $\Phi$ in a proper way. In relatively simple situations, both topology and bornology can be lifted from the initial phase space, but in more complicated cases (like 3D Navier-Stokes equations) this does not work. Moreover, there are several alternative non-equivalent ways to do this, so the unified approach introduced in the previous section works in full strength here. The purpose of this section is, in particular, to give the comparison of known alternative approaches to the trajectory attractors on the example of the 3D Navier-Stokes equations. To the best of our knowledge this has never been done before.
\par
Section \ref{s4} extends the unified approach to attractors in the non-autonomous case. We consider the most general case, where phase space for the dynamical process $U(t,\tau)$, $t\ge\tau$, associated with the considered problem may depend on time, so $U(t,\tau):\Phi_\tau\to\Phi_t$ and $\{\Phi_t\}_{t\in\R}$ is a family of Hausdorff topological spaces and use the  pullback attraction property in order to define attractors. In this case, an attractor is understood as a time dependent set $\Cal A(t)\subset\Phi_t$, $t\in\R$, the bornology $\Bbb B$ (which is often referred as a universe in this theory) also consists of time dependent sets $B(t)\subset\Phi_t$, $t\in\R$, and the attraction property is pullback in time, i.e. if you fix $t\in\R$ and start from $B(\tau)\in\Bbb B$, $\tau<t$, then the image $U(t,\tau)B(\tau)$ will be close to $\Cal A(t)$ if $t-\tau$ is large enough, see section \ref{s4} for the details. The particular case where $\Phi_t$ are Banach spaces and $\Bbb B$ consists of uniformly (in time) bounded sets was considered in \cite{CPT13}, the suggested extension allows us also to treat from the unified point of view the case of the bornology of tempered sets (which is important for random attractors) as well as many other interesting examples.
\par
This general theory covers, in particular, the  case of cocycles (or skew-products). We recall that a family of maps $\Cal S_\xi(t):\Phi\to\Phi$, $\xi\in\Psi$, $t\ge0$, is a cocycle over the group $T(h):\Psi\to\Psi$, $h\in\R$ if
$$
\Cal S_\xi(0)=Id.\ \ \Cal S_\xi(t+h)=\Cal S_{T(h)\xi}(t)\circ\Cal S_\xi(h),\ \ t,h\ge0.
$$
These objects are natural for the theory of non-autonomous and random dynamical systems. Roughly speaking, the underlying group $T(h):\Psi\to\Psi$ describes the evolution of a time-dependent  symbol of the considered non-autonomous PDE under time shifts and $\Cal S_\xi(t)$ are solution operators (from zero time moment to time moment $t$) of the considered PDE with a given symbol $\xi\in\Psi$, see \cites{ChVi93,ChVi02,CLR12,Har91,KlL07,KlR11} and references therein. Recall that the dynamical process which corresponds to symbol $\xi$ can be recovered from the cocycle by a simple formula $U_\xi(t,\tau)=\Cal S_{T(\tau)\xi}(t-\tau)$ and exactly this relation allows us to extend the theory from dynamical processes to cocycles. Note also that, introducing a Borel probability measure on $\Psi$, which is invariant (and usually ergodic) with respect to the group $T(h)$, allows us to link the theory with the theory of random dynamical systems and their attractors, see \cites{A98,CF94,CF98,CLR12,D98,KlR11} and references therein.
\par
In  section \ref{s4}, we also discuss an alternative approach to attractors of non-autonomous dynamical systems, which is based on the reduction of the cocycle related with the considered PDE to the autonomous semigroup acting on the  extended phase space $\Phi\times\Psi$. This approach leads to an object which is {\it independent of time} and the rate of attraction to it is uniform in time as
well as with respect to $\xi\in\Psi$, for this reason, it is referred as a {\it uniform} attractor. We present here the classical results related with weak uniform attractors and strong ones (in the case of translation-compact external forces), see \cites{ChVi95a,ChVi02} and references therein, as well as more recent results concerning strong uniform attractors for the case of non translation-compact external forces, see \cite{Z15} and references therein.
\par
Section \ref{s5}, which is the second central section of the survey, discusses the dimensions of  attractors. As we have already mentioned, the fact that the attractor $\Cal A$ of the considered PDE has a finite fractal dimension allows us to build up an IF for this equation and this is one of possible ways to justify the finite-dimensionality conjecture for the associated limit dynamics, so getting realistic upper and lower bounds for this dimension is one of the most fashionable branches of the attractors theory. This activity was originated by the seminal paper of Mallet-Paret \cite{MP76} (see also \cite{Lad82}), where the method of estimating the dimension of the negatively invariant sets based on some smoothing/squeezing properties for differences of solutions has been proposed. In a modern interpretation, the main result can be formulated as follows: let $\Phi$ and $\Phi_1$ be two Banach spaces such that $\Phi_1$ is compactly embedded in $\Phi$ and let also $B$ be a bounded set in $\Phi_1$ which is negatively invariant with respect to some map $S$. Assume that the map $S$ enjoys the following kind of squeezing property:
\begin{equation}\label{0.sq}
\|S(u_1)-S(u_2)\|_{\Phi_1}\le\kappa\|u_1-u_2\|_{\Phi_1}+L\|u_1-u_2\|_\Phi, \ \ \forall u_1,u_2\in B
\end{equation}
for some $\kappa\in[0,1)$ and $L>0$. Then the fractal dimension of $\Cal A$ in $\Phi_1$ is finite, see e.g. \cite{PRSE05} or \cite{8}. This result and various its generalizations to non-autonomous and random cases are presented in section \ref{s5}. Note that this approach has a tremendous number of applications in the modern theory of attractors, see e.g. \cites{8,EMZ00,MN96,MaP02,Z00,KSZ21} and references therein, in particular, most  of the results concerning the  {\it exponential} attractors are strongly based on this method, see the discussion below.
\par
We also discuss there  the  volume contraction method for obtaining upper bounds for the Hausdorff and fractal dimension of the attractor suggested by Douady and Oesterle \cite{DO80} (see also \cite{Il83} and \cite{CF85} for generalizations to the infinite-dimensional case). Roughly speaking, the key result here is that if you have a negatively invariant (with respect to some smooth map $S$) compact set $B$ of a Hilbert space,  such that the infinitesimal $k$-dimensional volumes in it are contracted by the map $S$, then the dimension of $B$ is less than $k$. For the case of the Hausdorff dimension this result was established in \cite{DO80} and \cite{Il83} for the finite and infinite dimensional cases respectively. The case of fractal dimension is more delicate and for a long time only partial results with extra unnecessary assumptions were known (see \cites{CF85,ChVi02} and references therein). The breakthrough here was done by Hunt \cite{Hunt96} where exactly the same result has been established for the fractal dimension in the finite-dimensional case. This result was extended later to the infinite-dimensional case in \cite{BlI99} under some extra technical assumptions which has been finally removed in \cite{ChIl04}. The advantage of this method is that, combined with the usage of  Lieb-Thirring inequalities suggested by Lieb in \cite{Lieb84}, it gives the best known upper bounds for the fractal dimension of the attractor of 2D Navier-Stokes equations, see e.g. \cite{tem}, see also \cite{FH21} and \cite{ChIl04} for the best analytic bounds for this dimension via Lieb-Thirring inequalities. This makes the volume contraction method
 very popular in the attractor theory (despite the fact that it also has drawbacks and is not applicable in many cases, in particular, when the solution operators are not differentiable with respect to the initial data, see \cite{MZ08} and references therein).  We give an exposition of this method as well in section \ref{s5}.
\par
We also discuss the lower bounds for the dimension of the attractors in this section. Recall that the most widespread method to get such estimates is based on the fact that an unstable manifold of any (hyperbolic) equilibrium always belongs to the attractor, so its dimension cannot be smaller than the dimension of this unstable manifold. This gives us the lower bounds for the dimension of the attractor in terms of the instability indexes of equilibria, see \cites{BV92,tem} and references therein. However, this is not enough in some cases for obtaining sharp lower bounds and there are examples where the instability indexes of all equilibria remain bounded, but the dimension of the attractor tends to infinity when the physical parameter of the considered system tends to zero. For this reason, we discuss also more exotic but promising alternative method (suggested by Turaev and Zelik \cite{TZ03}), which is based on the homoclinic bifurcation theory. Roughly speaking, this method utilizes the fact that, under some natural assumptions, an invariant torus of very high dimension (which is restricted by the Lyapunov dimension of the corresponding equilibrium) may bifurcate from a homoclinic orbit under the proper choice of the perturbation. This allows us to relate the lower bounds with the Lyapunov dimension (similarly to the upper bounds obtained via the volume contraction method). This  is very important for weakly dissipative equations and gives us the sharp upper and lower bounds of the same order for some class of damped wave equations.
\par
In section \ref{s6}, we discuss the theory of inertial manifolds (IMs) for dissipative PDEs. Recall that, by definition, an IM is a smooth invariant finite-dimensional sub-manifold of the phase space of the considered problem which is globally stable and is normally hyperbolic. The existence of such a manifold gives a perfect justification of the finite-dimensionality conjecture for the limit dynamics. Indeed, the restriction of our system to the IM gives the desired smooth IF which is governed by the system of ODEs on the base of the manifold. On the other hand, due to the normal hyperbolicity, we have the exponential tracking (asymptotic phase) property which guarantees that any other trajectory of the considered PDE attracts exponentially to the corresponding trajectory on the IM, so we do not lose any important information about the limit dynamics when passing to the reduced IF on the manifold. To the best of our knowledge, such an object was first constructed by Man\'e in \cite{Man77} for the case of a reaction-diffusion equation and became popular  after \cite{FST88}, where this result was extended to more general class of equations and  was associated  with  the dream to understand turbulence (even the denomination "inertial" introduced there is motivated by the inertial scale in the conventional theory of turbulence, see e.g. \cite{firsh} and references therein).
\par
The classical construction of an IM requires the considered PDE to satisfy rather restrictive conditions, which allow us to present the associated DS as a slow-fast system and slave the fast modes to the slow ones using the standard methods of hyperbolic theory, see \cites{FST88,Mik91,kok98,RT96,SY02,tem, Ro94,Rob11,Zel14} and references therein. These conditions are usually formulated in terms of {\it spectral gap} conditions on the leading linear part of the  equation under consideration. On the one hand, these conditions are  not satisfied, for instance, for the case of 2D Navier-Stokes equations, so the existence or non-existence of an IM for Navier-Stokes equations is one of the most challenging open problems of the theory. On the other hand, it is also known that the spectral gap conditions are sharp at least on the level of abstract functional models (e.g., in the class of abstract semilinear parabolic equations), see \cites{EKZ13,Zel14}, so further progress in constructing IMs beyond the spectral gap conditions is possible only utilizing some special properties of the concrete classes of PDEs.
\par
The first result in this direction was obtained by Mallet-Paret and Sell in \cite{MPS88}, where  IMs have been constructed for the case of scalar reaction-diffusion equations in 3D with periodic boundary conditions, utilizing  the  {\it spatial averaging} method, see also \cite{MPSS93} where it was shown that this method does not work in space dimensions higher than three. Taking into  account  recent progress in this area (for instance, extending the spatial averaging methods to 3D Cahn-Hilliard equation, see \cite{KZ15}, various truncated or regularized versions of 3D Navier-Stokes equations, see \cites{GG18,K18,KLSZ}, developing the  spatio-temporal averaging method and applying it to 3D complex Ginzburg-Landau equation, see \cites{K21,KSZ22}), we give a brief exposition of this method in section~\ref{s6}.
\par
An alternative method for constructing IMs is based on  transforming the initial PDE or embedding it into a new system of PDEs in such a way that the obtained new system satisfies the spectral gap conditions. This method is originally related with the erroneous attempt  of Kwak \cite{kwak92} to prove the existence of an IM for 2D Navier-Stokes equation, see also \cites{kwak92a,TW93}, and also \cite{KZ221} for clarifying the nature of the error. For this reason, the whole method has been forgotten for a long time and was considered as suspicious and potentially erroneous. The situation is changed recently after the works \cites{K17,K18}, where the proper modification of this method has been applied to solve the long-standing open problem about the existence of IMs for general 1D reaction-diffusion-advection systems, see also \cites{Vuk1,Vuk2,Vuk3} for some preliminary results in this direction. We  include a brief exposition of this alternative method to section \ref{s6}.
\par
We also discuss the smoothness of IMs. It is known that, in general, even if the initial system is analytic, the IM related with this PDE is only $C^{1+\eb}$-smooth for some small $\eb>0$ and further regularity of IMs requires much stronger version of spectral gap conditions which are not satisfied even in the simplest examples, see \cites{CLS92,kok98}. This looks like a big drawback of the theory since the regularity of an IM is important from both a theoretical and an applied point of view. Indeed, even the analysis of simplest bifurcations requires more smoothness than $C^{1+\eb}$ (e.g. for the analysis of the Andronov-Hopf bifurcation, we need $C^3$, see, for instance, \cite{Shi01}). On the other hand, the low regularity of the IM and the related IF prevents us from the usage of higher order numerical methods. This problem looked unsolvable for a long time, but as shown very recently, see \cite{KZ22}, it nevertheless can be overcome. Namely, by increasing the dimension of the manifold and by using the clever cut-off procedure (based on the Whitney extension theorem), we may kill resonances and other  obstacles to the existence of smooth invariant manifolds and  get the $C^k$-smooth IMs for every finite $k$.
\par
Section \ref{s7} of the present survey is devoted to  {\it exponential} attractors. These objects have been introduced in \cite{EFNT94} as, in a sense, intermediate objects between usual attractors and IMs in order to overcome key drawbacks of the attractors theory. The main of these drawbacks is exactly the slow rate of attraction and, which is even more important, the fact that it is impossible, in a more or less general situation, to control this rate of attraction in terms of physical parameters of the considered PDE. This makes the attractor, in a sense, unobservable in experiments: no matter how long we wait, we never can be sure that we are close to the attractor. The absence of such a control also leads to the sensitivity of the attractor to perturbations.
\par
Roughly speaking, the idea of the construction of an exponential attractor, is to add some special points (e.g., metastable states) to the usual attractor in such a way that, on the one hand, the rate of attraction to the new object becomes exponential and controllable and, on the other hand, the size of this object does not grow too much, in particular, it should still have the finite fractal dimension, so the Man\'e projection theorem still allows us to construct an IF on it.
 \par
Following  \cite{EMZ00}, the modern theory of exponential attractors is based on the squeezing property \eqref{0.sq} and various its generalizations, see \cites{BN95,7,EMZ00,EMZ05,EZ07,FMGZ,GMPZ10,MZ08} and references therein. Thus, in contrast to IMs, the exponential attractors are as general as the usual finite-dimensional global attractors, see also \cites{EMZ04,EMZ03} for infinite-dimensional exponential attractors in the case where the dimension of a usual attractor is infinite.
\par
In the present survey, we do not present most general conditions of the form \eqref{0.sq} which guarantee the existence of  exponential attractors (we refer the interested reader to surveys \cite{EMZ04} and \cite{MZ08}). Instead, we discuss the impact of the theory of exponential attractors to non-autonomous and random attractors. Our exposition follows mainly the papers \cite{PRSE05} and \cite{ShZ13} and is based on the straightforward extensions of \eqref{0.sq} to the non-autonomous case.
\par
We recall that the theory of exponential attractors allows us to overcome the extra drawbacks of the theory of usual attractors which arise when  the non-autonomous case is considered. Namely, again because of the non-controllable and non-uniform rate of attraction to, say, pullback attractors, we lose in general the attraction forward in time. The situation is a bit better in the case of random attractors where we have attraction in probability forward in time, but the lack of the uniformity causes a lot of difficulties, in particular, it does not allow to get the robustness of random attractors in the random-deterministic limit, see e.g. \cites{CF98,Chu02} and also \cite{CDLM17} and references therein for the contemporary approach to this problem. In contrast to this, non-autonomous  exponential attractors can be constructed in such a way that the rate of attraction is uniform and exponential both forward and pullback in time and this remains true for random exponential attractors as well, see section \ref{s7} for the details.
\par
Since the sensitivity of attractors with respect to perturbations is closely related with the rate of attraction, we also include some elements of perturbation theory of attractors to the survey as well as a theory of  the so-called {\it regular} attractors in the terminology of Babin and Vishik, see \cite{BV83}. Such attractors are typical for the systems which possess a global Lyapunov function. Then, under the extra generic assumption that the number of equilibria $\Cal R$ is finite and all of them are hyperbolic, the attractor consists of a finite union of finite-dimensional unstable manifolds of these equilibria and every complete trajectory lying on the attractor is a heteroclinic orbit between these equilibria. These attractors possess a number of nice properties and they have many  similarities  with exponential attractors. In particular, the rate of attraction to them is exponential and they are H\"older continuous with respect to perturbations (including the non-autonomous ones), see \cites{BV83,BV92,CL09,Hal04,HR89,VZCh13} and references therein.
\par
In section \ref{s8}, we discuss an alternative approach to the problem of the finite-dimensional reduction (which was suggested in \cite{FP67}, see also \cite{Lad72} and was historically the first one) related with {\it determining} functionals. By definition, a system $\Cal F:=\{\Cal F_1,\cdots,\Cal F_N\}$ of continuous functionals $\Cal F_i:\Phi\to\R$ on the phase space $\Phi$ is (asymptotically) determining for the semigroup $S(t):\Phi\to\Phi$ if for every two trajectories $u_1(t)$ and $u_2(t)$ of this semigroup, the convergence
$$
\lim_{t\to\infty}(\Cal F_i(u_1(t))-\Cal F_i(u_2(t))=0,\ \ i=1,\cdots,N
$$
implies that $\lim_{t\to\infty}\|u_1(t)-u_2(t)\|_\Phi=0$. Thus, the asymptotic behaviour of trajectories of the considered system is determined by the behavior of finitely many quantities $\xi_i(t):=\Cal F_i(u(t))$, $i=1,\cdots, N$. Note, however, that the determining functionals do not give a true finite-dimensional reduction since the quantities $\xi_i(t)$ do not obey a finite system of ODEs. Moreover, typically they satisfy some kind of a system of ODEs with delay (and realize the  Lyapunov-Schmidt reduction), so the phase space of the reduced system of ODEs with delay remains infinite-dimensional, and not any finite-dimensional reduction can be actually constructed in this way. Nevertheless, this topic remains interesting and fashionable (see \cites{6,8,9,14,15,OT08}) since determining functionals and the related reduction to delay ODEs have many important applications, for instance, for establishing the controllability of an initially infinite dimensional system by finitely many modes (see e.g. \cite{AT14}), verifying the uniqueness of an invariant measure for
random/stochasitc PDEs (see e.g. \cite{KuSh12}),  data assimilation problems where the values of functionals $\Cal F_i((u(t))$ are interpreted as the results of observations and the theory of determining functionals allows us to build up new methods of restoring the trajectory $u(t)$ by the results of observations,
see \cites{AT14,AOT13,OT08} and references therein.
 \par
 In our exposition, we mainly follow the recent paper \cite{KAZ22} and try to clarify the nature of determining functionals as well as the nature of  the minimal number of such functionals (the so-called  {\it determining dimension} $\dim_{det}(S(t))$ of the considered system). The trivial lower bound for this dimension is related with the size of the set of equilibria $\Cal R$, namely its embedding dimension $\dim_{emb}(\Cal R)$, i.e. the minimal number $N$ such that there is a continuous injective map $F:\Cal R\to\R^N$. Surprisingly, this lower bound is sharp and we have two sided estimate
 $$
 \dim_{emb}(\Cal R)\le \dim_{det}(S(t))\le\dim_{emb}(\Cal R)+1,
 $$
  see \cite{KAZ22} for the details. The proof of this result is based on the famous Takens delay embedding theorem and its generalizations for H\"older continuous maps, see \cites{Rob11,SIC91,Tak81} and references therein.
 \par
 In particular, in the generic case where the set $\Cal R$ is finite, almost every continuous function $\Cal F:\Phi\to\R$ is a determining functional for the considered PDE and such a determining functional can be chosen in the class of polynomial maps of a sufficiently big degree. Thus, in contradiction to the widespread paradigm, the minimal number of determining functionals describes the structure/size of the set of equilibria $\Cal R$ of the considered system and is related neither with the complexity of the corresponding dynamics on the attractor, nor with its dimension, nor even with the dissipativity of the considered PDE. The corresponding examples, illustrating this statement, are also presented in section \ref{s8}.
 \par
 Finally, the most important for our purposes function spaces and their properties, which are used throughout of the survey, are collected in appendix \ref{a1}.
 \par
To conclude, we note that, because of the restricted size of this survey, it is not possible to pay the proper attention to all important works in the area of attractors, so the choice of the material is somehow subjective and reflects the personal preferences of the author, and many interesting areas (like PDEs with delay, attractors in unbounded domains, approximate inertial manifolds, infinite-dimensional center manifolds and the corresponding hyperbolic theory, etc.) are out of this survey. The author apologises for the inconvenience caused.

 \section{Attractors: basic theory and  model examples}\label{s1}
The aim of this section is to discuss  briefly  various concepts related with attractors and illustrate them by simple examples.   We assume here that we are given a phase space $\Phi$ which will be for a moment a normed space (more general situation where $\Phi$ is a Hausdorff topological space will be considered later)  and a semigroup $S(t):\Phi\to\Phi$ satisfying
\begin{equation}
S(t+h)=S(t)\circ S(h),\ \ t,h\ge0,\ \ S(0)=\operatorname{Id}.
\end{equation}
We refer this semigroup $S(t)$ acting on $\Phi$ as a {\it dynamical system} (DS) and denote it as $(S(t),\Phi)$.
Usually $S(t)$ are the solution operators of an ODE or an evolutionary PDE considered (which maps the initial data to the   corresponding solution at time moment $t$). Thus, we implicitly assume that the corresponding initial value problem is globally solvable and this solution is unique (although this concept may be used even in cases without uniqueness, see section \ref{s3} for more details). The phase space $\Phi$ is often either some  Sobolev space or its subspace endowed with an appropriate topology, see examples below.
\par
The key concept  behind the attractors theory is the concept of an $\omega$-limit set.
\begin{definition}\label{Def1.omega} Let $B\subset\Phi$ be an arbitrary non-empty subset of $\Phi$. Then the $\omega$-limit set of $B$ is defined via
\begin{equation}\label{1.omega}
\omega(B):=\cap_{T\ge0}[\cup_{t\ge T} S(t)B]_{\Phi},
\end{equation}
where $[V]_\Phi$ stands for the closure of a set $V$ in the topology of $\Phi$. In the case of metric spaces, we may give an equivalent, but sometimes more convenient sequential definition:
\begin{equation}\label{1.omega-seq}
\omega(B)=\big\{u_0\in\Phi,\ \exists u_n\in B,\ \exists t_n\to\infty,\ \ u_0=\lim_{t\to\infty}S(t_n)u_n\big\}.
\end{equation}
In general case when $\Phi$ is not metrizable, these two definitions may give different objects.
\end{definition}
It is well-known that, without further assumptions, an $\omega$-limit set can easily be empty or, even if it is occasionally not empty, it may not possess important  properties like invariance or/and attraction. In order to keep them, we need some kind of compactness which is the central assumption of the attractors theory. We summarize the properties of an $\omega$-limit set in the following lemma.
\begin{lemma}\label{Lem1.p-omega} Let the semigroup $S(t)$ be asymptotically compact on $B$, i.e. for any sequences $u_n\in B$ and $t_n\to\infty$ the closure $[S(t_n)u_n]_\Phi$ is a compact set in $\Phi$. Then,
\par
1) The set $\omega(B)$ is not empty.
\par
2) It attracts the images of the set $B$ in the following sense: for every neighbourhood $\Cal O(\omega(B))$, there exists a time moment $T=T(\Cal O)$ such that
\begin{equation}
S(t)B\subset\Cal O(\omega(B)) \ \ t\ge T.
\end{equation}
Let us assume in addition that the operators $S(t)$ are continuous for every fixed $t$. Then the $\omega$-limit set is strictly invariant:
\begin{equation}
S(t)\omega(B)=\omega(B),\ \ t\ge0.
\end{equation}
\end{lemma}
This lemma, which will be proved in the next section, is the main building block in the attractors theory and all versions of attractors (known to the author) use implictly or explicitly the formula for an $\omega$-limit set. We also emphasize here that the concept of an  $\omega$-limit set requires asymptotic compactness, so if we want to have an attractor in our phase space, we need to fix the topology in $\Phi$ in such a way that the asymptotic compactness assumption holds.
\par
The situation is much simpler in the finite-dimensional case $\Phi=\R^n$, where the  {\it dissipative estimate}:
\begin{equation}\label{1.dis}
\|S(t)u_0\|_{\Phi}\le Q(\|u_0\|_\Phi)e^{-\alpha t}+ C_*
\end{equation}
(where $\alpha, C_*>0$ and $Q$ is some monotone increasing function) is enough to get the asymptotic compactness. However, this estimate is not enough in infinite-dimensional case, so an extra work is required in order to verify this compactness.
\par
The next principal question is {\it what objects should be attracted by the attractor?}
\par
There are different answers to this question. First of all, there are {\it local} attractors which are very natural for the  modern theory of dynamical systems, see e.g.\cite{KH95}, and which attract the trajectories starting from some (small) neighbourhood of the attractor only. More advanced version is a Milnor attractor where the attraction property holds for the initial data up to a zero measure set, see \cite{M85}, see also \cites{Con78,GonT17,Rue81,HSZ01,Il13} for more delicate versions of attractors.
\par
However, some {\it global} versions of attractors are traditionally preferable in the theory of dissipative PDEs which is partially explained by the ultimate goal of this theory -- to justify the finite-dimensional reduction. The most widespread is a global attractor which attracts the images of all {\it bounded} sets, see \cites{BV92,Hal88,SY02,tem}.

\begin{definition}\label{Def1.attr} A set $\mathcal A$ is a global attractor for  the dynamical system $S(t):\Phi\to\Phi$ if
\par
1.\ The set $\mathcal A$ is compact in $\Phi$;
\par
2. It is strictly invariant, i.e., $S(t)\mathcal A=\mathcal A$ for all $t\ge0$;
\par
3. The set $\mathcal A$ is an attracting set for the semigroup $S(t)$, i.e., for any bounded $B\subset \Phi$ and any neighbourhood $\mathcal O(\mathcal A)$ of the set $\mathcal A$, there exists $T=T(\mathcal O,B)$ such that
$$
S(t)B\subset\mathcal O(\mathcal A)
$$
for all $t\ge T$.
\end{definition}

Thus, from the one hand, the global attractor consists of complete trajectories of a DS  and contains all non-trivial dynamics (due to the attraction property). On the other hand, due to the compactness property, it is essentially smaller than the initial (usually infinite-dimensional) phase space, so the existence of a global attractor already gives some kind of the reduction of degrees of freedom for the limit dynamics of the considered dissipative system.
\par
The second possibility is a {\it point} attractor which attracts single trajectories only (not bounded sets of trajectories). Other possibility is a $(\Psi,\Phi)$-attractor which has been introduced in \cite{BV92} and which attracts bounded sets in the space $\Psi$ in the topology of the space $\Phi$. It may be also natural to consider the attraction of only {\it compact} sets of initial data, etc. The unified approach to these attractors will be considered in the next section and here we state only the simplest and most popular version of the attractor's existence theorem, see \cites{BV92,tem} for more details.

\begin{theorem}\label{Th1.atr-sim} Let the operators  $S(t)$ be continuous for every fixed $t\ge0$ and let the semigroup $S(t)$ possess a compact attracting set $\Cal B$ in $\Phi$, i.e., for every bounded set $B$ and every neighbourhood $\Cal O(\Cal B)$ there exists $T=T(\Cal O,B)$ such that
$$
S(t)B\subset\Cal O(\Cal B),\ \ t\ge T.
$$
Then the semigroup $S(t)$ possesses a global attractor $\Cal A$ which is a subset of $\Cal B$.
\end{theorem}
This theorem is an almost immediate corollary of the properties of $\omega$-limit sets formulated in Lemma \ref{Lem1.p-omega}. Indeed, the desired attractor can be found via the following two equivalent formulas:
$$
\Cal A=\omega(\Cal B)=\big[\bigcup_{B-\text{bounded}}\omega(B)\big]_{\Phi}.
$$
However, the second expression is more general and works also for other types of attractors, without the continuity assumption, etc.
The similar theorems also hold for other types of attractors, one just needs to understand the attracting set in a proper way, e.g. point-attracting, $(\Phi,\Psi)$-attracting, etc. For instance, the point attractor can be found via
$$
\Cal A_{point}=\big[\cup_{u_0\in\Phi}\omega(u_0)\big]_{\Phi}.
$$
We also mention an important property of global attractors, namely, the  representation formula:
\begin{equation}\label{1.rep}
\Cal A=\Cal K\big|_{t=0},
\end{equation}
where $\Cal K$ is a set of all complete bounded orbits of the semigroup $S(t)$:
$$
\Cal K:=\{u:\R\to\Phi,\ \ S(t)u(h)=u(t+h), \ \ t\in\R_+,\ \ h\in\R,\ \  \sup_{t\in\R}\|u(t)\|_{\Phi}\le C_u\}.
$$
Formula \eqref{1.rep} is actually one of the main technical tools to work with global attractors. In particular, as we will see below, it is extremely useful in perturbation theory, in estimates of Hausdorff and fractal dimensions, etc. However, it may be not true for different types of attractors, e.g. it fails in general for point attractors or in the case where the operators $S(t)$ are not continuous.
\par
We postpone the general theory of attractors till next sections and turn to the examples which illustrate their basic properties.

\begin{example}\label{Ex1.simple} Let us start with the first order scalar ODE
$$
\frac {dy}{dt}=-y((y-1)^2-\eb),\ \ y\big|_{t=0}=y_0,
$$
 where $\eb\in\R$ is a parameter. First of all, multiplying this equation by $y$ and using the Gronwall inequality, it is easy to see that for every value of $\eb$ there is a bounded attracting (and even absorbing) set for the associated semigroup. Thus, we have the global attractor $\Cal A_{gl}(\eb)$ as well as the point attractor $\Cal A_{point}(\eb)$ whose structure depends on $\eb$. For $\eb<0$ the zero equilibrium is globally exponentially stable, so we have
 $$
 \Cal A_{gl}(\eb)=\Cal A_{point}(\eb)=\{0\}.
 $$
 For $\eb\ge0$, two extra equilibria $z=1\pm\sqrt\eb$ appear and the global attractor will be the closed interval containing these points:
 $$
 \Cal A_{gl}(\eb)=[\min\{0,1-\sqrt\eb\},1+\sqrt\eb].
 $$
 However, since any trajectory still stabilizes to one of these equilibria, the point attractor will consist of these three equilibria:
 $$
 \Cal A_{point}(\eb)=\{0,1\pm\sqrt\eb\}.
 $$
 Even on this simplest example we can already see one of the major drawbacks of global attractors, namely, they are not robust with respect to perturbations. Indeed, at the bifurcation point $\eb=0$ we see a "jump" of the attractor from a single point $\{0\}$ to the whole interval $[0,1]$. As we will see below, in general, global attractors are only {\it upper semicontinuous} with respect to perturbations and the lower semicontinuity can be proved in exceptional cases only. In the considered example, the point attractor is also upper semicontinuous with respect to $\eb$, but as next examples show, even the upper semicontinuity may be lost on the level of point attractors.
 \end{example}
\begin{example}\label{Ex1.2D1} More interesting things may happen when we consider equations on the plane. It is more transparent to plot the phase portraits on a plane than to write down the explicit equations generating the corresponding dynamics, so we  describe bellow the dynamics by plotting the corresponding pictures. In this example, we consider the following dynamical systems:
 \begin{figure}[h]
\begin{subfigure}{0.4\textwidth}
\includegraphics[scale=0.33]{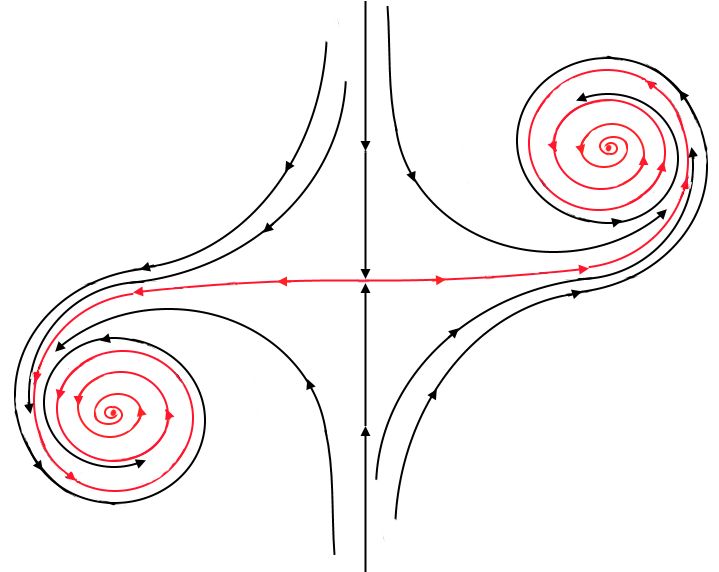}
\caption{The attractor before Andronov-Hopf}
\end{subfigure}
\begin{subfigure}{0.4\textwidth}
\includegraphics[scale=0.33]{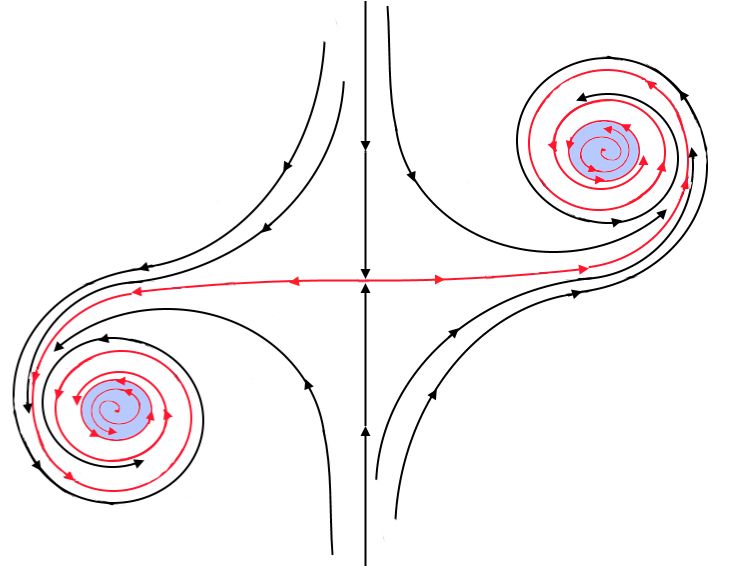}
\caption{The attractor after Andronov-Hopf}
\end{subfigure}
\end{figure}

These systems  have a saddle at the origin and two foci, say at $(1,1)$ and $(-1,-1)$. In the left picture, the foci are stable, so all of the trajectories converge to one of these equilibria. Thus, $\Cal A_{point}$ consists of these equilibria and the global attractor $\Cal A_{gl}$ contains these equilibria together with two unstable separatrices of the saddle plotted by red in the picture.  Thus, topologically the global attractor is still a segment and both Hausdorff and fractal dimensions of it are equal to one.
\par
In the right picture, the foci become unstable via the Andronov-Hopf bifurcation and two extra limit cycles are born, see e.g. \cite{Shi01}. In this case, the point attractor will consist of three equilibria together with newly born limit cycles. The global attractor will now contain extra two disks restricted by limit cycles which are indicated by blue on the picture (together with the equilibria, limit cycles and two unstable separatrices of the saddled, marked by red). This follows, e.g. from the fact that the global attractor consists of all complete bounded orbits. Both Hausdorff and fractal dimensions are equal to two in this case. Remarkable here is that the global attractor remains connected, but becomes not {\it linearly connected} after the bifurcation.
\par
It is also interesting to look at the moment of bifurcation. At this moment, the phase portrait remains topologically the same as on the left picture, but the foci become degenerate. In a generic case the radius $R$ of the spirals satisfy $\frac{ d R}{dt}\sim R^3$, so $R\sim t^{-1/2}\sim\varphi^{-1/2}$, where $\varphi$ is an angle of the spiral. An elementary computation shows that the fractal dimension of such a spiral is equal to $4/3$. Thus, since the unstable separatrices of the saddle form such spirals, we have
\begin{equation}\label{1.frac}
\dim_f(\Cal A_{gl})=\frac43\ne \dim_H(\Cal A_{gl})=1
\end{equation}
(the Hausdorff dimension is one since the attractor is a countable union of 1D segments). This is the simplest mechanism which generates global attractors with non-integer fractal dimension as well as with the difference between Hausdorff and fractal dimensions. We also mention that the Lyapunov dimension of the attractor at the bifurcation moment is obviously equal to two.
\end{example}
\vskip10pt
\begin{example}\label{Ex1.det} One more nontrivial example of 2D dynamics (introduced in \cite{KAZ22}) is given by the next picture:

\vspace{2mm}

{
\hspace{35mm}
\begin{tikzpicture}[scale=0.35]
\fill[outer color=red!5] (0,1) circle (2);
\draw[thick, ->] (0,-1) arc (270:360:0.5);
\draw[thick, ->] (0.5,-0.5) arc (0:180:0.5);
\draw[thick, ->] (-0.5,-0.5) arc (180:270:0.5);
\draw[thick, ->] (1,0) arc (0:120:1);
\draw[thick, ->] (-0.5,0.8660254) arc (120:360:1);
\draw[thick, ->] (0,-1) arc (270:360:1.5);
\draw[thick, ->] (1.5,0.5) arc (0:180:1.5);
\draw[thick, ->] (-1.5,0.5) arc (180:270:1.5);
\draw[color=red!70, thick, ->] (0,-1) arc (270:360:2);
\draw[color=red!70, thick, ->] (2,1) arc (0:90:2);
\draw[color=red!70, thick, ->] (0,3) arc (90:180:2);
\draw[color=red!70, thick, ->] (-2,1) arc (180:270:2);
\draw[thick, ->](0,3.8) arc (90:230:2.6);
\draw[thick, ->](-1.6712478,-0.791715544) to [out=330, in=195] (0,-1);
\draw[thick, ->](2.7,1.1) arc (0:90:2.7);
\draw[thick](2.7,1.1) arc (360:300:2.7);
\draw[thick, ->] (0.5,-3) to [out=85, in=195] (1.35,-1.23826858);
\draw[thick, ->] (0.5,-4)--(0.5,-3);
\draw[thick, ->] (0,-4)--(0,-2);
\draw[thick, ->] (0,-2)--(0,-1);
\draw[thick, ->](0,5.8) arc (90:180:4.7);
\draw[thick, ->](-4.7,1.1) arc (180:230:4.7);
\draw[thick, ->](-3.021101767,-2.500408868) to [out=330, in=260] (0,-1);
\draw[thick](0,5.8) arc (90:30:4.9);
\draw[thick, ->](4.24352446,-1.55) arc (330:390:4.9);
\draw[thick, ->] (2.5,-4) to [out=85, in=230] (4.24352446,-1.55);
\draw[thick, ->] (0.5,-4)--(0.5,-3);
\draw[thick, ->](-5,-4) to [out=350, in=265] (-0.05,-1.4);
\draw[thick,](0,-1) -- (-0.05,-1.4);
\end{tikzpicture}
}

This phase portrait consists of a saddle-node at the origin $(0,0)$ glued with the disc $\{(x,y)\in\R^2\,:\  x^2+(y-1)^2\le 1\}$ filled by homoclinic orbits to the origin. The key feature of this DS is that the $\omega$-limit set of any single trajectory coincides with the origin and, therefore, the point attractor $\Cal A_{point}=\{0\}$ is trivial, but the global attractor
$$
\Cal A_{gl}=\{(x,y)\in\R^2,\ \ x^2+(y-1)^2\le1\}
$$
is not trivial. Note that this phase portrait is extremely degenerate and an arbitrary small perturbation will destroy the homoclinic orbits and may born a "big" limit cycle. Thus, in contrast to the global attractor, the point attractor is in general not upper semicontinuous with respect to perturbation and this is one more reason why the global attractor looks preferable.
\end{example}

\begin{example}\label{Ex1.lor} The situation becomes much more interesting when the dimension of the phase space is larger than two, since we no more have Poincare-Bendixon theorem and more complicated structures than equilibria, limit cycles and homo/heteroclinic connections between them may appear.  One of the most popular 3D model examples here is the Lorenz attractor, see \cite{Lor63} which illustrates the possibility of chaotic "unpredictable" behaviour in deterministic systems - the so-called deterministic chaos. It is strongly believed that similar effects are responsible for the complicated behaviour of more realistic systems (including PDEs) arising e.g. in hydrodynamics, weather prediction, chemical reactions, etc.
\par
The Lorenz system consists of the following equations:
\begin{equation}\label{1.lorenz}
\frac{dx}{dt}=\sigma(y-x),\ \frac{dy}{dt}=x(\rho-z)-y,\ \frac{dz}{dt}=xy-\beta z
\end{equation}
and the standard choice of the parameters is $\sigma=10$, $\beta=\frac83$, $\rho=28$. The corresponding attractor is plotted below:

\begin{figure}[H]
\center{\includegraphics[scale=0.3]{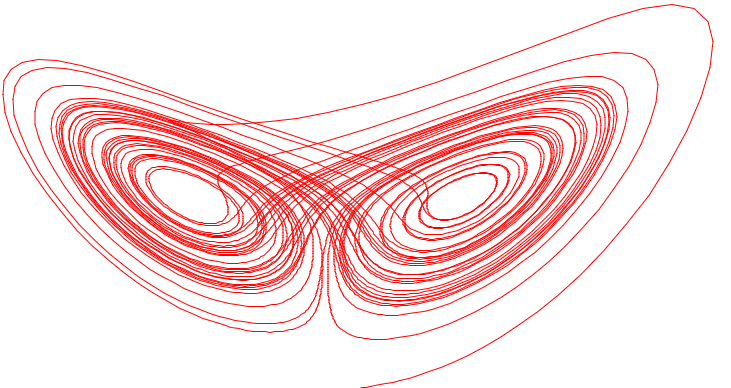}}
\caption{Lorenz attractor: numeric simulation with standard parameters}
\label{fig:4.2.1}
\end{figure}

(the attractor may be essentially different for other values of parameters).
\par
The structure of the Lorenz attractor is well-understood nowadays, but this theory is far beyond the scopes of the survey, so we restrict ourselves by mentioning a few number of interesting properties of this object. More details can be found in \cites{ABS77,GH83,HoM07,Gon21,Kif88,Kuz20,Rand78}.
\par
It is known that the Lorenz system is dissipative and possesses an absorbing ellipsoid  $\Cal B\subset \R^3$, which  is semi-invariant (see e.g. \cite{tem}) and, therefore the global attractor $\Cal A_{gl}$ of \eqref{1.lorenz}  exists. It is a fractal set whose Hausdorff and fractal dimensions are strictly inbetween 2 and 3:
$$
2<\dim_H(\Cal A)\le\dim_f(\Cal A)<3
$$
(at least for the standard values of parameters, see \cite{MPI20}). In particular, the upper bounds follow from the volume contraction method and the explicit formula for the Lyapunov dimension of this attractor:
$$
\dim_L(\Cal A)= 3-\frac {2(\sigma +\beta +1)}{\sigma +1+\sqrt {(\sigma -1)^{2}+4\sigma \rho }},
$$
which holds for all values of the parameters for which all three equilibria are hyperbolic, see~\cite{LKKK16}.
\par
It is worth to emphasize that this global attractor $\Cal A_{gl}$ {\it does not coincide} with the classical Lorenz attractor $\Cal A_{lor}$ or geometric Lorenz attractor, whose existence was verified by Tucker using the interval arithmetics and  the computer assistant proof, see \cite{Tuck02}. Namely, he verified the existence of a compact seminvariant domain $U$ in $\R^3$ (inside of the absorbing ellipsoid $\Cal B$). This domain  looks like a filled double torus which contains zero saddle equilibrium $C_0$, but does not contain two other saddle-foci $C_\pm$. The local attractor of this domain $U$ is exactly the Lorenz attractor $\Cal A_{lor}$, which possesses a geometric description in terms of iterations of 1D discontinuous Poincare map suggested in \cites{ABS77,GH83}, and the main achievement of \cite{Tuck02} is the rigorous numerical verification of the fact that the Lorenz system with classical values of parameters satisfies all of the hypothesis stated in the geometric model. In particular, it follows from the verified assumptions that $\Cal A_{lor}$ is topologically transitive, pseudo-hyperbolic, possesses an invariant measure with nice properties (SRB measure)  and  periodic orbits are dense in it.
\par
On the other hand, there is a strong numeric evidence that there is no complete bounded trajectories belonging to the domain $\Cal B\setminus U$ in the Lorenz system with the standard parameters, although, to the best of our knowledge, this is not rigorously verified yet on the level of computer assistant proofs. If we believe this fact, then the relations between global, Lorenz and point attractor for the Lorenz system are given by

$$
\Cal A_{point}=\Cal A_{lor}\cup C_+\cup C_-,\ \ \Cal A_{gl}=\Cal A_{lor}\cup \Cal M^+(C_+)\cup \Cal M^+(C_-),
$$
where $\Cal M^+(C_\pm)$ are two-dimensional unstable manifolds of the saddle-foci $C_\pm$ and all unions are disjoint.
\end{example}

\begin{example}\label{Ex1.heat} We now turn to examples of PDEs and start with the heat equation in $\R^d$:
\begin{equation}\label{1.heat}
\Dt u=\Dx u,\ \ u\big|_{t=0}=u_0\in\Phi=L^2(\R^d).
\end{equation}
This problem is simple enough to be solved explicitly
$$
u(t,x)=\int_{y\in\R^d}K(x,y)u_0(y)\,dy,\ \ K(x,y)=\frac1{(4\pi t)^{d/2}}e^{\frac{|x-y|^2}{4t}}
$$
and from this formula we see that the solution semigroup $S(t): \Phi\to\Phi$ is well-defined. Moreover, it is not difficult to prove that $u(t)\to0$ in $\Phi$ as $t\to\infty$ for every fixed $u_0\in\Phi$. Thus, the point attractor exists and consists of a zero equilibrium:
$$
\Cal A_{point}=\{u_0\}.
$$
On the other hand, the global attractor $\Cal A_{gl}$ {\it does not exist} here. Indeed, since the considered equation is linear, the existence of a global attractor implies the {\it exponential} stability of the equilibrium which is clearly not the case for equation \eqref{1.heat}.
\par
There are several possibilities to extend the theory of global attractors to this case based on different classes of "bounded" sets    which must be attracted to it or/and different topologies of the phase space. First of all, we may restrict our attention  to one-point sets, then we already have the attractor. However, doing that, we lose a lot of information (clearly the attraction is actually stronger).
\par
The second  possibility would be to restrict the attraction property to {\it compact} sets only. This is probably the most appropriate choice for equation \eqref{1.heat} since we indeed have the uniformity of the attraction for compact sets only.
\par
Other possibilities are related with the change of the  topology of the attraction. Indeed, the attraction of all bounded sets of $\Phi=L^2(\R^d)$ will hold if we endow the phase space  $\Phi$ with the topology induced by the embedding $\Phi\subset \Phi_{loc}:=L^2_{loc}(\R^d)$ (again a straightforward corollary of the explicit formula for the solutions),  This corresponds to the  {\it locally compact} attractor which attracts bounded sets in the local topology only and which is typical for PDEs in unbounded domains, see \cite{MZ08} and references therein for more details.
\par
Alternatively, we may endow the space $\Phi=L^2(\R^d)$ with the {\it weak} topology. Then all bounded sets will be attracted in this weak topology. In all these cases, the attractor will consist of a single point $\{0\}$.
\end{example}
\begin{example}\label{Ex1.1D-par} We complete this section by the model example of a nonlinear heat equation in a bounded domain which demonstrates the standard way how the abstract Theorem \ref{Th1.atr-sim} is used in order  to verify the existence of an attractor for nonlinear PDEs. Namely, we consider the following problem:
\begin{equation}\label{1.ce}
\Dt u=\partial_x^2u+a u-u^3, \ \ x\in(0,\pi), \ \ u\big|_{x=0,\pi}=0,\ \ u\big|_{t=0}=u_0
\end{equation}
in the interval $\Omega=(0,\pi)$ with Dirichlet boundary conditions. Here $a>0$ is a given parameter. We first recall the basic dissipative energy estimate in the phase space $\Phi=L^2(\Omega)$. In order to avoid technicalities, we give below only the formal derivation of the required estimates without their justification (and even without formalizing what is a weak solution of equation \eqref{1.ce}). The skipped  details can be found e.g. in \cites{BV92,tem}.
\par
 Indeed, multiplying  equation \eqref{1.ce} by $u$ and integrating with respect to $x$, we have
$$
\frac12\frac{d}{dt}\|u(t)\|^2_{L^2}+\|\partial_x u(t)\|_{L^2}^2+\|u(t)\|^4_{L^4}=a\|u(t)\|^2_{L^2}.
$$
Using the Poincare inequality $\|\partial_x u\|^2_{L^2}\ge\alpha\|u\|^2_{L^2}$ together with the obvious estimate $a\|u\|^2_{L^2}\le \|u\|^4_{L^4}+Ca^2$, we arrive at
\begin{equation}\label{1.en}
\frac{d}{dt}\|u(t)\|^2_{L^2}+\alpha\|u(t)\|^2_{L^2}+\|\Nx u(t)\|^2_{L^2}\le Ca^2,
\end{equation}
where $C$ is some constant which is independent on $a$,
and, after the integration in time, we get the desired dissipative estimate in $\Phi$:
\begin{equation}\label{1.dis-par}
\|u(t)\|^2_{L^2}\le \|u(0)\|^2_{L^2}e^{-\alpha t}+C_*a^2,
\end{equation}
where $C_*=C/\alpha$.
The global existence  of a solution is straightforward here and can be verified e.g. using the Galerkin approximations, see e.g. \cites{BV92,tem} for more details. Let us verify the uniqueness. Indeed, let $u_1(t)$ and $u_2(t)$ be two solutions of \eqref{1.ce} and let $v(t)=u_1(t)-u_2(t)$. Then this function solves
\begin{equation}\label{1.dif-dif}
\Dt v-\partial_x^2 v=av-(u_1^2+u_1u_2+u_2^2)v, \ \ v\big|_{t=0}=u_1(0)-u_2(0).
\end{equation}
Multiplying this equation by $v(t)$ and using that $u_1^2+u_1u_2+u_2^2\ge0$, we have
$$
\frac{d}{dt}\|v(t)\|^2_{L^2}+2\|\partial_x v(t)\|^2_{L^2}\le 2a\|v(t)\|^2_{L^2}
$$
and
\begin{equation}\label{1.lip}
\|u_1(t)-u_2(t)\|^2_{L^2}+2\int_0^t e^{2a(t-s)}\|\partial_x v(s)\|^2_{L^2}\,ds\le e^{2at}\|u_1(0)-u_2(0)\|^2_{L^2},
\end{equation}
which proves the  uniqueness.
  Therefore, we have constructed a solution semigroup $S(t):\Phi\to\Phi$ associated with equation \eqref{1.ce}. This semigroup possesses a dissipative estimate \eqref{1.dis} in the phase space $\Phi=L^2(\Omega)$ and, by this reason,  the set
$$
\Cal B:=\{u_0\in\Phi\, :\ \ \|u_0\|^2_{\Phi}\le 2C_*a^2\}
$$
is a bounded attracting (and even absorbing) set for this semigroup. However, this is still not enough to verify the existence of a global attractor since the set $\Cal B$ is not {\it compact} in $\Phi$. To overcome this difficulty, we need to utilize the parabolic {\it smoothing} property. To obtain the desired smoothing property, we multiply equation \eqref{1.ce} by $t\partial_x^2 u$ and integrate in $x$. This gives
$$
\frac12\frac d{dt}(t\|\partial_x u(t)\|^2_{L^2})=-t\|\partial^2_x u(t)\|^2_{L^2}-3t\|u\partial_x u\|^2_{L^2}+(at+\frac12)\|\partial_x u(t)\|^2_{L^2}\le (at+\frac12)\|\partial_x u(t)\|^2_{L^2}.
$$
Integrating this estimate with respect to $t\in[0,1]$, we arrive at
$$
\|\partial_x u(1)\|^2_{L^2}\le (2a+1)\int_0^1\|\partial_x u(s)\|^2_{L^2}\,ds.
$$
Finally, we estimate the right-hand side of this inequality by integrating estimate \eqref{1.en} with respect to  $t\in[0,1]$ and using the dissipative estimate \eqref{1.dis-par}. This gives
$$
\|\partial_x u(1)\|_{L^2}^2\le C(a+1)\|u(0)\|^2_{L^2},
$$
where the constant $C$ is independent of the parameter $a$. By shifting time, we also see that
$$
\|\partial_x u(t+1)\|_{L^2}^2\le C(a+1)\|u(t)\|^2_{L^2}.
$$
Combining this estimate with \eqref{1.dis-par}, we see that the set
\begin{equation}\label{1.abs-a}
\Cal B_1:=\big\{u_0\in H^1_0(\Omega),\ :\ \|\partial_x u_0\|_{L^2}^2\le 2CC_*a^2(a+1)\big\}
\end{equation}
is also  the attracting (and even absorbing) set for the semigroup $S(t)$. The difference is that this set is compact in $\Phi$. The continuity of the operators $S(t)$ with respect to the initial data follows from estimate \eqref{1.lip}. Thus, all of the assumptions of Theorem \ref{Th1.atr-sim} are verified and the existence of a global attractor $\Cal A_{gl}\subset\Phi$ for equation \eqref{1.ce} is also verified.
\par
The next question is the structure of the global attractor $\Cal A_{gl}$. The key role in its further investigation  plays the global {\it Lyapunov} function. Indeed, multiplying equation \eqref{1.ce} by $\Dt u$ and integrating in $x$, we derive that
$$
\frac{d}{dt}\Cal L(u(t)):=\frac{d}{dt}\(\frac12\|\partial_x u(t)\|^2_{L^2}+\frac14\|u(t)\|^4_{L^4}-\frac a2\|u(t)\|^2_{L^2}\)=-\|\Dt u(t)\|^2_{L^2}
$$
and see that the function $u\to\Cal L(u)$ is strictly decreasing along the non-equilibria trajectories. This allows us to conclude that any trajectory of \eqref{1.ce} tends as $t\to\infty$ to the set $\Cal R$ of equilibria. The equilibria $u_0\in\Cal R$ solve the second order scalar ODE and their structure can be completely understood. In particular, it is known that new equilibria can bifurcate from zero equilibrium only and all non-zero equilibria remain hyperbolic for all values of the parameter $a$. This fact shows that there are exactly $N=2\lceil \sqrt a\rceil-1$ different equilibria in $\Cal A_{gl}$ if $a>1$ (the case $a<1$ is not interesting since, in this case, we have a single globally exponentially stable zero equilibrium). Thus, from this fact and the existence of a global Lyapunov function, we conclude that any trajectory of the considered equation stabilizes as $t\to\infty$ to one of these equilibria:
$$
\Cal A_{point}=\Cal R=\{u_1,u_2,\cdots,u_N\}
$$
and the global attractor $\Cal A_{gl}$ consists of equilibria and heteroclinic orbits connecting them. More detailed analysis based on Morse-Smale theory allows us to show the robustness of the dynamics on the attractor for $a\in(n^2,(n+1)^2)$, $n\in\Bbb N$, as well as  to determine what equilibria are connected by heteroclinic orbits and to compute the dimensions of the corresponding sets of heteroclinic orbits, see \cites{BV92,FR10,Hen85} for more details.
\par
To conclude we note that \eqref{1.ce} is a rare exception in the attractors theory, where the structure of the global attractor can be completely understood. This is possible thanks to three different nice properties of this equation. The first one is the existence of the global Lyapunov function $\Cal L$ discussed above. The second one is the existence of another type of (discrete) Lyapunov function, namely, the number $Z(u(t))$ of zeros of the profile $x\to u(t,x)$ which is non-increasing function of time. Exactly this discrete Lyapunov function is responsible for the transversality of stable and unstable manifolds of equilibria and the Morse-Smale property. And the third property is that the considered equation is order preserving with respect to the cone of non-negative functions. This allows us to apply the Perron-Frobenious theory and essentially simplify the analysis of spectral properties of equilibria, see \cites{BV92,CI74,Hen85,Kost95} for more details. In contrast to this, the structure of the global attractor remains a mystery  in a more or less general situation (say, for Navier-Stokes equations) and our knowledge is restricted to the general facts that it is compact, connected, consists of bounded trajectories and has the finite Hausdorff and fractal dimensions.
\end{example}

\section{Attractors: a unified approach}\label{s2}
In this section, we present a unified approach to various types of attractors. As we have seen in the previous section, we need to specify two major things:
\par
 1) Bornology: what sets will be attracted by the attractor (we will call such sets "bounded" which explains the denomination bornology);
 \par
  2) Topology: in what sense the attraction will hold.
\par
We assume here that our phase space $\Phi$ is a Hausdorff topological space. On the one hand, main results of the theory discussed in the previous section can be easily extended to general  Hausdorff topological spaces and, on the other hand, this extension makes the theory more convenient and more elegant especially when the attraction in a weak or weak-star topology is considered.
\par
Then we fix a family of sets $\Bbb B\subset 2^{\Phi}$ which will be referred as a {\it bornology} (or bounded structure) on the phase space $\Phi$. Recall that usually the definition of a bounded structure includes some assumptions on $\Bbb B$ such as the stability with respect to inclusions and finite unions as well as the fact that $\Phi=\cup_{B\in\Bbb B}B$, see e.g. \cite{Hog77}. Since the attraction of a set $B$ automatically implies the attraction of all its subsets and the same is true for finite unions, the first two assumptions are actually not restrictive (we always can extend our bornology to satisfy them without changing the attraction properties). However, the third assumption is a big restriction since it excludes local attractors, so we prefer not to pose it. We also do not pose the first two assumptions since they can be satisfied by the straightforward extension mentioned above, so, in our theory, $\Bbb B$ is an arbitrary family of non-empty sets.
\par
Finally, we are given a semigroup $S(t):\Phi\to\Phi$ acting in our phase space which is not assumed to be continuous, so we will need some modifications in the theory. We start with the standard definitions.

\begin{definition} A set $\Cal B\subset\Phi$ is an absorbing ($\Bbb B$-absorbing) set for the semigroup $S(t)$ if, for every $B\in\Bbb B$, there exists time $T=T(B)$ such that
$$
S(t)B\subset\Cal B,\ \ t\ge T.
$$
A set $\Cal B$ is an attracting ($\Bbb B$-attracting) set for the semigroup $S(t)$ if, for every $B\in\Bbb B$ and every neighbourhood $\Cal O(\Cal B)$ of the set $\Cal B$, there exists $T=T(B,\Cal O)$ such that
$$
S(t)B\subset\Cal O(\Cal B),\ \ t\ge T.
$$
\end{definition}
We now turn to the attractors. The key difference here is that the invariance of the attractor $\Cal A$ requires some kind of continuity of $S(t)$ which we do not assume, so we need to replace the invariance by minimality.

\begin{definition} A set $\Cal A\subset\Phi$ is an attractor ($\Bbb B$-attractor) of the semigroup $S(t)$ acting in the topological space $\Phi$ if the following assumptions are satisfied:
\par
1) The set $\Cal A$ is a compact set in $\Phi$;
\par
2) The set $\Cal A$ is an attracting set ($\Bbb B$-attracting set) for the semigroup $S(t)$;
\par
3) The set $\Cal A$ is a minimal (by inclusion) set which satisfies properties 1) and 2).
\end{definition}
  We are now ready to state the main result of this section related with the existence of an attractor.

\begin{theorem}\label{Th2.main} Let $S(t):\Phi\to\Phi$ be a semigroup acting in a Hausdorff topological space $\Phi$ endowed with a bornology $\Bbb B$. Assume that $S(t)$ possesses a compact $\Bbb B$-attracting set $\Cal B$. Then there exists a $\Bbb B$-attractor $\Cal A$ which can be obtained as follows:
\begin{equation}\label{2.attr}
\Cal A=\big[\cup_{B\in\Bbb B}\omega(B)\big]_\Phi,
\end{equation}
where $\omega(B)$ is the $\omega$-limit set of $B$ defined by \eqref{1.omega}.
\end{theorem}
\begin{proof} Although the proof looks more or less standard and various particular cases of this theorem can be found in the literature, see \cites{ChVi02,MarNe02}, for the convenience of the reader we give all details  here. As usual, we need to verify the analogue of \eqref{Lem1.p-omega} for our case.
\par
{\it Step 1.} The set $\omega(B)$ is not empty for all non-empty $B\in\Bbb B$. Let
$$
K_T:=\big[\cup_{t\ge T}S(t)B\big]_{\Phi},\ \ U_T=\Phi\setminus K_T.
$$
Assume that $\omega(B)=\cap_{T\ge0}K_T=\varnothing$. Then $\{U_T\}_{T\ge0}$, is an open covering of $\Phi$ and, in particular, an open covering of the compact set $\Cal B$. Since $U_T$ are non-empty and nested, there exists $T\ge0$ such that $\Cal B\subset U_T$ and, therefore,
$$
K_T\cap\Cal B=  [\cup_{t\ge T}S(t)B\big]_{\Phi}\cap\Cal B=\varnothing.
$$
We claim that this contradicts the attraction property. Indeed, let $x\in\Cal B$ be arbitrary. Then there exists a neighbourhood $U_x$ of $x$ such that $U_x\cap \cup_{t\ge T}S(t)B=\varnothing$. Let $U=\cup_{x\in\Cal B}U_x$. Then, on the one hand, $U$ is a neighbourhood of $\Cal B$ and, on the other hand
$$
U\cap \cup_{t\ge T}S(t)B=\varnothing.
$$
Therefore, $U\cap S(t)B=\varnothing$ for all $t\ge T$ and $\Cal B$ is not an attracting set for $B$. This contradiction proves that $\omega(B)$ is not empty.
\par
{\it Step 2.} $\omega(B)\subset\Cal B$. Indeed, let $x\in\omega(B)\setminus\Cal B$. Then, since $\Phi$ is Hausdorff and $\Cal B$ is compact, there exist neighbourhoods $U_{\Cal B}$ and $U_x$ of the set $\Cal B$ and the point $x$ respectively such that $U_{\Cal B}\cap U_x=\varnothing$.   Then, from the attraction property we know that $S(t)B\subset U_{\Cal B}$ for all $t\ge T$. On the other hand, we know that $x\in K_T$ for all $T\ge0$. From here we conclude that $\cup_{t\ge T}S(t)B\cap U_x\ne\varnothing$ for all $T\ge0$. This contradicts the fact that $U_{\Cal B}\cap U_x=\varnothing$ and proves the statement.
\par
{\it Step 3.} $\omega(B)$ attracts $B$.  Note  that $\omega(B)$ is compact as a closed subset of a compact set. Let $U$ be an arbitrary neighbourhood of  $\omega(B)$. Let us consider closed sets $C_T:=K_T\setminus U=K_T\cap(\Phi\setminus U)$. Then $\cap_{T\ge0}C_T=(\cap_{T\ge0}K_T)\setminus U=\varnothing$ (since $\cap_{T\ge0}K_T=\omega(B)\subset U$). Thus, the sets $V_T:=\Phi\setminus C_T$ cover the compact set $\Cal B$.
 Taking a finite covering and using that the sets $V_T$ are increasing, we get that $\Cal B\subset V_T$ for some $T$. Thus, $C_T\cap B=\varnothing$. Since $C_T$ is closed and $\Cal B$ is compact, there is a neighbourhood $U_{\Cal B}$ of $\Cal B$ such that $C_T\cap U_{\Cal B}=\varnothing$. In particular, $(S(t)B\setminus U)\cap U_{\Cal B}=\varnothing$ for all $t\ge T$ (since $S(t)B\subset K_T$ for $t\ge T$). Since $S(t)B\subset U_{\Cal B}$ for a sufficiently large $t$ due to the attraction property, we must have $S(t)B\subset U$ if $t$ is large enough which proves the attraction property.
\par
{\it Step 4.} The attractor. We define it by formula \eqref{2.attr}. Then since  $\omega(B)\subset\Cal B$ for all $B\in\Bbb B$, $\Cal A$ is a closed subset of the compact set $\Cal B$, so it is compact in $\Phi$. The attraction property is also obvious since, by definition, $\Cal A$ contains $\omega(B)$ for every $B\in\Bbb B$. Thus, we only need to check the minimality. To this end, it is enough to prove that $\omega(B)$ is a minimal set which attracts $B$. Indeed, let $\Omega(B)$ be another compact attracting set for $B$. Then, arguing as in Step 2, we get that $\omega(B)\subset\Omega(B)$ and $\omega(B)$ is indeed minimal. This finishes the proof of the theorem.
\end{proof}
Let us now discuss the invariance of the attractor $\Cal A$. This requires some kind of continuity of the operators $S(t)$.
\begin{proposition}\label{Prop1.inv} Let the assumptions of Theorem \ref{Th2.main} hold and let, in addition, the operators $S(t):\Phi\to\Phi$ be continuous for every fixed $t$. Then the attractor $\Cal A$ is strictly invariant with respect to $S(t)$, i.e.
$$
S(t)\Cal A=\Cal A
$$
for all $t\ge0$.
\end{proposition}
\begin{proof} We first prove that $\Cal A\subset S(t)\Cal A$. Indeed, $S(t)\Cal A$ is compact as a continuous image of a compact set. Let us prove that $S(t)\Cal A$ is an attracting set. Let $U$ be a neighbourhood of $S(t)\Cal A$. Then, by continuity, $V:=S(t)^{-1}(U)$ is a neighbourhood of $\Cal A$. Due to the attraction property for $\Cal A$, for any $B\in\Bbb B$, there exists time $T=T(B,V)$ such that $S(h)B\subset V$ for all $h\ge T$. Then $S(h)B\subset U$ for all $h\ge T+t$ and the attraction property is proved. Thus, by minimality of $\Cal A$, we have the desired inclusion.
 \par
 Let us prove the opposite inclusion. Indeed, for any $B\in\Bbb B$, we have
\begin{multline*}
 S(t)\omega(B)=S(t)\cap_{T\ge0}[\cup_{h\ge T}S(h)B]_{\Phi}\subset \cap_{T\ge0}S(t)[\cup_{h\ge T}S(h)B]_{\Phi}\subset\\\subset \cap_{T\ge0}[S(t)\cup_{h\ge T}S(h)B]_{\Phi}=\cap_{T\ge0}[\cup_{h\ge T+t}S(h)B]_{\Phi}=\omega(B).
\end{multline*}
Using again that, due to the continuity, $S(t)[A]_\Phi\subset [S(t)A]_\Phi$ together with \eqref{2.attr}, we conclude that $S(t)\Cal A\subset\Cal A$ and finish the proof of the proposition.
\end{proof}
The next example shows that the continuity assumption can not be omitted.
\begin{example}\label{Ex2.ni} Let $\Phi:=(-\infty,0]\cup\{1\}$ with the standard topology and bornology  induced by the embedding to $\R$ and let
$$
S(t)x:=\begin{cases} e^{-t}x,\ \ x<0,\\
              1,\ \ x\ge0.
              \end{cases}
$$
Then, obviously, $\Cal A=\{0,1\}$. However, $\omega(\Cal A)=\{1\}\ne\Cal A$ and $S(t)\Cal A=\{1\}\ne \Cal A$. This example shows that without continuity assumptions we cannot expect strict invariance of the attractor. Moreover, if $\Cal B$ is a compact attracting set of $S(t)$, then in general we cannot expect that
\begin{equation}\label{2.n-om}
\Cal A=\omega(\Cal B).
\end{equation}
Indeed, in our case, we may take $\Cal B=\Cal A$. In addition, this example shows that the condition that $S(t): \Cal B\to \Phi$ are continuous is not enough to get the invariance (again, $S(t)$ are continuous on $\Cal B:=\Cal A$ in our example) and we really need to require continuity on a larger set than $\Cal B$.
\end{example}
We now discuss the validity of the representation formula \eqref{1.rep}. We now define a bounded complete trajectory $u:\R\to\Phi$ as a full trajectory of $S(t)$ such that
$$
B_u:=\cup_{t\in\R}u(t)\in\Bbb B
$$
and consider the set $\Cal K$ of all complete bounded trajectories with respect to a given bornology~$\Bbb B$.
\par
It is immediate to see from the attraction property that any complete bounded trajectory belongs to the attractor $\Cal A$, so we have
\begin{equation}\label{2.emb}
\Cal K\big|_{t=0}\subset\Cal A.
\end{equation}
However, the opposite embedding is not true in general and requires some extra assumptions. In particular, we need some kind of continuity in order to get the invariance of the attractor (see Example \ref{Ex2.ni} where there are no complete bounded trajectories passing through the point $0\in\Cal A$). This continuity is still not enough to get the representation formula \eqref{1.rep}. For instance, if the bornology $\Bbb B$ consists of one-point sets only, then complete bounded trajectories are equilibria only, so \eqref{1.rep} holds for the corresponding point attractor $\Cal A_{point}$ if and only if it coincides with the set of equilibria ($\Cal A_{point}=\Cal R$). The next proposition gives a useful sufficient condition for the validity of \eqref{1.rep}.

\begin{proposition}\label{Prop2.rep} Let $\Phi$ be a Hausdorff topological space with a bornology $\Bbb B$ which is stable with respect to inclusions. Let also $S(t):\Phi\to\Phi$ be a continuous semigroup in $\Phi$ (i.e., all maps $S(t)$, $t\ge0$, are continuous) and possess a bounded compact attracting set $\Cal B$. Then this semigroup possesses a $\Bbb B$-attractor $\Cal A\subset\Cal B$ which is strictly invariant and is generated by all complete bounded trajectories (i.e., \eqref{1.rep} holds).
\end{proposition}
\begin{proof} Indeed, the existence of the attractor $\Cal A$ follows from Theorem \ref{Th2.main}. Its strict invariance is proved in Proposition \ref{Prop1.inv}, embedding \eqref{2.emb} follows from the attraction property, so we only need to verify that every point $u_0\in\Cal A$ belongs to some complete bounded trajectory. This trajectory can be constructed as follows: for $t\ge0$, we just define $u(t):=S(t)u_0\in\Cal A$. By the proved invariance, we have $u(t)\in\Cal A$ for all $t\ge0$.
Let us now define $u(-n)$ for $n\in\Bbb N$ by induction. Indeed, by the strict invariance, there exists $u(-1)\in\Cal A$ such that $S(1)u(-1)=u_0$, then there exists $u(-2)\in\Cal A$ such that $S(1)u(-2)=u(-1)$ and so on. Finally, for $t\in[-n,-n+1]$, we set $u(t):=S(t+n)u(-n)$. Then, by the invariance, $u(t)\in\Cal A$ for all  $t\in\R$ and by the construction $u(0)=u_0$. Since $B_u\in\Cal A$ and $\Cal A$ is bounded, using the stability of $\Bbb B$ by inclusion, we conclude that $B_u\in\Bbb B$ and $u\in\Cal K$. This proves the proposition.
\end{proof}
\begin{remark} As we have seen, the continuity assumption for the semigroup $S(t)$ cannot be omitted. However, there are several possibilities to weaken it. One of them is to require that the graph of the map $S(t)$ is closed in $\Phi\times\Phi$ for every fixed $t\ge0$, see \cites{BV92,PZ07}. Another possibility is to verify this continuity not in the whole phase space $\Phi$, but on the {\it absorbing} set $B_0$  only. Namely, it is not difficult to verify that the strict invariance of the attractor remains true if $S(t):B_0\to \Phi$ are continuous for every fixed $t$. This is especially useful for the case where $\Phi$ is endowed with weak or weak star topology. Although this topology is not metrizable on the whole space $\Phi$, very often its restriction on the properly chosen {\it absorbing} set is metrizable. This allows us to verify only the {\it sequential} continuity of the maps $S(t)$. This may be a great simplification since the topological continuity of the nonlinear maps $S(t)$ is much harder to check than the sequential one (at least when we are speaking about weak and weak-star topologies).
\par
We also recall that in a general topological space compactness and sequential compactness are different (even unrelated), so there are  two corresponding types of attractors: the topological and sequential ones which also may be completely different. We will not go into the details here (see \cite{KZ23} for these details), but only mention two natural cases where they  coincide. The first one is when we have a metrizable absorbing set and the second, more interesting case is when $\Phi$ is a Banach space endowed with a weak topology. Then compactness and sequential compactness coincide due to the Eberlein-Smulian theory.
\end{remark}
\begin{remark} To conclude this section, we also mention that in many cases, the key dissipative estimates give us the existence of an absorbing set $\Cal B$ which is in a sense "bounded" in $\Phi$ and by this reason it is often a complete {\it metrizable} space even if the initial phase space $\Phi$ is not metrizable. Then the verification of the existence of a {\it compact} attracting set is reduced to checking the so-called asymptotic compactness, namely, we need to check that, for every sequences $t_n\to\infty$ and $u_n\in\Cal B$, the sequence $\{S(t_n)u_n\}_{n=1}^\infty$ is pre-compact. Actually, it is not difficult to check that under the above assumptions this property is {\it equivalent} to the existence of a compact attracting set, but the usage of asymptotic compactness arguments allows us to avoid the explicit construction of this compact attracting set, which is useful in many applications.
\par
Up to the moment, there are many powerful methods for verifying the asymptotic compactness, e.g. based on the Kuratowski measure of non-compactness or on the  energy method or on the compensated compactness arguments. We refer the interested reader to \cites{7,Hal88,MRW98} (see also references therein) for more details.
\end{remark}

\section{Trajectory attractors}\label{s3}
In this section, we demonstrate how the attractors theory can be applied for studying equations without the uniqueness of solutions  of the corresponding initial value problem. One of possible approaches to handle the non-uniqueness problem is to consider semigroups of multivalued maps and to extend the concept of an attractor to such semigroups, see \cites{BV86,Ball78,Ba94,MV98}, however, there exists an alternative more elegant approach which has been developed in \cites{ChVi95,sell96} for the study of 3D Navier-Stokes equations and which we will use here (see also \cite{VZ96} for the case of elliptic equations). Under this approach, one constructs a  {\it trajectory} dynamical system associated with the considered problem  and then applies the usual theory of attractors to this dynamical system. We illustrate below the main idea on a number of examples.
\par
\subsection{ODEs and reaction-diffusion equations} We start with considering relatively simple case of ODEs and reaction-diffusion systems (RDS) without uniqueness.

\begin{example}\label{Ex3.1} Let us consider a system  of ODEs
\begin{equation}\label{3.ODE}
u'+f(u)=0,\ \ u=(u^1,\cdots, u^N), \ \ u\big|_{t=0}=0.
\end{equation}
We assume that the nonlinearity $f$ is locally Lipschitz and satisfies the standard dissipativity assumption:
\begin{equation}\label{3.dis}
f(u).u\ge-C+\alpha|u|^2,\ \ u\in\R^N
\end{equation}
for some positive constants $C$ and $\alpha$. Then problem \eqref{3.ODE} is locally uniquely solvable for any $u_0\in\R^N$. Moreover, taking a scalar product of \eqref{3.ODE} with $u$, we arrive at the inequality
$$
\frac12\frac d{dt}|u(t)|^2+\alpha|u(t)|^2\le C
$$
and integrating this inequality, we arrive at the key dissipative estimate
\begin{equation}\label{3.dis1}
|u(t)|^2\le |u_0|^2e^{-2\alpha t}+\frac C{\alpha}.
\end{equation}
This estimate gives us the global well-posedness of problem \eqref{3.ODE} for any $u_0\in\R^N$ and, therefore, this problem generates a dissipative semigroup $S(t):\Phi\to\Phi$ in the phase space $\Phi=\R^N$. Note that this semigroup is not necessarily a group since the solutions $u(t)$ may blow up for negative times.
\par
To construct the attractor for this semigroup, we fix the standard bornology $\Bbb B$ in $\R^N$ which consists of  usual bounded sets of $\R^N$. Then the ball $\Cal B:=\{u\in\R^N,\ \ |u|^2\le \frac{2C}\alpha\}$ will be a compact and bounded absorbing set for the semigroup $S(t)$. Moreover, the maps $S(t)$ are continuous with respect to the initial data. Thus, according to the general theory,
there exists a global attractor $\Cal A\subset\Phi$.
\par
To explain the idea of trajectory attractors, we consider the set $\Cal K_+\subset C_{loc}(\R_+,\Phi)$ which consists of all semi-trajectories $u(t)$, $t\ge0$, of the constructed solution semigroup $S(t)$ starting from all $u_0\in\Phi$. Then the semigroup of time shifts $T(h)$, $h\ge0$, acts on this set:
\begin{equation}\label{3.tr}
(T(h)u)(t):=u(t+h),\ \ t,h\in\R_+,\ \ T(h)\Cal K_+\subset \Cal K_+.
\end{equation}
Thus, we have defined a dynamical system $T(h)$  on $\Cal K_+$. This dynamical system is called a {\it trajectory} dynamical system associated with equation \eqref{3.ODE} and the set $\Cal K_+$ is referred as the {\it trajectory} phase space.
\par
Note that, in the case where the uniqueness theorem holds, the constructed trajectory dynamical system is topologically conjugate to the semigroup  $(S(t),\Phi)$ acting on the usual phase space $\Phi$. Indeed, let $\Bbb S: \Phi\to \Cal K_+$ be the solution operator $u_0\to u(\cdot)$ associated with equation \eqref{3.ODE}. Then this map is one-to-one with the inverse map defined via $\Bbb S^{-1}u:=u(0)$ and is a homeomorphism if we endow the set $\Cal K_+$ with the topology induced by the embedding to $C_{loc}(\R_+,\Phi)$. Moreover, we have an obvious relation
$$
S(t)=\Bbb S^{-1}\circ T(t)\circ\Bbb S,
$$
which show that the trajectory dynamical system $(T(h),\Cal K_+)$ is equivalent to the classical one $(S(t),\Phi)$. Finally, if we define the bornology $\Bbb B_{tr}$  on $\Cal K_+$ as follows:
\begin{equation}\label{3.t-born}
\Bbb B_{tr}=\{B\subset\Cal K_+,\ \ B\big|_{t=0} \ \text{is bounded in}\ \Phi\},
\end{equation}
then we will have also one-to-one correspondence between bounded sets in $\Cal K_+$ and $\Phi$. Thus, since $(S(t),\Phi)$ possesses a global attractor $\Cal A$ (as explained above), the trajectory dynamical system $(T(h),\Cal K_+)$ also possesses an attractor $\Cal A_{tr}:=\Bbb S\Cal A$ which is called the {\it trajectory} attractor associated with equation \eqref{3.ODE}. This attractor is strictly invariant with respect to the shift semigroup and consists of complete bounded trajectories:
\begin{equation}\label{3.tr-comp}
\Cal A_{tr}=\Cal K\big|_{t\ge0}.
\end{equation}
Note also that, due to the dissipative estimate \eqref{3.dis1}, the bornology $\Bbb B_{tr}$ consists of all bounded sets in $C_{loc}(\R,\Phi)\cap \Cal K_+$ (or even $C_b(\R,\Phi)\cap \Cal K_+$).
\par
Up to the moment, the above constructions look as a tautology, however, the construction of the trajectory dynamical system $(T(h),\Cal K_+)$ {\it does not} require the uniqueness theorem to hold, so it can be extended to the cases where the uniqueness theorem is either violated or is not known. For instance, in the considered example, we may relax the assumption of the Lipschitz continuity of $f$ and assume that $f$ is only {\it continuous}, but satisfies the dissipativity assumption \eqref{3.dis}. Then we still have the existence of a solution for any $u_0\in\Phi$ and the validity of dissipative estimate \eqref{3.dis1}, so we may construct the trajectory dynamical system $(T(h),\Cal K_+)$ and the bornology $\Bbb B_{tr}$ in the same way as before, so we may speak about the (trajectory) attractor $\Cal A_{tr}$ of this system despite the fact that the uniqueness theorem does not hold here in general and the classical dynamical system $(S(t),\Phi)$ is even not properly defined.
\begin{proposition}\label{Prop3.ODE-tr} Let the function $f$ be continuous and satisfy the dissipativity assumption \eqref{3.dis}. Then the trajectory DS $(T(h),\Cal K_+)$ endowed with the topology of $C_{loc}(\R,\Phi)$ and the bornology $\Bbb B_{tr}$ described above possesses an attractor $\Cal A_{tr}\subset\Cal K_+$ which is strictly invariant and consists of all complete bounded trajectories of \eqref{3.ODE}, i.e. \eqref{3.tr-comp} holds.
\end{proposition}
\begin{proof} In contrast to the case with uniqueness, we are now unable to get the result by lifting the global attractor $\Cal A$ to the trajectory phase space (the dynamical system $(S(t),\Phi)$ does not exist any more), so we will use Proposition \ref{Prop2.rep} instead. Indeed, the shift maps $T(h)$ are continuous in the topology of $C_{loc}(\R,\Phi)$ and the dissipative estimate \eqref{3.dis1} gives us a bounded absorbing set $\Cal B$ for the semigroup $T(h)$. Namely, we may take
$$
\Cal B:=\{u\in\Cal K_+,\ \ \|u\|_{C_b(\R,\Phi)}\le \frac{2C}\alpha\}.
$$
We claim that this set is actually compact in $\Cal K_+\subset C_{loc}(\R_+,\Phi)$. Indeed, from equation \eqref{3.ODE}, we conclude that $u'$ is uniformly bounded with respect to $u\in\Cal B$ and the Arzela theorem gives the desired compactness. Thus, we have proved the existence of the trajectory attractor $\Cal A_{tr}\subset\Cal B$ which is strictly invariant with respect to $T(h)$ and is generated by all complete bounded trajectories of \eqref{3.ODE} (being pedantic, we have verified only that it generates by all complete bounded trajectories of $T(h)$, but it is immediate to see that there is a one-to-one correspondence between such trajectories and complete bounded solutions of \eqref{3.ODE}).
\end{proof}
\end{example}
\begin{example}\label{Ex3.RDS-tr} We now consider (following mainly \cite{CVZ10}) a bit more complicated case of a reaction-diffusion equation in a bounded domain $\Omega\subset\R^d$
\begin{equation}\label{3.RDS}
\Dt u=a\Dx u-f(u),\ \ u\big|_{\partial\Omega}=0,\ \ u\big|_{t=0}=u_0.
\end{equation}
Here $u=(u^1,\cdots,u^N)$ is an unknown vector-valued function, $\Dx$ is the Laplacian with respect to the variable $x=(x_1,\cdots,x_d)$, $a$ is a given diffusion matrix which satisfies the condition $a+a^*>0$ and $f$ is a given nonlinearity, which satisfies a slightly stronger dissipativity assumption:
\begin{equation}\label{3.dis2}
-C+\alpha|u|^{p+1}\le f(u).u\le C(1+|u|^{p+1}),
\end{equation}
where $p>0$ is a given exponent. We say that the function
$$
u\in L^2_{loc}(\R_+;W^{1,2}_0(\Omega))\cap L^{p+1}_{loc}(\R_+,L^{p+1}(\Omega)):=\Theta_{loc}
$$
is a weak solution of \eqref{3.RDS} if it satisfies the equation in the sense of distributions, see e.g. \cite{ChVi02}. Important for us that from \eqref{3.RDS}, we see that for any such  solution
\begin{equation}\label{3.der}
\Dt u\in L^2_{loc}(\R_+,W^{-1,2}(\Omega))+L^q_{loc}(\R_+,L^q(\Omega)),
\end{equation}
where $\frac1q+\frac1{p+1}=1$. Therefore, the function $t\to\frac12\|u(t)\|_{L^2}^2$ is well defined and  absolutely continuous in time and the following energy identity:
\begin{equation}\label{3.en}
\frac12\frac d{dt}\|u(t)\|^2_{L^2}+(a\Nx u(t),\Nx u(t))+(f(u(t)),u(t))=0
\end{equation}
holds for almost all $t\ge0$. In turn, this identity implies that $u\in C([0,T],L^2(\Omega))$ and, therefore, the initial condition $u\big|_{t=0}=u_0$ is well-defined, see \cites{ChVi02,CVZ10}.
\par
It is not difficult to verify, using e.g. Galerkin approximations and the energy identity \eqref{3.en}, that thus defined weak solutions exist for every $u_0\in L^2(\Omega):=\Phi$ and satisfy the dissipative estimate:
\begin{equation}\label{3.dis3}
\|u(t)\|^2_{L^2}+\int_t^{t+1}\|\Nx u(s)\|^2_{L^2}\,ds+\int_t^{t+1}\|u(s)\|^{p+1}_{L^{p+1}}\,ds\le C\|u_0\|^2_{L^2}e^{-\beta t}+C_*
\end{equation}
for some positive constants $C$, $C_*$ and $\beta$, see \cites{ChVi02,CVZ10}.
\par
Let us now turn to attractors. As before, we define the trajectory phase space $\Cal K_+\subset\Theta_{loc}$ as a set of all weak solutions of problem \eqref{3.RDS}. Then the semigroup of time shifts $T(h)$ acts on $\Cal K_+$, so the trajectory dynamical system $(T(h),\Cal K_+)$ is well-defined. As in the previous example, we lift the standard bornology of the phase space $\Phi=L^2(\Omega)$ to the set $\Cal K_+$, namely, $B\in\Bbb B$ if and only if $B\big|_{t=0}$ is bounded in $\Phi$. Then estimate \eqref{3.dis3} and the embedding
$$
\Cal K_+\subset C_{loc}(\R_+,\Phi)
$$
 guarantee that the set
$$
\Cal B:=\{u\in\Cal K_+,\ \|u\|_{L^2_b(\R_+,W^{1,2}_0)}^2+\|u\|_{L^{p+1}_b(\R_+, L^{p+1})}\le 2C_*\}
$$
is a bounded absorbing set for the semigroup $T(h)$.  Moreover, since this set is bounded in the reflexive locally convex space $\Theta_{loc}$, it is precompact in the weak topology $\Theta_{loc}^w$ of the space $\Theta_{loc}$ by the Banach-Alaoglu theorem. The set $\Cal B$ is also closed in $\Theta_{loc}^w$ since it is precompact and is sequentially closed (the last fact is proved similarly to the proof of the existence of a weak solution, see \cite{ChVi02}), so it is compact by Eberlein-Smulian theorem (actually, we may avoid the usage of Eberlein-Smulian theory here since $\Cal B$ is bounded in $\Theta_{loc}$ and the weak topology on it is metrizable).
\par
Thus, we have found a compact bounded absorbing set $\Cal B$ for the trajectory DS $(T(h),\Cal K_+)$ endowed with the weak topology. Mention also that the shift semigroup $T(h)$ is obviously continuous.  Then, due to Proposition \ref{Prop2.rep}, there exists a weak trajectory attractor $\Cal A_{tr}^w\subset \Cal K_+$ which is the attractor of the trajectory dynamical system $(T(h),\Cal K_+)$ in the weak topology of $\Theta_{loc}$. It is strictly invariant and is generated by all complete bounded solutions of equation \eqref{3.RDS}.
\par
We now prove that the constructed trajectory attractor is actually an attractor in the strong topology of $\Theta_{loc}$ as well. To this end, we need an absorbing set which will be compact in a strong topology as well. We claim that $\Cal B_1:=T(1)\Cal B$ is such a set. To verify the compactness, we use the  energy method developed in \cites{MRW98,Ball04,CVZ10}. This method is based on the fact that for many Banach spaces (including Hilbert ones, the spaces $L^p$ with $1<p<\infty$, etc.), the weak convergence $\xi_n\to\xi$ together with the convergence of the norms imply the strong convergence.
\par
Let $u_n\in\Cal B$ be a sequence of solutions. Then, without loss of generality we may assume that $u_n\to u\in\Cal K_+$ in a weak topology of $\Theta_{loc}$. To verify the compactness of $T(1)\Cal B$ in the strong topology, it is sufficient to prove that $u_n\to u$ strongly in $L^2(T,T+1;W^{1,2}_0)$ and in $L^{p+1}(T,T+1;L^{p+1})$ for every $T\ge1$. For simplicity we consider only the case $T=1$ (the general case is analogous). To this end, we multiply the energy identity \eqref{3.en} for $u_n$ by $t$ and integrate over $t\in[0,2]$ to get
\begin{equation}\label{3.en-n}
\frac12\|u_n(2)\|^2_{\Phi}+\int_0^2t(a\Nx u_n(t),\Nx u_n(t))\,dt+\int_0^2t(f(u_n(t)),u_n(t))\,dt=\frac12\int_0^2\|u_n(t)\|^2_\Phi\,dt.
\end{equation}
Using the boundedness of $u_n$ in $\Theta_{loc}$ together with the control \eqref{3.der} of time derivatives $\Dt u_n$, we conclude that $u_n\to u$ weakly in $C(0,2;\Phi)$ and strongly in $L^2(0,2;\Phi)$. Thus, without loss of generality, we also have the convergence $u_n\to u$ almost everywhere. Passing to the subsequence if necessary, we extract from the last equality that
\begin{multline}
\frac12\lim_{n\to\infty}\|u_n(2)\|^2_{\Phi}+\lim_{n\to\infty}\int_0^2t(a\Nx u_n(t),\Nx u_n(t))\,dt+\alpha\lim_{n\to\infty}\int_0^2t\|u_n(t)\|^{p+1}_{L^{p+1}}\,dt+\\+\lim_{n\to\infty}\int_0^2t(f(u_n(t)),u_n(t))-\alpha t\|u_n(t)\|^{p+1}_{L^{p+1}}\,dt=\frac12\int_0^2\|u(t)\|^2_\Phi\,dt.
\end{multline}
Using now the weak lower semicontinuity of convex functions and the Fatou lemma (in order to handle the term containing the nonlinearity $f$) together with the identity \eqref{3.en-n} for the limit solution $u(t)$, we conclude that
$$
\lim_{n\to\infty}\int_0^2t(a\Nx u_n(t),\Nx u_n(t))\,dt=\int_0^2t(a\Nx u(t),\Nx u(t))\,dt
$$
and
$$
\lim_{n\to\infty}\int_0^2t\|u_n(t)\|^{p+1}_{L^{p+1}}\,dt=\int_0^2t\|u(t)\|^{p+1}_{L^{p+1}}\,dt.
$$
Taking into  account the weak convergence, we conclude that $u_n\to u$ strongly in $L^2(1,2;W^{1,2}_0)$ as well as in $L^{p+1}(1,2;L^{p+1})$. Thus, we have verified the compactness of the set $\Cal B_1$ in the strong topology of $\Theta_{loc}$. In turn, this gives the existence of a trajectory attractor $\Cal A_{tr}$ in the strong topology of $\Theta_{loc}$ which is generated by all complete bounded trajectories of equation \eqref{3.RDS} and coincides with the weak trajectory attractor constructed above.
\end{example}
\begin{remark} We recall that the possible non-uniqueness of solutions of \eqref{3.RDS} is formally caused by the absence of local Lipschitz continuity of the nonlinearity $f$. However, in contrast to ODEs, adding this natural assumption would not change the situation drastically due to possible blow up of smooth solutions which may occur despite the dissipative energy estimate \eqref{3.dis3}, see \cite{BRW05} for the case of 3D complex Ginzburg-Landau equation and \cites{HV98,KaZ16,PS00} for different classes of RDS.
\end{remark}
\begin{remark} We emphasize another important property of system \eqref{3.RDS} which does not hold for more general classes of PDEs, namely, the fact that {\it every} weak solution of \eqref{3.RDS} satisfies the energy identity \eqref{3.en} and energy estimate \eqref{3.dis3}. In particular, this fact allows us to define the bornology in the trajectory phase space $\Cal K_+$  just by lifting bounded sets in the classical phase space $\Phi$ to  the space of trajectories, see \eqref{3.t-born}, which is not typical for the trajectory attractors theory, see examples below. As a result, we may define the generalized multivalued semigroup $S(t):\Phi\to2^{\Phi}$ via
\begin{equation}\label{3.gen}
S(t)u_0:=\{u(t)\in\Phi\, : u \ \text{ is a weak solution of \eqref{3.RDS} on the interval $[0,t]$}\}
\end{equation}
and consider the (generalized) global attractor $\Cal A_{gl}$ for this semigroup
avoiding the usage of any trajectory spaces. Of course, we will have the relation $\Cal A_{gl}=\Cal A_{tr}\big|_{t=0}$, see \cites{ChVi02,CVZ10,VC11} for more details. Again, in a more general situation, the straightforward formula \eqref{3.gen} does not work and we cannot define the associated multi-valued semigroup without using the trajectory phase space $\Cal K_+$ and a non-trivial bornology on it. We also mention the  concatenation property, namely, if $u_1(t)$, $t\in[0,T_1]$ and $u_2(t)$, $t\in[T_1,T_2]$ are two weak solutions of \eqref{3.RDS} and $u_1(T_1)=u_2(T_1)$, then the compound function
$$
u(t)=\begin{cases}u_1(t),\ t\in[0,T_1],\\ u_2(t),\ t\in[T_1,T_2]\end{cases}
$$
is a weak solution of \eqref{3.RDS} on the interval $t\in[0,T_2]$. This property also fails in more complicated applications of the trajectory attractors theory.
\end{remark}

\subsection{Trajectory attractors for elliptic PDEs} We continue by  more simple and a bit artificial example which demonstrates some important features of trajectory attractors and gives us a toy example for  trajectory attractors of elliptic PDEs.

\begin{example}  Let us consider  the following second order ODE:
\begin{equation}\label{3.hyp}
u''+\gamma u'+u^3-u=0,\ \ u\big|_{t=0}=u_0\in\R.
\end{equation}
Of course, the solution of this problem is not unique since we "forgot" to pose the initial data for $u'\big|_{t=0}$. However, multiplying this equation by $u'$, we arrive at the following energy identity:
\begin{equation}\label{3.lyap}
\frac d{dt}\(\frac{u'^2}2+\frac{u^4}4-\frac{u^2}2\)+\gamma u'^2=0
\end{equation}
which shows that all solutions of this equation are bounded in time. Moreover, a bit more accurate arguments related with multiplication of this equation by $u'+\alpha u$ where $\alpha>0$ is properly chosen, give the dissipative estimate:
\begin{equation}\label{3.hyp-dis}
u'^2(t)+u^2(t)\le C(u^2(0)+u'^2(0))^2e^{-\kappa t}+C_*,
\end{equation}
where $C$, $\kappa$ and $C_*$ are some positive constants which are independent of $u$.
\par
Let us define the trajectory phase space $\Cal K_+\subset \Theta_{loc}:=C^1_{loc}(\R)$ as the set of all solutions $u(t)$, $t\ge0$ of problem \eqref{3.hyp} which correspond to all $u_0\in\R$. The bornology $\Bbb B$ on it is defined as follows: $B\subset\Cal K_+$ is bounded if $B\big|_{[0,1]}$ is bounded in $C^1[0,1]$. Then the dissipative estimate \eqref{3.hyp-dis} can be rewritten in the form of
$$
\|u\|_{C^1[T,T+1]}\le C\|u\|_{C^1[0,1]}^2e^{-\kappa T}+C_*
$$
and, therefore, the set $\Cal B:=\{u\in\Cal K_+, \|u\|_{C^1_b(\R_+)}\le 2C_*\}$ is a bounded absorbing set for the trajectory dynamical system $(T(h),\Cal K_+)$. Moreover, expressing the second derivative $u''$ from the equation, we see that $u''$ is also bounded if $u\in\Cal B$, so by the Arzela theorem $\Cal B$ is compact in $\Theta_{loc}$. Thus, according to the general theory, the trajectory dynamical system $(T(h),\Cal K_+)$ possesses an attractor $\Cal A_{tr}\subset\Cal K_+$ which is generated by all complete bounded solutions of \eqref{3.hyp}.
\par
The structure of the constructed attractor can be easily understood since we have the global Lyapunov
function (due to the identity \eqref{3.lyap}), so we have only the equilibria $u=0,\pm1$ and heteroclinic orbits between them. Moreover, since the map $u\to(u(0),u'(0))$ is one-to-one as the map from $\Cal K_+\to\R^2$, the obtained attractor is actually two-dimensional and qualitatively looks like
in Example \ref{Ex1.2D1} (left picture up to rotation on $\pi/4$). Thus, the trajectory approach allows us to restore the "forgotten" initial condition on $u'\big|_{t=0}$ and end up with the standard attractor for this equation.
\end{example}
\begin{remark}\label{Rem3.strange} We see that in the previous example, we are unable to obtain the bornology on the trajectory phase space just by lifting bounded sets from the space $\R$ of the initial data (we will not have the dissipative estimate for such "bounded" sets since $u'(0)$ will be out of control) and the concatenation property is also lost here.
\par
 Actually, we may define the topology and bornology on $\Cal K_+$ in several equivalent ways. For instance, we may take $\Theta_{loc}:=C_{loc}(\R_+)$ and understand first and second derivative of $u$ in the distributional sense (it is not difficult to show that such distributional solutions are actually classical ones and a posteriori $u\in C^2_{loc}(\R_+)$). Bounded sets in $\Cal K_+$ also can be defined alternatively as sets which are bounded in the Frechet space $\Theta_{loc}$ or in the Banach space $C^1_b(\R)$ (indeed, the dissipative estimate \eqref{3.hyp-dis} guarantees that in all these cases we will have the same bornology on $\Cal K_+$). We mention here a bit exotic construction of $\Cal K_+$ and the associated bornology suggested by Vishik and Chepyzhov   to handle trajectory attractors for 3D Navier-Stokes and damped wave equations (see \cites{ChVi95,ChVi02}). Namely, let us define the trajectory space $\Cal K_+$ as the set of all solutions $u$ of \eqref{3.hyp} which satisfy the following estimate:
$$
\|u\|_{C^1[T,T+1]}\le C_ue^{-\kappa t}+2C_*,\ \ T\in\R_+,
$$
where $C_*$ and $\kappa$ are fixed (the same as in \eqref{3.hyp-dis}) and the constant $C_u$ may depend on $u$. Note that the set $\Cal K_+$ thus defined is shift-invariant (so the trajectory dynamical system $(T(h),\Cal K_+)$ is well-posed), but a priori may be smaller than the set of all solutions of \eqref{3.hyp}. However, for our model example,  thanks to \eqref{3.hyp-dis} we know that a posteriori this is not true. The bornology on $\Cal K_+$ is then naturally defined as follows: $B\in\Bbb B$ if and only if $C_u\le C_B<\infty$ for all $u\in B$. It is also easy to see that, in our case, this choice of $\Bbb B$ gives the same bounded sets as the previous constructions.
\end{remark}
\begin{example}\label{Ex3.ell} We now turn to a bit more interesting example of trajectory attractors related with elliptic boundary problems in cylindrical domains. Namely, let us consider the following elliptic boundary value problem:
\begin{equation}\label{3.ell}
a(\partial^2_t u+\Dx u)+\gamma\Dt u-f(u)=0, \  \ u\big|_{\partial\Omega}=0,\ \ u\big|_{t=0}=u_0,\  t\ge0,\ \ x\in\Omega,
\end{equation}
where $\Omega$ is a bounded domain of $\R^d$, $u=(u^1(t,x),\cdots,u^n(t,x))$ is an unknown vector valued function, $a=a^*>0$ is a given diffusion matrix, $\gamma\in\R$ is a given parameter and $f\in C(\R^d,\R^d)$ is a given nonlinearity which satisfies the dissipativity assumption
\begin{equation}\label{3.dis4}
f(u).u\ge-C+\alpha|u|^{2+\eb}
\end{equation}
for some positive $C$, $\alpha$ and $\eb$.
\par
We emphasize that originally the PDE \eqref{3.ell} is not an evolutionary one, but may appear, for instance,  under the study of traveling wave solutions for evolutionary PDEs. Then the parameter $\gamma$ is naturally interpreted as a wave-speed, see \cites{CMS93,Ba94,VZ96,MieZ02,VVV94} and references therein. One of possible and rather popular approach to study these equations is related with the interpretation of the variable along the axis of the cylinder as time and apply the methods of the theory of dynamical systems (center manifolds, inertial manifolds, attractors, etc., see \cites{Ba94,Ba95,K82,Mie94} for more details).
\par
Since the Cauchy problem is ill-posed for elliptic equations, as a rule, we do not have uniqueness of solutions for problem \eqref{3.ell} and, for this reason, it is natural (following \cite{VZ96}) to use the trajectory approach. To this end, we fix the space $\Phi:=C_0(\Omega)$ as the space of initial data and consider the set $\Cal K_+\subset \Theta_{loc}:=C_{loc}(\R_+,\Phi)$ of all (weak) solutions $u(t,x)$ which are defined in the whole semi-cylinder $\R_+\times\Omega$. Note from the very beginning that due to the interior regularity for elliptic equations, any solution $u\in W^{2,p}((t,t+1)\times\Omega)$ for all $p<\infty$ and the regularity of a solution is restricted by the smoothness of $f$ only.
 \par
The following lemma gives the crucial dissipative estimate for the solutions of \eqref{3.ell}.
\begin{lemma}\label{Lem3.ell-dis} Let $u$ be a solution of problem \eqref{3.ell} defined on the interval $t\in[0,N]$. Then the following estimate holds:
\begin{equation}\label{3.dis-est5}
\|u(t)\|_{\Phi}\le Q(\|u(0)\|_\Phi) H(1-t)+Q(\|u(N)\|_\Phi)H(t-N+1)+C_*,\ \ t\in[0,N],
\end{equation}
where $C_*$ and $Q$ are some constant and monotone increasing function respectively, which are independent of $u$, $t$ and $N$, and $H(z)$ is the standard Heaviside function.
\end{lemma}
\begin{proof}[Sketch of the proof] Let $w(t,x):=au(t,x).u(t,x)$. Then, taking the dot product of equation \eqref{3.ell} with $u(t,x)$, we arrive at
$$
\frac12(\partial^2_t w(t,x)+\Dx w(t,x))=f(u(t,x)).u(t,x)+a\nabla_{t,x}u(t,x).\nabla_{t,x}u(t,x)-\gamma\Dt u(t,x).u(t,x)
$$
and using the dissipativity condition on $f$ and the positivity of the matrix $a$, we get
$$
\Dt^2 w(t,x)+\Dx w(t,x)+\alpha w(t,x)^{1+\eb/2}\ge -C.
$$
Thus, due to the maximum/comparison principle it is enough to verify the analogue of estimate \eqref{3.dis-est5} for the following ODE:
$$
y''(t)+\alpha y(t)^{1+\eb/2}=-C,\ \ y\big|_{t=0}=\|a u(0).u(0)\|_C,\ \ y\big|_{t=N}=\|a u(N).u(N)\|_C.
$$
But such an estimate for the ODE is straightforward and we leave it to the reader, so the lemma is proved.
\end{proof}
The proved estimate allows us to verify the existence of a solution for problem \eqref{3.ell} for every $u_0\in\Phi$. Indeed, to this end we first solve the corresponding boundary value problem on the finite interval $t\in[0,N]$ with an extra boundary condition $u_N\big|_{t=N}=0$ and then pass to the limit $N\to\infty$. The possibility to do this is guaranteed by this estimate, see \cite{VZ96} for more details. Moreover, passing to the limit $N\to\infty$ in \eqref{3.ell}, we get that {\it any} solution $u\in\Cal K_+$ satisfies the estimate
\begin{equation}\label{3.ell-dis}
\|u(t)\|_{\Phi}\le Q(\|u_0\|_\Phi) H(1-t)+C_*.
\end{equation}
In particular, any solution $u\in\Cal K_+$ is bounded as $t\to\infty$.
\par
Thus, the trajectory DS $(T(h),\Cal K_+)$ associated with equation \eqref{3.ell} is constructed. We endow the space $\Cal K_+$ with the topology of $\Theta_{loc}$ and with the bornology $\Bbb B$ lifted from the phase space $\Phi$, namely, $B\subset\Cal K_+$ is in $\Bbb B$ if $B\big|_{t=0}$ is bounded in $\Phi$. Then estimate \eqref{3.ell-dis} guarantees that the set
$$
\Cal B:=\{u\in\Cal K_+,\ \|u\|_{C_b(\R_+,\Phi)}\le C_*\}
$$
is a bounded absorbing set for the above defined trajectory DS. Moreover, due to the interior regularity estimate for elliptic equation, we know that the set $\Cal B_1:=T(1)\Cal B$ is bounded in $C^1_b(R_+\times\Omega)$ and, by the Arzela theorem it is compact in $\Theta_{loc}$. Thus, a bounded and compact absorbing set for the trajectory DS is constructed and the following result holds.
\begin{proposition}\label{Prop3.ell} Under the above assumptions, problem \eqref{3.ell} possesses a trajectory attractor $\Cal A_{tr}\subset\Theta_{loc}$ which consists of all bounded solutions of \eqref{3.ell} defined on the whole cylinder $\R\times\Omega$:
$$
\Cal A_{tr}=\Cal K\big|_{t\ge0}.
$$
\end{proposition}
\end{example}
\begin{remark} Note that, at least for the application to traveling waves, exactly the set $\Cal K$ of all bounded solutions of \eqref{3.ell} on the whole cylinder consists of traveling waves and is the main object of interest. This set clearly contains the equilibria $\Cal R$ of \eqref{3.ell} (trivial solutions which are independent of $t$) and an interesting and important question whether or not it contains anything else (whether or not non-trivial traveling waves exist). In the case of evolutionary equations we know that the attractor is usually connected (see also \cites{Val95,KlV10} and the material below for connectedness of attractors without uniqueness), therefore we would expect that $\Cal K\ne\Cal R$ if $\Cal R$ is disconnected. Surprisingly, this may be not true for the attractors of elliptic equations. Indeed, let us consider equation
\eqref{3.ell} with {\it Neumann} boundary conditions, $\gamma=0$ and $f(u)=u(u-1)^2\cdots(u-N)^2$. Then, introducing the function $z(t):=\int_\Omega a u(t,x).u(t,x)\,dx$, we get
$$
z''(t)=2(f(u(t)),u(t))+2(a\nabla_{t,x}u,\nabla_{t,x}u)\ge0.
$$
So, the function $z(t)$ is bounded on $\R$ and convex and, therefore, is constant. Since $z''(t)\equiv0$, we conclude that $\nabla_{t,x}u\equiv0$ and $u=const$. Thus, $\Cal A_{tr}=\R=\{0,1,2,\cdots,N\}$ is totally disconnected  and we do not have non-trivial traveling waves.
\par
Fortunately, the example given above is an exception, not a rule. Indeed, as proved in \cite{FSV98}, we have $\Cal K\ne\Cal R$ under some extra  mild assumptions, namely, we need to have at least two {\it non-degenerate} (hyperbolic) equilibria to guarantee the existence of non-trivial traveling waves. Note that in the above counterexample we have exactly one non-degenerate equilibrium $u=0$ and all others are degenerate. The proof of this fact is based on the analogous result for the Galerkin system of ODEs approximating \eqref{3.ell} (where it can be established using the Conley index) and the upper semicontinuity of the corresponding trajectory attractors.
\end{remark}
\begin{remark} We note that in the case of equation \eqref{3.ell}, the corresponding bornology on the trajectory phase space is defined via lifting the bornology on the usual phase space $\Phi$, exactly as in the case of RDS considered above. Nevertheless, we do not have the concatenation property here and cannot define the generalized multivalued semigroup $S(t)$ analogously to \eqref{3.gen}. The reason is that we cannot guarantee that the solution $u(t)$ of \eqref{3.ell} initially defined on the interval $t\in[0,N]$ can be continued for $t\ge N$ (it may blow up immediately for $t\ge N$), so to define this semigroup properly, we need to use only the solutions which are defined for all $t\ge0$ (interpreting this assumption as some kind of the second boundary condition at $t=\infty$), see \cite{Ba94}. Alternatively, it can be proved that the projection
$$
\Pi:\Cal K_+\to\Phi\times\Phi, \ \ \Pi u:=\(u\big|_{t=0},\Dt u\big|_{t=0}\)
$$
is injective and therefore is a homeomorphism between $\Cal K_+$ and $\Cal G:=\Pi\Cal K_+\subset\Phi^2$. Thus, we may define the equivalent DS $\Cal S(t):\Cal G\to\Cal G$ associated with equation \eqref{3.ell} and its trajectory DS. Then we will have its global attractor $\Cal A_{gl}:=\Pi\Cal A_{tr}$. However, the structure of the set $\Cal G$ remains unclear (and for this reason such an approach does not give essential advantages in comparison with the trajectory one). In addition, the semigroup is only H\"older continuous in general (and is not Lipschitz) which, in turn, allows the attractors $\Cal A_{gl}$ and $\Cal A_{tr}$ to be infinite dimensional, see \cite{MieZ02} for the corresponding example.
\par
We also mention that the situation becomes much better in the case of the so-called {\it fast} traveling waves $\gamma\gg1$, where the uniqueness of a solution can be established for problem \eqref{3.ell}. Then the theory becomes very similar to the standard situation related e.g. with RDS satisfying the conditions of the uniqueness theorem, see \cites{CMS93,VZ99}.
\end{remark}
\begin{remark} The trajectory approach is applicable for elliptic boundary value problems in non-cylindrical domains as well. Indeed, if an unbounded domain $\Omega\subset\R^d$ is invariant with respect to shifts in some fixed direction $\vec l\in\R^d$, i.e.
$$
\Cal T_{\vec l}(h)\Omega\subset\Omega,\ \ h\ge0,\ \Cal T_{\vec l}(h)x:=x+h\vec l,
$$
then the corresponding semigroup of shifts $(T_{\vec l}(h)u)(x):=u(x+h\vec l)$ will acts on the properly defined set $\Cal K_+$ of solutions of the considered elliptic boundary value problem. Thus, we may interpret the direction $\vec l$ as a direction of "time",  construct the associated trajectory dynamical system $(T_{\vec l}(h),\Cal K_+)$ and study its attractors. If in addition, the domain $\Omega$ satisfies
$$
\cup_{h\le0}\Cal T_{\vec l}(h)\Omega=\R^d,
$$
 the trajectory attractor will be generated by all bounded solutions of the considered problem defined for all $x\in\R^d$, see \cites{Z-VIN, Z98a} for the details. There is also an interesting possibility to apply the attractors theory for the case where the domain $\Omega$ is not semi-invariant with respect to shifts in any direction (e.g. where $\Omega$ is an exterior domain). This possibility is based on the recently developed theory of attractors for semigroups with multi-dimensional time, see \cite{KZ23b} for the details.
\end{remark}

\subsection{3D Navier-Stokes system}\label{s4.NS} In this subsection, we consider various approaches to attractors of the 3D Navier-Stokes system:
\begin{equation}\label{3.NS}
\Dt u+(u,\Nx)u+\Nx p=\nu\Dx u+g,\ \ \divv u=0,\ \ u\big|_{t=0}=u_0
\end{equation}
in a bounded domain $\Omega\subset\R^3$ with smooth boundary. Here $u=(u^1,u^2,u^3)$ is an unknown velocity vector field, $p$ is an unknown pressure,
$$
(u,\Nx)u=\sum_{i=1}^3u_i\partial_{x_i}u,
$$
$\nu>0$ is a given viscosity and $g$ are given external forces.
\par
As usual, we denote by $\Cal V:=\{\varphi\in C_0^\infty(\Omega),\ \ \divv\varphi=0\}$ the space of divergence free test functions. Then the phase space  $\Phi=H$ and the space  $V$ are defined as the closure of $\Cal V$ in $[L^2(\Omega)]^3$ and $[H^1(\Omega)]^3$ respectively and the space $V^{-1}$ is a dual space to $V$ with respect to the duality generated by the standard inner product of $H$. We also assume that $u_0\in H$ and $g\in V^{-1}$, see e.g. \cites{BV92,tem} for more details.
\par
By definition, a function $u\in L^\infty(\R_+,H)\cap L^2_b(\R_+,V)$ is a weak energy solution of \eqref{3.NS} if, for every test function $\phi\in C_0^\infty(\R_+\times\Omega)$ such that $\divv \phi=0$, the following identity holds:
$$
-\int_\R(u(t),\Dt\phi(t))\,dt+\int_\R((u(t),\Nx)u(t),\phi(t))\,dt+\nu\int_\R(\Nx u(t),\Nx \phi(t))\,dt=\int_\R(g,u(t))\,dt.
$$
It is well-known that any such a solution $u\in C([0,T], H_w)$ for all $t\ge0$ and, moreover, $\Dt u\in L^{4/3}_b(\R_+,L^{4/3}(\Omega))$. In particular, the initial data $u\big|_{t=0}=u_0$ is well-posed.
\par
It is also well-known that, for any $u_0\in H$ and $g\in V^{-1}$, equation \eqref{3.NS} possesses at least one weak energy solution $u$ which is, in addition,  continuous at $t=0$ as a function with values in $H$ (with the strong topology) and satisfies the following energy inequality:
\begin{equation}\label{3.NS-en}
\frac12\frac d{dt}\|u(t)\|^2_H+\nu\|\Nx u(t)\|^2_{L^2}\le(g,u(t))
\end{equation}
which is understood in the sense of distributions. By definition, an energy  solution $u$ which satisfies these extra properties is called a Leray-Hopf (LH) solution of \eqref{3.NS}, see \cites{L34,Hopf 51}. In turn, the energy inequality \eqref{3.NS-en} is equivalent to the inequality
\begin{equation}\label{3.NS-en-int}
\|u(t)\|^2_H+2\nu\int_s^t\|\Nx u(\tau)\|^2_{L^2}\,d\tau\le \|u(s)\|^2_{H}+2\int_s^t(g,u(\tau))\,d\tau,
\end{equation}
which holds for almost all $s\in\R_+$ and all $t\ge s$. The solution $u(t)$ satisfying this inequality is strongly continuous at $t=0$ (i.e. it is a LH-solution) if and only if the inequality holds for $s=0$ (i.e. $s=0$ is not in the exceptional zero measure set), see \cites{ChVi02,RRS16,tem1} for more details. Note also that, applying the Gronwall inequality to \eqref{3.NS-en-int}, we get the dissipative energy estimate of the form:
\begin{equation}\label{3.NS-dis}
\|u(t)\|_H^2+\nu\int_0^{t}e^{-\beta(t-\tau)}\|\Nx u(\tau)\|^2_{L^2}\,d\tau \le (\|u(s)\|^2_H-C\|g\|^2_{V^{-1}})e^{-\beta(t-s)}+C\|g\|^2_{V^{-1}}
\end{equation}
which holds for almost all $s\ge0$ and all $t\ge s$ with some positive constants $C$ and $\beta$ which are independent of $t$, $s$, $g$ and $u$.
\par
We emphasize that the class of LH-solutions is a priori not invariant with respect to time shifts. Indeed, if $u$ is an LH-solution and $T(s)u$ is its time shift, we cannot guarantee that estimate \eqref{3.NS-en-int} holds for $s$ and, therefore, $T(s)u$ may be discontinuous at $t=0$. Important for us is the fact that $T(s)u$ is an LH-solution for {\it almost all} $s\in\R_+$. We also introduce the class of generalized LH-solutions as the weak solutions which satisfy \eqref{3.NS-en}, but not necessarily strongly continuous at $t=0$. Obviously, such solutions will be shift-invariant and $T(s)u$ is a LH-solution for almost every $s\in\R_+$.
\par
The above mentioned facts about the solutions of \eqref{3.NS} give a base for the attractors theory, but the construction of the corresponding trajectory attractor can be realized in several alternative ways. We start with the most general scheme suggested by Vishik and Chepyzhov, see \cites{ChVi95,ChVi02}.

\begin{example}\label{Ex3.NS1} Let us consider the set $\Cal K_+^{VC}$ of all weak energy solutions $u$ of \eqref{3.NS} which correspond to all $u_0\in H$ and satisfy the following estimate:
\begin{equation}\label{3.VC-dis}
\|u(t)\|^2_H+\nu\int_0^{t}e^{-\beta(t-s)}\|\Nx u(s)\|^2_{L^2}\,ds\le (C_u-C\|g\|^2_{V^{-1}})e^{-\beta t}+C\|g\|^2_{V^{-1}},
\end{equation}
where the positive constants $C$ and $\beta$ are the same as in \eqref{3.NS-dis} and the constant $C_u$ depends on the solution $u$. Then, on the one hand, the set $\Cal K_+^{VC}$ is not empty and contains all LH-solutions and, on the other hand, it is shift invariant: $T(h)\Cal K_+^{VC}\subset\Cal K_+^{VC}$, so the associated trajectory DS $(T(h),\Cal K_+^{VC})$ is well-defined. We endow the space $\Cal K_+^{VC}$ with the weak-star topology of the space $\Theta_{loc}:=L^\infty_{loc}(\R_+,H)\cap L^2_{loc}(\R_+,V)$ as well as with the bornology $\Bbb B$ discussed in Remark \ref{Rem3.strange}. Namely, $B\subset\Cal K_+^{VC}$ belongs to $\Bbb B$ if and only if $\sup_{u\in B} C_u:=C_B<\infty$.
\par
Then, due to assumption \eqref{3.VC-dis}, the set
\begin{equation}\label{3.absVC}
\Cal B:=\big\{u\in\Cal K_+^{VC},\ \ \|u(t\|_H^2+\nu\int^{t}_0e^{-\beta(t-s)}\|\Nx u(s)\|_{L^2}^2\,ds\le 2C\|g\|^2_{V^{-1}}\big\}
\end{equation}
 is a bounded absorbing set for the constructed trajectory DS.
Moreover, since
$$
L^\infty(0,T;H)\cap L^2(0,T;V)=[L^1(0,T;H)+L^2(0,T;V^{-1})]^*,
$$
for all $T$, the set $\Cal B$ is precompact in $\Cal K_+^{VC}$ endowed with the topology of $\Theta_{loc}^{w^*}$. Since the weak-star topology on $\Cal B\subset\Theta_{loc}$ is metrizable, to verify that $\Cal B$ is closed, it is sufficient to verify its sequential closeness. In turn, this can be done by passing to the limit in \eqref{3.NS} similarly to the proof of the existence of a weak solution, see \cite{ChVi02} for more details.
\par
Thus, we have verified that $\Cal B$ is a bounded compact absorbing set for the trajectory DS $(T(h),\Cal K_+^{VC})$ and, according to the general theory, there exists a trajectory attractor $\Cal A^{VC}_{tr}\subset \Cal K_+^{VC}$ which is generated by all complete bounded weak solutions of \eqref{3.NS}, i.e., by all weak solutions which are defined for all $t\in\R$ and satisfy the inequality
\begin{equation}
\|u(t)\|^2_H+\int_{-\infty}^{t}e^{-\beta(t-s)}\|\Nx u(s)\|^2_{L^2}\,ds\le C\|g\|^2_{V^{-1}}.
\end{equation}
We denote the set of all such solutions by $\Cal K^{VC}$.
\end{example}
\begin{remark} The constructed attractor $\Cal A^{VC}_{tr}$ attracts, in particular, all LH and generalized LH-solutions, however, it is somehow "too big" and contains a lot of non-physical solutions. In particular, it depends on the choice of the constant $C$ in estimate \eqref{3.VC-dis}. Indeed, due to \cite{BuVi19}, we know that, for any given sufficiently regular function $E(t)$, there exists a weak energy solution $u(t)$ of problem \eqref{3.NS} which satisfies the additional assumption $\|u(t)\|_H=E(t)$, $t>0$. Based on this, we see that even the set $\Cal K^{VC}$ of complete bounded trajectories depends on the choice of the constant $C$. For this reason, it is interesting to consider alternative constructions which allow us to drop out most part of non-physical solutions.
\end{remark}
\begin{example}\label{Ex3.LH} Let us consider the set $\Cal K_+^{gLH}$ of all generalized LH-solutions. As we have already mentioned, this set is semi-invariant with respect to time shifts and, therefore, the corresponding trajectory DS $(T(h),\Cal K_+^{gLH})$ is well-defined. The topology on $\Cal K_+^{gLH}$ is naturally defined by the embedding to $\Theta_{loc}^{w^*}$ as in the previous case and we only need the bornology.
\par
We say that $B\subset\Cal K_+^{gLH}$ is bounded if $B\big|_{t\in[0,1]}$ is a bounded set in $L^\infty(0,1,H)$. Then, due to estimate \eqref{3.NS-dis}, we have
\begin{equation}\label{3.NS-dis1}
\|u(t)\|_H^2+\nu\int_s^te^{-\beta(t-\tau)}\|\Nx u(\tau)\|^2_{L^2}\,d\tau \le C\|u\|^2_{L^\infty(s,s+1,H)}e^{-\beta(t-s)}+C\|g\|^2_{V^{-1}}
\end{equation}
already for all $t\ge s\ge0$ and, therefore,
\begin{equation}\label{3.abs1}
\Cal B:=\big\{u\in \Cal K_+^{gLH},\ \|u(t)\|_H^2+\nu\int_0^{t}e^{-\beta(t-\tau)}\|\Nx u(\tau)\|^2_{L^2}\,d\tau \le2C\|g\|^2_{V^{-1}}\big\}
\end{equation}
is a bounded absorbing set for the considered trajectory DS. This set is precompact due to the Banach-Alaoglu theorem and its compactness follows from the standard fact that the weak limit of generalized   LH-solutions is a generalized LH-solution.  Thus, according to the general theory, the considered trajectory DS possesses an attractor $\Cal A^{gLH}_{tr}$ which is generated by all bounded LH-solutions of \eqref{3.NS} defined for all $t\in\R$. We denote the set of such solutions by $\Cal K^{LH}$.
\end{example}
\begin{remark} The constructed attractor $\Cal A^{gLH}_{tr}$ looks as the most natural trajectory attractor for 3D Navier-Stokes equations and appears in slightly different, but equivalent forms in many papers starting from the pioneering work of Sell \cite{sell96}. For instance, one can replace the topology of $\Theta_{loc}^{w^*}$ by the strong topology of $L^2_{loc}(\R_+,H)$ avoiding the usage of weak topologies in definitions. Indeed, as it is not difficult to show, using the control of the proper norm of time derivative $\Dt u$ from equation \eqref{3.NS}, see e.g. \cite{sell96}, the topologies of $L^2_{loc}(\R_+,H)$ and $\Theta_{loc}^{w^*}$ coincide on the absorbing ball $\Cal B$. Moreover, due to the trick with an alternative generalization of LH-solutions suggested in \cite{sell96}, one can use  bounded sets of $L^2_{loc}(\R_+,H)$ to define the bornology on the trajectory phase space as well. However,  we cannot use the bornology generated by bounded in $H$ sets of initial data. As we have already mentioned, estimate \eqref{3.NS-dis} may fail at $s=0$ for  generalized LH-solutions and the boundedness of the set of $u(0)$ does not imply that the corresponding set of trajectories is bounded in $\Theta_{loc}$.
\end{remark}
\begin{example} There is an alternative way which allows us to construct a trajectory attractor using the bornology related with bounded sets in $H$, namely,  one may consider the set $\Cal K_+^{LH}$ of standard (not generalized) LH-solutions as a "trajectory phase space". The problem here is that, as discussed before,  $T(h)\Cal K_+^{LH}$ may be not a subset of $\Cal K_+^{LH}$, so we need to deal with non-invariant sets of trajectories. The extension of the attractors theory to this case has been developed in \cite{ZvK14}. We will not discuss this theory in more details since there is a simple trick which allows us to reduce it to the general scheme considered in the previous section. Namely, let us consider the {\it whole space} $\Theta_{loc}^{w^*}$ as the trajectory phase space for problem \eqref{3.NS}. Then the semigroup $T(h)$ of time shifts obviously acts continuously on $\Theta_{loc}^{w^*}$, so the trajectory DS $(T(h),\Theta_{loc}^{w^*})$ is well-defined.
\par
Of course, up to the moment this trajectory DS looks as an abstract nonsense since it is completely unrelated with the initial Navier-Stokes equation. But we still did not introduce the bornology for this DS and exactly the bornology will relate it with the initial problem. Namely, a set $B\subset \Theta_{loc}^{w^*}$ belongs to the bornology $\Bbb B$ if
\par
1) $B\subset \Cal K_+^{LH}$, i.e. $B$ consists of LH-solutions of the Navier-Stokes problem;
\par
2)  The set $B\big|_{t=0}:=\{u(0),\ u\in B\}$ is bounded in $H$.
\par
Indeed, due to the dissipative estimate \eqref{3.NS-dis} and our choice of the bornology, we see that the compact set \eqref{3.abs1} constructed above is an absorbing set for this new trajectory DS and, therefore, according to the general theory, we have an attractor  $\Cal A_{tr}^{LH}\subset \Cal B$ (and therefore $\Cal A_{tr}^{LH}\subset \Cal K_+^{gLH}$). Since $T(h)$ is continuous, this theory also guarantees that $\Cal A_{tr}^{LH}$ is strictly invariant and, therefore, is generated by some subset of complete bounded LH-solutions:
$$
\Cal A_{tr}^{LH}\subset \Cal A_{tr}^{gLH}=\Cal K^{LH}\big|_{t\ge0}.
$$
Note that the general theory does not give us the coincidence of two attractors since we cannot claim that the constructed absorbing set belongs to $\Bbb B$ (we do not know whether or not $\Cal B\subset\Cal K_+^{LH}$). But in our case it follows immediately from the fact that for every $u\in\Cal K^{LH}$, we have that $T(h)u\big|_{t\ge0}\in \Cal K_+^{LH}$ for almost every $h\in\R$. Thus,
$$
\Cal A_{tr}^{LH}=\Cal A_{tr}^{gLH}
$$
and considering the non-invariant spaces of trajectories does not bring anything new to the theory of trajectory attractors for 3D Navier-Stokes equations. We also mention that this is a lucky exception for the theory of attractors of non-invariant sets of trajectories that we are able to establish the last equality and clarify the structure of the corresponding attractor. In more general situation, we usually do not know how this attractor is related with the solutions of the considered equation and what complete bounded trajectories  generate it. This is actually the main drawback of the theory.
\end{example}
\begin{example} We now consider one more approach to trajectory attractors which has been suggested in \cite{Zel04a} (see also \cites{CVZ11,GSZ10,MZ08}) for the study of damped wave equations with supercritical nonlinearities and which is based on approximations of the original system by the system for which we have the uniqueness of solutions. We restrict ourselves to considering the Galerkin approximations of the original 3D Navier-Stokes system only. Namely, let $\{e_n\}_{n=1}^\infty$ be the orthonornal in $H$ system of eigenvectors of the classical Stokes operator $A:=-\Pi\Dx$ in $\Omega$ with Dirichlet boundary conditions and let $P_N$ be the corresponding orthoprojector to the first $N$ eigenvectors of $A$:
$$
P_Nu:=\sum_{n=1}^N(u,e_n)e_n, \ \ H_N:=P_N H.
$$
Then, for a given $N\in\Bbb N$, the corresponding Galerkin approximation system reads
\begin{equation}\label{3.gal}
\Dt u_N+P_N(u_N,\Nx)u_N)=-\nu Au_N+P_Ng,\ \ u_N\big|_{t=0}=u_N^0,\ \ u_N=\sum_{i=1}^Nu_N^i(t)e_i\in H_N.
\end{equation}
These equations are  a smooth system of ODEs with respect to functions $u^1_N(t),\cdots, u_N^N(t)$. In addition, any solution of this system satisfies exactly the same energy estimates as the limit Navier-Stokes system and, for this reason, we have the unique global solvability of \eqref{3.gal} as well as the uniform with respect to $N$ dissipativity estimate \eqref{3.NS-dis} for solutions $u_N$. Moreover, arguing in a standard way, we conclude that any sequence $u_N$ of the Galerkin solutions such that $u_N(0)$ are uniformly bounded in $H$ has a subsequence, convergent in $\Theta_{loc}^{w^*}$ to a generalized LH solution $u$ of the initial Navier-Stokes problem \eqref{3.NS} (we do not assume here that $u_N(0)$ converges strongly in $H$, so we cannot guarantee the continuity of $u(t)$ in $H$ at $t=0$). This is the standard way how the LH-solutions are usually constructed, see \cites{ChVi02,RRS16} for more details.
\par
We now define the trajectory phase space $\Cal K_+^{gal}\subset\Cal K_+^{gLH}\subset\Theta_{loc}$ as the set of all generalized LH-solutions of \eqref{3.NS} which can be obtained as weak-star limits of the Galerkin solutions:
\begin{equation}\label{3.tr-gal}
\Cal K_+^{gal}:=\big\{u\in\Theta_{loc}:\,  u= \lim_{k\to\infty}u_{N_k}\big\},
\end{equation}
where the limit is taken in the topology of $\Theta_{loc}^{w^*}$.
Since we do not assume that $u_{N_k}(0)\to u(0)$ strongly in $H$, the set $\Cal K_+^{gal}$ is invariant with respect to time shifts and, therefore, the trajectory DS $(T(h),\Cal K_+^{gal})$ is well-defined. We endow this DS with the topology of $\Theta_{loc}^{w^*}$, so, to speak about attractors, it remains to introduce the bornology on $\Cal K_+^{gal}$. This requires some accuracy since we want  the weak-star limit of  such solutions, belonging to a bounded set, to be also such a solution, i.e. it should be possible to  obtain it as a weak-star limit of Galerkin solutions. To this end, following \cite{Zel04a}, we introduce the so-called $M$-functional:
\begin{equation}
M_u(t):=\inf_{u_{N_k}\rightharpoondown u}\liminf_{k\to\infty}\|u_{N_k}(t)\|^2_{H},\ \ u\in\Cal K_+^{gal}.
\end{equation}
The external infinum is taken over all subsequences of Galerkin solutions which converge to a given $u\in\Cal K_+^{gal}$. By  definition, the $M$-functional is well-defined on $\Cal K_+^{gal}$ and the straightforward arguments show that
\begin{multline}\label{3.Mdis}
1.\ \ M_{T(s)u}(t)\le M_u(t+s),\ \ 2.\ \ \|u(t)\|^2_H\le M_u(t),\\ 3.\ \  M_u(t+s)+\nu\int_s^{t+s}\|\Nx u(\tau)\|_{L^2}^2\,d\tau\le CM_u(s)e^{-\beta t}+C\|g\|_{V^{-1}}^2,\ \ t,s\ge0.
\end{multline}
The key property of the $M$-functional is stated in the lemma below.
\begin{lemma}\label{Lem3.M} Let the sequence $u_n\in\Cal K_+^{gal}$ be such that $u_n\to u$ in $\Theta_{loc}^{w^*}$. Then $u\in\Cal K_+^{gal}$ and
$$
M_u(t)\le \liminf_{n\to\infty}M_{u_n}(t),\ \ t\ge0.
$$
\end{lemma}
The proof of this lemma is based on the diagonal procedure and the fact that the topology of $\Theta_{loc}^{w^*}$ is metrizable on bounded subsets of $\Theta_{loc}$, see \cite{Zel04a} for the details.
\par
We are now ready to define the bornology $\Bbb B$ for the trajectory DS $(T(h),\Cal K_+^{gal})$. Namely, the set $B\subset\Cal K_+^{gal}$ is an element of $\Bbb B$ if and only if $\sup_{u\in B}M_u(0)<\infty$. Then, according to Lemma \ref{Lem3.M} and dissipative estimate \eqref{3.Mdis}, the set
$$
\Cal B:=\{u\in\Cal K_+^{gal},\ M_u(0)\le 2C\|g\|^2_{V^{-1}}\}
$$
is a compact and bounded absorbing set for the trajectory DS $(T(h),\Cal K_+^{gal})$, so according to the general theory, we have the following result.
\begin{theorem} Under the above assumptions, the trajectory DS $(T(h),\Cal K_+^{gal})$ associated with the Navier-Stokes problem \eqref{3.NS} possesses an attractor $\Cal A_{tr}^{gal}\subset \Cal A_{tr}^{gLH}$ which is generated by all complete bounded trajectories of \eqref{3.NS} which can be obtained as a weak-star limit of the corresponding Galerkin approximations. Namely,
$$
\Cal A_{tr}^{gal}=\Cal K^{gal}\big|_{t\ge0},
$$
where $u\in\Cal K^{gal}$ if and only if there exists a sequence $t_k\to-\infty$, a bounded in $H$ sequence of the initial data and a sequence of the Galerkin solutions $u_{N_k}(t)$, $t\ge t_k$, such that $u$ is a weak-star limit of $u_{N_k}$, see \cite{Zel04a} for more details.
\end{theorem}
\end{example}
\begin{remark} The key advantage of the approach related with approximations and the $M$-functional is that we can use not only the energy estimates, but also other type of estimates available on the level of approximations in order to study the corresponding trajectory attractor. Indeed, this approach has been originally suggested in \cite{Zel04a} for the study of damped wave equations with super-critical nonlinearities in order to verify that any complete bounded trajectory $u(t)$, $t\in\R$, of such an equation is actually smooth for $t\le T_u$. This fact follows from the smoothness of the corresponding set of equilibria and the gradient structure of the equation, but its proof requires rather delicate estimates which can be justified on the level of Galerkin approximations only, so the analogous fact is not known for other types of trajectory attractors. As a drawback, we mention that the equality $\Cal A_{tr}^{gLH}=\Cal A_{tr}^{gal}$ is not known and the obtained attractor may a priori depend on the way of approximation (e.g. on the choice of the Galerkin base).
\end{remark}

\subsection{Connectedness of trajectory attractors} It is well-known that global attractors of evolutionary PDEs are usually connected, see \cites{BV92,tem}. This is based on a simple topological fact that an $\omega$-limit set $\omega(B)$ of a connected set $B$ is connected if the corresponding DS is continuous (may be not true if it only has a closed graph, see \cite{PZ07}). However, the situation with trajectory attractors is more delicate since it is a priori not clear whether or not the corresponding trajectory phase space is connected (and as we have seen above, it may be not connected at least for the trajectory attractors of elliptic PDEs).  Nevertheless, in many cases, the trajectory attractors related with evolutionary PDEs remain connected. We briefly discuss below this theory  based on the model example of 3D Navier-Stokes equations (utilizing the ideas from  \cites{Val95,KlV10}) although the method has a general nature and works in many other cases.
\par
We start with the simplest case of the trajectory DS $(T(h),\Cal K_+^{gal})$.
\begin{proposition}\label{Prop3.gal} Let $(T(h),\Cal K_+^{gal})$ be the trajectory DS associated with Galerkin solutions of the Navier-Stokes equations constructed above. Then the set $\Cal K_+^{gal}$, the absorbing set $\Cal B$ and the corresponding trajectory attractor $\Cal A_{tr}^{gal}$ are connected in the topology of $\Theta_{loc}^{w^*}$.
\end{proposition}
\begin{proof}[Sketch of the proof]   Indeed, let $u,\bar u\in\Cal B$ be two Galerkin solutions and let $\bar u(t)\equiv u_2$ be a stationary solution. It is well-known that $\bar u$ exists and is regular enough ($\bar u\in H^1_0(\Omega)$), so the functions $P_N\bar u\to \bar u$ as $N\to\infty$ strongly in $H^1_0(\Omega)$ and, in particular, in $\Theta_{loc}^{w^*}$. Important for us that the whole sequence $P_N\bar u$ (not only up to a subsequence) is convergent.
\par
By definition, there exists a sequence of Galerkin approximations $u_{N_k}$ such that $u_{N_k}\to u$ in $\Theta_{loc}^{w^*}$ as $k\to\infty$. Let us consider a continuous curve $\gamma_{k,s}\subset\Theta_{loc}^{w^*}$, $s\in[0,1]$. To construct this curve, we take $s u_{N_k}(0)+(1-s)P_{N_k}\bar u$ as the initial data for the Galerkin system \eqref{3.gal} and take $\gamma_{k,s}(t)$ as a unique solution of this problem.
\par
It can be shown that $\cup_{k\in\Bbb N, s\in[0,1]}\gamma_{k,s}$ is a precompact set in $\Theta_{loc}^{w^*}$ and any limit point of this set as $k\to\infty$ belongs to $\Cal B$. Since $u$ and $\bar u$ belong to the set of such limit points, they must be in the same  connected component of $\Cal B$ and, since $u$ is arbitrary, we conclude that $\Cal B$ is connected and $\Cal A_{tr}^{gal}$ is also connected as an $\omega$-limit set of a connected set.
\end{proof}
We now turn to the other types of trajectory attractors.
\begin{proposition} The trajectory attractors $\Cal A_{tr}^{gLH}$ and $\Cal A_{tr}^{VC}$ are connected in $\Theta_{loc}^{w^*}$.
\end{proposition}
\begin{proof}[Sketch of the proof] The idea of the proof is similar: to construct a family of continuous curves $\gamma_{k,s}$ connecting different points of $\Cal K_+$ in $\Theta_{loc}^{w^*}$ and pass to the limit $k\to\infty$. However, we now do not have a good way to approximate the trajectories of our equation, so we need to proceed in a more accurate way. Namely, Galerkin approximations are no more appropriate for our purposes, so we consider an alternative Leray-$\alpha$ approximations:
\begin{equation}\label{3.aNS}
\Dt u_\alpha+(v_\alpha,\Nx)u_\alpha+\Nx p=\nu\Dx u_\alpha+g,\ \divv u_\alpha=0, \ v_\alpha=(1-\alpha A)^{-1}u_\alpha,\ \ u_\alpha\big|_{t=0}=u_\alpha(0),
\end{equation}
where $\alpha>0$ is a regularization parameter. It is well-known that this problem possesses a unique solution which continuously depends on the initial data $u_\alpha(0)\in H$ for every $\alpha>0$. Moreover, the solution $u_\alpha$ satisfies all the energy estimates stated above uniformly with respect to $\alpha$ and,  when $\alpha\to0$,  we have the convergence (up to extracting a subsequence) in $\Theta_{loc}^{w^*}$ to generalized LH-solutions of the limit Navier-Stokes system \eqref{3.NS}. In addition, if $u_\alpha(0)\to u(0)$ strongly in $H$, the limit solution will be an LH-solution, see \cites{CTV07,RRS16} for more details.
\par
Let us start with the trajectory DS $(T(h),\Cal K_+^{gLH})$. It is enough to prove that the absorbing set $\Cal B$ defined by \eqref{3.abs1} is connected in $\Theta_{loc}^{w^*}$. In turn, in order to verify this fact, it is enough to check that the restriction $\Cal B_T:=\Cal B\big|_{t\in[0,T]}$ is connected in $\Theta_T^{w^*}=\Theta_{loc}^{w^*}\big|_{t\in[0,T]}$ for any fixed $T>0$.
\par
Let now $u_1,u_2\in\Cal K_+^{gLH}$. For a given $\alpha>0$, we construct a continuous curve $\gamma_{\alpha,s}\in\Theta_T^{w^*}$, $s\in[-T-1,T]$ as follows. For $s\in[0,T]$, we define
\begin{equation}\label{1}
\gamma_{\alpha,s}(t):=\begin{cases} u_1(t),\ t\le s,\\ u_{\alpha,1}(t),\ \ t\ge s,\end{cases}
\end{equation}
where $u_{\alpha,1}(t)$ is a unique solution of the Leray-$\alpha$ approximation \eqref{3.aNS} with
$u_{\alpha,1}(s)=u_1(s)$. For $s\in[-1,0]$ $\gamma_{\alpha,s}(t)$ solves \eqref{3.aNS} with the initial data $(s+1)u_1(0)-su_2(0)$. Finally, for $s\in[-T-1,-1]$,
\begin{equation}\label{2}
\gamma_{\alpha,s}(t):=\begin{cases} u_2(t),\ t\le -s+1,\\ u_{\alpha,2}(t),\ \ t\ge -s+1,\end{cases}
\end{equation}
where $u_{\alpha,2}(t)$ solves \eqref{3.aNS} with the initial data $u_{\alpha,2}(-s+1)=u_2(-s+1)$. It is not difficult to see that $\gamma_{\alpha,s}$ is indeed a continuous curve in $\Theta_T^{w^*}$ with $\gamma_{\alpha,T}=u_1$ and $\gamma_{\alpha,-T-1}=u_2$.
\par
As in the previous case, one can show that the set $\cup_{\alpha,s}\gamma_{\alpha,s}$ is precompact in $\Theta_T^{w^*}$ and that all limit points as $\alpha\to0$ are the generalized LH-solutions of the limit Navier-Stokes system and belong to $\Cal B$ (to verify this we essentially use the fact that the compound trajectories \eqref{1} and \eqref{2} satisfy the energy  equality \eqref{3.NS-en} on the whole time interval $t\in[0,T]$). Thus, $u_1$ and $u_2$ must belong to the same connected component of $\Cal B_T$ and $\Cal B_T$ is connected.
\par
We now turn to the case of trajectory DS $(T(h),\Cal K_+^{VC})$. The proof in this case is analogous: we use exactly the same continuous curves $\gamma_{\alpha,s}$  and the only difference that we need to check that the compound trajectories \eqref{1} and \eqref{2} satisfy the energy estimate involved into the definition \eqref{3.absVC} uniformly with respect to $\alpha$. The last fact is an immediate corollary of \eqref{3.NS-dis} for $u_\alpha$. Thus, the connectedness of $\Cal A_{tr}^{VC}$ is verified and the proposition is proved.
\end{proof}
\begin{remark} Note that the validity of the concatenation property is not known for the trajectory spaces $\Cal K_+^{gLH}$ and $\Cal K_+^{gal}$, nevertheless, the corresponding trajectory attractors are connected. On the other hand, as we can see from the proof, the possibility to construct a compound solutions by "gluing" together two different solutions  is crucial for this result. For instance, the concatenation property is not verified by two generalized LH-solutions, but it holds if the second one is continuous at the left endpoint and this is the key point of the proof. It contrast to this, we are unable in general to extend a solution of an elliptic PDE considered in Example \ref{Ex3.ell} which is defined for $t\in[0,T]$ for $t\in T$ by solving the appropriate boundary value problem with the "initial" condition at $t=T$ and this allows the trajectory phase space to be not connected.
\end{remark}

\section{Attractors for non-autonomous problems}\label{s4}

In this section, we discuss the attractors theory for the dynamical processes related with dissipative PDEs which depend explicitly on time. Up to the moment, there are two major approaches to extend the concept of an attractor to non-autonomous equations: the first one treats such an attractor as a time depending set $\Cal A(t)$ which naturally leads to  {\it pullback} attractors (or  kernel sections in the terminology of Chepyzhov and Vishik), see \cites{CF94,ChVi93,CLR12,KlR11} and references therein; the second one is based on the reduction of a non-autonomous problem to the autonomous one by the proper extension of the phase space and leads to  {\it uniform} attractors which remain independent of time, see \cites{ChVi95a,ChVi02,Har91}. We also mention that there are several intermediate approaches, e.g. the so-called forward attractors, etc.  which are essentially more difficult for study and less popular. In addition,  they often somehow accumulate drawbacks of both pullback and uniform attractors and, for this reason, will be not considered here, see \cites{Har91,KlY20} for the details. We start our exposition with pullback attractors.

\subsection{Pullback attractors: a general approach} The aim of this subsection is to give a natural extension of the main Theorem \ref{Th2.main} to the non-autonomous case . We first give a general definition of a dynamical process.
\begin{definition} Let $\{\Phi_\tau\}_{\tau\in\R}$ be a family of Hausdorff topological spaces. Then a family of maps $U(t,\tau):\Phi_\tau\to\Phi_t$, $\tau\in\R$, $t\ge\tau$, is a dynamical process (DP) if
\begin{equation}\label{4.dp}
U(\tau,\tau)=Id,\ \ U(t,s)=U(t,\tau)\circ U(\tau,s),\ \ t\ge\tau\ge s.
\end{equation}
\end{definition}
As it was in the autonomous case, in order to speak about attractors, we need to specify what sets are "bounded". In the non-autonomous case and under the pullback approach, it is natural to consider bounded sets which also depend on time: $t\to B(t)\in\Phi_t$, $t\in\R$. Namely, we specify the bornology $\Bbb B$ as a collection of such time dependent sets: $B=\{B(t)\}_{t\in\R}\in\Bbb B$. We put the only requirement for the elements $B\in\Bbb B$: for any $t\in\R$, $B(t)\ne\varnothing$. We also mention that the bornology $\Bbb B$   is often referred as a "universe" in the theory of pullback attractors. We however prefer to keep the denomination "bornology" to be consistent with our general theory for the autonomous case.

\begin{definition}\label{Def4.abs} Let $U(t,\tau)$ be a DP in Hausdorff topological spaces $\Phi_t$. A time dependent set $\Cal B=\{\Cal B(t)\}_{t\in\R}$ is a pullback absorbing set with respect to the bornology $\Bbb B$ if for every $B\in\Bbb B$ and every $t\in\R$, there exists $T=T(B,t)>0$ such that
\begin{equation}\label{4.abs}
U(t,t-s)B(t-s)\subset\Cal B(t),\ \ \text{ for all }\ \ s\ge T.
\end{equation}
Analogously, a time dependent set $\Cal B$ is an attracting set with respect to the bornology $\Bbb B$ if for every $B\in\Bbb B$, every $t\in\R$ and every neighbourhood $\Cal O(\Cal B(t))$ of the set $\Cal B(t)$ in the topology of $\Phi_t$, there exists $T=T(B,t,\Cal O)$ such that
\begin{equation}\label{4.attr}
U(t,t-s)B(t-s)\subset\Cal O(\Cal B(t)),\ \ \text{ for all }\ \ s\ge T.
\end{equation}
\end{definition}
\begin{remark} In may look a bit surprising, but exactly the {\it pullback} attraction property is a natural generalization of the attraction property to the non-autonomous case (at least if we treat the attractor as a time dependent set). Roughly speaking, if we fix a present time $t$ and start the evolution from a bounded set $B(t-s)$ sufficiently far in the past, its image $U(t,t-s)B(t-s)$ at present time will be arbitrarily close to the attracting set $\Cal B(t)$. In contrast to this, the forward attraction property (i.e., if you start from a bounded set $B(t)$ at present time then the image $U(t+s,t)B(t)$ will be close to $\Cal B(t+s)$ if $s$ is large enough) is much more delicate and is not convenient since it may fail for pullback attractors, see Example \ref{Ex4.bad} below. Moreover, the author is not aware about any more or less general constructions which give a "forward" attractor with reasonably good properties. Actually, exactly the lack of forward attraction is a key drawback of the theory of pullback attractors. There are two alternative ways to overcome this drawback, one of them is to consider random external forces with some ergodicity (then the forward attraction will hold in the sense of convergence in measure), the alternative is to consider non-autonomous {\it exponential} attractors where we have uniform in time attraction, see the exposition below.
\end{remark}
As in the autonomous case, the theory is based on a proper generalization of $\omega$-limit sets.
\begin{definition}\label{Def4.om} Let $U(t,\tau)$ be a DP in the Hausdorff topological spaces $\Phi_t$, $t\in\R$ endowed with some bornology $\Bbb B$. Then the pullback $\omega$-limit set of $B\in\Bbb B$ is defined as follows:
\begin{equation}\label{4.om}
\omega_B(t):=\cap_{s\ge0}\big[\cup_{\tau\ge s}U(t,t-\tau)B(t-\tau)\big]_{\Phi_t}.
\end{equation}
\end{definition}
\begin{definition}\label{Def4.pattr} A time dependent set $\Cal A=\{\Cal A(t)\}_{t\in\R}$ is a pullback attractor for a DP $U(t,\tau)$ in Hausdorff topological spaces $\Phi_t$ endowed with some bornology $\Bbb B$ if
\par
1) For any $t\in\R$, the set $\Cal A(t)$ is compact in $\Phi_t$;
\par
2) $\Cal A$ is a pullback attracting set;
\par
3) $\Cal A$ is a minimal (by inclusion) time dependent set which satisfies properties 1) and 2).
\end{definition}
The next theorem gives the analogue of Theorem \ref{Th2.main} and is the main result of this subsection. We mention also that, in the case where $\Phi_t$ are Banach spaces and the bornology $\Bbb B$ consists of uniformly bounded sets, this theorem is proved in \cite{CPT13}.
\begin{theorem}\label{Th4.main} Let a DP $U(t,\tau)$, $t\ge\tau\in\R$,  acting in Hausdorff topological spaces $\{\Phi_t\}_{t\in\R}$ endowed with some bornology $\Bbb B$, possess a compact pullback attracting set $\Cal B$. Then the DP $U(t,\tau)$ possesses a pullback attractor $\Cal A(t)$, $t\in\R$. Moreover, for all $t\in\R$, we have $\Cal A(t)\subset\Cal B(t)$.
\end{theorem}
\begin{proof}[Sketch of the proof] As in the autonomous case, the statement follows from the corresponding properties of $\omega$-limit sets. Namely, we need to verify that under the assumptions of the theorem, for every $B\in\Bbb B$, $\omega_B(t)$ is non-empty compact subset of $\Cal B(t)$  which pullback attracts the images of $B$ and that is the minimal compact set which possesses this property. Then the desired attractor is defined via the standard formula
$$
\Cal A(t):=\big[\cup_{B\in\Bbb B}\omega_B(t)\big]_{\Phi_t},\ \ t\in\R.
$$
The verification of the above properties of pullback $\omega$-limit sets repeats almost word by word the proof of Theorem \ref{Th2.main} and by this reasons is left to the reader.
\end{proof}
Analogously to the autonomous case, the pullback attractor is strictly invariant with respect to the corresponding DP if the maps $U(t,\tau)$ are continuous.
\begin{proposition}\label{Prop4.inv} Let the assumptions of Theorem \ref{Th4.main} hold and let, in addition, $U(t,\tau):\Phi_\tau\to\Phi_t$ be continuous for all fixed $t\ge\tau\in\R$. Then the pullback attractor $\Cal A(t)$ is strictly invariant:
\begin{equation}\label{4.inv}
\Cal A(t)=U(t,\tau)\Cal A(\tau).
\end{equation}
\end{proposition}
The proof of this statement also repeats word by word the proof of Proposition \ref{Prop1.inv} and thus is omitted.
\begin{remark} We see that, analogously to the autonomous case, we have the invariance of the pullback attractor if the maps $U(t,\tau)$ are continuous and this property remains true if we replace the continuity by the assumption that the graphs of all maps $U(t,\tau)$ are closed. However, in contrast to the autonomous case, we cannot in general replace the minimality assumption in Definition \ref{Def4.pattr} by the strict invariance \eqref{4.inv} even if all maps $U(t,\tau)$ are continuous. Indeed, as simplest examples show, see e.g. \cites{CLR12,MZ08} and Example \ref{Ex4.bad} below, such a modification of the definition may lead to the non-uniqueness of a pullback attractor.
\end{remark}
We now discuss the validity of the representation formula for pullback attractors. We say that the complete trajectory $u(t)$, $t\in\R$, is bounded ($\Bbb B$-bounded) if $\{u(t)\}_{t\in\R}\in\Bbb B$ and denote by $\Cal K$ the set of all bounded complete trajectories of the DP $U(t,\tau)$ (following Vishik and Chepyzhov we refer to $\Cal K$ as a $\Bbb B$-kernel of the DP $U(t,\tau)$). Then, obviously
$$
\Cal K\big|_{t=\tau}\subset \Cal A(\tau)
$$
if the pullback attractor exists. The next statement is the non-autonomous analogue of Proposition \ref{Prop2.rep}.
\begin{proposition}Let the assumptions of Theorem \ref{Th4.main} hold and let, in addition, the maps $U(t,\tau)$ be continuous for all fixed $t\ge\tau\in\R$ and the compact absorbing set $\Cal B$ be bounded ($\Cal B\in\Bbb B$). Assume also that the bornology $\Bbb B$ is stable with respect to inclusions (if $B\in\Bbb B$ and $B_1(t)\subset B(t)$ for all $t\in\R$ and $B_1(t)\ne\varnothing$, then $B_1\in\Bbb B$). Then
\begin{equation}
\Cal A(\tau)=\Cal K\big|_{t=\tau},\ \ \tau\in\R,
\end{equation}
where $\Cal K$ is the $\Bbb B$-kernel of the DP $U(t,\tau)$.
\end{proposition}
The proof of this statement is similar to the proof of Proposition \ref{Prop2.rep} and is omitted.
\begin{example}\label{Ex4.bad} Let us consider the following ODE in $\Phi_t\equiv\R$:
\begin{equation}\label{4.bad}
y'(t)=f(t,y),\ \ y\big|_{t=\tau}=y_\tau,\  \ f(t,y):=\begin{cases} -y,\ \ t<0,\\ y-y^3,\ \ t\ge0.\end{cases}
\end{equation}
Let us also fix the standard bornology $\Bbb B$ as the bornology of uniformly bounded sets, namely, $B\in\Bbb B$ if and only if $B(t)\ne\varnothing$ and $\sup_{t\in\R}\|B\|<\infty$. Then, as it is not difficult to see, the solution operators $U(t,\tau)$ of the considered ODE generate a DP in $\R$ which possesses a pullback attractor $\Cal A(t)=\{0\}$. Note that this attractor does not possess a forward in time attraction property. Indeed, for $t\ge0$, the equilibrium  $y=0$ is exponentially unstable and all of the close trajectories run away from the neighbourhood of zero (and approach the interval $[-1,1]$ which is a uniform attractor for this case). This example demonstrates the key drawback of the theory of pullback attractors which forces us to identify the {\it repelling} point $y=0$ with the "attractor".
\par
Let us fix two points $a>0$ and $b<0$ and consider the corresponding solutions $a(t)$ and $b(t)$, $t\in\R$, of equation \eqref{4.bad} satisfying $a(0)=a$ and $b(0)=b$. Then the time-dependent set
$$
\bar{\Cal A}(t):=[b(t),a(t)]
$$
 satisfies the properties 1) and 2) of Definition \ref{Def4.pattr} together with the strict invariance \eqref{4.inv}. This example shows that the minimality property of pullback attractors cannot be replaced by strict invariance without the risk to lose the uniqueness of a pullback attractor.
\end{example}

\subsection{Cocycles and random attractors} The abstract construction of a pullback attractor is somehow "too general" for practical applications, so it looks reasonable to consider its particular cases. One of the most interesting particular cases is  given by a cocycle which, in particular, bridges  the theory of pullback and random attractors.
\begin{definition}\label{Def4.co}Let $\Phi$ be a Hausdorff topological space, $\Psi$ be an arbitrary set and let $T(h):\Psi\to\Psi$, $h\in\R$, be a group of operators acting on $\Psi$. Then the family of maps $\Cal S_\xi(t):\Phi\to\Phi$, $\xi\in\Psi$, $t\ge0$, is a cocycle over the group $T(h)$ if
\begin{equation}\label{4.co}
1)\ \ \Cal S_\xi(0)=Id,\ \ 2)\ \ \Cal S_\xi(t+s)=\Cal S_{T(s)\xi}(t)\circ\Cal S_\xi(s),\ \ t,s\ge0,\ \ \xi\in\Psi.
\end{equation}
\end{definition}
As it is not difficult to check, any cocycle $\Cal S_\xi(t)$ generates a family of DP on $\Phi$ via
\begin{equation}\label{4.dpco}
U_\xi(t,\tau):=S_{T(\tau)\xi}(t-\tau),\ \ t\ge\tau\in\R,\ \ \xi\in\Psi,
\end{equation}
depending on a parameter $\xi\in\Psi$. This family satisfy the additional translation identity
\begin{equation}\label{4.trans}
U_\xi(t+s,\tau+s)=U_{T(s)\xi}(t,\tau),\ \ \xi\in\Psi, \ t\ge\tau\in\R,\ \ s\in\R.
\end{equation}
Vice versa, any family of DP $U_\xi(t,\tau):\Phi\to\Phi$ which satisfies the translation identity \eqref{4.trans} generates a cocycle $\Cal S_\xi(t):=U_\xi(t,0)$ in $\Phi$. In applications, we are usually given a non-autonomous PDE
\begin{equation}\label{4.npde}
\Dt u=A(u,\xi(t)),\ \ u\big|_{t=\tau}=u_\tau,
\end{equation}
where $A(\cdot,\cdot)$ is a nonlinear operator which we will not specify here and $\xi=\xi(t)$ accumulates all terms of the equation which depend explicitly on time, $\Psi$ is an appropriate space of time dependent functions (e.g., some shift invariant subspace of $L^2_{loc}(\R,H)$ where $H$ is some Banach space) and $T(h):\Psi\to\Psi$, $h\in\R$ is a group of time shifts: $(T(h)\xi)(t)=\xi(t+h)$. Then, if  problem \eqref{4.npde} is globally well-posed for all $u_\tau\in\Phi$ and all $\xi\in\Psi$, the corresponding solution operators $U_\xi(t,\tau):\Phi\to\Phi$ generate a translation invariant family of DP and, therefore, also generate a cocycle over $T(h):\Psi\to\Psi$ in the phase space $\Phi$, see \cites{CLR12,ChVi02,KlY20} for more details.
\par
In the case of cocycles, it is natural to define the bornology $\Bbb B$ as a collection of non-empty $\xi$-dependent sets $\{B(\xi)\}_{\xi\in\Psi}\in\Bbb B$. Then absorbing, attracting sets will be also $\xi$-dependent sets. For instance,  the set $\{\Cal B(\xi)\}_{\xi\in\Psi}$ is pullback attracting if for every $B\in\Bbb B$, every $\xi\in\Psi$, and every neighbourhood $\Cal O(\Cal B(\xi))$ of the set $\Cal B(\xi)$ in $\Phi$, there exists $\tau=\tau(\xi,B,\Cal O)$ such that
$$
\Cal S_{T(-t)\xi}(t)B(T(-t)\xi)\subset\Cal O(\Cal B(\xi)),\ \ t\ge\tau.
$$
Analogously, the set $\{\Cal A(\xi)\}_{\xi\in\Psi}$ is a pullback attractor for the cocycle $\Cal S_\xi(t)$ if
\par
1) The set $\Cal A(\xi)$ is compact for all $\xi\in\Psi$;
\par
2) The set $\{\Cal A(\xi)\}_{\xi\in\Psi}$ is a pullback attracting set for $\Cal S_\xi(t)$;
\par
3) It is the minimal (by inclusion) set which satisfies properties 1) and 2).
\par
The next theorem gives the analogue of Theorem \ref{Th4.main}.
\begin{theorem}\label{Th4.main1} Let $\Phi$ be a  Hausdorff topological space, $\Psi$ be a set and let $\Cal S_\xi(t):\Phi\to\Phi$ be a cocycle over the group $T(h):\Psi\to\Psi$, $h\in\R$. Assume that this cocycle possesses a compact pullback attracting set $\{\Cal B(\xi)\}_{\xi\in\Psi}$. Then there exists a pullback attractor $\{\Cal A(\xi)\}_{\xi\in\Psi}$ which is a compact subset of $\{\Cal B(\xi)\}_{\xi\in\Psi}$.
\end{theorem}
\begin{proof}[Sketch of the proof] This statement is a straightforward corollary of Theorem \ref{Th4.main}. Indeed, for every fixed $\xi\in\Psi$ let us consider the corresponding DP $U_\xi(t,\tau)$ in the phase space $\Phi$ endowed with the bornology $\Bbb B_\xi$ which consists of time dependent sets $B_\xi(t):=B(T(t)\xi)$ for all $B\in\Bbb B$.  Then the DP $U_\xi(t,\tau)$ satisfies all of the assumptions of Theorem \ref{Th4.main} and, therefore, there exists a pullback attractor $\Cal A_\xi(t)$, $t\in\R$, for this process. Moreover, due to the translation identity and the uniqueness of a pullback attractor, we have $\Cal A_{T(h)\xi}(t)=\Cal A_\xi(t+h)$, $t,h\in\R$ and $\xi\in\Psi$. Thus, the $\xi$-dependent set
$$
\Cal A(\xi):=\Cal A_\xi(0)
$$
is well-defined and gives the desired pullback attractor for the cocycle $\Cal S_\xi(t)$.
\end{proof}
As in the previous case, if in addition, the maps $\Cal S_\xi(t)$ are continuous for every fixed $t\ge0$ and $\xi\in\Psi$, we have the strict invariance of the pullback attractor which now reads
$$
\Cal S_\xi(t)\Cal A(\xi)=\Cal A(T(t)\xi),\ \ \xi\in\Psi,\ \ t\ge0.
$$
Moreover, if in addition the compact attracting set $\Cal B\in\Bbb B$ and the bornology $\Bbb B$ is stable by inclusions, we also have the representation formula
$$
\Cal  A(\xi)=\Cal K_\xi\big|_{t=0},
$$
where $\Cal K_\xi$ is the $\Bbb B_\xi$-kernel of the DP $U_\xi(t,\tau)$.
\par
As we have seen in Example \ref{Ex4.bad}, the forward attraction property may fail even in the case of relatively simple dependence of the symbols $\xi(\cdot)\in\Psi$ on time. The situation becomes much better if we assume some recurrence properties of time-dependent external forces. Namely, let us assume that $\Psi$ possesses an invariant (with respect to $T(h)$) probability measure $\mu$, i.e. $(\Psi,\Cal F,\mu)$, where $\Cal F$ is a $\sigma$-algebra, is a probability space and the maps $T(h)$ are measure preserving. Then it is also natural to assume that the maps $\xi\to \Cal S_\xi(t)$ are measurable (or even continuous) and for every $B\in\Bbb B$ the map $\xi\to B(\xi)$ is a measurable set-valued map (the sets satisfying this property are called {\it random} sets and the corresponding cocycle is called  a random dynamical system (RDS)). Moreover, in order to work with random sets it is usually assumed that $\Phi$ is a Polish space (i.e. it is separable, metrizable and complete). We also mention that it is natural for RDS that all properties starting from \eqref{4.co} hold not for all $\xi\in\Psi$, but for $\mu$-almost all of them.
\par
The detailed exposition of the theory of RDS is out of scope of this survey, we refer the interested reader to \cites{CF94,Cra01,CraS18,CraK15,KlY20} and references therein for more details, and restrict ourselves by only mentioning
the random analogue of Theorem \ref{Th4.main1} and giving some examples.

\begin{theorem}\label{Th4.main2} Let $\Cal S_\xi(t):\Phi\to\Phi$ be a RDS over the measure preserving group $T(h):\Psi\to\Psi$ acting on the probability space $\Psi$ and let $\Phi$ be Polish. Assume also that this RDS possesses  a compact random attracting set $\Cal B$ (i.e. $\Cal B(\xi)$ is compact for almost all $\xi\in\Psi$) with respect to some bornology $\Bbb B$ which consists of random sets. Then the RDS $\Cal S_\xi(t)$ possesses a pullback attractor which is a compact random set in $\Phi$. Moreover, if the maps $\Cal S_\xi(t)$ are continuous for almost all $\xi$, then the attractor $\Cal A(\omega)$ is strictly invariant.
\end{theorem}
The proof of this theorem is reduced to Theorem \ref{Th4.main1} under the extra assumption that the bornology $\Bbb B$ is separable, i.e. that there exists a {\it countable} bornology $\Bbb B_0$ such that for every $B\in\Bbb B$ there exists $B_0\in\Bbb B_0$ such that $B\subset B_0$, see \cite{Cra01}. In a general case, when this property fails (for instance when the point  random attractors are considered), the proof is more delicate. In particular, it is not a priori clear why the set $\xi\to\big[\cup_{B\in\Bbb B}\omega_B(\xi)\big]_{\Phi}$ which contains a non-countable union will be a random set. Moreover, as shown in \cite{CraS18}, this union can be essentially larger that the desired random attractor $\Cal A(\xi)$. In the example given there, the random pullback attractor is a one point set for almost all $\xi\in\Psi$ and the above union coincides with the whole space $\Phi$ almost surely. Nevertheless, as it is shown in \cite{CraS18}, the theorem remains true in the general setting as well (and even the measurability of sets from $\Bbb B$ can be removed). We also mention the paper \cite{Sch02}, where the examples of different type of random attractors are given, in particular, the examples where the random attractor does not attract bounded sets forward in time are given there.
\par
\begin{corollary} Under the assumptions of Theorem \ref{Th4.main2}, the random attractor $\Cal A(\xi)$ possesses the  forward attraction property in the following sense: for every $B\in\Bbb B$ and every~$\eb>0$
\begin{equation}\label{4.for}
\lim_{t\to+\infty}\mu\big\{\xi\in\Psi\,:\, \dist_{\Phi}\(\Cal S_\xi(t)B(\xi),\Cal A(T(t)\xi)\)>\eb\big\}=0,
\end{equation}
where $\dist_{\Phi}(A,B):=\sup_{y\in A}\inf_{x\in B}d(x,y)$ is a non-symmetric Hausdorff distance in $\Phi$.
\end{corollary}
\begin{proof} Indeed, due to the pullback attraction property, for almost all $\xi\in\Psi$, we have
$$
\lim_{t\to\infty} \dist_\Phi(U_{T(-t)\xi}(-t,0)B(\xi),\Cal A(\xi))=0.
$$
Therefore, by the Lebesgue dominating convergence theorem,
$$
\lim_{t\to\infty}\int_{\xi\in\Psi}\frac{ \dist_\Phi(U_{\xi}(0,-t)B(T(-t)\xi),\Cal A(\xi))}{1+ \dist_\Phi(U_{T(-t)\xi}(-t,0)B(\xi),\Cal A(\xi))}\mu(d\xi)=0.
$$
Since $T(t)$ is measure preserving, the change of the independent variable $\xi\to T(t)\eta$ together with the translation identity $U_{T(t)\eta}(0,-t)=U_\eta(t,0)$ gives
$$
\lim_{t\to\infty}\int_{\eta\in\Psi}\frac{ \dist_\Phi(U_{\eta}(t,0)B(\eta),\Cal A(T(t)\eta))}{1+ \dist_\Phi(U_{\eta}(t,0)B(\eta),\Cal A(T(t)\eta))}\mu(d\eta)=0
$$
which gives the desired convergence in measure.
\end{proof}
We illustrate the theory by few model examples, more examples can be found e.g. in \cites{Sch02,KlY20}.
\begin{example} Let us return to example \ref{Ex4.bad}. It can be written in the form of \eqref{4.npde} with
$$
A(y,\xi):=(2\xi-1)y-\xi y^3,\ \ \xi(t):=H(t)=\begin{cases} 0,\ \ t<0,\\ 1,\ \ t\ge0.\end{cases}
$$
This non-autonomous DS can be embedded in
a cocyle, e.g. by introducing the {\it hull}
$$
\Cal H(\chi):=\{0\}\cup\{1\}\cup\{T(h)H,\ h\in\R\}
$$
of the external forces $H(t)$. We endow $\Psi$ with the topology of $L^1_{loc}(\R)$ which makes it a compact metric space and $T(h)\Psi=\Psi$ acts continuously on $\Psi$. After that, we may consider the corresponding family of DP $U_\xi(t,\tau)$, $\xi\in\Psi$ which satisfies the translation identity, so the cocycle associated with problem \eqref{4.bad} is defined.
\par
The DS $(T(h),\Psi)$ carries exactly two egrodic invariant probability measures: $\mu_0=\delta(\xi)$ and $\mu_1=\delta(\xi-1)$. The supports of both of these measures correspond to the autonomous equations with attractors $\Cal A_{\mu_0}=\{0\}$ and $\Cal A_{\mu_1}=[-1,1]$ both of which are forward "random" attractors.
\end{example}
\begin{remark} We see that although the measures constructed in the previous example satisfy formally all of the assumptions of the theory, they are a posteriori useless and do not give any additional information about the considered DS. This indicates a general problem of the theory of DS and, in particular, ergodic theory. Namely, although the invariant measure can be constructed on a general compact space by the Prohorov theorem, this is not enough to get a reasonable information about the considered system and to do so we need some "good" properties of this measure (for instance, to have the so-called "physical measure"), e.g. its absolute continuity with respect to the Lebesgue measure if $\Psi$ is a subset of $\R^n$. Unfortunately, such a good measure may not exist and it is not clear how to construct it in a more or less general situation, see \cites{KH95,Y02} and references therein for more details.
\end{remark}

\begin{example}\label{Ex4.hyp} Let us consider the damped nonlinear pendulum equation with a sign-changing dissipation:
\begin{equation}\label{4.hyp}
y''+\gamma(t)y'+y|y|^p-\alpha y=0,\ \ \xi_y:=\{y,y'\}\big|_{t=\tau}=\xi_\tau,
\end{equation}
where $\gamma\in L^\infty(\R)$ is a given damping coefficient, $p>0$ and $\alpha>0$ are fixed parameters. We assume that $\gamma\in \Psi$, where $\Psi$ is a {\it compact} set in $L^1_{loc}(\R)$ which is strictly invariant with respect to time translations $T(h):\Psi\to\Psi$. Then the solution operators $U_\gamma(t,\tau): \Phi\to\Phi$, $\Phi=\R^2$ satisfy the translation identity and, therefore, generate a cocycle $\Cal S_\gamma(t)$ in $\Phi$.  Let us introduce the standard energy functional
$$
E_y(t):=\frac12y'(t)^2+\frac1{p+2}|y(t)|^{p+2}.
$$
It can be shown using some refined energy arguments that the following estimate holds
\begin{equation}\label{4.hyp-dis-est}
E_y(t)\le C E_y(s)e^{-\int_s^t\(2\frac{p+2}{p+4}\gamma(\tau)-\kappa\)d\tau}+
C_\kappa\int_s^te^{-\int_m^t\(2\frac{p+2}{p+4}\gamma(\tau)-\kappa\)d\tau}\,dm
\end{equation}
for every $t\ge s\in\R$, every $\kappa>0$ and every $\gamma\in\Psi$, see \cite{CSZ23} where this estimate is verified in much more general setting of a hyperbolic PDE.
\par
To get a RDS, we assume that there exist a Borel measure $\mu$ on the set $\Psi$ which is invariant with respect to shift operators $T(h)$, $h\in\R$, and is {\it ergodic}. We also assume that the following dissipativity assumption holds:
\begin{equation}\label{4.hyp-dis}
\kappa_0:=\int_{\gamma\in\Psi}\int_0^1\gamma(t)\,dt\,\mu(d\gamma)>0.
\end{equation}
In order to specify the appropriate bornology $\Bbb B$ on $\Phi$, we need to recall that the function $t\to f(t)$, $t\in\R$, is called {\it tempered} if, for every $\beta>0$,
$$
\lim_{t\to-\infty}e^{\beta t}|f(t)|=0,
$$
see \cites{CLR12,KlL07} for more details.
We say that a random set $\{B(\xi)\}_{\xi\in\R}$ is  tempered if the function $t\to \|B(T(-t)\xi)\|_{\Phi}$ is tempered for $\mu$-almost all $\xi\in\Phi$. The standard bornology which consists of all tempered sets $\{B(\xi)\}_{\xi\in\Psi}$ will be denoted by $\Bbb B_{temp}$. At the next step, we use the Birkhoff ergodic theorem in order to establish that there exists a shift-invariant subset $\Psi_{erg}\subset\Psi$ of full $\mu$-measure such that
\begin{equation}\label{4.Birk}
\lim_{T\to\infty}\frac1T\int_{-T}^0\gamma(t)\,dt=\kappa_0>0,\ \ \gamma\in\Psi_{erg}
\end{equation}
and the existence of this limit allows us to establish that the function
\begin{equation}\label{4.b}
b_\gamma(t):=2C_\kappa\int_{-\infty}^te^{-\int_{m}^t\(2\frac{p+2}{p+4}\gamma(\tau)-\kappa\)d\tau}\,dm
\end{equation}
is well-defined and tempered for all $\gamma\in\Psi_{erg}$ if $\kappa>0$ is small enough (the exponent $\kappa$ is fixed from now on in such a way that this condition is satisfied), see \cite{CSZ23} for the details. In turn, this fact together with \eqref{4.hyp-dis-est} allow us to verify that the tempered random set
\begin{equation}\label{4.abs-temp}
\Cal B(\gamma):=\{\xi_y(0)\in\Phi,\, \ E_{y}(0)\le b_\gamma(0)\}
\end{equation}
is a compact pullback absorbing set for the random DS $\Cal S_\gamma(t):\Phi\to\Phi$ associated with equation \eqref{4.hyp}, see \cite{CSZ23}. Thus, according to the general theory, this random dynamical system possesses a tempered random attractor $\Cal A(\gamma)\subset\Cal B(\gamma)$, $\gamma\in\Psi_{erg}$, which is strictly invariant and possesses the standard representation formula:
\begin{equation}
\Cal A(\gamma)=\Cal K_\gamma\big|_{t=0},\ \ \gamma\in\Psi_{erg},
\end{equation}
where the tempered kernel $\Cal K_\gamma$ consists of all complete tempered trajectories of equation \eqref{4.hyp}.
\par
As a natural example of a  set $\Psi$ and measure $\mu$ let us consider the set of random piece-wise constant dissipation rates
$$
\gamma_l(t):=l_n,\ \ t\in[n,n+1),\ \ n\in\Bbb Z,
$$
where $l=\{l_n\}_{n\in\Bbb Z}\in \Gamma:=\{a,-b\}^{\Bbb Z}$. In other words, the dissipation rate $\gamma_l(t)$ may take only two values $a$ and $-b$ on the interval $[n,n+1)$, $a,b>0$ which are parameterized by the Bernoulli shift scheme $\Gamma$. We assume that the values $a$ and $-b$ have probabilities $q$ and $1-q$ to appear, so we endow $\Gamma$ with the standard Bernoulli product measure $\mu_q$. It is well-known, see e.g. \cite{KH95}, that discrete Bernoulli shifts $\Cal T(n):\Gamma\to\Gamma$ preserve this measure and that $\mu_q$ is ergodic with respect to these shifts. Using the obvious fact that
\begin{equation}\label{4.com}
T(n)\gamma_l=\gamma_{\Cal T(n)},\ \ l\in\Gamma, \ n\in\Bbb Z,
\end{equation}
we may lift the DS $(\Cal T(n),\Gamma)$ to the shift-invariant set $\Psi\subset L^\infty(\R)$. Namely,
let
$$
\Psi:=\{T(h)\gamma_l,\ \ h\in[0,1], \ l\in\Gamma\}.
$$
Then $\Psi$ is obviously shift invariant and is compact in $L^1_{loc}(\R)$. The commutation \eqref{4.com} also allows us to lift the Bernoulli measure $\mu_q$ to the appropriate shift-invariant measure $\tilde\mu_q$ on $\Psi$. Finally, condition \eqref{4.hyp-dis} now reads
$$
\kappa_0=aq-b(1-q)>0
$$
and gives us the sufficient condition for the existence of a random attractor $\Cal A(\gamma)$, see \cite{CSZ23} for the details.
\end{example}
\begin{remark}\label{Rem4.typ} The above example demonstrates many typical features of the attractors theory for random dynamical systems. First, we see the crucial role of the Birkhoff ergodic theorem or its probabilistic analogues in constructing  random attractors. Second, since the existence of a Birkhoff limit \eqref{4.Birk} is guaranteed not for all $\gamma\in\Psi$, but only for almost all of them, the usage of "$\mu$-almost all" in all definitions related with random attractors also looks unavoidable. In the exceptional measure zero choices of $\gamma$, the corresponding DP is simply not dissipative and does not possess an attractor (e.g. it will be so if we take $\gamma(t)=-b<0$ for all $t\in\R$).  We also mention that, at least in the example related with Bernoulli shifts, the function $b_\gamma(t)$ is {\it not bounded} as $t\to\pm\infty$ with probability one. This explains why it is important to develop the pullback attractors theory for non-autonomous equations with unbounded in time coefficients, see also next example. The introduced above bornology $\Bbb B_{temp}$ of tempered sets looks as an appropriate and natural generalization of the bornology $\Bbb B_{bound}$ of uniformly bounded in time sets (which is widely used in the deterministic case, see \cites{BV92,tem}) to the DP related with random and stochastic dynamical systems, see \cites{KlL07,KlY20}. Note however, that in many case, the usage of the bornology which consists of deterministic (uniformly bounded) sets is enough to recover a random attractor in a proper way, see \cites{CF94,CraS18} and references therein.
\end{remark}
Our next example will be related with the simplest nonlinear {\it stochastic} ODE.
\begin{example}\label{Ex4.sode} Let us consider the following stochastic ODE:
\begin{equation}\label{4.sde}
y'+y^3-y=\eb\eta'(t),\ \ y\big|_{t=\tau}=y_\tau,
\end{equation}
where  $\eta'(t)$ is a two-sided white noise, and $\eb>0$ is a small parameter.
\par
Following the standard procedure, we realize the two-sided Wiener process on the space $\Psi:=C_0(\R)$ of continuous functions $\eta:\R\to\R$ which are equal to zero at $t=0$ endowed with the locally compact topology and the standard Wiener measure $\mu$. Then the modified group of temporal shifts $(T(h)\eta)(t):=\eta(t+h)-\eta(h)$ acts on $\Psi$ and the Wiener measure $\mu$ is invariant and ergodic, see e.g. \cite{Str93}. We present the solution $y(t)=y_\eta(t)$ in the form of $y_\eta(t)=\eb v_\eta(t)+w_\eta(t)$, where $v_\eta(t)$ is the stationary Ornstein-Uhlenbeck process:
\begin{equation}\label{3.OU}
v_\eta'+v_\eta=\eta'(t),\ \ \text{ i.e. } v_\eta(t)=\eta(t)-\int_{-\infty}^t\eta(s)e^{s-t}\,ds
\end{equation}
and the reminder solves
\begin{equation}\label{4.reduced}
w'_\eta-w_\eta+(\eb v_\eta(t)+w_\eta)^3=2\eb v_\eta(t),\ \ w_\eta\big|_{t=0}=y_0-\eb v_\eta(0).
\end{equation}
Thus, the random dynamical system $\Cal S_\eta(t):\R\to\R$ over the group of shifts $T(h):\Psi\to\Psi$ is well-defined and we may speak about its pullback random attractor $\Cal A(\eta)$. According to the general theory, to this end we only need to construct a compact random absorbing set $\Cal B(\eta)$ with respect to the bornology $\Bbb B_{temp}$ of tempered random sets on $\Phi=\R$.  In turn, this means that we need to construct such an absorbing set for equation \eqref{4.reduced} for almost all $\eta$.  This can be done in many ways, for instance, multiplying the equation by $w_\eta(t)$ and using the Young's inequality, one easily derives that
$$
\frac d{dt}(w_\eta^2)+w_\eta^2\le C(1+\eb^4v_\eta(t)^4)
$$
for some deterministic constant $C$ which is independent of $\eb$. Therefore, for all $t-s\ge0$, we have
\begin{equation}
w_\eta(t)^2\le w_\eta(s)^2e^{s-t}+C\int_{s}^te^{\kappa-t}(1+\eb^4v_\eta(\kappa)^4)\,d\kappa.
\end{equation}
It only remains to note that the function
$$
b_\eta(t):=C\int_{-\infty}^te^{\kappa-t}(1+\eb^4v_\eta(\kappa)^4)\,d\kappa
$$
is tempered for almost all $\eta$ as well as the function $v_\eta(t)$ and, consequently the random set
$$
\Cal B(\eta):=v_\eta(0)+\big\{y_0\in\R\,:\ y_0^2\le b_\eta(0)\big\}
$$
is a compact tempered pullback absorbing set  for the constructed random dynamical system $\Cal S_\eta(t)$ associated with equation \eqref{3.OU}. Thus, according to the general theory, this system possesses a tempered pullback random attractor $\Cal A_\eb(\eta)$ which is generated by all complete tempered trajectories of \eqref{3.OU}, see \cite{CF98a} for the details.
\par
The constructed random DS possesses some remarkable properties. Namely, solving the associated Kolmogorov-Fokker-Plank equation, we see that
\begin{equation}\label{4.measure}
\nu(dy)=F_\eb(y)\,dy,\ \ F_\eb(y):=\frac{e^{\eb^{-1}\(-\frac{y^4}4+\frac{y^2}2\)}}
{\int_{\R}e^{\eb^{-1}\(-\frac{y^4}4+\frac{y^2}2\)}\,dy}
\end{equation}
is an invariant (stationary) measure for the considered stochastic process. Moreover, this measure is uniquely ergodic and mixing, see \cites{CF98a,Kha12} for the details. Due to the ergodicity, the corresponding Lyapunov exponent is given by
\begin{equation}\label{4.un-exp}
\lambda(\eb):=\int_{\R}(1-3y^2)F_\eb(y)\,dy=
-\frac12-\frac32\,\,\frac{I_{-3/4}(1/8\eb)+I_{3/4}(1/8\eb)}{I_{-1/4}(1/8\eb)+I_{1/4}(1/8\eb)}\le-\frac12<0,
\end{equation}
where $I_\kappa(z)$ is the modified Bessel function of the first kind and order $\kappa$.
\par
Since the Lyapunov coefficient is negative for all $\eb>0$, the attractor $\Cal A_\eb(\eta)$ consists of a single point $\Cal A_\eb(\eta)=\{u_{\eb,\eta}(0)\}$ and the corresponding tempered kernel $\Cal K_{\eb,\eta}$ consists of a single exponentially stable trajectory $u_{\eb,\eta}(t)$ for almost all $\eta>0$. Thus, adding an arbitrary small additive noise destroys the pitchfork instability and makes the random system exponentially stable. In particular, the limit deterministic attractor $\Cal A_0=[-1,1]$ is not robust with respect to random perturbations, see \cite{CF98a} for more details and related discussion.
\end{example}
\begin{remark} The stabilization effect of adding a small additive noise appears in much more general situations including stochastic reaction-diffusion equations, etc. The key ingredients here are: 1) the uniqueness and ergodicity of the invariant measure which is known in more or less general situation if the noise is not too degenerate, see \cites{Kha12,KuSh12} and references therein; 2) the order preserving structure which allows us to construct the maximal and minimal solutions belonging to the attractor. If these two solutions coincide  almost surely, the attractor consists of a single point as in the above example and if not, every of them carries an invariant measure which contradicts the uniqueness, see \cite{Chu02} for the details. Note that this mechanism does not work if the limit deterministic dynamics is really chaotic and have positive Lyapunov exponents, for instance, in the case of Lorenz system \eqref{1.lorenz} perturbed by the small additive noise, so we expect non-trivial attractors and reach dynamics in the stochastic case as well. Unfortunately, despite the solid numerical evidence, we failed to find a rigorous mathematical proof of this fact.
\end{remark}

\subsection{Uniform attractors}\label{ss5.3} We now turn to the alternative approach to attractors for non-autonomous equations which is based on their reduction to the autonomous ones. We start with a cocycle $\Cal S_\xi(t):\Phi\to\Phi$, over $T(h):\Psi\to\Psi$, where $\Phi$ and $\Psi$ are Hausdorff topological spaces and the corresponding family $U_\xi(t,\tau)$, $t\ge\tau$, $\xi\in\Psi$, of DPs.  Let us consider the {\it extended} phase space $\Bbb P:=\Phi\times\Psi$ and the associated extended DS $\Bbb S(t)$, $t\ge0$ on it defined by
\begin{equation}\label{4.eDS}
\Bbb S(t)(u_0,\xi):=(\Cal S_\xi(t)u_0,T(t)\xi),\ \ u_0\in\Phi,\ \ \xi\in\Psi.
\end{equation}
Indeed, it follows from the cocycle property that $\Bbb S(t)$ is a semigroup. Therefore, we end up with the autonomous DS acting on the extended phase space $\Bbb P$ and may speak about its attractors. To this end, we need to fix a bornology $\Bbb B$ on the space $\Phi$ (in contrast to the previous cases, $B\in \Bbb B$ are time-independent sets $B\subset\Phi$) and to define the extended bornology $\Bbb B_{ext}$ on $\Bbb P$ via $\frak B\subset \Bbb B_{ext}$ if $\Pi_1\frak B\in\Bbb B$ where $\Pi_1:\Bbb P\to\Phi$ is a projection to the first component of the Cartesian product.
\begin{definition} A set $\Cal B\subset\Phi$ is a uniformly attracting set for the cocycle $\Cal S_\xi(t)$, $\xi\in\Psi$, with respect to the bornology $\Bbb B$ if for any $B\in\Bbb B$ and any neighbourhood $\Cal O(\Cal B)$ of $\Cal B$ in $\Phi$, there exists $T=T(\Cal O,B)$ such that
$$
\Cal S_\xi(t)B\subset\Cal O(\Cal B),\ \ \forall\xi\in\Psi,\ \text{ if }\ t\ge T.
$$
It is not difficult to see that if the set $\Cal B$ is a  uniformly attracting set for the cocycle $\Cal S_\xi(t)$ then the set $\Cal B_{ext}:=\Cal B\times\Psi$ is an attracting set for the extended DS $\Bbb S(t):\Bbb P\to\Bbb P$ with respect to the bornology $\Bbb B_{ext}$. Vice versa, if $\frak B$ is an attracting set for $\Bbb S(t)$ then $\Cal B:=\Pi_1\frak B$ is a uniformly attracting set for the cocycle $\Cal S_\xi(t)$, $\xi\in\Psi$.
\end{definition}
\begin{definition} Let $\Phi$ be a Hausdorff topological space and let $\Cal S_\xi(t):\Phi\to\Phi$
be a cocycle over the DS $T(h):\Psi\to\Psi$, $h\in\R$. Then the set $\Cal A_{un}\subset \Phi$ is a uniform attractor for this cocycle with respect to some bornology $\Bbb B$ if
\par
1) $\Cal A_{un}$ is a compact set in $\Phi$;
\par
2) $\Cal A_{un}$ is a uniformly attracting set for the cocycle $\Cal S_\xi(t)$;
\par
3) $\Cal A_{un}$ is a minimal (by inclusion) set which satisfies properties 1) and 2).
\end{definition}
The analogue of the attractor's existence theorem for this case reads.
\begin{theorem}\label{Th4.main4} Let $\Phi$ and $\Psi$ be Hausdorff topological spaces and let $\Psi$ be a compact space. Let, in addition, $\Cal S_\xi(t):\Phi\to\Phi$ be a cocycle over the group $T(h):\Psi\to\Psi$, $h\in\R$, and let this cocycle possess a compact uniformly attracting set $\Cal B$ with respect to some bornology $\Bbb B$ on $\Phi$. Then the extended semigroup $\Bbb S(t):\Bbb P\to\Bbb P$ possesses an attractor $\Bbb A\subset\Cal B\times\Psi$. Moreover, its projection $\Cal A_{un}:=\Pi_1\Bbb A$ is a uniform attractor for the cocycle $\Cal S_\xi(t)$.
\par
Assume, in addition, that the map $(u_0,\xi)\to(\Cal S_\xi(t)u_0,T(t)\xi)$ is continuous for every fixed $t$. Then the attractor $\Bbb A$ is strictly invariant with respect to $\Bbb S(t)$. If also $\Cal B\in \Bbb B$ and $\Bbb B$ is stable with respect to inclusions, then $\Cal A_{un}$ is generated by all complete $\Bbb B$-bounded trajectories of $\Cal S_\xi(t)$, i.e.
\begin{equation}\label{4.rep-un}
\Cal A_{un}=\cup_{\xi\in\Psi}\Cal K_\xi\big|_{t=0},
\end{equation}
where $\Cal K_\xi$ is a $\Bbb B$-kernel of the DP $U_\xi(t,\tau)$. In particular, since the pullback attractor $\Cal A_{\xi,pb}(\tau)$ of the DP $U_\xi(t,\tau)$ is equal to $\Cal K_\xi\big|_{t=\tau}$, we have the relation
$$
\Cal A_{un}=\cup_{\xi\in\Psi}\Cal A_{\xi,pb}(0).
$$
\end{theorem}
\begin{proof}[Sketch of the proof] All of the statements of the theorem are straightforward corollaries of the key Theorem \ref{Th2.main} as well as Propositions \ref{Prop1.inv} and \ref{Prop2.rep}. Indeed, if $\Cal B$ is a uniformly attracting set for the cocycle $\Cal S_\xi(t)$ then $\frak B:=\Cal B\times\Psi$ is a compact attracting set for the extended semigroup $\Bbb S(t):\Bbb P\to\Bbb P$. Thus, the existence of the attractor $\Bbb A$ is verified. The fact that $\Cal A_{un}=\Pi_1\Bbb A$ is the desired unform attractor is also straightforward. Indeed, compactness and the attraction property are obvious and we only need to check the minimality. Let $\Cal B_1$ be another uniformly attracting set, then by minimality of $\Bbb A$, we have $\Bbb A\subset\Cal B_1\times\Psi$ and $\Cal A_{un}\subset \Cal B_1$. The rest statements are immediate corollaries of Propositions \ref{Prop1.inv} and \ref{Prop2.rep}.
\end{proof}
Note that the definition of a uniform attractor given above does not require any topology on the space $\Psi$. Moreover, its existence can be obtained based on the existence of a uniformly attracting set, namely, as in the autonomous case, we may define the uniform $\omega$-limit set
\begin{equation}
\omega_{un}(B):=\cap_{T\ge0}[\cup_{\xi\in\Psi}\cup_{t\ge T}\Cal S_\xi(t)B]_{\Phi}
\end{equation}
and construct the uniform attractor via $\Cal A_{un}:=[\cup_{B\in\Bbb B}\omega_{un}(B)]_{\Phi}$, see \cite{ChVi02} for the details. Then the compactness of $\Psi$ (as well as the topology on it) is not necessary to get a uniform attractor. In particular, we even may construct a uniform attractor for a single DP $U(t,\tau):\Phi\to\Phi$ by introducing  the group of shifts $T(h):\Psi\to\Psi$ acting on the space $\Psi=\R$  via $T(h)\xi:=\xi+h$ and defining the corresponding cocycle $\Cal S_\xi(t):=U(t+\xi,\xi)$. This corresponds to the trivial reduction of a non-autonomous equation
$
y'=f(y,t)
$
to the autonomous system
$$
\begin{cases} \dot y=f(t,y),\\ \dot t=1.\end{cases}
$$
The drawback of this approach is that it does not give any information on the structure of the uniform attractor (e.g. the representation formula \eqref{4.rep-un} fails without the compactness of $\Psi$) and it is not clear how to relate the attractor $\Cal A_{un}$ with the solutions of the considered PDE.  Thus, if we want to have the representation formula \eqref{4.rep-un}, we need to consider not only all time shifts of our initial equation, but also their limits in the appropriate topology. The next proposition shows that, under natural assumptions, taking this closure does not affect the size of the attractor.
\begin{proposition}\label{Prop4.closure} Let $\Phi$ and $\Psi$ be two Hausdorff topological spaces and let $\Cal S_\xi(t):\Phi\to\Phi$ be a cocycle over DS $T(h):\Psi\to\Psi$. Assume also that $\Psi_0\subset\Psi$ be a dense and invariant with respect to $T(h)$, $h\in\R$, set and that the map $\xi\to \Cal S_\xi(t)u_0$ is continuous for every fixed $t$ and $u_0\in\Phi$.
Then the uniform attractors $\Cal A_{\Psi}$ and $\Cal A_{\Psi_0}$ of $\Cal S_\xi(t)$ considered as the cocycles over $T(h):\Psi\to\Psi$ and $T(h):\Psi_0\to\Psi_0$ exist or do not exist simultaneously and coincide:
\begin{equation}
\Cal A_{\Psi}=\Cal A_{\Psi_0}.
\end{equation}
\end{proposition}
\begin{proof}[Sketch of the proof] Indeed, it is not difficult to show using the continuity that a set $\Cal B$ is a compact uniformly attracting set for the cocycle $\Cal S_\xi(t)$, $\xi\in\Psi$ if and only if it is a compact uniformly attracting set for $\Cal S_\xi(t)$, $\xi\in\Psi_0$, and this gives the desired result, see e.g. \cite{ChVi02}.
\end{proof}
We see that in order to get the key representation formula \eqref{4.rep-un}, we need to take a closure of time shifts of the considered DP  in a topology which, on the one hand, makes the closure a compact Hausdorff topological space and, on the other hand, preserves the continuity. This is a non-trivial task which does not always have a positive solution  (in this case, the representation formula will be lost), see e.g.,  \cite{SZ20}. We restrict ourselves to consider here only an important particular case (well adapted to the study of equations with additive non-autonomous external forces), where this problem possesses a more or less complete solution. Namely, we assume that $\Phi$ is a separable reflexive Banach space and the elements of $\Psi$ which represent the non-autonomous external forces are functions $\xi:\R\to H$ with values in the other separable reflexive Banach space $H$. Moreover, assume that we start from a given external force $\xi_0$ which satisfies
\begin{equation}\label{4.tr-b}
\xi_0\in L^p_b(\R,H)
\end{equation}
for some $1<p<\infty$. Then, due to the Banach-Alaoglu theorem, any bounded set in $L^p_{loc}(\R,H)$ is pre-compact in the weak topology. Therefore, the hull
\begin{equation}\label{5.w-hull}
\Cal H(\xi_0):=\left[T(h)\xi_0,\ h\in\R\right]_{L_{loc}^{p,w}(\R,H)}
\end{equation}
is a {\it compact} subset of $L^{p,w}_{loc}(\R,H)$, so we may naturally take $\Psi=\Cal H(\xi_0)$. The scheme works as follows: we start with equation \eqref{4.npde} with a given non-autonomous external force $\xi_0\in L^p_b(\R,H)$ and consider the whole family of similar problems generated by its shifts in time and the proper limits. Namely, we consider the family
\begin{equation}\label{4.npde-h}
\Dt u=A(u,\xi(t)),\ \ \xi\in\Psi:=\Cal H(\xi_0),\ \ u\big|_{t=0}=u_0\in\Phi.
\end{equation}
If these equations are uniquely solvable in the proper sense, they define a
cocycle $\Cal S_\xi(t):\Phi\to\Phi$ over the group $T(h):\Psi\to\Psi$ of time shifts. We fix weak topologies on both spaces $\Phi$ and $\Psi$ and let the bornology $\Bbb B$ on $\Phi$ consists of all bounded subsets of the Banach space $\Phi$. Since $\Phi$ is reflexive, any $B\in\Bbb B$ is precompact in a weak topology, so if we find a uniformly absorbing/attracting  set $\Cal B\in\Bbb B$ for the considered cocycle $\Cal S_\xi(t)$, the closed convex hull of it will be a bounded and compact uniformly absorbing set. This, together with Theorem \ref{Th4.main4} gives us the following result,
\begin{corollary}\label{Cor4.weak} Let the above assumptions hold and let the cocycle $\Cal S_\xi(t):\Phi\to\Phi$ over $T(h):\Psi\to\Psi$ with $\Psi=\Cal H(\xi_0)$ possesses a bounded uniformly absorbing set. Assume also that the map $(u_0,\xi)\to\Cal S_\xi(t)u_0$ is continuous for every fixed $t\ge0$. Then this cocycle possesses a uniform attractor $\Cal A_{un}^w$ in a weak topology of $\Phi$ and this attractors can be described as follows:
\begin{equation}\label{4.w-rep}
\Cal A_{un}^w=\cup_{\xi\in\Cal H(\xi_0)}\Cal K_\xi\big|_{t=0},
\end{equation}
where $\Cal K_\xi$ is a bounded kernel of the DP $U_\xi(t,\tau)$.
\end{corollary}
Note that, as usual, the continuity assumption can be replaced by the closed graph assumption.
\par
We now turn to the case of strong attractors, so we want to fix a {\it strong} topology on the Banach space $\Phi$. Then, according to the general theory, the existence of a uniform attractor will be guaranteed if we find a compact (in a strong topology) uniformly attracting set for the cocycle $\Cal S_\xi(t)$ associated with \eqref{4.npde-h}. As elementary examples show, assumption \eqref{4.tr-b} is usually {\it not enough} to get this compactness, so some extra assumptions are needed, see \cite{Z15}. The most natural would be to assume that the hull $\Cal H(\xi_0)$ is compact not only in a weak topology of $L^p_{loc}(\R,H)$, but also in the {\it strong} topology. The functions $\xi_0\in L^p_b(\R,H)$ which satisfy this extra assumption are called {\it translation-compact} ($\xi_0\in L^p_{tr-c}(\R,H)$).  Then $\Psi$ is compact in the strong topology as well and we can fix the strong topology in both spaces $\Phi$ and $\Psi$ and get a complete analogue of Corollary \ref{Cor4.weak} for the  case of strong topology as  well. This scheme, which works for many important dissipative PDEs, is studied in details in \cite{ChVi02}, see also \cites{ChVi95,MZ08}, so we will not give more details here. Instead, we discuss a bit more delicate case where we take strong topology on $\Phi$, but try to keep
weak topology on the hull $\Psi=\Cal H(\xi_0)$. The possibility to get the strong uniform attractors for non-translation-compact external forces has been indicated in \cite{LWZ05}, see also \cites{Zel04b,Lu06,Lu07,MZS09,MCL08,MaZ07,Z15}. Then we still have the analogue of Corollary \ref{Cor4.weak}, since we take strong and weak topologies on the first and second components of the Cartesian product $\Bbb P:=\Phi\times\Cal H(\xi_0)$, so Theorem \ref{Th4.main4} is still applicable if we have a compact uniformly attracting set for the associated cocycle. Moreover, it is not difficult to show that in this case the uniform attractor in the strong topology coincides with already constructed $\Cal A_{un}^w$:
$$
\Cal A_{un}^s=\Cal A_{un}^w
$$
and we may use \eqref{4.w-rep} to describe the structure of the strong attractor $\Cal A_{un}^s$. Note that we need not to verify the continuity of maps $(\xi,u_0)\to\Cal S_\xi(t)u_0)$ with respect to the weak topology on $\Psi$ and the strong one on $\Phi$ (which is usually not true) and may check continuity (or the closeness of the graph) only in weak topologies, see \cite{Z15} for more details.
\par
To continue, we  need to introduce some classes of the external forces.
\begin{definition} Let $\xi_0\in L^p_b(\R,H)$, where $H$ is a reflexive Banach space and $1<p<\infty$. We say that $\xi_0$ is time regular if there exists a sequence $\xi_n\in C^1_b(\R,H)$ such that
\begin{equation}\label{4.lim}
\lim_{n\to\infty}\|\xi_n-\xi_0\|_{L^2_b(\R,H)}=0.
\end{equation}
We denote the class of such functions by $L^2_{t-reg}(\R,H)$.
\par
Analogously, we say that $\xi_0$ is space regular, if there exists a sequence of {\it finite-dimensional} Banach spaces $H_n\subset H$ and a sequence of functions $\xi_n\in L^2_b(\R,H_n)$ such that \eqref{4.lim} holds. The class of such functions is denoted by $L^p_{sp-reg}(\R,H)$.
\par
The function $\xi_0\in L^p_b(\R,H)$ is called normal if
$$
\lim_{\tau\to0}\sup_{t\in\R}\int_t^{t+\tau}\|\xi_0(s)\|^p_{H}\,ds=0.
$$
The class of such functions is denoted by $L^p_{norm}(\R,H)$.
\end{definition}
\begin{remark} It is shown in \cite{Z15} that
$$
L^p_{tr-c}(\R,H)=L^p_{t-reg}(\R,H)\cap L^p_{sp-reg}(\R,H),
$$
and both spaces in the right-hand side are strictly larger than the space of translation-compact external forces. The machinery for obtaining  uniform attractors in a strong topology for various PDEs (including damped wave equations, reaction-diffusion ones, etc.) with non-autonomous external forces from these classes is also presented there.
\par
Note that the assumption $1<p<\infty$ is crucial for the theory since it guarantees the reflexivity of the space $L^p_{loc}(\R,H)$ and the possibility to use the Banach-Alaoglu theorem. However, the case $p=1$ is also interesting from the point of view of applications especially for wave equations where it is naturally related with Strichartz estimates, see \cites{BSS09,BLP08,KSZ16,YKSZ23}. The situation here is much more delicate since starting from a regular function $\xi_0\in L^1_b(\R,H)$, we may easily get an $H$-valued {\it measure} when taking the closure in the proper topology. In turn, it naturally leads to DP with non-continuous in time trajectories, see \cite{SZ20}. The theory developed there is based on presenting a function $\xi_0\in L^1_b(\R,H)$ as a regular Borel measure $\xi_0\in M_b(\R,H)$. Then using the fact that $M(0,1;H)=[C(0,1;H)]^*$, we may endow the space $M_b(\R,H)$ with the local $w^*$-topology and consider the hull $\Cal H(\xi_0)$ in this topology. This is the way how to restore the compactness of the hull, but as the price to pay, we may lose the continuity of the map $\xi\to\Cal S_\xi(t)$, see \cite{SZ20} for more details.
\end{remark}
We illustrate the theory by the example of 2D Navier-Stokes system with non-autonomous normal external forces.
\begin{example}\label{Ex4.NS-norm} Let us consider the following system
\begin{equation}\label{4.NS}
\Dt u+(u,\Nx)u+\Nx p=\nu\Dx u+\xi_0(t),\ \ \divv u=0, \ u\big|_{t=0}=u_0.
\end{equation}
in a bounded 2D domain endowed with Dirichlet boundary conditions. We will use the notations of subsection \ref{s4.NS} adapted to the 2D case. The initial data $u_0$ is taken from the phase space $\Phi$ which is the closure of the space $\Cal V$ of divergence free test functions in the $L^2$-norm and the spaces $V$ and $V^{-1}$ are defined analogously and we assume that $\xi_0\in L^2_b(\R,V^{-1})$. Also we define a weak energy solution of \eqref{4.NS} exactly as in subsection \ref{s4.NS}.
\par
It is also well-known that, in contrast to the 3D case, in the 2D case, a weak energy solution is unique and satisfies the energy {\it identity}:
\begin{equation}\label{4.NS-id}
\frac12\frac d{dt}\|u(t)\|^2_\Phi+\nu\|\Nx u(t)\|^2_{L^2}=(\xi_0(t),u(t)),
\end{equation}
see \cites{BV92,tem} for the details. This identity gives us the dissipative estimate
\begin{equation}\label{4.NS-dis}
\|u(t)\|^2_\Phi+\nu\int_0^te^{-\beta(t-s)}\|\Nx u(s)\|^2\,ds\le \|u(0)\|_{\Phi}^2e^{-\beta t}+C\|\xi_0\|^2_{L^2_b(\R,V^{-1})},
\end{equation}
for some positive constants $\beta$ and $C$.
We now consider the hull $\Psi=\Cal H(\xi_0)$ of the external force $\xi_0$ in $L^{2,w}_{loc}(\R,H)$ and for every $\xi\in\Psi$, we define a map $\Cal S_\xi(t):\Phi\to\Phi$ as a solution operator at time moment $t$ for problem \eqref{4.NS} where $\xi_0$ is replaced by $\xi$. Obviously, $\Cal S_\xi(t)$ is a cocycle over $T(h):\Psi\to\Psi$ and the symbol space $\Psi$ is compact if we fix a weak topology on it. We also fix a standard bornology $\Bbb B$ which consists of all bounded subsets of a Banach space $\Phi$. Moreover, it is not difficult to see that estimate \eqref{4.NS-dis} is uniform with respect to $\xi\in\Cal H(\xi_0)$, see e.g. \cite{ChVi02}, so by Banach-Alaoglu theorem, the closed ball $\Cal B_R:=\{u_0\in\Phi\,,\ \|u_0\|_\Phi\le R\}\in\Bbb B$ will be a compact uniformly absorbing set for this cocycle if $R$ is large enough. The continuity of the map $(\xi,u_0)\to \Cal S_\xi(t)u_0$ in the chosen weak topologies on $\Phi$ and $\Psi$ is also straightforward and, due to Corollary \ref{Cor4.weak}, this cocycle possesses a uniform attractor $\Cal A_{un}^w$ which enjoys the representation formula \eqref{4.w-rep}.
\par
We now turn to the case of {\it strong} topology in $\Phi$. Since the existence of a uniform attractor $\Cal A_{un}^w$ in a weak topology of $\Phi$ together with the representation formula is already established, we only need to find a compact (in the strong topology of $\Phi$) uniformly attracting (or even absorbing) set for this cocycle. We assume, in addition, that $\xi_0\in L^2_{norm}(\R,V^{-1})$ and claim that the following set:
\begin{equation}
\Cal B:=\{S_\xi(1)\Cal B_R,\ \ \xi\in\Psi\}
\end{equation}
is such a set. Indeed, from estimate \eqref{4.NS-dis} we see that $\Cal B$ is a bounded uniformly absorbing set, so we only need to verify the compactness. To this end, we use the energy method. Namely, we consider an arbitrary sequences $\xi_n\in\Psi$ and $u_0^n\in\Cal B_R$ and the sequence of the corresponding solutions $u_n(t):=\Cal S_{\xi_n}(t)u_n$. Without loss of generality, we may assume that $\xi_n\to\xi$ and $u_0^n\to u_0$ in a weak topology and, due to the weak continuity of $\Cal S_\xi(t)$, we conclude that $u_n(1)\rightharpoondown u(1)$ where $u(t)=\Cal S_\xi(t)u_0$. Thus, we only need to prove that $u_n(1)\to u(1)$ in the strong topology of~$\Phi$.
In turn, this will be proved if we check that $\|u_n(1)\|_\Phi^2\to\|u(1)\|^2_\Phi$. To get this convergence we rewrite  the energy identity \eqref{4.NS-id} in the form of
$$
\frac d {dt}(t\|u_n(t)\|^2_\Phi)+N(t\|u_n(t)\|^2_{\Phi})+2\nu t\|\Nx u_n(t)\|^2_{L^2}=(Nt+1)\|u_n(t)\|^2_\Phi+2t(u_n(t),\xi_n(t)),
$$
where $N$ is an arbitrary positive number, multiply it by $e^{Nt}$ and integrate over $t\in[0,1]$ to get the integral identity
\begin{multline}\label{4.en-NS}
\|u_n(1)\|^2_\Phi+2\nu\int_0^1e^{-N(1-t)}t\|\Nx u_n(s)\|^2_{L^2}\,ds=\\ \int_0^1e^{-N(1-t)}(Nt+1)\|u_n(t)\|^2_\Phi\,ds+
2\int_0^1e^{-N(1-t)}(Nt+1)(u_n(t),\xi_n(t))\,ds.
\end{multline}
We want to pass to the limit $n\to\infty$ in \eqref{4.en-NS}. Using the compactness lemma and arguing in a standard way, see e.g. \cite{Z15}, we conclude that $u_n\to u$ strongly in $L^2(0,1,\Phi)$, so the passage to the limit in the first term in the RHS is straightforward. To pass to the limit in the second term in the LHS we use the weak lower semicontinuity of convex functions, so it only remains to estimate the last term in the RHS.  To this end, we use the key property of normal functions, namely, that
$$
\lim_{N\to\infty}\sup_{\xi\in\Cal H(\xi_0)}\sup_{t\in\R}\int_0^t e^{-N(t-s)}\|\xi(s)\|_{V^{-1}}^2\,ds=0,
$$
see e.g. \cite{Z15}. Using this fact together with the uniform boundedness of $u_n$ in $L^2(0,1;V)$, we see that, for every $\eb>0$ there exists $N=N(\eb)$ such that
$$
\big|\int_0^1e^{-N(1-t)}(u_n(t),\xi_n(t))\,dt\big|\le \eb
$$
and the same is true for the limit functions $u$ and $\xi$. Passing to the limit $n\to\infty$ in \eqref{4.en-NS}, we now get
\begin{multline}\label{4.en-NS1}
\limsup_{n\to\infty}\|u_n(1)\|^2_\Phi+2\nu\int_0^1e^{-N(1-t)}t\|\Nx u(s)\|^2_{L^2}\,ds\le \\ \le \int_0^1e^{-N(1-t)}(Nt+1)\|u(t)\|^2_\Phi\,ds+
2\eb
\end{multline}
and the comparison with the analogue of \eqref{4.en-NS} for the limit functions $u$ and $\xi$ gives
\begin{equation}\label{4.en-NS2}
\|u(1)\|^2_{\Phi}\le\liminf_{n\to\infty}\|u_n(1)\|^2_\Phi\le \limsup_{n\to\infty}\|u_n(1)\|^2_\Phi\le \|u(1)\|^2_\Phi+4\eb,
\end{equation}
where the first inequality is again a weak lower semicontinuity for a convex function. Finally, passing to the limit $\eb\to0$, we arrive to
$$
\lim_{n\to\infty}\|u_n(1)\|^2_\Phi=\|u(1)\|^2_\Phi
$$
which finishes the proof that $\Cal B$ is compact. A general theory now gives the existence of the uniform attractor $\Cal A^s_{un}$ in the strong topology of $\Phi$ and its coincidence with $\Cal A^w_{un}$.
\end{example}
\begin{remark} It is straightforward to see that, $L^2_{t-reg}(\R,H)\subset L^2_{norm}(\R,H)$, so the obtained result immediately gives the existence of a strong uniform attractor for the case of time-regular external forces, but gives nothing for the space-regular ones. The concept of a normal function can be weakened (following  \cite{Lu06}, see also \cite{Z15}), namely, the function $\xi_0\in L^p_b(\R,H)$ is {\it weakly normal} ($\xi_0\in L^p_{w-norm}(\R,H)$) if for every $\eb>0$ there exists a finite-dimensional subspace $H_\eb\subset H$ and a function $\xi_\eb\in L^p_b(\R,H_\eb)$ such that
$$
\limsup_{h\to0}\sup_{t\in\R}\int_t^{t+h}\|\xi_0(s)-\xi_\eb(s)\|^p_H\,ds\le \eb.
$$
Then, on the one hand
$$
L^p_{t-reg}(\R,H)+L^p_{s-reg}(\R,H)\subset L^p_{w-norm}(\R,H),
$$
so the class of weakly regular functions includes both space and time regular functions. On the other hand, the method presented in the previous example can be easily extended to the case of weakly normal external forces $\xi_0\in L^2_{w-norm}(\R,V^{-1})$, see \cite{Z15} for the details. In particular, this gives a unified proof of the existence of a strong uniform attractor for 2D Navier-Stokes with space regular or time regular external forces.
\par
We however note that the class of normal external forces is mainly adapted to {\it parabolic} equations and, in contrast to space or time regularity, the normality of the external forces in the proper space is not sufficient to get a strong uniform attractor for, say, damped wave equations, see \cite{Z15}.
\end{remark}
\begin{remark} To conclude this section, we briefly mention that general theorems \ref{Th2.main} and \ref{Th4.main} are well adapted also for developing the theory of {\it trajectory} attractors for problems in the form of \eqref{4.npde-h} without the uniqueness of solutions. The detailed exposition of this topic can be found in \cite{ChVi02}, so we just explain schematically the main ideas. To this end, we introduce, for every $\xi\in\Cal H(\xi_0)$, the corresponding set $\Cal K^+_\xi$ of solutions of \eqref{4.npde-h} defined on a semi-interval $\R_+$. Analogously to the autonomous case, this may be the set of all weak solutions or some subset of it consisting of some special solutions, but we need to satisfy the key invariance assumption:
$$
T(h):\Cal K^+_\xi\subset\Cal K^+_{T(h)\xi},\ \ \xi\in\Cal H(\xi_0),\ \ t\ge0.
$$
Then we define the set $\Cal K^+:=\cup_{\xi\in\Cal H(\xi_0)}\Cal K^+_\xi$ and consider the trajectory DS $(T(h),\Cal K^+)$ which gives the trajectory analogue  to the extended DS \eqref{4.eDS}. Then we may construct an attractor for this trajectory DS by verifying the conditions of Theorem \ref{Th2.main} and this gives us a {\it uniform} trajectory attractor for problem \eqref{4.npde-h}, see \cite{ChVi02} for more details.
\par
Alternatively, we may consider maps $T(h):\Cal K_\xi^+\to\Cal K_{T(h)\xi}^+$, $h\ge0$, as a cocycle over the DS $T(h):\Cal H(\xi_0)\to\Cal H(\xi_0)$, $h\in\R$. Then, fixing some topologies on the spaces $\Cal K^+_\xi$, we get a family of DP $U_\xi(t,\tau):=T(t-\tau):\Cal K^+_{T(\tau)\xi}\to\Cal K^+_{T(t)\xi}$ and may use the key Theorem \ref{Th4.main} in order to construct the trajectory analogues of pullback attractors, see \cite{ZZ08} for more details.
\end{remark}

\section{Dimensions of the attractor}\label{s5}
In this section, we start to discuss the finite-dimensionality of attractors related with dissipative PDEs. Note from the very beginning that attractors are usually not regular, but fractal subsets of the phase space, so we need to use the proper generalizations of a dimension which are suitable for fractal sets. Actually there are many such generalizations like Lebesgue covering, Hausdorff, fractal, Lyapunov, Assaud dimensions, etc. In general, all of them may be different, see e.g. \cites{EKZ13,Rob11} and also Example \ref{Ex1.2D1}.
\par
One of the main motivations to study the dimensions of attractors is to give a rigorous justification of the heuristic idea that despite the infinite-dimensionality of the initial phase space, the limit dynamics of many important dissipative PDEs is essentially finite-dimensional and can be described by finitely many parameters (the order parameters in the terminology of I. Progogine, see \cite{Pri77}) whose evolution is governed by  a system of ODEs. This finite-dimensional reduction would allow us to reduce the study, say,   turbulence which is described by Navier-Stokes equations to a system of ODEs which can be further investigated by the methods of classical dynamics. Unfortunately, despite many efforts in this direction, the above mentioned finite-dimensional reduction remains a "mystery" and the existing theory is still far from being complete, see \cites{EKZ13,Zel14,Rob11} and references therein for more details.
\par

\subsection{Man\'e projection theorem and finite-dimensional reduction}

In this subsection we restrict ourselves to the most studied case of a fractal dimension.

\begin{definition}\label{Def5.frac} Let $\Cal A$ be a (pre)compact set in a metric space $\Phi$. Then, by the Hausdorff criterion, for any $\eb>0$ it can be covered by finitely many balls of radius $\eb$ in $\Phi$. Let $N_\eb(\Cal A,\Phi)$ be the minimal number of such balls. Then the Kolmogorov entropy of $\Cal A$ in $\Phi$ is the following number:
$$
\Bbb H_\eb(K,\Phi):=\log_2 N_\eb(K,\Phi),
$$
where the base $2$ in the logarithm comes from the information  theory, see \cite{KT93} and references therein for more details. The (upper) fractal dimension of $K$ in $\Phi$ is defined as follows:
\begin{equation}
\dim_f(K,\Phi):=\limsup_{\eb\to0}\frac{\Bbb H_\eb(K,\Phi)}{\log_2\frac1\eb}.
\end{equation}
It is well-known that $\dim_f(K,\Phi)=n$ if $K$ is an $n$-dimensional Lipschitz manifold, but may be not integer if $K$ has a fractal structure, e.g. $\dim_f(K,[0,1])=\frac{\ln2}{\ln3}$ for the standard ternary Cantor set in $[0,1]$. Roughly speaking, $N_\eb(K,\Phi)\sim \(\frac1\eb\)^\kappa$ if $\dim_f(K,H)=\kappa$. Also mention that it can easily be in the case where $\Phi$ is infinite-dimensional that $\dim_f(K,\Phi)=\infty$. This simply means that $N_\eb(K,\Phi)$ has a stronger divergence rate as $\eb\to0$ than $\(\frac1\eb\)^\kappa$. In this case, the problem of finding/estimating the fractal dimension naturally transforms to the problem of finding the asymptotic behavior of $N_\eb(\Cal A,\Phi)$ as $\eb\to0$, see \cites{KT93,Tri78,Zel04,MZ08} for more details. We also mention that sometimes, instead of covering the set $\Cal A$ by $\eb$-balls, the covering by sets of diameter less than or equal to $\eb$ are used. Although it does not affect the value of the fractal dimension, it may be more suitable for its  estimation since the problem of whether or not the centers of $\eb$-balls belong to $\Cal A$ disappears under this more general setting.
\end{definition}
The applications of the fractal dimension to the above mentioned finite-dimensional reduction problem are based on the following Man\'e projection theorem, see \cite{Man81} and also \cites{HK99,Rob11} and references therein.

\begin{theorem}\label{Th5.mane} Let $\Cal A$ be a compact subset of a Hilbert space $H$ such that $\dim_f(\Cal A,H)\le n$ for some $n\in\Bbb N$. Then the orthoprojector $P_L$ to a "generic" plane $L\subset H$ of dimension $\dim L\ge 2n+1$ is one-to-one on $\Cal A$.
\end{theorem}

 Thus, since $\Cal A$ is compact, $P_L:\Cal A\to\bar {\Cal A}:=P_L\Cal A\subset L$ is a homeomorphism. If $\Cal A$ is an attractor of a continuous  DS $S(t): H\to H$, then we may project this semigroup to the one acting on a finite-dimensional space $L\sim\R^{2n+1}$:
 \begin{equation}\label{5.DS-red}
 \bar S(t):\bar{\Cal A}\to\bar{\Cal A},\ \ \bar S(t):=P_L\circ S(t)\circ P_L^{-1}
 \end{equation}
and, therefore, the  dynamics on $\Cal A$ is indeed described by a continuous semigroup acting on a compact subset of $\R^{2n+1}$. This is exactly the way how we may realize the desired finite dimensional reduction based on the finiteness of the fractal dimension of the attractor and the Man\'e projection theorem. Moreover, we may even write out a system of ODEs on the order parameters $y(t):=P_Lu(t)$ if $u(t)\in\Cal A$ solves a general PDE of the form
$$
\Dt u=A(u).
$$
Namely, applying the projector $P_L$ to both sides of this equation, we formally get
\begin{equation}\label{5.IF}
\frac d{dt}y(t)=P_LA(P_L^{-1}y(t)):=\Cal F(y(t)),\ \ y(t)\in\bar{\Cal A}\subset L\sim\R^{2n+1}.
\end{equation}
This ODE is often referred as an Inertial Form (IF) of the initial PDE. Thus, under this approach, the fractal dimension of the attractor $\Cal A$ is interpreted as a number of effective degrees of freedom in the reduced IF. This, in turn, motivates a great interest to various methods for obtaining upper and lower bounds for that dimension, see \cites{BV92,tem,MZ08} and references therein.
\begin{remark}\label{Rem5.bad} There are many versions of the Man\'e projection theorem, in particular, the  H\"older Man\'e projection theorem which allows us to establish the H\"older continuity of the inverse map $P_L^{-1}:\bar{\Cal A}\to\Cal A$. Moreover, under some extra conditions which are usually satisfied at least in the case of semilinear parabolic equations, the H\"older exponent can be made arbitrarily close to one by increasing the dimension of $L$. The fact that $H$ is Hilbert is also not essential and all these results remain true in Banach spaces, see \cite{Rob11} and references therein.
\par
However, the finiteness of the fractal dimension is not enough to guarantee that it is possible to find $L$ in such a way that $P_L^{-1}$ is Lipschitz, so in general the IF constructed in such a way has only H\"older continuous vector field $\Cal F$ and this is the key drawback of the described approach.
Indeed, the H\"older continuity of $\Cal F(y)$ is not enough even to establish the uniqueness of solutions for the IF \eqref{5.IF}, so it is not clear how to select the "physically relevant" solutions of this system without referring to the initial PDE. In addition, there are more and more examples where the dynamics on the attractor $\Cal A$ remains, in a sense, infinite-dimensional (e.g. admits super-exponentially stable limit cycles, decaying traveling waves in the Fourier space, etc.) despite the finiteness of the fractal dimension. These attractors cannot be embedded in any finite-dimensional Lipschitz or log-Lipschitz submanifold of the phase space, see \cites{EKZ13,Zel14,KZ18} and also next section for more details. Thus, despite the widely accepted paradigm, the finite-dimensional reduction based on the Man\'e projection theorem does not look as a proper solution of the problem and the fractal dimension of the  attractor is not an appropriate tool to estimate the effective number of degrees of freedom of the reduced system. We will discus the alternative methods in the next sections and the rest of this section is devoted to upper and lower bounds for the fractal dimension.
\end{remark}
\subsection{Upper bounds via squeezing property: the autonomous case}\label{s5.sq} Let us assume that the attractor $\Cal A\subset\Phi$ of the considered DS is already constructed and the maps $S(t)$ are continuous. Then, for every $t_0\in\R_+$, we have
\begin{equation}\label{5.inv}
S(\Cal A)=\Cal A,\ \  S:=S(t_0)
\end{equation}
and the key question is "under what assumptions on the map $S$ we can guarantee that $\Cal A$ has a finite fractal dimension?" The next theorem gives the simplest of such conditions.
 \begin{theorem}\label{Th5.lad} Let $\Phi$ and $\Phi_1$ be two Banach spaces and the embedding $\Phi_1\subset\Phi$ be compact. Assume also that $\Cal A$ is a bounded subset of $\Phi_1$ and the map $S:\Cal A\to\Cal A$ satisfies \eqref{5.inv} as well as the following squeezing/smoothing property:
 \begin{equation}\label{5.sq1}
 \|S(u_1)-S(u_2)\|_{\Phi_1}\le L\|u_1-u_2\|_\Phi,\ \ u_1,u_2\in\Phi.
 \end{equation}
 Then the fractal dimension of $\Cal A$ in $\Phi_1$ is finite and enjoys  the following estimate:
 \begin{equation}\label{5.est11}
 \dim_f(\Cal A,\Phi_1)\le\Bbb H_{\frac1{4L}}(\Phi_1\hookrightarrow\Phi),
 \end{equation}
 where $\Bbb H_{\frac1{4L}}(\Phi_1\hookrightarrow\Phi)$ is the entropy of the unit ball of the space $\Phi_1$ considered as a pre-compact set in $\Phi$.
 \end{theorem}
 \begin{proof} Indeed, let us assume that we have already constructed an $\eb$-net $\{u_k\}_{k=1}^{N(\eb)}$ of the set $\Cal A$ in $\Phi_1$ (i.e.  the $\eb$-balls of $\Phi_1$ with centers in $u_k\in\Cal A$ cover $\Cal A$). Let us cover every of these balls by $N$ balls of radius $\eb/(4L)$ in $\Phi$. It is possible to do due to the compactness of the embedding $\Phi_1\subset\Phi$.  Moreover, the number of balls which is necessary to cover every of such balls does not exceed
 $$
 N:=N_{\eb/(4L)}(B_\eb(u_k,\Phi_1),\Phi)=N_{1/(4L)}(B_1(0,\Phi_1),\Phi).
 $$
 Crucial for us that this number is independent of $\eb$ and $u_k$. This gives us the covering of $\Cal A$ by $\eb/(4L)$-balls in $\Phi$ and the number of these balls does not exceed $NN(\eb)$. Moreover, increasing the radii of the balls by the factor of two, we may assume also that the centers $\{v_k\}_{k=1}^{NN(\eb)}$ of these balls belong to $\Cal A$.
 \par
  Then, due to the invariance of $\Cal A$ and condition \eqref{5.sq1}, the $\eb/2$-balls in $\Phi_1$ centered at $\{S(v_k)\}$ cover $\Cal A$. Thus, starting from $\eb$-covering of $\Cal A$ which consists of $N(\eb)$ elements, we end up with a new $\eb/2$-covering with the number of elements $N(\eb/2)\le N N(\eb)$. Since $\Cal A$ is bounded, we may start with some $\eb_0=R_0$ such that $\Cal A\subset B_{R_0}(u_0,\Phi_1)$ and therefore $N(\eb_0)=1$. Iterating the above described procedure, we finally get the $\eb_n:=R_02^{-n}$-coverings which consist of
  $$
  N(\eb_n)=N(R_02^{-n})\le N^n
  $$
  elements.
  Thus,
  $$
  \Bbb H_{\eb}(\Cal A,\Phi_1)\le n\log_2N\le  n \Bbb H_{\frac1{4L}}(\Phi_1\hookrightarrow\Phi),\ \ R_02^{-n}\le\eb\le R_02^{-n+1}
  $$
  and this gives us the desired estimate for the fractal dimension of $\Cal A$.
 \end{proof}
 \begin{remark} To the best of our knowledge the key idea used in the proof of this theorem is due to Mallet-Paret \cite{MP76} (where he used it for estimating the Hausdorff dimension under the extra assumption that $S$ is a $C^1$-map) and the stated theorem has been proved by Ladyzhenskaya, see \cite{Lad82}. We present a more or less complete proof here since it is simple and elegant on the one hand and, on the other hand, all known estimates of the fractal dimension of attractors are based on similar iteration schemes.
 \end{remark}
\begin{example}\label{Ex5.1D-par} We return to Example \ref{Ex1.1D-par} of 1D semilinear parabolic equation \eqref{1.ce}. The existence of solution semigroup $S(t)$ associated with this equation as well as the existence of a global attractor $\Cal A\subset H^1_0(0,1)\subset C[0,1]$ for this semigroup are already verified in Example \ref{Ex1.1D-par}. Moreover, due to the maximum/comparison principle we also know that any complete bounded solution $u(t)$, $t\in\R$, of this equation satisfies
\begin{equation}\label{5.max}
-a^{1/2}\le u(t,x)\le a^{1/2},\ \ t\in\R,\ \ x\in [0,1].
\end{equation}
Let now $u_1(t)$ and $u_2(t)$ be two trajectories belonging to the attractor $\Cal A$ and let $v(t):=u_1(t)-u_2(t)$. Then, multiplying equation \eqref{1.dif-dif} by $t\partial_x^2v$, integrating by parts and using \eqref{5.max}, we get
$$
\frac d{dt}(t\|\partial_x v(t)\|_{L^2}^2)-2a(t\|\partial_x v(t)\|^2_{L^2})\le Ca^2 t\|v(t)\|^2_{L^2}+\|\partial_x v(t)\|^2_{L^2},
$$
where $C$ is independent of $a$. Integrating this inequality in time and using \eqref{1.lip}, we end up with
$$
t\|\partial_x v(t)\|^2_{L^2}\le\int_0^te^{2a(t-s)}\(C a^2 s\|v(s)\|^2_{L^2}+\|\partial_x v(s)\|^2_{L^2}\)\,ds
\le C_1(at+1)e^{at}\|v(0)\|^2_{L^2}.
$$
for some $C_1>0$ which is independent of $a$. Fixing now $t_0=a^{-1}$ and $S:=S(t_0)$, we prove that for any two points $u_1(0),u_2(0)\in \Cal A$, estimate \eqref{5.sq1} is satisfied with $\Phi=L^2(0,1)$, $\Phi_1=H^1_0(0,1)$ and $L=Ca^{1/2}$ for some $C$ which is independent of $a$. Thus,
\begin{equation}\label{5.ent-est0}
\dim_f(\Cal A,H^1_0)\le \Bbb H_{\frac1{4L}}(H^1_0\hookrightarrow L^2)\le Ca^{1/2},
\end{equation}
where we have used the well-known result about the entropy of embeddings of Sobolev spaces in bounded domains $\Omega\subset\R^d$, namely
\begin{equation}\label{5.ent}
\Bbb H_\nu(W^{s_1,p_1}\hookrightarrow W^{s_2,p_2})\le C\(\frac1\nu\)^{\frac d{s_1-s_2}},
\end{equation}
see \cite{Tri78}.
\end{example}
\begin{remark} As we can see from Example \ref{Ex5.1D-par}, the squeezing property \eqref{5.sq1} is a straightforward corollary of the parabolic smoothing property for the linear PDE (equation of variations) for the difference of two solutions associated with  the initial nonlinear PDE. Since such a smoothing property is typical for dissipative PDEs (for non-parabolic equations it should be replaced by the proper asymptotic smoothing property discussed below), this explains why, in many cases, we have the finiteness of the fractal dimension for attractors of dissipative PDEs.
\par
Note also that the obtained upper bound \eqref{5.ent-est0} is sharp with respect to $a\to\infty$ (the lower bounds of the same order in $a$ are available). Remarkable that, in order to get reasonably sharp estimates, one should consider the squeezing property \eqref{5.sq1} on a {\it small} time interval $t_0\to0$ as $a\to\infty$. In contrast to this, the estimates based on volume contraction method  usually work better for large $t_0\to\infty$, see \cites{BV92,tem}.
\par
We mention also that the key advantage of the squeezing property \eqref{5.sq1} is that, in comparison with the volume contraction, it does not require the semigroup $S(t)$ to be {\it differentiable} with respect to the initial data. This is crucial e.g. for applications to singular or/and degenerate PDEs, where such a differentiability usually does not take place or is very difficult/impossible to verify. As the price to pay, this method may give essentially worse estimates in comparison with the volume contraction. For instance, for the 2D Navier-Stokes equation in a bounded domain, the best known upper bound for the fractal dimension of the attractor reads
$$
   \dim_f(\Cal A,H)\le C\nu^{-2},
$$ where $\nu$ is a kinematic viscosity, see \cites{BV92,tem} and references therein, but even the polynomial in $\nu^{-1}$ bounds via the squeezing property are not known so far.
\end{remark}
We now state a natural analogue of the squeezing property \eqref{5.sq1} which, in particular, is suitable for  non-parabolic equations.
\begin{theorem}\label{Th5.2lad}
 Let $\Phi$ and $\Phi_1$ be two Banach spaces and the embedding $\Phi_1\subset\Phi$ be compact. Assume also that $\Cal A$ is a bounded subset of $\Phi_1$ and the map $S:\Cal A\to\Cal A$ satisfies \eqref{5.inv} as well as the following squeezing/smoothing property:
 \begin{equation}\label{5.sq2}
 \|S(u_1)-S(u_2)\|_{\Phi_1}\le \kappa\|u_1-u_2\|_{\Phi_1}+L\|u_1-u_2\|_\Phi,\ \ u_1,u_2\in\Phi.
 \end{equation}
 Then the fractal dimension of $\Cal A$ in $\Phi_1$ is finite and satisfies the following estimate
 \begin{equation}\label{5.est1}
 \dim_f(\Cal A,\Phi_1)\le\frac{\Bbb H_{\frac{1-\kappa}{4L}}(\Phi_1\hookrightarrow\Phi)}{\log_2\frac2{1+\kappa}},
 \end{equation}
 where $0\le\kappa<1$ and $L>0$ are some constants.
 \end{theorem}
 The proof of this theorem repeats almost word by word the arguments given in the proof of the previous theorem. The only difference is that, to be able to work with $\kappa>1/2$, we need to use the coverings by sets of diameter less than $\eb$ in the iteration scheme, see \cite{7} for more details.
 \par
 We also mention that nowadays there are  a huge amount of various formulations of the squeezing property (which give the finiteness of the fractal dimension of an invariant set) adapted to the concrete classes of problems. The detailed exposition of them is out of scope of this survey, so we refer the interested reader to \cites{EMZ05,MZ08} and  references therein.

\begin{example} Let us consider the degenerate version of real Ginzburg-Landau equation:
 \begin{equation}\label{5.deg}
 \Dt u=\Dx(u^3)+u-u^3,\ \ u\big|_{t=0}=u_0,\ \ u\big|_{\partial\Omega}=0
 \end{equation}
 in a bounded smooth domain $\Omega$ of $\R^d$.  This equation generates a dissipative semigroup $S(t)$, say, in the space $\Phi=L^1(\Omega)$ (the proof of this fact is based on the Kato inequality and the related multiplication of the equation by $\sgn(u)$, see \cite{Cle87}). Moreover, this semigroup is globally Lipschitz in $L^1(\Omega)$ and,  due to the H\"older continuity results for solutions of degenerate parabolic problems, it possesses a compact absorbing set (with respect to the bornology of bounded sets in $\Phi$) which is bounded in the space $C^\alpha(\Omega)$ for some $\alpha>0$, therefore, the attractor $\Cal A$ exists and is generated by all complete bounded trajectories which are H\"older continuous in space and time, see \cites{EZ07,Iv82,DiB93}.
 \par
 However, as shown in \cite{EZ07} this attractor has an infinite Hausdorff and fractal dimensions:
 $$
 \dim_H(\Cal A, \Phi)=\infty.
 $$
 Moreover, the infinite-dimensional family of complete bounded solutions can be constructed almost explicitly based on the finite propagation speed of compactly supported solutions.
 \par
 This infinite-dimensionality can be somehow explained by the fact that the formal "linearization" of \eqref{5.deg} on a degenerate solution $u=0$ is unstable, so the energy income in the system is possible in the "area" where the equation is degenerate. In the case where the areas where the equation is degenerate or singular and where the energy income may happen are separated,  typically the corresponding attractor has finite fractal dimension, see \cites{EZ07,EZ09,MirZ09,MZ08,MirZ07,MirZ05,SZ09} for the justification of this heuristic principle for concrete classes of singular/degenerate equations. In particular, if we consider a slightly modified  version of equation \eqref{5.deg}:
 \begin{equation}\label{5.deg1}
 \Dt u=\Dx(u^3)+3u^2-2u-u^3,\ \ u\big|_{t=0}=u_0,\ \ u\big|_{\partial\Omega}=0,
 \end{equation}
  the formal linearization $\Dt v=-2v$ on $u=0$ is exponentially stable and the proper version of the squeezing property gives us the finiteness of the fractal dimension of the attractor $\Cal A$, see~\cite{EZ07}.
\end{example}
\subsection{Upper bounds via squeezing property: the non-autonomous case}

  The standard situation in this case
is when we have a cocycle $\Cal S_\xi(t):\Phi\to\Phi$, $\xi\in\Psi$, over DS $T(h):\Psi\to\Psi$ and an invariant set $\Cal A_\xi$ such that
$$
\Cal S_\xi(1)\Cal A_\xi=\Cal A_{T(1)\xi}, \ \ \xi\in \Psi.
$$
Instead of time step $t=1$, one may take an arbitrary $t_0>0$, but for simplicity we assume that $t_0=1$ here.
The non-autonomous analogue of the squeezing property \eqref{5.sq2} reads
\begin{equation}\label{5.sq3}
\|\Cal S_\xi(1)u_1-\Cal S_\xi(1)u_2\|_{\Phi_1}\le \kappa\|u_1-u_2\|_{\Phi_1}+L(\xi)\|u_1-u_2\|_{\Phi},
\end{equation}
which holds for all $\xi\in\Psi$ and $u_1,u_2\in\Cal A_\xi$. For simplicity, we take a "deterministic" value of $\kappa\in[0,1)$ (it is independent of $\xi$, but its generalization to the "random" case $\kappa=\kappa(\xi)$ is straightforward.
\par
The estimate for the fractal dimension of $\Cal A_\xi$ can be obtained in the same way as in Theorems \ref{Th5.lad} and \ref{Th5.2lad}. The only difference is that, instead of iterating a single map $S$, we now need to iterate different maps from the cocycle and use that
$$
\Cal S_{T(-1)\xi}(1)\circ\cdots\circ\Cal S_{T(n-1)\xi}(1)\circ\Cal S_{T(-n)\xi}(1)\Cal A_{T(-n)\xi}=\Cal A_\xi
$$
for all $\xi\in\Psi$ and $n\in\Bbb N$. This gives the following result, see \cites{PRSE05,ShZ13} for more details.
.
\begin{theorem}\label{Th5.3lad} Let $\Phi_1\subset\Phi$ be two embedded Banach spaces such that the embedding is compact and let $\Cal S_\xi(t):\Phi\to\Phi$, $\xi\in\Psi$, be a cocycle over $T(h):\Psi\to\Psi$. Assume that $\xi\to\Cal A_\xi\subset\Phi_1$ is a family of bounded strictly invariant sets which satisfy the squeezing property \eqref{5.sq3} with some $\kappa\in[0,1)$. Then the fractal dimensions of $\Cal A_\xi$ satisfies the following estimate:
\begin{multline}\label{5.est-main}
\dim_f(\Cal A_\xi,\Phi_1)\le \frac1{\log_2\frac2{1+\kappa}-\limsup_{n\to\infty}\frac{\log_2 R(T(-n)\xi)}{n}}\times\\\times \limsup_{n\to\infty}\frac{\Bbb \sum_{k=1}^n \Bbb H_{\frac{1-\kappa}{4L(T(-k)\xi)}}(\Phi_1\hookrightarrow\Phi)}n,
\end{multline}
where $R(\xi):=\|\Cal A_\xi\|_{\Phi_1}$ and we assume that the right-hand side is infinite if the denominator is negative.
\end{theorem}
The most straightforward application of this general theorem is related with the "deterministic" (uniform) case, where the size $R(\xi)$ of the attractors $\Cal A_\xi$ is uniformly bounded: $R(\xi)\le R_0$ and the expanding factor $L(\xi)\le L$ is also uniformly bounded. This is typical for non-autonomous equations with uniformly in time bounded external forces considered in \cite{PRSE05}, see also  references therein. Then we have exactly the same estimate for the dimensions of $\Cal A_\xi$ as in the autonomous case.
\begin{corollary} Let the assumptions of Theorem \ref{Th5.3lad} hold and let, in addition, $R(\xi)\le R_0$ and $L(\xi)\le L$ for all $\xi\in\Psi$. Then the fractal dimensions of $\Cal A_\xi$ are finite and satisfy estimate \eqref{5.sq2} uniformly with respect to $\xi\in\Psi$.
\end{corollary}
The applications of this general theorem to the random case are more delicate and interesting. In this case, we have an invariant ergodic measure $\mu$ for the DS $T(h):\Psi\to\Psi$ and assume that the function $\xi\to\Cal A_\xi$ is $\mu$-measurable. We also assume that the expanding factor $L(\xi)$ is also measurable and that the entropy of the embedding $\Phi_1\subset\Phi$ possesses the estimate
 \begin{equation}\label{5.sob-ent}
 \Bbb H_\nu(\Phi_1\hookrightarrow\Phi)\le C\(\frac1\nu\)^\theta
 \end{equation}
 for some positive $C$ and $\theta$. Thus assumption is not restrictive since usually in applications $\Phi$ and $\Phi_1$ are Sobolev spaces for which such an estimate holds. Then we have the following result.
 \begin{corollary} Let the assumptions of Theorem \ref{Th5.3lad} hold and let, in addition, the above mentioned random structure be introduced and  \eqref{5.sob-ent} hold. We also assume that the family of the attractors $\Cal A_\xi$ is tempered in  $\Phi_1$ and that the expanding factor $L(\xi)$ has the finite $\theta$'s momentum:
 \begin{equation}\label{5.mom}
 \Bbb E(L^\theta):=\int_{\Psi}L(\xi)^\theta\mu(d\xi)<\infty.
 \end{equation}
 Then, for almost all values $\xi\in\Psi$, the fractal dimension of $\Cal A_\xi$ is finite and satisfies the estimate:
\begin{equation}\label{5.r-est}
\dim_f(\Cal A_\xi,\Phi_1)\le \frac{C\frac{4^\theta}{(1-\kappa)^\theta}\Bbb E(L^\theta)}{\log_2\frac{2}{1+\kappa}}.
\end{equation}
 \end{corollary}
Indeed, since $\Cal A_\xi$ is tempered, $\lim_{n\to\infty}\frac{\log_2R(T(-n)\xi)}n=0$. The second multiplier in the right-hand side of \eqref{5.est-main} is estimated using \eqref{5.sob-ent}, \eqref{5.mom} and the Birkhoff ergodic theorem.

\begin{remark} It is typical for random DS that, in order to have the finite-dimensionality of the attractor, we need to verify that some random variable (which is responsible to the expansion rate of the distance between two solutions or Lyapunov's exponents) has the finite mean, see \cites{CF98,D98,ShZ13} and  references therein. This condition is non-trivial and is often the most difficult to verify. On the other hand, as we will see in the next toy example (suggested in \cite{CSZ23}), the dissipation mechanism may be not strong enough to provide the finite-dimensionality of a random attractor if this condition is violated.
\end{remark}
\begin{example}\label{Ex6.inf} Let $H=l_2$ (space of square summable sequences) and consider the random dynamical system in $H$ generated by the following equations:
\begin{equation}\label{5.odes}
\frac d{dt} u_1+\gamma(t)u_1(t)=1,\ \ \frac d{dt}u_k+k^4u_k=u_1(t)u_k-u_k^3, \ k=2,3,\cdots,
\end{equation}
where $u=(u_1,u_2,\cdots)\in H$ and $\gamma\in\Psi$ is exactly the Bernoulli process used in Example \ref{Ex4.hyp}. The first equation of this system models the energy evolution of \eqref{4.hyp} and the rest equations give some coupling of the first equation with a parabolic PDE.
\par
System of ODEs \eqref{5.odes} can  be solved explicitly. In particular, if $aq-(1-q)b>0$,
\begin{equation}\label{6.good}
u_1(t)=u_{1,\gamma}(t)=\int_{-\infty}^te^{-\int_s^t\gamma(l)\,dl}\,ds
\end{equation}
is a unique tempered complete solution of the first equation (for almost all $\gamma\in\Psi$). Moreover, it is not difficult to show that
\begin{equation}\label{6.uinf}
\int_{\gamma\in\Psi}u_{1,\gamma}(0)\mu(d\gamma)=\infty
\end{equation}
if
\begin{equation}\label{6.infinf}
\ln(qe^{-a}+(1-q)e^{b})>0>-aq+(1-q)b.
\end{equation}
In this case, the following result holds.

\begin{proposition}\label{Prop5.main} Let the exponents $a,b>0$ and  $q\in(0,1)$ satisfy \eqref{6.infinf}. Then the random attractor $\mathcal A_\gamma$ for system \eqref{5.odes} in $H$ has infinite Hausdorff and fractal dimensions:
\begin{equation}\label{6.dim}
\dim_H(\mathcal A_\gamma,H)=\dim_f(\mathcal A_\gamma,H)=\infty
\end{equation}
for almost all $\gamma\in\Psi$.
\end{proposition}
\begin{proof}[Sketch of the proof] The detailed proof of this result is given in \cite{CSZ23}. Here we just briefly discuss the main ideas behind it. The existence of a random tempered absorbing set for the first component $u_1$ of system \eqref{5.odes} follows from the explicit formula for the solution. Let us now assume that $u_1(\tau)$ is already in this absorbing set and get the estimates for $u_k$. To this end, multiplying the $k$th equation by $\operatorname{sgn}(u_k)$ and taking a sum, after the standard estimates, we get
$$
\frac d{dt}\(\sum_{k=2}^\infty|u_k|\)+\sum_{k=2}^\infty k^4|u_k|\le C\(\sum_{k=2}^\infty\frac1{k^2}\)|u_1|^{3/2}\le C|u_1(t)|^{3/2}.
$$
Integrating this inequality and using that $u_1(t)$ is tempered, we get a tempered absorbing ball for $u$ in $l_1\subset H$. To obtain a compact absorbing set, it is enough to use the parabolic smoothing property in a standard way. Thus, the existence of a random attractor $\mathcal A_\gamma$ is verified.
\par
Recall also that a random attractor consists of all complete tempered trajectories of the system considered, so it is enough to find all such trajectories. The first equation is linear and is independent of the other equations, so such a trajectory is unique and is given by \eqref{6.good}. Thus, to find $\mathcal K_\gamma$, we need to fix $u_1(t)=u_{1,\gamma}(t)$ in the other equations of \eqref{6.good} and find the tempered attractor $\mathcal A^k_\gamma$, $k=2,3,\cdots$ for every component of \eqref{6.good} separately. Then the desired attractor $\mathcal A_\gamma$ for the whole system is presented as a product
\begin{equation}\label{6.pro}
\mathcal A_\gamma=\{u_{1,\gamma}(0)\}\times\otimes_{k=2}^\infty\mathcal A^k_\gamma.
\end{equation}
Moreover, since the attractor is always connected, every $\mathcal A^k_\gamma$ is a closed interval, so to prove the desired infinite-dimensionality, it is enough to prove that $\mathcal A^k_\gamma\ne\{0\}$ for all $k$. In other words, we need to find a non-zero tempered trajectory $u_k=u_k(t)$ for the equation
\begin{equation}\label{6.k}
\frac d{dt} u_k+k^4u_k=u_{1,\gamma}(t)u_k-u_k^3,
\end{equation}
using that $u_{1,\gamma}(t)$ is tempered and have an infinite mean.
\par
To this end, we first note that every solution of equation \eqref{6.k} is either tempered as $t\to-\infty$ or blows up backward in time (this can be easily shown by comparison using the fact that $u_{1,\gamma}(t)$ is tempered), so any complete solution $u_k(t)$, $t\in\R$, is automatically tempered.
\par
We construct this solution by solving equation \eqref{6.k} explicitly. Namely,
$$
u_{k,\gamma}(t):=\frac1{\sqrt{2\int_{-\infty}^te^{2\int_s^t(k^4-u_{1,\gamma}(l))\,dl}\,ds}}.
$$
Indeed, the finiteness of the integral is guaranteed by \eqref{6.uinf}, see \cite{CSZ23} for more details. Thus, we have proved that $\mathcal A^k_\gamma\ne \{0\}$ for all $k$ and almost all $\gamma$ and the proposition is proved.
\end{proof}
\end{example}
\begin{remark}Actually, we have found explicitly the random attractor in the previous example:
$$
\mathcal A_\gamma=\{u_{1,\gamma}(0)\}\times\otimes_{k=2}^\infty[-u_{k,\gamma}(0),u_{k,\gamma}(0)].
$$
This example shows that, in a contradiction with the commonly accepted paradigm, adding random terms may not only simplify the dynamics, but also make it essentially more complicated and even infinite-dimensional. We expect that this phenomena has a general nature and can be observed in more realistic equations.
\end{remark}
We conclude this subsection by considering the applications of the dimension estimate to {\it uniform} attractors.  In this case, we typically have a cocycle $S_\xi(t):\Phi\to\Phi$, $\xi\in\Psi$ over the DS $T(h):\Psi\to\Psi$, a family of invariant sets (pullback attractors) $\Cal A_\xi$, $\xi\in\Psi$, and a union
\begin{equation}\label{5.rep}
\Cal A_{un}=\cup_{\xi\in\Psi}\Cal A_\xi,
\end{equation}
which is exactly the uniform attractor whose dimension we want to estimate. For simplicity, we assume that $\Psi\subset L^p_b(\R,H)$ for some Banach space $H$ and $1\le p<\infty$ is a hull of some translation-compact external force $\xi_0\in L^p_b(\R,H)$, so we assume that
$$
\Psi:=\Cal H(\xi_0)\subset\subset L^p_{loc}(\R,H),
$$
see \cite{ChVi02} for more details and more general exposition. It is also well-known that, despite the finiteness of the fractal dimension of every $\Cal A_\xi$, the dimension of the union \eqref{5.rep} is usually infinite, because of the infinite-dimensionality of the hull $\Cal H(\xi_0)$. For this reason, it looks natural (following Vishik and Chepyzhov) to study the Kolmogorov $\eb$-entropy of $\Cal A_{un}$.  To this end, we need to include to the squeezing property \eqref{5.sq2} also some kind of Lipschitz/H\"older continuity with respect to $\xi$, namely, to assume that
\begin{equation}\label{5.sq5}
\|\Cal S_{\xi_1}(1)u_1-\Cal S_{\xi_2}(1)u_2\|_{\Phi_1}\le \kappa\|u_1-u_2\|_{\Phi_1}+L\|u_1-u_2\|_{\Phi}+K\|\xi_1-\xi_2\|_{L^p(0,1;H)},
\end{equation}
where $0\le\kappa<1$, $L$ and $K$ are independent of $\xi_1,\xi_2\in\Psi$ and $u_1,u_2\in\Cal A_\xi$. Then we get the following analogue of Theorem \ref{Th5.2lad}.

\begin{theorem}\label{Th5.un-lad} Let $\Cal S_\xi(t):\Phi\to\Phi$ be a cocycle over the DS $T(h):\Psi\to\Psi$ and $\Psi=\Cal H(\xi_0)$ be a hull of some function $\xi_0\in L^p_b(\R,H)$ which is translation-compact in this space. Assume also that there is a family of invariant sets $\Cal A_\xi$ which   satisfies the squeezing property \eqref{5.sq5}, where $\Phi_1$ is some other Banach space compactly embedded to $\Phi$, and also is uniformly bounded in the Banach space $\Phi_1$. Then the Kolmogorov entropy of the union \eqref{5.rep} possesses the following estimate:
\begin{multline}\label{5.ent-est}
\Bbb H_\eb(\Cal A_{un},\Phi_1)\le C_0+\frac{\Bbb H_{\frac{1-\kappa}{4L}}(\Phi_1\hookrightarrow\Phi)}{\log_2\frac2{1+\kappa}}\log_2\frac{\eb_0}\eb+\\+\Bbb H_{\frac\eb {K_0}}(\Cal H(\xi_0)\big|_{t\in[0,L_0\log_2\frac{\eb_0}\eb]}, L^p_b([0,L_0\log_2\frac{\eb_0}\eb],H)),
\end{multline}
where $\eb\le \eb_0$ and $\eb_0,C_0,K_0,L_0$ are some positive numbers which are independent of $\eb$.
\end{theorem}
The derivation of this estimate is similar to the proof of Theorem \ref{Th5.lad}, so we omit it here and refer the interested reader to \cite{EMZ04} or \cite{YKSZ23}, see also \cite{ChVi02} for the analogous results obtained via the volume contraction method.
\par
As we already mentioned, the dimension of $\Cal A_{un}$ may be infinite due to the second term in the right-hand side of \eqref{5.ent-est}, see \cite{ChVi02} and references therein. However, there are interesting particular cases where the dimension of a uniform attractor remains finite, e.g. it is so for periodic or quasi-periodic DPs as well as for the DP stabilizing to the autonomous ones as time goes to $\pm\infty$, see \cites{ChVi02,EMZ04,YKSZ23} for more examples. To extract this finite-dimensionality from the key estimate \eqref{5.ent-est}, we introduce the fractal dimension of the hull $\Cal H(\xi_0)$ as follows:
\begin{equation}
\dim_f(\Cal H(\xi_0)):=\limsup_{\eb\to0}\frac{\Bbb H_{\eb}(\Cal H(\xi_0)\big|_{t\in[0,\log_2\frac{1}\eb]}, L^p_b([0,\log_2\frac{1}\eb],H))}{\log_2\frac1\eb}.
\end{equation}
Then estimate \eqref{5.ent-est} reads
\begin{equation}\label{5.un-est-dim}
\dim_f(\Cal A_{un},\Phi_1)\le \frac{\Bbb H_{\frac{1-\kappa}{4L}}(\Phi_1\hookrightarrow\Phi)}{\log_2\frac2{1+\kappa}}+\dim_f(\Cal H(\xi_0)),
\end{equation}
see \cite{ChVi02} for details. The first term in the right-hand side can be interpreted as (the upper bound of) the dimension of every $\Cal A_\xi$ and the second term is added due to the union.
\begin{remark} It is not difficult to see that $\dim_f(\Cal H(\xi_0))$ is indeed the standard fractal dimension of the set $\Cal H(\xi_0)$ in the weighted norm of $L^p_{e^{-|x|}}(\R,H)$. Note that, for a translation compact $\xi_0\in L^p_b(\R,H)$, the hull $\Cal H(\xi_0)$ is compact in a local topology of $L^p_{loc}(\R,H)$ and, in order to speak about the dimension, we need to fix a metric on $\Cal H(\xi_0)$. Estimate \eqref{5.ent-est} guesses that the weighted metric of $L^p_{e^{-|x|}}(\R,H)$ is the most appropriate for estimating the entropy of uniform attractors.
\par
Note also that estimate \eqref{5.ent-est} makes no sense if $\xi_0$ is not translation-compact (since the second term in the RHS will be infinite). However, there is a natural generalization of this formula suggested in \cite{YKSZ23} which allows us to work with the classes of non-translation-compact external forces considered in subsection \ref{ss5.3}.
\par
It is worth  mentioning that the result of Theorem \ref{Th5.un-lad} remains true with minor changes if we replace the Lipschitz continuity with respect to $\xi$ in \eqref{5.sq5} by the H\"older continuity with an arbitrary  exponent $\alpha\in(0,1]$. In contrast to this, one may lose the finite-dimensionality (even in the autonomous case) if you replace Lipschitz continuity with respect to the initial data $u_i$ by H\"older or $\log$-Lipschitz continuity, see \cite{MieZ02} for the counterexample related with attractors of elliptic PDEs.
\end{remark}

\subsection{Volume contraction and Lyapunov's dimension} We now turn to the  scheme for estimating the dimension of the attractor, which is based on volume contraction arguments. This scheme becomes extremely popular due to applications to Navier-Stokes equations, where it gives the best available so far estimates, which can not be obtained by other methods (although this scheme has essential drawbacks, e.g. it requires the phase space to be Hilbert and the corresponding DS to be differentiable with respect to the initial data, so it cannot completely replace other methods). Since there are nice expositions of this material in the literature, see e.g. \cites{BV92,tem} and references therein, we restrict ourselves only to a brief discussion.
\par
We start with introducing the  volume contraction factor. Let $\Phi$ be a separable Hilbert space and let $\varphi_1,\cdots,\varphi_d\in\Phi$. By definition, a wedge product $\varphi_1\wedge\cdots\wedge\varphi_d$ is a $d$-linear antisymmetric form on $\Phi$ defined by
$$
\varphi_1\wedge\cdots\wedge\varphi_d(\psi_1,\cdots,\psi_d):=\det\((\varphi_i,\psi_j)_{i,j=1}^d\),
$$
where $(\cdot,\cdot)$ is the inner product in $\Phi$. A $d$-linear form on $\Phi$ which is a wedge product of $d$ vectors of $\Phi$ is called decomposable. Let us denote
by $\tilde\Lambda^d\Phi$
the space of $d$-linear antisymmetric forms which can be presented as a finite linear combination
of decomposable functionals. For two decomposable forms $\varphi_1\wedge\cdots\varphi_d$ and
$\psi_1\wedge\cdots\wedge\psi_d$, their inner
product is defined as follows:
$$
 (\varphi_1\wedge\cdots\varphi_d,  \psi_1\wedge\cdots\wedge\psi_d) := \det\((\varphi_i,\psi_j)_{i,j=1}^d\)
$$
and being extended by linearity to $\tilde\Lambda^d\Phi$ it defines an inner product on the space $\tilde\Lambda^d\Phi$, see \cite{tem} for the
details. Finally, the completion of $\tilde\Lambda^d\Phi$
 with respect to this norm is called $d$th exterior power of the
space $\Phi$ and is denoted by $\Lambda^d\Phi$. Note that $\Lambda^d\Phi$ is a subspace of all continuous antisymmetric $d$-linear forms on $\Phi$ and this subspace is proper if $d>1$ and $\dim\Phi=\infty$. Let now $L\in\Cal L(\Phi,\Phi)$ be a linear continuous operator on $\Phi$. Then its $d$th exterior power $\Lambda^dL$ is defined as follows:
$$
(\Lambda^dL\xi)(\psi_1,\cdots,\psi_d):=\xi(L^*\psi_1,\cdots,L^*\psi_d),\ \ \xi\in \Lambda^d\Phi,
$$
where $L^*$ is an adjoint operator to $L$.  It is not difficult to check that $\Lambda^dL:\Lambda^d\Phi\to\Lambda^d\Phi$ is a linear continuous operator and its norm does not exceed $\|L\|^d$. The following lemma gives a geometric interpretation for the norm of this operator.
\begin{lemma} The norm of $d$th exterior power of $L$ satisfies
$$
\|\Lambda^dL\|_{\Cal L(\Lambda^d\Phi,\Lambda^d\Phi)}=\sup_{\varphi_1\wedge\cdots\wedge\varphi_d\ne0}
\frac{\|(L\varphi_1)\wedge\cdots\wedge(L\varphi_d)\|_{\Lambda^d\Phi}}
{\|\varphi_1\wedge\cdots\wedge\varphi_d\|_{\Lambda^d\Phi}}=\sup_{\Pi^d\subset\Phi}
\frac{\operatorname{vol}_d(L\Pi)}{\operatorname{vol}_d(\Pi)},
$$
where the last supremum is taken over all non-degenerate $d$-dimensional parallelepipeds in $\Phi$ and $\operatorname{vol}_d$ stands for the $d$-dimensional Lebesgue measure.
\end{lemma}
The proof of this lemma can be found e.g. in \cite{tem}. Thus,
geometrically $\|\Lambda^dL\|$ is the maximal expanding factor for $d$-dimensional volumes under the action of the operator $L$.
\par
Let us now assume that we are given a compact set $\Cal A\subset\Phi$ and a map $S:\Cal A\to\Cal A$. Assume that this map is uniformly Frechet (quasi)differentiable on $\Cal A$, i.e. that there exist linear continuous maps $S'(u_0):\Phi\to\Phi$, $u_0\in\Cal A$  such that
\begin{equation}\label{5.quasi}
\|S(u_1)-S(u_2)-S'(u_1)(u_1-u_2)\|_{\Phi}=o(\|u_1-u_2\|_{\Phi})
\end{equation}
uniformly with respect to $u_1,u_2\in\Cal A$. Assume also that the map $u_0\to S'(u_0)$ is continuous on $\Cal A$. Then we define an infinitesimal $d$-volume contraction factor $\omega_d(S,\Cal A)$ as follows:
$$
\omega_d(S,\Cal A):=\sup_{u_0\in\Cal A}\|\Lambda^dS'(u_0)\|_{\Cal L(\Lambda^d\Phi,\Lambda^d\Phi)}.
$$
This definition can be extended also for non-integer values of $s=s_0+\alpha$, $\alpha\in(0,1)$ via
$$
\omega_d(S,\Cal A):=\sup_{u_0\in\Cal A}\{\|\Lambda^dS'(u_0)\|_{\Cal L(\Lambda^d\Phi,\Lambda^d\Phi)}^{1-\alpha}\|\Lambda^{d+1}S'(u_0)\|_{\Cal L(\Lambda^{d+1}\Phi,\Lambda^{d+1}\Phi)}^{\alpha}\}.
$$
The main result of the theory is the following theorem.
\begin{theorem}\label{Th5.vc} Let $\Cal A$ be a compact set in a Hilbert space $\Phi$ which is strictly invariant with respect to a map $S:\Cal A\to\Cal A$. Assume also that the map $S$ is uniformly quasidifferentiable on $\Cal A$ and that the map $u_0\to S'(u_0)$ is continuous. Finally, let  $\omega_d(S,\Cal A)<0$ for some $d\in\R_+$. Then
$$
\dim_f(\Cal A,\Phi)<d.
$$
\end{theorem}
To the best of our knowledge, this theorem has been proved by Duady and Oesterle \cite{DO80} in the finite-dimensional case for the Hausdorff dimension and has been extended to the infinite-dimensional case in \cite{Il83}, see also \cite{CF85}. The case of fractal dimension is more delicate and has been initially treated with some rather annoying extra conditions, see \cite{CF85}. These conditions have been removed by Hunt \cite{Hunt96} for finite-dimensional case, see also \cite{BlI99} for $C^2$ diffeomorphisms in ininite-dimensional case. The result of the theorem in the form stated above appeared in \cite{ChIl04}.

\begin{remark} The proof of this theorem follows the general scheme presented in Theorem \ref{Th5.lad}. Namely, we start with the covering of $\Cal A$ by very small $\eb$-balls and obtain a better covering of the set $S(\Cal A)$ using our assumptions on $S$. Then we iterate this process in order to get the estimate for the dimension. Indeed, since $\eb$ is small and $S$ is smooth enough, the image $S(B_\eb(u_0,\Phi))$ is close to the ellipsoid $S'(u_0)B_\eb(u_0,\Phi)$ and, by the assumptions of the theorem, $d$-dimensional volume of this ellipsoid is strictly less than $d$-dimensional volume of the initial ball. In the case when the Hausdorff dimension is considered, it is more or less straightforward to cover every of such an ellipsoid by smaller balls of {\it different} radii in such a way that the corresponding $d$-dimensional Hausdorff measure will be contracted at every iteration step and this gives the result, see \cite{tem} for the details. The difficulties in the case of fractal dimension are related with the fact that, in contrast to the case of conditions \eqref{5.sq1}, different ellipsoids obtained as images of $\eb$-balls from the initial covering may have essentially different "shapes" (only the volumes are contracted), but we still need to cover them by balls of {\it the same} radius, say $\eb/K$. For this reason, much more delicate estimates are required.
\par
Note that the quasi-differentiability condition is automatically satisfied if the map $S:\Phi\to\Phi$ is $C^1$-smooth. However, in applications it is often easier to verify estimate \eqref{5.quasi} on the attractor only since the attractor is usually more smooth. This is the reason why quasidifferentiability is introduced. Note also that, in contrast to Frechet derivative, the linear operator $S'(u_0)$ satisfying \eqref{5.quasi} may be not unique.
\par
It is worth  mentioning that, in the applications to  PDEs, we normally use the theorem for integer values of $d$ only. However, the possibility to use non-integer $d$s is crucial in many applications to ODEs, e.g. to the Lorentz system.
\end{remark}
Theorem \ref{Th5.vc} can be reformulated in a more elegant way using the fact that the map $S$ can be replaced by any its iteration $S^n$. Indeed, it is not difficult to show using the chain rule that the sequence $\ln\omega_d(S^n,\Cal A)$ is subadditive, so the following limit:
$$
\bar\omega_d(S,\Cal A)=\lim_{n\to\infty}\(\omega_d(S^n,\Cal A)\)^{1/n}=\inf_{n\in\Bbb N}\(\omega_d(S^n,\Cal A)\)^{1/n}
$$
exists for every $d$. In the case of continuous time and a semigroup $S(t)$, we just replace $n\in\Bbb N$ by $t\in\R_+$ and $S^n$ by $S(t)$ everywhere. Let us also recall that the {\it uniform} Lyapunov dimension of the map $S$ on $\Cal A$ is defined as follows:
\begin{equation}\label{5.dim-L}
\dim_L(S,\Cal A):=\sup_{d\in\R_+}\{\bar\omega_d(S,\Cal A)\ge0\},
\end{equation}
see \cites{tem,Zel08} and references therein. Then Theorem \ref{Th5.vc} reads.

\begin{corollary} Let the assumptions of Theorem \ref{Th5.vc} hold. Then
\begin{equation}\label{5.est-fL}
\dim_f(\Cal A,\Phi)\le \dim_L(S,\Phi).
\end{equation}
\end{corollary}
At the next step, we discuss how to evaluate or estimate the Lyapunov dimension. This can be done for dynamical systems with continuous time using the analogues of the Liouville formula for the evolution of volumes. Namely, let us assume that, for any $v_0\in\Phi$ and any $u_0\in\Cal A$ the (quasi)derivative $v(t):=S'(t)(u_0)v_0$ of the map $S(t):\Cal A\to\Cal A$ solves the corresponding equation of variations:
\begin{equation}\label{5.var}
\Dt v(t)=\Cal L(u(t))v(t),\ \ v\big|_{t=0}=v_0,\ \ u(t):=S(t)u_0,
\end{equation}
where $\Cal L(u): D\to\Phi$ are linear unbounded operators such that $D$ is dense in $\Phi$ and the operators $\Cal L(u)$ are bounded from above, i.e.
$$
(\Cal L(u)\eta,\eta)\le C\|\eta\|^2_{\Phi},\ \eta\in D,\ u\in\Cal A,
$$
where $C$ is independent of $\eta$ and $C$. The following lemma is a key technical tool in estimating the Lyapunov dimension.
\begin{lemma} Under the above assumptions, the following identity holds:
\begin{equation}\label{5.liu}
\frac 12\frac d{dt}\|v_1(t)\wedge\cdots\wedge v_d(t)\|^2_{\Lambda^d\Phi}=\operatorname{Tr}(Q(t)\circ\Cal L(u(t)\circ Q(t))\|v_1(t)\wedge\cdots\wedge v_d(t)\|^2_{\Lambda^d\Phi},
\end{equation}
where $v_i(t)$, $i=1,\cdots, d$ are solutions of the equation of variations \eqref{5.var}, $Q(t)$ is the orthoprojector to $d$-dimensional subspace spanned by vectors $\{v_i(t)\}_{i=1}^d$ and $\operatorname{Tr}$ is the trace of a matrix in $\R^d$.
\end{lemma}
This statement is the extension of the classical Liouville theorem to an infinite-dimensional case, see e.g. \cite{tem} for the details. To proceed further, we define the  $d$-dimensional trace of operators $\Cal L(u(t))$, namely,
\begin{equation}\label{5.trace}
\operatorname{Tr}_d(\Cal L(u(t)):=\sup\left\{\sum_{i=1}^d(\Cal L(u(t))\psi_i,\psi_i)\ :\ \psi_i\in D, \ (\psi_i,\psi_j)=\delta_{ij}\right\}.
\end{equation}
Then identity \eqref{5.liu} together with the obvious estimate
$$
\operatorname{Tr}(Q(t)\circ\Cal L(u(t))\circ Q(t))\le\operatorname{Tr}_d(\Cal L(u(t))
$$
gives us the estimate
$$
\(\omega_d(S(T),\Cal A)\)^{1/T}\le e^{2\sup_{u_0\in\Cal A}\{\frac1T\int_0^T\operatorname{Tr}_d(\Cal L(u(t)))\,dt\}}.
$$
This, in turn, gives the corollary of Theorem \ref{Th5.vc} which is suitable for applications, see \cites{ChVi02,tem} for more details.
\begin{corollary}\label{Cor5.vc-ap} Let the above assumptions hold. Let also
\begin{equation}\label{5.qd}
\bar q_d:=\inf_{T\ge0}\left\{\frac1T\sup_{u_0\in\Cal A}\int_0^T\operatorname{Tr}_d(\Cal L(S(t)u_0))\,dt\right\}.
\end{equation}
Assume that $\bar q_{d}<0$ for some $d\in\Bbb N$. Then
\begin{equation}
\dim_f(\Cal A,\Phi)\le\dim_L(S(t),\Cal A)< d.
\end{equation}
\end{corollary}
\begin{remark} Corollary  \ref{Cor5.vc-ap} reduces the problem of estimating the fractal dimension  of the attractor to estimating the traces of linear operators related with the corresponding equation of variations. This bridges the theory of attractors with the classical topic of the operators theory, namely, estimating traces of various differential operators, see \cite{FLW22} and references therein. In particular, in relatively simple cases, such estimates can be obtained using the min-max principle, in more complicated cases (like Navier-Stokes equations or other hydrodynamical equations), an essential progress has been achieved using  Lieb--Thirring inequalities as well as other collective Sobolev inequalities, see \cites{tem,IlZ21,IKZ} and references therein.
\par
Note also that, although the definition of the Lyapunov dimension is independent of the concrete choice of the equivalent inner product in the space $\Phi$,  the clever choice of this inner product (which may depend on the point of the phase space), may essentially improve the estimates, see e.g. \cites{Ghi88,SZ16} for estimates related with damped Schroedinger and hyperbolic Cahn-Hilliard equations respectively) as well as \cite{LKKK16} for calculating the {\it exact} value of the Lyapunov dimension of the Lorentz attractor. We will not go into more details here and restrict ourselves to the example related with damped wave equations which is important for the discussion in the next subsection.
\end{remark}
\begin{example}\label{Ex5.hyp} Let us consider an abstract damped wave equation in a Hilbert space $\Phi$:
\begin{equation}\label{5.hyp}
\Dt^2 u+\gamma\Dt u+A u=F(u,\Dt u),\ \ \xi_u\big|_{t=0}=\xi_0,\ \ \xi_u(t):=\{u(t),\Dt u(t)\},
\end{equation}
where $A: D(A)\to \Phi$ is a positive self-adjoint linear operator with compact inverse, $\gamma>0$ is a given dissipation coefficient and $F$ is a given nonlinearity. We define a scale $\Phi^s:=D(A^{s/2})$, $s\in\R$, of Hilbert spaces associated with the operator $A$ and the associated scale of energy spaces $E^s:=\Phi^{s+1}\times\Phi^s$. For simplicity, we assume that $F$ is smoothing and bounded, i.e.
\begin{equation}\label{5.inf-good}
F\in C^\infty_b(E^{-m},\Phi^m),\ \ \forall m\in\Bbb N.
\end{equation}
Then it is straightforward to prove that problem \eqref{5.hyp} is globally well posed in the energy phase space $E:=E^0$, the corresponding solution semigroup $S(t):E\to E$ is dissipative with respect to the standard bornology of bounded sets in $E$ and possesses an attractor $\Cal A$ in $E$, see e.g. \cite{TZ03} for more details.
\par
Our main task here is to estimate the fractal dimension of this attractor and its dependence on a small parameter $\gamma$. Since we have extremely strong assumptions \eqref{5.inf-good} on the nonlinearity $F$, it is also immediate to verify that the solution semigroup $S(t)$ is $C^\infty$-smooth in $E$ with respect to the initial data and that the Frechet derivative $\xi_v(t):=S'(\xi_0)(t)\xi_{v_0}$ is defined as a solution of the following equation of variations:
\begin{equation}\label{5.hypvar}
\Dt^2 v+\gamma\Dt v+Av=F'_u(\xi_u(t))v+F'_{\Dt u}(\xi_u(t))\Dt v,\ \ \xi_v\big|_{t=0}=\xi_{v_0}, \ \xi_u(t):=S(t)\xi_{u_0},
\end{equation}
so we only need to estimate $d$-dimensional traces of the operator
\begin{equation}\label{5.cl-hyp}
\Cal L(\xi_u(t)):=\(\begin{matrix} 0&1\\-A+F'_u&-\gamma +F'_{\Dt u}\end{matrix}\)
\end{equation}
in the space $E$. However, the volume contraction scheme does not work properly in the initial inner product in $E$ induced by the Cartesian product and we need to modify it, namely, we set
$$
(\{v,v'\},\{w,w'\})_\gamma:=(Av,w)+(v',w')+\frac\gamma2((v,w')+(v',w)),\ \ \{v,v'\},\{w,w'\}\in E.
$$
Then this inner product is equivalent to the standard one if $\gamma>0$ is small enough and
$$
(\Cal L\xi_v,\xi_v)_\gamma=-\frac\gamma2\(\|v\|^2_{\Phi^1}+\|v'\|^2_{\Phi}\)+(F'_{\Dt u}v',v')+(F'_uv,v')+\frac{\gamma}2\((F'_uv,v)+(F'_{\Dt u}v',v)-\gamma(v,v')\).
$$
To estimate the terms in the right-hand side containing $F$, we use assumption \eqref{5.inf-good}. For instance, for the first term we have
$$
|(F'_{\Dt u}v',v')|\le \|F'_{\Dt u}v'\|_{\Phi^m}\|v'\|_{\Phi^{-m}}\le C\|v'\|_{\Phi^{-m}}^2=C(A^{-m}v',v')
$$
and estimating the remaining terms analogously, for sufficiently small $\gamma>0$, we arrive at
$$
(\Cal L(\xi_u(t))\xi_v,\xi_v)_\gamma\le (\Bbb L_m\xi_v,\xi_v)_\gamma,\ \ \  \Bbb L_m:=\(\begin{matrix}-\frac\gamma4+C_mA^{-m}&0\\0&-\frac\gamma4+C_mA^{-m}\end{matrix}\),
$$
where $m\in\Bbb N$ is arbitrary and the constant $C_m$ is independent of $\xi_v=\{v,v'\}$ and $\xi_u$, see \cite{TZ03}.  Then, using the min-max principle, we finally arrive at
$$
\operatorname{Tr}_d(\Cal L(\xi_u(t)))\le\operatorname{Tr}_d(\Bbb L_m)=2\(-\frac\gamma4 d+C_m\sum_{n=1}^d\lambda_n^{-m}\),
$$
where $\{\lambda_n\}_{n=1}^\infty$ are the eigenvalues of the operator $A$ enumerated in the non-decreasing order. If we now assume that
$$
\sum_{n=1}^\infty\lambda_n^{-m}<\infty
$$
for some big $m$ (which is always true if $A$ is a uniformly elliptic differential operator in a bounded domain),  estimate \eqref{5.qd} will give us that
$$
\bar q_d\le -\frac\gamma4 d+\bar C
$$
and then Corollary \ref{Cor5.vc-ap} gives us the final result that the fractal dimension of the attractor $\Cal A$ associated with the equation \eqref{5.hyp} possesses the following upper bound:
\begin{equation}\label{5.attr-frac}
\dim_f(\Cal A,E)\le\dim_L(S(t),\Cal A)\le \frac C\gamma,
\end{equation}
where $\gamma$ is small enough and $C$ is independent of $\gamma$, see \cite{TZ03} for details.
\end{example}

\subsection{Lower bounds for the dimension} Most part of lower bounds for
 the attractors dimensions available in the literature are based on the estimates of the instability index of a properly constructed equilibrium. Assume that we are given a dissipative semigroup $S(t):\Phi\to\Phi$ acting on a Banach space $\Phi$ with the standard bornology $\Bbb B$ which consists of all bounded sets of $\Phi$. Assume also that $u_0$ is an equilibrium of this semigroup and let us consider the unstable set of this equilibrium:
 \begin{equation}\label{5.uns}
 \Cal M_+(u_0):=\{v_0\in\Phi, \ \exists \ u\in\Cal K \ \text{\ such that\ }\ u(0)=u_0,\ \lim_{t\to-\infty}u(t)=u_0\},
 \end{equation}
 where $\Cal K$ is a set of complete bounded trajectories of $S(t)$. Note that, due to dissipativity, it is enough to construct a negative semi-trajectory $u(t)$, $t\le0$ with the above properties. Then it follows from the definition of an attractor that
 \begin{equation}\label{5.emb-atr}
 \Cal M_+(u_0)\subset\Cal A
 \end{equation}
 if the attractor $\Cal A$ exists. Thus, the fractal dimension of the attractor $\Cal A$ is estimated in terms of the dimension of the unstable set $\Cal M_+(u_0)$ which is usually a submanifold by itself or contains a submanifold of  a sufficiently big dimension.
\par
Indeed, let us assume that $S(t)$ is Frechet differentiable near the equilibrium point $u_0$ and the Frechet derivative $S'(u_0)$ is given by the equation of variations \eqref{5.var} and is uniformly continuous in a small neighbourhood of $u_0$. Assume also that the part of the spectrum $\sigma(\Cal L(u_0))$ belonging to the positive semi-plane $\Ree\lambda>0$ consists of finitely many eigenvalues of finite algebraic multiplicity. Denote by $\Cal N_+(u_0)\in\Bbb N$ the number of all such eigenvalues taking into account their multiplicity. Then the standard theorem about unstable manifolds, see e.g. \cite{BV92}, claims that the set $\Cal M_+(u_0)$ contains an $\Cal N_+(u_0)$-dimensional local submanifold which is generated by all complete trajectories which approach $u_0$ as $t\to-\infty$ exponentially fast. Thus, we have
\begin{equation}\label{5.lower-bound}
\dim_f(\Cal A,\Phi)\ge\dim_f(\Cal M_+(u_0),\Phi)\ge \Cal N_+(u_0).
\end{equation}
Thus, obtaining lower bounds for the dimension of an attractor is also reduced to the classical problem  of
of the spectral theory, namely, finding/estimating the number of unstable eigenvalues of a given differential operator. Of course, the above described scheme can be used not only in the case where $u_0$ is an equilibrium. It may be something more complicated, e.g. periodic or quasi-periodic or even chaotic orbit. The theory of unstable manifolds for such objects is well-developed nowadays, see e. g. \cites{BV92,KH95} and references therein. However, such objects are almost never used  in the literature for estimating the attractor dimension since it is extremely difficult to find them and estimate explicitly their instability index. An exception from this, which is based on the theory of homolcinic bifurcations and as believed has a general nature, will be considered in examples below.
\par
We also mention that the lower bounds for the Lyapunov dimension of the attractor can be obtained via equilibria and the obvious inequality
$$
\dim_L(S(t),\Cal A)\ge\dim_L(S(t),u_0).
$$
Important that the Lyapunov dimension $\dim_L(S(t),u_0)$ also can be estimated via the spectrum of the infinitesimal generator $\Cal L(u_0)$ of the linear semigroup $S'(t)(u_0)$. To this end, we need the following definition.
\begin{definition}\label{Def5.Lyap-exp} Let $\Phi$ be a Hilbert space and $\Cal L(u_0): D\to\Phi$ be a linear closed bounded from above operator in $\Phi$. Let also $\mu_\infty:=\sup\Ree\sigma_{ess}(\Cal L(u_0))$. Then we have only finite or countable number of eigenvalues $\lambda_j$ of $\Cal L(u_0)$ which belong to the semi-plane $\Ree\lambda>\mu_\infty$, which can be enumerated in the non-increasing order. Then we define the Lyapunov exponents $\mu_n$, $n\in\Bbb N$, as follows: $\mu_n=\Ree\lambda_n$ if there exists at least $n$ eigenvalues satisfying $\Ree\lambda>\mu_\infty$ and $\mu_n=\mu_\infty$ otherwise. Then the Lyapunov dimension of $\Cal L(u_0)$ is defined as follows
\begin{equation}
\dim_L(\Cal L(u_0)):=\sup_{d\in\R_+}\{\bar\mu_d\ge0\},
\end{equation}
where $\bar\mu_s=\sum_{n=1}^s\mu_n$ for integer $s$ and $\bar\mu_s=\bar\mu_{[s]}+(s-[s])\mu_{[s]+1}$, see e.g. \cite{tem} for more details.
\end{definition}
 The next proposition gives the desired estimate.
 \begin{proposition}\label{Prop5.a-l} Let the assumptions of Theorem \ref{Th5.vc} hold and let $u_0\in\Cal A$ be an equilibrium. Then the following estimates hold:
 \begin{equation}\label{5.est-main-dim}
 \begin{cases} 1.\ \ \Cal N_+(u_0)\le \dim_f(\Cal A,\Phi)\le\dim_L(S(t),\Cal A),\\  2. \ \dim_L(\Cal L(u_0))\le\dim_L(S(t),u_0)\le\dim_L(S(t),\Cal A).
 \end{cases}
 \end{equation}
 \end{proposition}
 Indeed, only the left-hand side of the second inequality is not proved yet. Its proof follows in a straightforward way from the definitions and the spectral mapping theorem, see \cite{Cle87} for more details. Note that, in contrast to the finite-dimensional case, this inequality may be strict if $\Phi$ is infinite-dimensional, since we have only one-sided spectral mapping theorem for the essential spectrum.
 \par
 We now turn to examples.
 \begin{example} We return once more to the simplest Example \ref{Ex1.1D-par}. We have already proved that the fractal dimension of the corresponding attractor $\Cal A$ is finite in $\Phi=L^2(0,\pi)$ and satisfies the upper bound \eqref{5.ent-est}. We now apply the volume contraction method and Proposition \ref{Prop5.a-l} to get more explicit upper and lower bounds. We first note that the corresponding solution semigroup is $C^\infty$-smooth here, see e.g. \cite{BV92}, so all smoothness assumptions are automatically satisfied and we only need to get good upper and lower bounds for the Lyapunov dimension and the instability indexes.
 \par
 The operator $\Cal L(u(t))$ for the equation of variations now reads $\Cal L(u)=\partial_x^2-3u^2+a$ and the eigenvalues of the operator $-\partial_x^2$ with the Dirichlet boundary conditions are $\lambda_n=n^2$. Therefore, by the min-max principle, for integer values of  $d$ we have
 $$
 \operatorname{Tr}_d(\Cal L(u))\le \operatorname{Tr}_d(\partial_x^2+a)=-\sum_{n=1}^d\lambda_n+ad=-\frac{d(d+1)(2d+1)}6+ad\ge -\frac{d^3}3-ad.
 $$
We recal that $u_0=0$ is an equilibrium and the first inequality in the above formula is attained at this equilibrium. This gives us the equiality
\begin{equation}\label{5.sharp-dl}
\dim_L(S(t),\Cal A)=\dim_L(S(t),0)=\dim_L(\Cal L(0))=\dim_L(\partial_x^2+a),
\end{equation}
where the second inequality is due to the fact that the spectral mapping theorem holds for self-adjoint operators. Note that the right-hand side of this equality is easy to compute explicitly, so the Lyapunov dimension of the attractor $\Cal A$ is not only estimated, but also explicitly computed. We will not write out this formula, but only mention that, for large $a$ we have
\begin{equation}
\dim_H(\Cal A,\Phi)\le\dim_F(\Cal A,\Phi)\le\dim_L(S(t),\Cal A)\sim \sqrt{3a}.
\end{equation}
Let us now look to the lower bounds for Hausdorff and fractal dimensions of the attractor $\Cal A$. We see that the instability index $\Cal N_+(0)\sim\sqrt{a}$. Moreover, since the considered equation possesses a global Lyapunov function, the corresponding attractor is a union of unstable sets of all equilibria,~i.e.
$$
\dim_H(S(t),\Cal A)\sim\max_{u_0\in\Cal R}\Cal N_+(u_0)=\Cal N_+(0)\sim \sqrt{a}.
$$
So, the Hausdorff dimension of the attractor also can be explicitly computed. In contrast to this, for the fractal dimension, we only have two-sided inequality
\begin{equation}\label{5.lower}
\sqrt{a}\sim \Cal N_+(0)\sim\dim_H(\Cal A,\Phi)\le\dim_f(\Cal A,\Phi)\le\dim_L(S(t),\Cal A)=\dim_L(\Cal L(0))\sim\sqrt{3a}.
\end{equation}
Note that even in this simplest example, the fractal dimension may a priori be larger than the Hausdorff one due to nontrivial intersections of stable and unstable manifolds of the equilibria, so we basically do not have anything more than two-sided bounds \eqref{5.lower}.
 \end{example}
 \begin{remark} Note that the positivity of the Lyapunov dimension $\dim_L(S(t),\Cal A)$ does not imply here the presence of chaotic dynamics of the considered system (in our example the system is gradient and the dynamics is trivial). This should not be misleading, since in contrast to the classical dynamics, where the Lyapunov exponents for {\it individual} trajectories (which exist for almost all trajectories with respect to the "physical measure") are usually considered, see e.g. \cite{KH95}, we consider here the so-called {\it uniform} Lyapunov exponents, which may be positive even for trivial gradient dynamics. Note also that the coincidence of the Lyapunov dymension of the whole attractor with the dimension of an equilibrium is not restricted to gradient systems. For instance, exactly the same happens (for some different reasons) for the case of chaotic Lorentz attractor considered in Example \ref{Ex1.lor}.
\par
We notice as well that, in the example considered above, the instability index $\Cal N_+(u_0)$ and the Lyapunov dimension $\Cal L(u_0)$ have the same asymptotic behavior as $a\to\infty$ (up to the non-essential multiplier $\sqrt3$), which allowed us to give sharp asymptotic behavior of $\dim_f(\Cal A,\Phi)$ as $a\to\infty$. This is somehow typical for parabolic equations but can be easily not true e.g. for hyperbolic equations.
\end{remark}
\begin{example}
 In particular, it is not true for the equation considered in Example \ref{Ex5.hyp}, to which we return now.  Indeed, due to assumptions \eqref{5.inf-good}, the eigenvalues $\nu_F$ of the linearized operator \eqref{5.cl-hyp} at any equilibrium $u_0$ are asymptotically  very close to the eigenvalues $\nu_0$ of the operator, which corresponds to $F=0$, namely,
$$
|\nu_{F,k}-\nu_{0,k}|\le C_N k^{-N}
$$
for all $N\in\Bbb N$, see \cite{TZ03}. The eigenvalues which correspond to $F=0$ are easy to compute:
$$
\nu^{\pm}_{0,k}=\frac{-\gamma\pm\sqrt{\gamma^2-4\lambda_k}}2.
$$
Thus, $\Cal N_+(u_0)\le C_N\gamma^{1/N}$, see \cite{TZ03} for any $N$ and, therefore, we cannot get sharp in $\gamma$ lower bound for \eqref{5.attr-frac} based on the instability index of any equilibrium. In contrast to this, if we look at the Lyapunov dimension of the equilibrium $u_0$, we see that the essential spectrum lies at the line $\Ree \lambda=-\frac\gamma2$ and consequently $\mu_d\ge\mu_\infty\ge-\frac\gamma2$ for all $d$. For this reason, if we find an equilibrium $u_0$ in such a way that it will have an unstable eigenvalue with $\Ree\lambda_1\sim O(1)>0$ as $\gamma\to0$, we will get the estimate
$$
\dim_L(\Cal L(u_0))\ge \frac C{\gamma}.
$$
Therefore, $\dim_L(S(t),\Cal A)\sim \gamma^{-1}$, so we can not get better upper bounds for the fractal dimension:
\begin{equation}\label{5.non-opt}
C_1\le \dim_f(\Cal A,E)\le C_2\gamma^{-1}.
\end{equation}
To overcome this difficulty and get sharp upper and lower bounds for the fractal dimension of $\Cal A$, a new method for lower bounds has been suggested in \cite{TZ03} based on the homoclinic bifurcation result proved in \cite{Tur96}. Roughly speaking, this result tells us that: if you are given a multi-dimensional system with an equilibrium $u_0=0$ and a homoclinic loop $u(t)$ to it such that $u(t)\sim e^{\mp\lambda}$ as $t\to\pm\infty$ for some $\lambda>0$ and you have sufficiently many  "intermediate" eigenvalues of the linearizarion matrix at $u_0=0$:
\begin{equation}\label{5.good-spec}
-\lambda<\Ree\lambda_1\le \cdots\le\Ree\lambda_N<\lambda,
\end{equation}
then, by arbitrarily smooth small perturbation of the system, we may born an invariant torus $\Bbb T$ whose dimension is proportional to the Lyapunov dimension of the equilibrium $u_0=0$. In other words, the dimension of an invariant manifold, which can be born near a homoclinic loop, is restricted by the Lyapunov dimension of the origin of this loop only (of course, under some natural extra assumptions on the loop).
\par
Finally, if we embed such a construction to the attractor $\Cal A$ of \eqref{5.hyp}, we end up with the desired sharp lower bound:
$$
\dim_f(\Cal A,\Phi)\ge C\gamma^{-1}.
$$
Exactly this approach has been realized in \cite{TZ03}. Namely, let us consider the following model decoupled system of ODEs
\begin{equation}\label{5.hyp-mod}
u_0''=u_0-u_0^3,\ \ u_n''+\gamma u_n'+\lambda_n u_n=0,\ \ n=1,\cdots,
\end{equation}
where $\gamma>0$ is small enough,  and $\gamma^2-4\lambda_n>0$. Then this system has a zero equilibrium and a homoclinic loop $\{u_0(t),0\}$ to it. Moreover, $u_0(t)\sim e^{-|t|}$ as $t\to\pm\infty$ and the remaining eigenvalues (which correspond to equations for $u_n$, $n=1,\cdots$ are
$$
\mu_n^{\pm}=\frac{-\gamma\pm\sqrt{\gamma^2-4\lambda_n}}2.
$$
Thus, $\Ree\mu_n^{\pm}=-\gamma/2$ and conditions \eqref{5.good-spec} are satisfied with $\lambda=1$ if $\gamma$ is small enough. It is proved in \cite{TZ03} that for every $m\in\Bbb N$ and every $\eb>0$ and $\gamma>0$ small enough, there exists a perturbation $F_i$, $\|F_i\|_{C^m}\le\eb$, such that the perturbed system
 \begin{equation}\label{5.perturb}
 u_0''=u_0-u_0'+F_0(u_0,u_0',u,\Dt u),\ \ u_n''+\gamma u_n'+\lambda_nu_n=F_n(u_0,u_0',u_1,u_1',\cdots),\ \ n=1,\cdots
 \end{equation}
 possesses a $C\gamma^{-1}$-dimensional invariant torus in a small neighbourhood of the unperturbed homoclinic loop. Moreover, only finitely many ($n_0=C\gamma^{-1}$) modes are really perturbed and the perturbation depends on $u_n$, $n\le n_0$ only.
 \par
 It remains to note that equations \eqref{5.perturb} can be easily embedded to the system of the form \eqref{5.hyp}, where the $C^m$-norm of the nonlinearity $F$ is uniformly bounded as  $\gamma\to0$. This shows that the upper bounds $\dim_f(\Cal A,E)\sim C\gamma^{-1}$ are indeed sharp.
 \end{example}
 \begin{remark} Note that, in general, we need a large number of parameters (of order proportional to $\dim_L(S(t),u_0)$) in order to be able to born the torus and use this scheme for constructing examples with big fractal dimension of the attractor. We expect that this method will be helpful for other types of equations of hyperbolic type, e.g. for damped Euler equations and their various regularizations.
 \end{remark}

\section{Inertial manifolds and finite-dimensional reduction}\label{s6}

In the previous section, we have constructed a reduced system of ODEs (IF of the initial PDE), which captures the limit dynamics on the attractor $\Cal A$, based on the finiteness of its fractal dimension and the Man\'e projection theorem, see \eqref{5.IF}. We have also pointed out the main drawback of such a reduction, namely, a drastic loss of smoothness. In this section, we discuss an alternative approach to the finite-dimensional reduction based on the center manifolds theory (or more general, on the theory of normally-hyperbolic invariant manifolds), which requires stronger assumptions on the considered DS, but gives much more appropriate construction of the IF suitable for applications. To the best of our knowledge, this approach has been suggested by Man\'e \cite{Man77} and has become very popular after the paper of Foias, Sell and Temam \cite{FST88}, see also \cites{An22,Mik91,Ro94,RT96,Zel14} and references therein.
\par
 Following this approach, the main object of the theory is not an attractor, but a globally stable  finite-dimensional invariant submanifold of the initial phase space, which is typically normally hyperbolic and, for this reason, is exponentially attractive. The reduced IF in this situation is
nothing else than the restriction of the initial PDE to this invariant manifold and is typically as smooth as the manifold. Keeping in mind possible applications to hydrodynamics and the hope to understand turbulence, this manifold has been referred as {\it inertial} manifold in relation with the so-called inertial scale in the conventional theory of turbulence, see e.g. \cite{firsh} and references therein. We start with the formal definition of an inertial manifold (IM).

\begin{definition}\label{Def6.IM} Let $S(t):\Phi\to\Phi$ be a DS acting in a Banach space $\Phi$. Then a strictly  invariant finite-dimensional Lipschitz submanifold $\Cal M$ of the phase space $\Phi$ is an IM for the DS $S(t)$ if it possesses an exponential tracking property (asymptotic phase) in the following form: for any semi-trajectory $u(t)=S(t)u_0$, $t\ge0$, of the considered DS, there exists a semitrajectory $\bar u(t)=S(t)\bar u_0$ with $\bar u_0\in\Cal M$ such that
\begin{equation}\label{5.track}
\|u(t)-\bar u(t)\|_\Phi\le Q(\|u_0\|_\Phi)e^{-\alpha t},\ \ t\ge0,
\end{equation}
where the positive constant $\alpha$ and a monotone function $Q$ are independent of $u_0$ and $t$.
\end{definition}
\begin{remark} In applications, an IM is often constructed as a graph of a Lipschitz function. Namely, let us assume that $\Phi=\Phi_+\oplus\Phi_-$ is presented as a direct sum of two Banach spaces (the corresponding projectors are denoted by $\Pi_+$ and $\Pi_-$ respectively) where $\dim\Phi_+=N$. Then $u(t)=\Pi_+u(t)+\Pi_-u(t)=u_+(t)+u_-(t)$. The variables $u_+(t)$ and $u_-(t)$ are treated as "slow" and "fast" ones respectively and the manifold $\Cal M$ slaves the fast variables $u_-(t)$ to the slow ones $u_+(t)$, i.e. $u_-(t)=M(u_+(t))$ for some at least Lipschitz function $M:\Phi_+\to\Phi_-$, so  $\Cal M$ is treated as a graph of this function:
\begin{equation}\label{6.graph}
\Cal M:=\{u_++M(u_+),\ \ u_+\in \Phi_+\}.
\end{equation}
We note from the very beginning that, in applications, one usually constructs an IM not for the initial PDE, but for the properly modified one (the so-called "prepared" equation), whose solutions have the same asymptotic behaviour, but it has better structure, e.g. its nonlinearity is globally Lipschitz continuous. Since an IM is a kind  of a (global) center manifold, it is not unique in general and we need to cut-off the nonlinearity properly to restore the uniqueness and the possibility to find it via the Banach contraction theorem.  In contrast to ODEs, this preparation procedure may be very delicate: sometimes it is enough just to cut-off the nonlinearity outside of the proper absorbing set making it globally Lipschitz (this usually works when the spectral gap conditions are satisfied), see \cite{FST88}, but in other cases, you may need to embed your initial PDEs into a larger system of PDEs, do some diffeomorphisms to restore the spectral gap conditions, change the leading order differential operator, etc., see \cites{MPS88,KZ15,KZ18,K18,KLSZ}. In addition, if you want to improve the smoothness of an IM, you need to change the vector field at the points of the attractor as well (in an accurate way in order not to affect the dynamics on the attractor, but kill the resonances), see \cite{KZ22}. We will not go into more details here and will always assume from now on that we already made some of these "preparations" and the nonlinearity is already globally Lipschitz.
\end{remark}

\subsection{Spectral gap conditions and inertial manifolds}
 We start our exposition with the classical theory of IMs, which usually deals with a semilinear equation of the form:
\begin{equation}\label{6.semPDE}
\Dt u+Au=F(u), \ \ u\big|_{t=0}=u_0\in\Phi.
\end{equation}
We assume for simplicity that $\Phi$ is a real Hilbert space, $A:D(A)\to\Phi$ is a positive self-adjoint linear operator with compact inverse and $F:\Phi\to\Phi$ is a given nonlinearity which is Lipschitz continuous with Lipschitz constant $L$:
\begin{equation}\label{6.lip}
\|F(u_1)-F(u_2)\|_{\Phi}\le L\|u_1-u_2\|_\Phi,\ \ u_1,u_2\in\Phi.
\end{equation}
From the Hilbert-Schmidt theorem, we know that the operator $A$ possesses a complete base of eigenvectors $\{e_i\}_{i=1}^\infty$ with the corresponding eigenvalues $\{\lambda_i\}_{i=1}^\infty$, which we enumerate in the non-decreasing order. The key result of the theory is the following theorem.
\begin{theorem}\label{Th6.IM-main} Let the nonlinearity $F$ satisfy \eqref{6.lip} and let $N\in\Bbb N$ be such that the following spectral gap condition be satisfied:
\begin{equation}\label{6.sg}
\lambda_{N+1}-\lambda_N>2L.
\end{equation}
Then equation \eqref{6.semPDE} possesses an $N$-dimensional  IM which can be presented as
a graph of a Lipschitz continuous function $M:\Phi_+\to\Phi_-$, where $\Phi_+$ is spanned by the first $N$ eigenvectors of $A$ and $\Phi_-$ is an orthogonal complement of $\Phi_+$. Moreover, if $F\in C^{1+\eb}(\Phi,\Phi)$, for a sufficiently small $\eb>0$, then the corresponding map $M$ is also $C^{1+\eb}$-smooth. We also mention that the attraction exponent $\alpha$ in the definition of an IM can be chosen as $\alpha=\frac{\lambda_N+\lambda_{N+1}}2$.
\end{theorem}
\begin{proof}[Idea of the proof] The construction of the manifold is based on perturbation arguments where the nonlinearity $F$ is considered as a perturbation. Indeed, for $F\equiv0$, we have $\Cal M=\Phi_+$, all trajectories on this manifold grow as $t\to-\infty$ not faster than $e^{-\theta t}$ where $\theta\in(\lambda_N,\lambda_{N+1})$ and this is a determining property for the manifold $\Cal M$.
\par
This observation guesses the way how to construct the desired IM $\Cal M$, namely, we need to solve the problem
\begin{equation}\label{6.IM-eq}
\Dt u+Au=F(u),\ \ t\le0,\ \ \Phi_+u\big|_{t=0}=u_+\in\Phi_+
\end{equation}
backward in time in the weighted space $L_{e^{\theta t}}(\R_-,\Phi)$ where $\theta\in(\lambda_N,\lambda_{N+1})$ is chosen in an optimal way ($\theta=\frac{\lambda_{N}+\lambda_{N+1}}2$ in our case). Then the desired map $M$ will be found as follows: $M: u_+\to\Pi_-u(0)$, see \cite{Zel14} for more details. Note that the idea to use the space $L^2$ in time here belongs to Miklavcic \cite{Mik91}. The usage of more natural at first glance space $C_{e^{\theta t}}(\R_-,\Phi)$ leads to the extra multiplier $\sqrt{2}$ in the right-hand side of \eqref{6.sg}.
\par
Following \cites{Mik91,Zel14} (see also \cite{FST88}), we solve equation \eqref{6.IM-eq} by Banach contraction theorem and the most important step here is to find the norm of the solution operator $\Cal L_\theta: h\to v$ for the following linear problem on the whole line $t\in\R$:
\begin{equation}\label{6.lin}
\Dt v+Av=h(t),\ \  h\in L^2_{e^{\theta t}}(\R,\Phi)
\end{equation}
in the weighted space $L^2_{e^{\theta t}}(\R,H)$, namely, to verify that
\begin{equation}\label{6.main-est}
\|\Cal L_\theta\|_{\Cal L(L^2_{e^{\theta t}},L^2_{e^{\theta t}})}=\max\bigg\{\frac1{\theta-\lambda_N},\frac1{\lambda_{N+1}-\theta}\bigg\}.
\end{equation}
The proof of this estimate is very elementary since the problem is reduced to the analogous one for the scalar ODEs
$$
v_n'+\lambda_n v_n=h_n(t),
$$
which are equations on the Fourier amplitudes for \eqref{6.lin}. An elementary calculation shows that
 \begin{equation}\label{6.ode-est}
 \|\Cal L_{\theta,n}\|_{\Cal L(L^2_{e^{\theta t}},L^2_{e^{\theta t}})}=\frac1{|\lambda_n-\theta|}
 \end{equation}
 and the Parseval equality gives us that $\|\Cal L_\theta\|=\max_{n\in\Bbb N}\|\Cal L_{\theta,
 n}\|$ (here we see the advantage of using the weighted $L^2$-spaces). This gives the desired estimate \eqref{6.main-est}.
 \par
 The rest of the proof is also straightforward. We just invert the linear part of \eqref{6.IM-eq} (after the proper restriction to the negative semi-axis $t\le0$) and apply the Banach contraction theorem. The optimal exponent $\theta=\frac{\lambda_N+\lambda_{N+1}}2$ gives us the value $\frac2{\lambda_{N+1}-\lambda_N}$ for the norm of $\Cal L_\theta$, so the Lipschitz constant of the composition $\Cal L_\theta$ with $F$ will not exceed $\frac{2L}{\lambda_{N+1}-\lambda_N}<1$ due to the spectral gap condition, so the constructed map is indeed a contraction. The exponential tracking property is also an almost immediate corollary of \eqref{6.main-est} and the Banach contraction theorem, see \cite{Zel14} for the missed details.
\end{proof}
\begin{remark} The spectral gap condition \eqref{6.sg} can be generalized to the case where the operator consumes smoothness. Namely, assume instead of \eqref{6.lip} that the operator $F$ is globally Lipschitz as the map from $\Phi$ to $\Phi^{-s}:=D(A^{-s/2})$, for some $s\in(0,2)$ with the same Lipschitz constant $L$. Then the analogue of \eqref{6.sg} reads
\begin{equation}\label{6.sg1}
\frac{\lambda_{N+1}-\lambda_N}{\lambda_N^{s/2}+\lambda_{N+1}^{s/2}}>L
\end{equation}
and this condition is sharp (moreover, $s<0$ is also possible under some natural extra assumptions), see \cite{Zel14}. In contrast to this, very few is known about sharp spectral gap conditions for the case where the operator $A$ is not self-adjoint. For instance, in the model case of a coupled system of two equations in $H=\Phi\times\Phi$ where the leading operator $\Bbb A$ has Jordan blocks:
$$
\Bbb A=\(\begin{matrix}1&1\\0&1\end{matrix}\)A
$$
and $F: H\to H$ is globally Lipschitz, the sharp spectral gap conditions read
\begin{equation}\label{6.sg-kwa}
\frac{(\lambda_{N+1}-\lambda_N)^2}
{\lambda_{N+1}+\lambda_N+2\sqrt{\lambda_N^2-\lambda_N\lambda_{N+1}+\lambda_{N+1}^2}}>L,
\end{equation}
see \cite{KZ221}. Such systems naturally appear after the so-called Kwak transform applied to, say, Navier-Stokes equations, see \cite{kwak92}. We see that  \eqref{6.sg-kwa} differs drastically from the self-adjoint case \eqref{6.sg} and are close to \eqref{6.sg1} with $s=1$. Thus, the presence of a Jordan block in the leading linear part of the equation is somehow equivalent to consuming one unit of smoothness by the nonlinearity (with self-adjoint linear part). Exactly this fact has been overseen in the erroneous construction of Kwak \cite{kwak92} (see also \cites{kwak92a,IO84,TW93}) of an IM for 2D Navier-Stokes problem.
\par
We mention here the paper \cite{CKZ19} (see also \cite{CGV05} and \cite{MS89} for weaker results) where the hyperbolic relaxation
$$
\eb\Dt^2 u+\Dt u+Au=F(u),\ \ \eb>0
$$
of \eqref{6.semPDE} is considered. It is shown there that the sharp spectral gap condition for this problem coincides with \eqref{6.sg} (where $\lambda_n$ are the eigenvalues of $A$) and is independent of $\eb$ if it is small enough. Inertial manifolds for elliptic boundary problems in cylindrical domains are studied in \cites{Ba95,Mie94}. We also mention that the necessary conditions for the existence of IMs are often formulated in the spirit of the theory of DS in terms of  invariant cones. We will not present the details here and refer the interested reader to \cites{An22,Ro94,Zel14} for more details.
\end{remark}
Let us discuss some examples, where the classical theory is applicable (more examples can be found in the survey \cite{Zel14}, see also \cite{tem} and references therein). We start with the system of reaction-diffusion equations in a bounded domain $\Omega$ of $\R^d$:
\begin{equation}\label{6.RDS}
\Dt u=a\Dx u-f(u), \ \ u=(u^1,\cdots,u^m)
\end{equation}
endowed with the proper boundary conditions. We assume that the diffusion matrix $a$ is self-adjoint and positive definite and $f$ is globally Lipschitz. Then the operator $A:=-a\Dx$ is self-adjoint and positive (non-negative in the case of periodic or Neumann boundary conditions) in $\Phi=[L^2(\Omega)]^m$ and, due to the Weyl theorem, we have
\begin{equation}\label{6.weyl}
\lambda_n\sim C_m n^{2/d},
\end{equation}
so the validity of spectral gap conditions strongly depends on the dimension $d$. When $d=1$, we have infinitely many $N$s such that $\lambda_{N+1}-\lambda_N\ge cN$ for some positive $c$, so we have spectral gaps of any size and, therefore, for any Lipschitz constant $L$ of $f$, we may find $N$ satisfying the spectral gap condition and this guarantees the existence of an IM.
\par
The case $d=2$ is more interesting. In this case, the Weyl theorem can guarantee only that $\lambda_{N+1}-\lambda_N\ge c$ for infinitely many values of $N$, so we may apply Theorem  \ref{Th6.IM-main} only if $L$ is small enough. However, we may still have spectral gaps of any size in the spectrum of the Laplacian despite that $\lambda_n\sim Cn$. For instance, for $\Omega=[-\pi,\pi]^2$ with periodic boundary conditions, we have infinitely many values of $N$ satisfying
$$
\lambda_{N+1}-\lambda_N\ge C\log\lambda_N,
$$
see \cite{Ri82}, so an IM exists for the case of periodic boundary conditions. The ideal situation from the point of view of IMs is the case of Laplace-Beltrami operator on the $d$-dimensional sphere $\Bbb S^d$, where the equation $\lambda_{N+1}-\lambda_N\ge c\lambda_N^{1/2}$ has infinitely many solutions  in any space dimension~$d$.  Thus, reaction-diffusion equations of the form \eqref{6.RDS} on a sphere always have IMs in any space dimension. Note that, to the best of our knowledge, the problem of existence of spectral gaps of any size for 2D Laplacian in a bounded domain $\Omega\subset\R^2$ is completely open. On the one hand, we do not know any examples where such gaps do not exist and, on the other hand, we do not know any reasonably general classes of domains which possess this property.
\par
In contrast to this, in dimension three or higher, spectral gaps of any size exist in very exceptional cases only (like sphere), so the classical IMs theory is not very helpful for 3D reaction-diffusion equations. Nevertheless, spectral gap conditions are still satisfied for higher order equations like the Swift-Hohenberg one:
$$
\Dt u+(\Dx+1)^2u=f(u),
$$
since for bi-Laplacian the Weyl theorem gives $\lambda_n\sim c n^{4/3}$ and this guarantees the existence of spectral gaps of any size.
\par
Let us now consider equations with nonlinearities decreasing the regularity. The classical example here is the Kuramoto-Sivashinski equation in 1D:
$$
\Dt u+\partial_x^4 u-a\partial_x^2 u+\partial_x(u^2)=0,\ \ u\big|_{t=0}=u_0, \ a>0
$$
in $\Omega=[-\pi,\pi]$ endowed with periodic boundary conditions. Since we have the conservation law here: $\<u\>:=\frac1{2\pi}\int_{-\pi}^\pi u(x)\,dx$, we need to consider this equation in the phase space $\Phi:=L^2(-\pi,\pi)\cap\{\<u\>=0\}$. It is well-known, see \cites{tem,Good94} that this equation generates a dissipative semigroup $S(t)$ in $\Phi$ which possesses a smooth absorbing set. We set
$A:=(\partial_x^2-a/2)^2+1$ and $F(u)=a^2/4 u+u-\partial_x(u^2)$. After the proper cut-off, the nonlinearity $F$ becomes globally Lipschitz as the map from $H$ to $H^{-1}(\Omega)=D(A^{1/4})$, therefore, we need to check the spectral gap condition \eqref{6.sg1} with $s=1/2$. We know that $\lambda_n\sim c n^4$, therefore,
$$
\frac{\lambda_{N+1}-\lambda_N}{\lambda_{N}^{1/4}+\lambda_{N+1}^{1/4}}\sim c_1N^2
$$
and we have spectral gaps of arbitrarily large size. Thus, this equation possesses an IM.
\par
One more example, which is interesting for what follows, is given by 1D system of reaction-diffusion advection equations:
\begin{equation}\label{6.rda}
\Dt u-\partial_x^2 u+u=f(u,\partial_x u), \ \ u=(u^1,\cdots, u^m),\ u\big|_{t=0}=u_0
\end{equation}
endowed with Dirichlet, Neumann or periodic boundary conditions. We assume that $f$ is smooth and both  $f'_u$ and $f'_{\partial_x u}$ are uniformly bounded. Here $A=-\partial_x^2+1$ and the nonlinearity $f$ consumes one unit of smoothness, so $s=1$. Since $\lambda_n\sim c n^2$, the corresponding spectral gap condition reads
$$
\frac{\lambda_{n+1}-\lambda_n}{\lambda_{n}^{1/2}+\lambda_{n+1}^{1/2}}\sim c>L.
$$
Thus, we will have an IM via the classical theory if the Lipschitz constant $L$ of $f$ is small enough. A bit more accurate analysis shows, see \cites{Man77,K17}, that the size of $f'_u$ is not essential and only $L_1:=\sup_{u,v\in\R}|f'_v(u,v)|$ should be small. As we will see below, this observation is crucial for the recent IM theory for these equations which will be discussed at the next subsection.
\par
We now turn to the smoothness of IMs. We first note that, in the case where $\Cal M$ is a graph of a function $M:\Phi_+\to\Phi_-$ and $\Phi_+$ is a spectral subspace of the operator $A$, the IF \eqref{5.IF} for equation \eqref{6.semPDE} is essentially simplified:
\begin{equation}\label{6.IF}
\Dt u_++Au_+=\Pi_+F(u_++M(u_+),\ \ u_+\in\Phi_+\sim\R^N,
\end{equation}
and we see that the regularity of the reduced equations is determined by the smoothness of the map $M$. Theorem \ref{Th6.IM-main} guarantees that these reduced equations are $C^{1+\eb}$-smooth for some {\it small} $\eb>0$ if the spectral gap conditions are satisfied and $F$ is smooth. This regularity cannot be improved in general since, similarly to the theory of center manifolds (or, more general, theory of normally hyperbolic invariant manifolds),  there are obstacles to further regularity of $\Cal M$, see e.g. counterexample of Sell \cite{CLS92} related with resonances. The nature of these obstacles can be explained as follows. Let us formally differentiate equation \eqref{6.IM-eq} in order to get the equation for the Frechet derivative of the map $M$:
\begin{equation}\label{6.IM-1der}
\Dt v+Av=F'(u(t))v,\ \ \Pi_+v\big|_{t=0}=\xi\in\Phi_+,
\end{equation}
where $M'(u_+)\xi:=\Pi_-v(0)$. Since $\|F'(u(t))\|\le L$, the spectral gap conditions still allow us to uniquely solve this equation in the space $L^2_{e^{\theta t}}(\R_-,\Phi)$ and define the map $M'(u_+)$. A bit more accurate analysis shows that the obtained map will be H\"older continuous with small positive H\"older exponent $\eb>0$ and the function $M$ is $C^{1+\eb}$, see e.g. \cites{Zel14,KZ22}.
\par
The situation changes drastically when we differentiate equation \eqref{6.IM-1der} once more and look at the second derivative:
\begin{equation}\label{6.IM-2der}
\Dt w+Aw-F'(u(t))w=F''(u(t))[v_\xi(t),v_\eta(t)]:=h_{\xi,\eta}(t),\ \ \Pi_+w\big|_{t=0}=0,
\end{equation}
where $v_\xi$ and $v_\eta$ are the solutions of \eqref{6.IM-1der} with the initial data $\xi$ and $\eta$ respectively. The problem here is that $v_\xi,v_\eta\in L^2_{e^{\theta t}}(\R_-,\Phi)$ for some
\begin{equation}\label{6.cond1}
\lambda_N+L<\theta<\lambda_{N+1}-L
\end{equation}
(at least we can guarantee the existence of a solution of \eqref{6.IM-1der} only for such values of $\theta$ based on the Banach contraction theorem). Therefore, $h_{\xi,\eta}\in L^2_{e^{2\theta}}(\R_-,\Phi)$ (since we have the parabolic smoothing property the product of two solutions belonging to the weighted $L^2$ will be also in $L^2$ with the appropriate weight). But in order to solve \eqref{6.IM-2der}, we need the exponent $2\theta$ to satisfy \eqref{6.cond1} and this is possible only if
\begin{equation}\label{6.cond2}
\lambda_{N+1}-2\lambda_N>3L.
\end{equation}
Analogously, if we want the IM to be $C^s$ for some $s>1$, we need the following spectral gap:
\begin{equation}\label{6.cond3}
\lambda_{N+1}-s\lambda_N>(s+1)L.
\end{equation}
This condition is actually sharp and there is an example of Sell \cite{CLS92} of an equation of the form \eqref{6.semPDE} which possesses a $C^{2-\eb}$-smooth IMs, but does not possess any $C^2$-smooth IM, see also \cite{KZ22} and references therein.
\par
Note that there is a principal difference between conditions \eqref{6.sg} and \eqref{6.cond2}, namely, the first condition requires the existence of gaps of arbitrarily large size in the spectrum of $A$ and can be satisfied at least for some elliptic operators in bounded domains (e.g. in the case of low space dimension). In contrast to this, condition \eqref{6.cond2} requires {\it exponentially} big lacunaes in the spectrum (e.g. $\lambda_n=a^n$ with $a>2$), which is difficult to expect in the case of elliptic operators. In  fact, we do not know any examples of such operators $A$ for which inequality \eqref{6.cond2} is solvable with respect to $N$ for any Lipschitz constant $L$. Thus, IMs constructed via Theorem \ref{Th6.IM-main} are {\it never} $C^2$-smooth in general if more or less realistic applications are considered. The only exception is the case of local bifurcations and associated local center manifolds, where $\lambda_N\sim0$ and $L$ can be chosen arbitrarily small by decreasing the size of the neighbourhood. Then condition \eqref{6.cond3} allows us to construct invariant manifolds of any finite smoothness.
\par
Nevertheless, there is a possibility to overcome (at least partially) the smoothness problem for IMs by increasing the dimension of the manifold and further cut-off of the nonlinearity $F$, namely, the following result is proved in \cite{KZ22}.
\begin{theorem}\label{Th6.ext} Let the assumptions of Theorem \ref{Th6.IM-main} hold and let in addition $F\in C^\infty(\Phi,\Phi)$ and the following stronger version of spectral gap conditions hold:
\begin{equation}\label{6.sg3}
\limsup_{N\to\infty}(\lambda_{N+1}-\lambda_N)=\infty.
\end{equation}
Let also $\Cal M_1$ be a $C^{1+\eb}$-smooth IM which corresponds to the first $N$ which satisfies \eqref{6.sg}. Then, for every $m\in\Bbb N$ and any small positive $\delta$, there exists a $C^m$-smooth modification $F_m$ of the initial nonlinearity $F$ such that
\par
1) The manifold $\Cal M_1$ remains an IM for the modified problem
\begin{equation}\label{6.mod}
\Dt u+Au=F_m(u),\ \ u\big|_{t=0}=u_0.
\end{equation}
\par
2)  Equation \eqref{6.mod} possesses a $C^m$-smooth IM $\Cal M_m$ such that $\Cal M_1$ is a normally hyperbolic exponentially stable invariant submanifold of $\Cal M_m$.
\par
3) The nonlinearity $F_m$ is $\delta$-close to $F$ in the $C^1$-norm.
\end{theorem}
\begin{proof}[Idea of the proof] Thanks to \eqref{6.sg3}, we have infinitely many  spectral gaps suitable for constructing IMs. The first of them is used to construct the manifold $\Cal M_1$. The key idea is to use {\it the second} spectral gap for solving equation \eqref{6.IM-2der} for the second derivative, then {\it the third} gap is used for finding the third derivative, etc. Then, for every $u\in\Cal M_1$, the above procedure will give us an $m$-jet which should correspond to the $m$-smooth extension $\Cal M_m$ of the manifold $\Cal M_1$. Such a jet can be constructed in many ways and the most difficult part of the proof is to fix it in such a way that the compatibility conditions of the Whitney extension theorem are satisfied. Then we get the desired manifold $\Cal M_m$ by this theorem and after that it is already not difficult to define the correction $F_m$ of the initial nonlinearity $F$ in such a way that $\Cal M_m$ will be an invariant manifold for the modified system \eqref{6.mod}, see \cite{KZ22} for the details.
\end{proof}
\begin{remark} Note the construction of Theorem \ref{Th6.ext} replaces the requirement \eqref{6.cond3} of the existence of one "huge" spectral gap (which is almost never satisfied) by the existence of many relatively small gaps which is satisfied if we are able to satisfy \eqref{6.sg} for any Lipschitz constant $L$ (so, it is not a big extra restriction).
\par
Note also that this theorem allows us to interpret the $C^{1+\eb}$-smooth IF \eqref{6.IF} (which is a system of ODEs in $\R^{N_1}$) as the reduced equations on the normally hyperbolic invariant manifold for the extended $C^m$-smooth system of ODEs in $\R^{N_m}$ which is the IF related with the IM $\Cal M_m$. Of course, such an extension does not exist for a general non-smooth system of ODEs and is strongly related with the fact that the considered system of ODEs is obtained as a reduction of a {\it smooth} infinite-dimensional system \eqref{6.semPDE}.
\end{remark}
\subsection{Beyond the spectral gap conditions} We now turn to the case where the spectral gap condition \eqref{6.sg} is violated. We start our exposition with the sharpness of conditions of Theorem \ref{Th6.IM-main}. Assume that \eqref{6.sg} is not satisfied for some fixed $N\in\Bbb N$, namely, that $\lambda_{N+1}-\lambda_N<2L$. Let us consider the following linear version of system \eqref{6.semPDE}:
\begin{equation}
\frac d{dt} u_n+\lambda_n u_n=0,\  n\ne N,N+1,\  \frac d{dt}u_N+\lambda_Nu_N=Lu_{N+1},\  \frac d{dt}u_{N+1}+\lambda_{N+1}u_{N+1}=-Lu_N.
\end{equation}
Then, for the corresponding $F(u)=Fu$, we have $\|F\|_{\Cal L(\Phi,\Phi)}=L$, so indeed $L$ is a Lipschitz constant for the perturbation $F(u)$. Then the eigenvalues of the perturbed system associated with the invariant subspace spanned by $e_N$ and $e_{N+1}$
$$
\mu_N^\pm=\frac{\lambda_{N+1}+\lambda_N}2\pm\sqrt{\(\frac{\lambda_{N+1}-\lambda_N}2\)^2-L^2}
$$
are complex conjugate with non-zero imaginary part. For this reason, we cannot decouple $u_N(t)$ and $u_{N+1}(t)$ and an invariant $N$-dimensional manifold (linear subspace) with the base $\Phi_+$ cannot exist. The same example shows that we also do not have a normally hyperbolic invariant subspace of dimension $N$ if the spectral gap condition is violated for $N$, see \cites{Ro00,Zel14} for more details.
\par
Let us now assume that \eqref{6.sg} are violated for any $N\in\Bbb N$, i.e. that
\begin{equation}\label{6.no-sg}
\sup_{N\in\Bbb N}(\lambda_{N+1}-\lambda_N)<2L.
\end{equation}
Then it is natural to find examples where we do not have an IM for any dimension $N$. Such examples are often based on the following lemma, see \cites{Ro00,Zel14} and its generalizations.
\begin{lemma}\label{Lem6.odd-even} Assume that the nonlinearity $F$ in equation \eqref{6.semPDE} belongs to $C^1$ and let $u_\pm\in\Phi$ be two equilibria of this equation. Let $\Cal L_{u_\pm}:=-A+F'(u_{\pm})$ be the linearization of \eqref{6.semPDE} at $u=u_{\pm}$. Assume that the spectrum $\sigma(\Cal L_{u_+})$ consists of complex conjugate eigenvalues with non-zero imaginary parts and that the spectrum of $\sigma(\Cal L_{u_-})$ contains one simple real and positive eigenvalue and the rest of it consists of complex conjugate eigenvalues with non-zero imaginary parts.
Then problem \eqref{6.semPDE} does not possess any finite-dimensional $C^1$-smooth IM.
\end{lemma}
\begin{proof}[Sketch of the proof] Indeed, assume that such a manifold $\Cal M$ exists. Let us consider its tangent planes $\Cal T_{u_\pm}\Cal M$ at $u=u_{\pm}$. Since the manifold is invariant, the planes $\Cal T_{u_{\pm}}\Cal M$ are invariant with respect to $\Cal L_{u_\pm}$. Therefore,
$$
\sigma(\Cal L_{u_\pm}\big|_{\Cal T_{u_\pm}})\subset \sigma(\Cal L_{u_\pm}).
$$
 Since the equation is real-valued and the spectrum of $\Cal L_{u_+}$ does not contain real eigenvalues, we conclude that $\dim\Cal M=\dim\Cal T_{u_+}\Cal M$ is even. On the other hand, since the IM always contains an attractor and the attractor always contains an unstable manifold, the direction of a simple real eigenvalue of $\Cal L_{u_-}$ must be in $\Cal T_{u_-}\Cal M$. Then the analogous arguments show that $\dim\Cal M=\dim\Cal T_{u_-}\Cal M$ is odd. This contradiction proves that the IM $\Cal M$ can not exist.
 \end{proof}
 \begin{corollary}\label{Cor6.ne} Let $\lambda_1<L$ and assumption \eqref{6.no-sg} be satisfied. Then there exists a smooth bounded and globally Lipschitz nonlinearity $F$ with Lipschitz constant $L$ such that equation \eqref{6.semPDE} does not possess any $C^1$-smooth finite-dimensional IM.
 \end{corollary}
 \begin{proof}[Sketch of the proof] According to the previous lemma, we only need to construct $F$ in equation \eqref{6.semPDE} in such a way that it will possess two equilibria $u=u_\pm$ such that the linearizations near $u_+$ and $u_-$ have the following forms:
 \begin{equation}\label{6.+}
 \frac d{dt}v_{2n-1}+\lambda_{2n-1}v_{2n-1}=Lu_{2n},\  \frac d{dt}v_{2n}+\lambda_{2n}v_{2n}=-Lu_{2n-1},\ n\in\Bbb N
 \end{equation}
 and
 \begin{equation}\label{6.-}
 \frac d{dt}v_{1}+\lambda_{1}v_{1}=Lu_{1},\ \frac d{dt}v_{2n}+\lambda_{2n}v_{2n}=Lu_{2n+1},\ \frac d{dt}v_{2n+1}+\lambda_{2n+1}v_{2n+1}=-Lu_{2n},\ n\in\Bbb N
 \end{equation}
 respectively. Then all assumptions of the previous corollary will be satisfied. The construction of such nonlinearity $F$ is straightforward and we omit it here, see \cite{Zel14} for the details.
 \end{proof}
 \begin{remark} Note that the fractal dimension of the global attractor $\Cal A$ which corresponds to equation \eqref{6.semPDE} is always finite if $F$ is bounded and  globally Lipschitz. Thus, according to the common paradigm, the corresponding dynamics on the attractor should be "finite-dimensional". The result of Corollary \ref{Cor6.ne} does not essentially contradict this heuristic principle, it just tells us that the dimension of the inertial "manifold" should be different in different parts of the phase space. For instance, instead of an IM, one may try to consider an "inertial CW-complex" with finite-dimensional dynamics on it. As we will see below this actually also does not work and the reduced dynamics may be essentially infinite-dimensional despite  the finiteness of the fractal dimension.
 \end{remark}
 The further progress in understanding the obstacles to existence of IMs and their consequences  is related with the recently discovered connections of the IM theory and the Floquet theory for infinite-dimensional differential equations, see \cites{Ku93,EKZ13,Zel14,KZ18} and references therein. Namely, let us consider the following linear time-periodic parabolic equation in a Hilbert space $\Phi$:
 \begin{equation}\label{6.per}
 \Dt v+Av=L(t)v,\ \ v\big|_{t=0}=v_0,
 \end{equation}
  where $L(t)$ is $T$-periodic in time linear operator satisfying $\|L(t)\|_{\Cal L(\Phi,\Phi)}\le L$. Roughly speaking, we fix the operator $L(t)$ in such a way that \eqref{6.per} will be close to \eqref{6.+} and \eqref{6.-} on the first and second half-periods respectively. Moreover, it can be done in such a way that the period map $\Cal P:u_0\to u(t)$ will have a form of  an infinite Jordan cell:
  \begin{equation}\label{6.cell}
  \cdots \Cal Pe_{2n}=\mu_{2n}e_{2n-2},\ \cdots, \Cal Pe_2=\mu_0e_1, \ \Cal Pe_1=\mu_1e_3,\cdots, \Cal Pe_{2n+1}=\mu_{2n+1}e_{2n+3},\cdots,
  \end{equation}
  where $\mu_n\sim e^{-\alpha |n|}$ for some positive $\alpha$. In particular, the map $\Cal P$ is compact and $\sigma(\Cal P)=\{0\}$, so the corresponding equation \eqref{6.per} does not possess any Floquet multipliers and any its solution decays super-exponentially to zero:
  \begin{equation}\label{6.sup}
  \|v(t)\|_{\Phi}\le C\|v_0\|_\Phi e^{-\alpha t^2/2},
  \end{equation}
  see \cites{EKZ13,Zel14} for more details. The next step is to realize operator $L(t)$ in the form of $L(u_0(t))$, where $u_0(t)$ is a time periodic solution of an ODE which is interpreted as an equation for the "zero-mode" of the abstract equation \eqref{6.semPDE}. This will give us a super-exponentially attracting periodic orbit $u(t):=\{u_0(t),0\}$ in the system of the form \eqref{6.semPDE} which belongs to the attractor. It only remains to modify the equations properly outside of this periodic orbit in order to find another periodic orbit $\bar u(t)$ belonging to the attractor which converges super-exponentially fast to $u(t)$:
  \begin{equation}\label{6.quad}
  \|u(t)-\bar u(t)\|_\Phi\le Ce^{-\alpha t^2/2},\ \ u,\bar u\in\Cal A,
  \end{equation}
  see \cites{EKZ13,Zel14} for more details. The presence of such two trajectories on the attractor clearly excludes the existence of  Lipschitz continuous IMs of any finite dimension. Moreover, the proper modification of the construction excludes also the possibility to embed the attractor $\Cal A$ in any finite-dimensional Lipschitz of even $\log$-Lipschitz continuous (not necessarily invariant) submanifolds. Namely, the following result is proved in \cite{EKZ13}.

  \begin{theorem}\label{Th6.no-IM} Let the assumptions of Corollary \ref{Cor6.ne} be satisfied. Then there exists a smooth and bounded nonlinearity $F$ with global Lipschitz constant $L$ such that there exist two solutions $u(t)$ and $\bar u(t)$ of equation \eqref{6.semPDE} belonging to the attractor and satisfying \eqref{6.quad}. Moreover, the corresponding attractor can not be embedded into any finite-dimensional $\log$-Lipschitz submanifold of $\Phi$. In particular, no IM exists for this equation.
  \end{theorem}
\begin{remark} Since under the assumptions of Theorem \ref{Th6.no-IM} the fractal dimension of the attractor $\Cal A$ is finite, due to the Man\'e projection theorem, we have a H\"older continuous IF on the attractor as well as its embedding to a H\"older continuous submanifold of $\Phi$. Moreover, the H\"older exponent can be made as close to one as we want by increasing the dimension of the manifold, see \cite{Rob11} and  references therein. Nevertheless, super-exponential attraction of trajectories is not observed in classical dynamics generated by smooth ODEs and hardly be interpreted as a finite-dimensional phenomenon. Thus, we see some kind of infinite-dimensional limit dynamics on the attractor of the finite fractal dimension. This phenomenon is not properly understood yet and definitely deserves further investigation.
\end{remark}
\begin{remark}\label{Rem6.per} Note that it is relatively easy to construct counterexamples to the Floquet theory on the level of {\it abstract} linear parabolic equations of the form \eqref{6.per}, for instance, constructing the desired operator $L(t)$ in the Fourier base of the operator $A$. The situation becomes much more complicated if we want \eqref{6.per} to be a true parabolic PDE or system of parabolic PDEs. The existence of counterexamples to the Floquet theory on the level of parabolic PDEs in bounded domains has been an open problem till recently. This problem is affirmatively solved in \cite{KZ18}, where the smooth space-time periodic $m\times m$-matrices $a(t,x)$ and $b(t,x)$ are constructed  for $m\ge4$  in such a way that all solutions of the corresponding linear reaction-diffusion-advection problem:
\begin{equation}\label{6.rda-lin}
\Dt v-\partial_x^2v=a(t,x)v+b(t,x)\partial_x v,\ \ v\big|_{t=0}=v_0,\ \ x\in(-\pi,\pi),\ v=(v^1,\cdots, v^m),
\end{equation}
endowed with periodic boundary conditions, decay super-exponentially in time:
$$
\|v(t)\|_\Phi\le \|v_0\|_\Phi e^{-\alpha t^3},\ \ t\ge0,\ \ \alpha>0.
$$
Moreover, an example of a system of semilinear reaction-diffusion-advection equations of the form
\begin{equation}\label{6.rda-s}
\Dt u-\partial_x^2u=f(u)+g(u)\partial_x u,\ \ x\in(-\pi,\pi),\ \ u=(u^1,\cdots,u^m),\ \ u\big|_{t=0}=u_0
\end{equation}
with $m\ge8$, periodic boundary conditions and with $C^\infty$-smooth function $f:\R^m\to\R^m$ and $g:\R^m\to M(m\times m)$, which does not possess any finite-dimensional IM and satisfies all of the assertions of Theorem \ref{Th6.no-IM}, is given in \cite{KZ18} based on the counterexample \eqref{6.rda-lin}.
\end{remark}
We now turn to examples, where the IM still exists despite the fact that the spectral gap condition is not satisfied. We have seen above that these conditions are sharp on the level of abstract parabolic equations, however, there is still a possibility to relax them if some specific sub-class of such equations is considered. For instance, the first such  example is due to Sell and Malet-Paret \cite{MPS88} where the IM for a scalar semilinear heat equation on the 3D torus $\Omega=[-\pi,\pi]^3$ has been constructed. We return to this example later, but prefer to consider first an alternative method related with finding the proper transformations or/and embedding of the initial system to a larger one for which the spectral gap conditions will be satisfied. We illustrate this approach on the example of the reaction-diffusion-advection system \eqref{6.rda-s} endowed with {\it Dirichlet} boundary conditions. Namely, let us do the change of an independent variable $u=a(t,x)w$, where the matrix $a(t,x)$ will be specified later. Then we arrive at the transformed equation:
\begin{multline}\label{6.RDS-trans}
\Dt w-\partial_x^2w=\{a^{-1}(2\partial_xa+g(aw)a)\partial_xw\}+\\+\{a^{-1}(\partial_x^2a-\Dt a+g(aw)\partial_xa)w+a^{-1}f(aw)\}:=\Cal F_1(w)+\Cal F_2(w).
\end{multline}
We see that the operator $\Cal F_2$ does not depend explicitly on $\partial_xw$ and this dependence is presented in $\Cal F_1$ only, so the naive idea would be to kill the term $\Cal F_1$ by the proper choice of the matrix $a=a(u)$, for instance, fixing it as a solution of a matrix ODE:
$$
\frac12\frac d{dx}a+g(u)a=0,\ \ a\big|_{x=-\pi}=Id.
$$
However, this naive idea will not work since the terms $\Dt a$ and $\partial_x^2a$ will implicitly depend on $\partial_xw$ and the operator $\Cal F_2$ will be not bounded from $\Phi:=H^1_0(-\pi,\pi)$ to $\Phi$. Fortunately, as we already mentioned discussing equation \eqref{6.rda}, we need not to kill the operator $\Cal F_1$ completely, it is sufficient  to make its Lipschitz
constant (as the map from $\Phi$ to $L^2(-\pi,\pi)$ small enough) and this guesses the proper transformation: we fix $a=a(u)$ as a solution of the following problem
\begin{equation}\label{6.gt}
\frac12\frac d{dt}a+g(P_Ku)a=0,\ \ a\big|_{x=-\pi}=Id,
\end{equation}
 where $P_K$ is some smoothing operator in $x$, for instance, we can fix it as the spectral projector to the first $K$ eigenvectors of $-\partial_x^2$ in $\Omega=(-\pi,\pi)$ with Dirichlet boundary conditions and fix $K$ being big enough. It is shown in \cite{K17} that $u=a(aw)w$ thus defined is indeed a diffeomorphism of the phase space $\Phi$, the Lipschitz norm of $\Cal F_1:\Phi\to L^2(-\pi,\pi)$ can be made arbitrarily small by choosing $K$ big enough and the map $\Cal F_2:\Phi\to\Phi$ is globally Lipschitz continuous. Thus, the proper version of spectral gap conditions is satisfied and we have the following result, see \cite{K17} for the details.
 \begin{theorem}\label{Th6.rda-D} Let the functions $f$ and $g$ be smooth and have finite supports. Then equation \eqref{6.rda-s} endowed with Dirichlet boundary conditions possesses an IM.
 \end{theorem}
\begin{remark} As we see from Remark \ref{Rem6.per} and Theorem \ref{Th6.rda-D}, the existence or non-existence of an IM for the reaction-diffusion-advection problems strongly depends on the type of boundary conditions. This is related with the fact that the used transform $u=a(t,x)w$ naturally preserves the Dirichlet boundary conditions, but does not preserve the Neumann or periodic ones. In the case of Neumann boundary conditions, there is a nice trick which allows us to overcome the problem, namely, we use that $v=\partial_xu$ satisfies Dirichlet boundary conditions and embed system \eqref{6.rda-s} into a larger system
$$
\Dt u-\partial_x u=f(u)+g(u)v,\ \partial_x u\big|_{x=\pm\pi}=0,\ \ \Dt v-\partial_x^2v=f'(u)v+g'(u)[v,v]+g(u)\partial_x v,\ \ v\big|_{x=\pm\pi}=0.
$$
Here we have only one dangerous term $g(u)\partial_xv$ and this term can be made small using the transformation of $v$ component only where the Dirichlet boundary conditions are preserved. This gives us the analogue of Theorem \ref{Th6.rda-D} for the case of Neumann boundary conditions, see \cite{KZ18} for more details. Note that, for the case of periodic boundary conditions, the analogue of this trick can not exist since we have a counterexample where an IM does not exist, see Remark \ref{Rem6.per}. We emphasize that the counterexample is constructed for systems $m\ge8$ only and it is proved in \cite{KZ18} that, in the case of a scalar equation $m=1$, we still have an IM.
\par
We also mention that the general case of equation \eqref{6.rda} can be also treated in a similar way by differentiating the equation in time and embedding the problem to a larger system of equations of the form \eqref{6.rda-s}, see \cite{KZ18} for more details. Thus, the assumption of smallness of the derivative $\partial_{u_x}f$ can be removed in the case of Dirichlet or Neumann boundary conditions, but can not be relaxed in general for systems with periodic boundary conditions.
\end{remark}
\begin{example} Let us consider a particular example of forced Burgers equation:
\begin{equation}\label{6.bur}
\Dt u-\nu\partial_x^2 u=\partial_x(u^2)+g(x), \ x\in(-\pi,\pi),\ u\big|_{t=0}=u_0
\end{equation}
endowed with Dirichlet boundary conditions (Neumann and periodic cases are similar, but we need to take care about the spatial mean value $\<u\>$). We assume, say, that $g\in \Phi:= L^2(-\pi,\pi)$ and $\nu>0$. Then equation \eqref{6.bur} is globally well-posed and dissipative in $\Phi$. Indeed, multiplying \eqref{6.bur} by $u$ and integrating by parts (which kills the nonlinear term), we end up with the energy identity:
$$
\frac12\frac d{dt}\|u(t)\|^2_\Phi+\nu\|\partial_x u(t)\|^2_{L^2}=(g,u(t)),
$$
which gives us the desired dissipativity, see \cite{tem} for more details. Moreover, using the parabolic smoothing property, we may construct an absorbing set for the corresponding solution semigroup $S(t)$ which is a bounded set in $H^2(-\pi,\pi)\subset C^1[-\pi,\pi]$. Thus, the semigroup $S(t)$ possesses an attractor $\Cal A$ in the phase space $\Phi$ endowed with the standard bornology of bounded sets in $\Phi$. In addition, we may cut-off the nonlinearity outside of the absorbing set and write out the problem \eqref{6.bur} in the form of \eqref{6.rda-s} with smooth nonlinearities with finite support. This gives us the existence of an IM for this equation, see \cite{K17} for more details.
\par
We however note that equation \eqref{6.bur} is "too simple" to have any non-trivial dynamics and its attractor always consists of a single equilibrium:
\begin{equation}\label{6.a-triv}
\Cal A:=\{G\},\ \ \nu G''=(G^2)'+g,\ \ G(\pm\pi)=0.
\end{equation}
Moreover, this equation does not generate any dynamical instability and every trajectory of it is Lyapunov stable and even asymptotically stable. Indeed, let $u_1(t)$ and $u_2(t)$ be two solutions of \eqref{6.bur} and $v(t):=u_1(t)-u_2(t)$. Then $v$ solves
$$
\Dt v-\partial_{x}^2 v=\partial_x((u_1+u_2)v), \ v\big|_{t=0}=v_0.
$$
Multiplying this equation by $\sgn(v(t))$ and using the Kato inequality, see \cite{Cle87}, we end up with
\begin{equation}\label{6.l1-lyap}
\frac d{dt}\|v(t)\|_{L^1}+\nu|\partial_xv(-\pi)|+\nu|\partial_xv(\pi)|\le 0.
\end{equation}
Thus, the quantity $\|u_1(t)-u_2(t)\|_{L^1}$ is not increasing along the trajectories. In particular, if we assume that $u_1,u_2\in\Cal A$, we infer from \eqref{6.l1-lyap} that
$$
v\big|_{x=\pm\pi}=\partial_x u\big |_{x=\pm\pi}=0
$$
and the Carleman type estimates give us that $v(t)\equiv0$, see \cite{RL}. Since, by the standard arguments,  equation \eqref{6.bur} possesses at least one equilibrium, we have that this equilibrium is unique and \eqref{6.a-triv} holds. Moreover, the standard compactness arguments show that this equilibrium is exponentially stable, i.e.
\begin{equation}
\|u_1(t)-G\|_\Phi\le C_\nu\|u_1(0)-G\|_\Phi e^{-\alpha_\nu t}
\end{equation}
for some positive $C_\nu$ and $\alpha_\nu$.
\par
Thus, constructing IMs and studying the limit dynamics on the attractor does not look very interesting for the Burgers equation \eqref{6.bur} (although this equation still may demonstrate a non-trivial {\it intermediate} behavior which may be interesting from the hydrodynamical point of view, see \cites{Bir01,Kuk99} and references therein, but this phenomenon is more appropriate to study by using the concept of an exponential attractor considered in the next section). Alternatively, we may return to the original Burgers model of turbulence, where the Burgers equation is coupled with an ODE:
\begin{equation}\label{6.bur-mod}
\frac d{dt} U+\nu U=P-\int_{-\pi}^\pi u^2(t,x)\,dx,\ \ \Dt u-\nu\partial_x^2 u+\partial_x(u^2)=Uu
\end{equation}
or two component Burgers equations
\begin{multline}\label{6.bur-2}
\frac d{dt} U+\nu U=P-\int_{-\pi}^\pi (u^2(t,x)+v^2(t,x))\,dx,\\ \Dt u-\nu\partial_x^2 u+\partial_x(u^2-v^2)=U(u-v),\ \Dt v-\nu\partial_x^2 v-\partial_x(2uv)=U(u+v),
\end{multline}
where the function $U(t)$ depends only on $t$ and the components $u$ and $v$ are endowed with Dirichlet boundary conditions, see \cite{Burgers} for the details. In contrast to  the single Burgers equation \eqref{6.bur}, we now have the instability if the parameter $P>0$ is large enough (this parameter plays the role of the external forces), so the attractor as well as the dynamics on it becomes non-trivial. On the other hand, we still have an energy identity and dissipativity for these equations as well as the existence of a smooth absorbing set. For instance, the energy identity for equation \eqref{6.bur-2} reads
$$
\frac12\frac d{dt}\(U^2+\|u\|^2_{L^2}+\|v\|^2_{L^2}\)+
\nu\(U^2+\|\partial_xu\|^2_{L^2}+\|\partial_xv\|^2_{L^2}\)=PU.
$$
Moreover, the general theory of IMs for  1D reaction-diffusion-advection problems is applicable here (we may make the terms containing $\partial_xu$ and $\partial_x v$ small enough to satisfy the spectral gap conditions by the transform described above). This, in turn, allows us to construct IMs for both equations \eqref{6.bur-mod} and \eqref{6.bur-2}. We  mention that this result was stated in \cite{IO84}, however, the proof given there is based on the erroneous idea of Kwak, so the correct proof becomes available only recently due to \cite{K17}.  We also mention that some particular cases of a general method suggested in \cite{K17} were considered in \cites{Vuk1,Vuk2,Vuk3}
\end{example}
\begin{remark} We recall that the 1D reaction-diffusion-advection equations can be considered as simplified models for 2D Navier-Stokes equations, so proving or disproving the existence of IMs for them was a longstanding open problem of a big theoretical and practical interest. A number of weaker results in this direction has been obtained. We mention here only the  Romanov theory, see \cite{Ro96} which allowed us to construct the Man\'e projectors with Lipschitz inverse on the attractors for some particular cases of these equations, see \cites{Ro00,Zel14,K21}. A more or less complete answer on this question which is obtained in \cites{K17,KZ18} and is discussed above is somehow unexpected and surprising. Indeed, periodic boundary conditions are used in hydrodynamics mainly because  they are "simpler" than more physical Dirichlet boundary conditions. The situation with reaction-diffusion-advection equations guesses that such an approach may be essentially wrong and in reality the periodic boundary conditions may be the "most complicated".
\end{remark}
We conclude our exposition of IMs by considering the spatial averaging method suggested by Sell and Mallet-Paret and its recent generalizations. We recall that in order to construct an IM for equation \eqref{6.semPDE}, we need to solve equation \eqref{6.IM-eq} backward in time in the properly chosen functional space and the main step here is to solve the corresponding equation of variations \eqref{6.IM-1der}. Moreover, looking on estimate \eqref{6.ode-est}, we see that the main impact to the norm \eqref{6.main-est} is given by the {\it intermediate} modes which correspond to eigenvalues $\lambda_N-k<\lambda_n<\lambda_{N+1}+k$. Let $\Cal I_{k,N}$ be a spectral orthoprojector on these modes. Then the most dangerous part of the operator $F'(u(t))$ in \eqref{6.IM-1der} is exactly its intermediate part $\Cal I_{k,N}\circ F'(u(t))\circ\Cal I_{k,N}$ for the properly chosen $k$ and $N$. The key assumption of the method is that this part is close to a scalar operator.

\begin{definition} We say that the nonlinearity $F\in C^1_b(\Phi,\Phi)$ and the operator $A$ satisfy the spatial averaging condition if there exist a bounded measurable function $a:\Phi\to\R$ and a positive number $\kappa$ such that, for every $\eb>0$ and every $k\in\R_+$, there exist infinitely many values of $N\in\Bbb N$ satisfying
\begin{equation}\label{6.sp-av}
\|\Cal I_{k,N}F'(u)\Cal I_{k,N}-a(u)Id\|_{\Cal L(\Phi,\Phi)}\le\eb,\ \ \lambda_{N+1}-\lambda_N>\kappa.
\end{equation}
Then the function $a(u)$ is referred as a spatial averaging of the operator $F'(u)$.
\end{definition}
\begin{theorem}\label{Th6.IM-sp} Let the operator $A:D(A)\to\Phi$, $A=A^*>0$ with compact inverse and the nonlinearity $F(u)$ satisfy the spatial average condition. Then equation \eqref{6.semPDE} possesses an IM (actually, infinitely many of them) which is a graph of a Lipschitz function over the properly chosen spectral subspace. Moreover, if $F$ and its spatial average are more regular, the corresponding IM is $C^{1+\eb}$-smooth for some small $\eb>0$.
\end{theorem}
\begin{proof}[Idea of the proof] Let us discuss how to solve the equation of variations \eqref{6.IM-1der} only, the full proof can be found in \cite{MPS88}, see also \cites{Zel14,CSZ23}. Let us look at the solution operator $\Cal L_\theta$ of equation \eqref{6.lin} with $\theta=\frac{\lambda_{N+1}+\lambda_N}2$. Then, due to \eqref{6.ode-est}, we have
$$
\|\Cal L_{\theta}\circ (Id-\Cal I_{k,N})\|_{\Cal L(\Phi,\Phi)}\le \frac1k,
$$
and this part of the operator $\Cal L_\theta$ can be made small by fixing $k$ large enough and taking the standard Lyapunov's metric
$$
\|u\|^2_{\Phi,\beta}:=\|(Id-\Cal I_{k,N})u\|^2_{\Phi}+\beta\|\Cal I_{k,N}u\|_{\Phi}^2,
$$
where the parameter $\beta$ is chosen properly, so, we only need to take care about the intermediate modes, see \cite{Zel14} for the details. For these modes we use the spatial averaging condition which allows us to reduce the problem to equations of the form
$$
\frac d{dt}v_n+\lambda_n v_n=a(u(t))v_n+\text{"small"}.
$$
Finally, the scalar term $a(u(t))v_n$ can be removed doing the change of independent variables $v_n(t)=e^{\int_0^sa(u(s))\,dx}w_n(t)$. Important that this transform is the same for all (intermediate) modes. This reduces the problem to
$$
\frac d{dt} w_n+\lambda_nw_n=\text{"small"}
$$
and this problem can be solved using that $\lambda_{N+1}-\lambda_N>\kappa$.
\par
This finishes the proof of solvability for equation of variations. The passage to the nonlinear case is a bit more complicated here since the weight $e^{\int_0^ta(u(s))\,ds}$ depends on the trajectory $u(t)$ which is also unknown from the very beginning. By this reason, the methods related with invariant cones are traditionally used here, see \cites{MPS88,Zel14,CSZ23} for the details, although this passage can be also done by perturbation arguments and Banach contraction theorem, see \cite{K18}.
\end{proof}
In applications, the operator $A$ is usually minus Laplacian with {\it periodic} boundary conditions mainly in 3D case, say, in $\Phi=L^2(\Omega)$, $\Omega=[-\pi,\pi]^3$ (see \cites{MPS88,kwe99} for applications to irrational tori in 2D or/and equilateral triangles). In this case, the eigenvalues $\lambda_n$ are naturally parameterized  by the points of a lattice $\Bbb Z^3$
$$
\lambda_{\vec n}=m^2+l^2+k^2,\ \ \vec n=(m,l,k)\in\Bbb Z^3
$$
and the number theory related with distributions of sums of squares naturally comes into play. Namely, assume in addition that the operator $F'(u)$ is a point-wise multiplication
$$
F'(u)v:=\psi(x)v(x)
$$
on a sufficiently smooth function $\psi=\psi_u(x)$ (or sums of combinations of point-wise multiplication and differentiation). Then, in the Fourier base, the action of $F'(u)$ becomes a convolution
\begin{equation}\label{6.conv}
(F'(u)v)_{\vec n}:=\sum_{\vec n_1\in\Bbb Z^3}\psi_{\vec n-\vec n_1}v_{\vec n_1}
\end{equation}
and we may use the following lemma to verify the spatial averaging condition.
\begin{lemma}
  Let
\begin{equation}
\Cal C^k_N:=\{\vec l\in\Bbb Z^3\,:\, N-k\le|\vec l|^2\le N+k\},\ \ \Cal B_r:=\{\vec l\in\Bbb Z^3\,:\,
|\vec l|\le r\}.
\end{equation}
Then, for every $k>0$ and $r>0$, there exist infinitely many $N\in\Bbb N$ such that
\begin{equation}\label{6.spav}
\(\Cal C^k_{N}-\Cal C^k_N\)\cap \Cal B_r=\{0\}.
\end{equation}
\end{lemma}
The proof of this lemma is given in \cite{MPS88}. This lemma shows that the intermediate part of the operator $F'(u)v$ contains the leading term $\psi_0 v$ and all other term contain $\psi_{\vec n}$ with $|\vec n|>\rho$ (which are small if $\psi$ is smooth enough). Thus, the spatial averaging condition will be satisfied with $a(u)=\psi_0=\<\psi_u\>$ and this justifies the interpretation of $a(u)$ as a spatial averaging in a general theory, see \cites{MPS88,Zel14} and references therein for more details.
\par
This scheme was suggested by Sell and Malet-Paret \cite{MPS88} to construct an IM for a {\it scalar} reaction-diffusion equation \eqref{6.RDS} in 3D with periodic boundary conditions. Clearly, the spectral gap conditions are not satisfied there and this was historically the first example of an IM beyond the spectral gap conditions. They also gave a counterexample showing that in 4D this method will not work and that the IM may not exist (at least in the class of normally hyperbolic IMs, see \cite{MPSS93}). The corresponding counterexample for a {\it system} of two reaction-diffusion equations in 3D (also with periodic boundary conditions) where an IM does not exist (again, in the class of normally-hyperbolic IMs) is given by Romanov \cite{Ro00}.    The IM for Cahn-Hilliard equations in 3D periodic domains are constructed by a similar method in \cite{KZ15}.
\par
As we have already mentioned, the spatial averaging method does not work in general for systems of PDEs. Indeed, in this case $a(u)$ will be not a scalar operator, but a matrix operator and this is not enough to solve the simplified equation of variations for the intermediate modes. One of the exceptions is the case where $a(u)$ is a zero matrix, then everything should work similarly to the scalar case. However, there was a non-trivial problem here related with the cut-off of the initial equation transforming it to the so-called "prepared" form. Usually the spatial averaging method requires very non-trivial cut-off procedure and it was not clear how to preserve the condition $\<a(u)\>=0$ under this procedure. The progress here is mainly related with the paper of Kostianko \cite{K18}, where such a cut-off procedure has been suggested using the proper modification of the nonlinearity $F(u)$ in the Fourier basis. In turn, this construction allowed us to use spatial averaging method for IMs for some hydrodynamical problems. For instance, the existence of an IM for the modified Leray-$\alpha$ model and hyperviscous Navier-Stokes equations on a 3D torus was proved in \cite{K18} and \cite{GG18} respectively and more general problems of the form
\begin{equation}\label{6.NS}
\begin{cases}
\Dt u+(u,\Nx\bar u)+(-\Dx)^{1+\gamma} u+\Nx p=g, \ \ u\big|_{t=0}=u_0,\\
\divv u=0,\ \ \bar u=(1-\alpha\Dx)^{-\bar\gamma}u,
\end{cases}
\end{equation}
where $\gamma,\bar\gamma\ge0$ and $\gamma+\bar\gamma=\frac12$ are considered in \cite{CSZ23}.
\par
Let us now return to the case of systems where $a(u)$ is not a scalar matrix and consider the following version of a cross-diffusion system:
\begin{equation}\label{6.cGL}
\Dt u-(1+i\omega)(\Dx u-u)=f(u,\bar u),\ \ u\big|_{t=0}=u_0,
\end{equation}
where now $u(t,x):=\Ree u(t,x)+i\Imm u(t,x)$ is a complex valued unknown function $\bar u$ is a complex conjugate and $f$ is a given smooth function with finite support and $\omega\ne0$ is a real cross diffusion parameter. We consider this equation on a 3D torus $\Omega=[-\pi,\pi]^3$. The typical example of such an equation is a complex Ginzburg-Landau equation:
\begin{equation}\label{6.cGL1}
\Dt u-(1+i\omega)(\Dx u-u)=\varphi(|u|^2)u.
\end{equation}
As usual, we assume that this equation possesses a smooth absorbing set, so without loss of generality we may assume that $\varphi$ is smooth and has a finite support.
\par
Applying the method of spatial averaging to the equation of variation related with problem \eqref{6.cGL}, we end up with the following problem for the intermediate modes:
\begin{equation}\label{6.cGL-sp}
\frac d{dt}v_n+(1+i\omega)(\lambda_n+1)v_n=\<f'_u\>v_n+\<f'_{\bar u}\>\bar v_n+\text{"small"}.
\end{equation}
We see that the first term in the right-hand side is a nice scalar operator, but the second one ($\<f'_{\bar u}\>\bar v_n$) is not, so the sole spatial averaging is {\it not sufficient} for constructing an IM (the corresponding example with $\omega=0$ is given in \cite{Ro00}). The key idea here is (following  \cite{K21}) to combine the spatial averaging method with the {\it temporal} averaging with respect to rapid in time oscillations generated by the large dispersion  term $i\omega(\lambda_n+1)v_n$ in equations \eqref{6.cGL-sp}. This average kills (makes small) the non-scalar term  $\<f'_{\bar u}\>\bar v_n$ in \eqref{6.cGL-sp} and allows us to complete the proof of the existence of an IM for equation \eqref{6.cGL} for the case $\omega\ne0$, see \cite{K21} and \cite{KSZ22} for the missed details.

\section{Exponential attractors and perturbation theory}\label{s7}
We have seen that IMs give us a perfect way of constructing  the finite-dimensional reduction for the limit dissipative dynamics. However, an IM requires much more restrictive conditions to exist than an attractor.
In this section, we discuss an intermediate concept between attractors and inertial manifolds introduced
in \cite{EFNT94}, namely, the concept of an {\it exponential} attractor, which, on the one hand, is almost as common as usual attractors and, on the other hand, allows us to overcome the most principal drawbacks of usual (global) attractors. We start with an illustrating example.
\begin{example}\label{Ex7.inter} Let us consider the following 1D real Ginzburg-Landau equation with non-zero boundary conditions:
\begin{equation}\label{7.cgl}
\Dt u=\nu^2\partial_x^2 u-u^3+u,\ \ u\big|_{x=1}=u\big|_{x=-1}=1,\ \ u\big|_{t=0}=u_0,
\end{equation}
where $\Omega=(-1,1)$ and $\nu>0$ is a small parameter. Obviously, the solution semigroup associated with this equation is dissipative and possesses a smooth absorbing set. We claim that the attractor $\Cal A$ of the corresponding solutions semigroup $S(t)$ in the phase space $\Phi:=L^2(\Omega)$, endowed with the standard bornology of bounded sets (the bornology of all subsets of $\Phi$ is also possible here and gives the same attractor), consists of a single equilibrium
\begin{equation}\label{7.triv}
\Cal A=\{u\equiv 1\}.
\end{equation}
Indeed, $u=1$ is a unique equilibrium of this problem no matter how small $\nu$ is. This follows by multiplication of the equation for the equilibria by $u'_x$ and integrating over $x$:
$$
\nu^2u'_x(x)^2=\nu^2u'_x(-1)^2+\frac12(u(x)^2-1)^2.
$$
 Thus, on the one side, $u'(x)^2$ is non-decreasing and, on the other hand $u'(1)^2=u'(-1)^2$ which is possible only if $u(x)\equiv1$. Another observation is that equation \eqref{7.cgl} possesses a global Lyapunov function:
 $$
 \frac d{dt}\(\nu^2\|u'_x\|^2_{\Phi}-\frac12((u^2-1)^2,1)\)=-2\|\Dt u\|^2_{\Phi}
 $$
 and therefore, any trajectory must converge to the set of equilibria and  \eqref{7.triv} holds.
 \par
 Thus, the limit dynamics is trivial for any $\nu>0$. It is just a single exponentially stable equilibrium. Moreover, linearizing our equation near $u=1$, we see that the attraction rate to the equilibrium is not slower than $e^{-2t}$ (and does not become worse as $\nu\to0$). So, we may conclude that, for any (bounded) set $B\subset\Phi$,
 \begin{equation}\label{7.false}
  \operatorname{dist}_\Phi(S(t)B,\Cal A)\le Ce^{-2t},
\end{equation}
 where the constant $C$ is actually independent of $B$, but still may depend on $\nu$.
\par
However, for small $\nu>0$, equation \eqref{7.cgl} possesses an interesting metastable dynamics related with slow evolution of multi-kinks. For instance, if you start from the profile which is close to
$$
u_0(x)=1-2H(1/3-|x|),
$$
where $H(x)$ is a standard Heaviside function, the corresponding solution will be close to $u_0(x)$ for an extremely long time, the life-span of this "almost" equilibrium $T\sim e^{-1/\nu}$, see \cites{ER98,MieZ09} for more details about the evolution of multi-kinks via the center manifold reduction. Thus, for the constant $C$ in \eqref{7.false}, we have an estimate
\begin{equation}\label{7.false1}
C\ge e^{e^{c/\nu}}
\end{equation}
for some positive $c$ which is independent of $\nu$.
\par
Thus, the behavior of trajectories of \eqref{7.cgl} for small $\nu$ is described mainly by the evolution of metastable states and the attractor $\Cal A$ becomes non-observable in experiments and has only a restricted theoretical interest.
\end{example}
\begin{remark} The above described situation where the limit behavior is trivial, but there is a rich and interesting metastable  dynamics, is somehow typical for dissipative systems and is observed in many other situations, for instance in 1D Burgers equation \eqref{6.bur}, see \cites{Bir01,Kuk99} and references therein. Similar situation also arises for stochastic equation \eqref{4.sde} for small values of $\eb>0$ where the corresponding random attractor is trivial (we even have uniform estimate \eqref{4.un-exp} for the Lyapunov exponent), but it takes a very long time (depending on $\eb$ in a way similar to \eqref{7.false1}) to reach this attractor. To avoid such problems, it is crucial to control the rate of attraction to the considered attractor in terms of physical parameters. Unfortunately, there are no ways to get this control on the level of usual (global) attractors in a more or less general situation. On the other hand, we often have such a control for the attraction rate to an IM, which allows to capture the metastable dynamics as well. Capturing this dynamics in more general cases, where an IM may not exist, is one of the main sources of motivation for the theory of {\it exponential} attractors. Actually, in the examples discussed above, we have an exponential rate of attraction to the corresponding global attractors, so the most important advantage of the theory of exponential attractors (at least in these cases) is {\it exactly} the ability to get a reasonable control (not like \eqref{7.false1}) for the exponential attraction rate in terms of physical parameters.
\end{remark}
The key idea of an {\it exponential} attractor is to add the metastable states to the global attractor in such a way that, on the one hand, to get an effective control for the rate of attraction and, on the other hand, to preserve the finite-dimensionality (in terms of fractal dimension), which allows us to use the Man\'e projection theorem.
\begin{definition}\label{Def7.exp} Let $S(t):\Phi\to\Phi$ be a DS acting in a metric space $\Phi$ with a fixed bornology $\Bbb B$. A set $\Cal M$ is an exponential attractor for the DS $S(t)$ if
\par
1) The set $\Cal M$ is compact in $\Phi$ and has a finite fractal dimension:
\begin{equation}
\dim_f(\Cal M,\Phi)<\infty;
\end{equation}
\par
2) It is semi-invariant, i.e. $S(t)\Cal M\subset\Cal M$, $t\ge0$;
\par
3) It attracts exponentially the images of $B\in\Bbb B$:
\begin{equation}
\dist_\Phi(S(t)B,\Cal M)\le Q_Be^{-\alpha t}
\end{equation}
for some positive constants $\alpha$ and $Q_B$.
\end{definition}
We recall some basic facts of the perturbation theory for global attractors before discussing the constructions of exponential attractors. We start with the upper semicontinuity which is based on the following fact from general topology.
\begin{proposition} Let $\Phi$ and $\frak A$ be two Hausdorff topological spaces and let $\Bbb A$ be a compact set in $\Phi\times\frak A$. Let also $\Pi_1,\Pi_2$ be the  projectors of $\Bbb A$ to the first and second  component of the Cartesian product and $\Cal A_\alpha:=\Pi_1\Pi_2^{-1}(\alpha)$, $\alpha\in\frak A$.
Then the family of sets $\{\Cal A_\alpha\}_{\alpha\in\frak A}$ is upper semicontinuous at every $\alpha_0\in\frak A$, namely, for every neighbourhood $\Cal O(\Cal A_{\alpha_0})$ of the set $\Cal A_{\alpha_0}$ in $\Phi$, there exists a neighbourhood $\Cal O(\alpha_0)$ in $\frak A$ such that
\begin{equation}\label{7.up-cont}
\Cal A_\alpha\subset\Cal O(\Cal A_{\alpha_0}),\ \ \forall\alpha\in\Cal O(\alpha_0).
\end{equation}
\end{proposition}
The proof of this proposition can be found, e.g. in \cite{BV92} for the case where $\Phi$ and $\frak A$ are metric spaces. The general case can be done analogously. Note also that in the metric case the upper semicontinuity can be rewritten in the equivalent way using the Hausdorff semi-distance:
\begin{equation}\label{7.metr}
\lim_{\alpha\to\alpha_0}\dist_{\Phi}(\Cal A_\alpha,\Cal A_{\alpha_0})=0.
\end{equation}
In applications, we usually take $\Bbb A:=\cup_{\alpha\in\frak A}\Cal A_\alpha\times\{\alpha\}$, where $\Cal A_\alpha$ are the attractors of DS $S_\alpha(t):\Phi\to\Phi$ depending on a parameter $\alpha\in\frak A$. Then, to verify the upper semicontinuity, we just need to consider an arbitrary sequence $u_n^0\in \Cal A_{\alpha_n}$ such that $\alpha_n\to\alpha_0$ and extract from it a subsequence $u_{n_k}^0$ which is convergent to some $u^0_{\alpha_0}\in\Cal A_{\alpha_0}$. In turn, in order to check this property, we may use  the representation formula for $\Cal A_{\alpha_n}$, so we only need to consider a sequence of complete bounded solutions $u_{\alpha_n}(t)\in\Cal K_{\alpha_n}$ and extract a subsequence $u_{\alpha_n}(t)$ which is convergent to $u_{\alpha_0}(t)\in\Cal K_{\alpha_0}$. This is usually true under the minimal assumptions on the DS $S_\alpha(t)$. We emphasize that this method {\it does not require} to verify the closeness of individual semi-trajectories of the perturbed and non-perturbed DS (which is often much more difficult task especially in the case of singular perturbations) and applicable, e.g. for trajectory attractors, see \cites{BV92,ChVi02,MZ08} and references therein.
\par
In contrast to the upper semicontinuity, the lower semicontinuity of attractors $\Cal A_\alpha$, $\alpha\in\frak A$ can be easily broken,  see Example \ref{Ex1.simple}. Recall, that in the metrizable case, the lower semicontinuity of the family $\{\Cal A_{\alpha}\}_{\alpha\in\frak A}$ reads
\begin{equation}\label{7.metr-low}
\lim_{\alpha\to\alpha_0}\dist_\Phi(\Cal A_{\alpha_0},\Cal A_\alpha)=0,\ \ \alpha_0\in\frak A.
\end{equation}
The continuity of the family of the attractors $\{\Cal A_\alpha\}_{\alpha\in\frak A}$ means that both \eqref{7.metr} and \eqref{7.metr-low} hold simultaneously, i.e.
\begin{equation}\label{7.metr-cont}
\lim_{\alpha\to\alpha_0}\dist_\Phi^{sym}(\Cal A_\alpha,\Cal A_{\alpha_0})=0,\ \ \alpha_0\in\frak A,
\end{equation}
where $\dist_\Phi^{sym}(U,V):=\max\{\dist_\Phi(U,V),\dist_\Phi(V,U)\}$ is a symmetric Hausdorff distance between sets. The next standard proposition shows that the continuity of attractors is determined by the rate of attraction.

\begin{proposition}\label{Prop7.low} Let $\Phi$ and $\frak A$ be subsets of normed spaces and let $S_\alpha(t):\Phi\to\Phi$ be a family of DS on $\Phi$ satisfying
\begin{equation}\label{7.cont}
\|S_{\alpha_1}(t)u_0-S_{\alpha_2}(t)u_0\|_\Phi\le C\|\alpha_1-\alpha_2\|_{\frak A}e^{Kt}
\end{equation}
for some positive constants $C$ and $K$ which are independent of $\alpha_i\in\frak A$ and $u_0\in\Phi$. Assume also that $\frak A$ is compact, $S_\alpha(t)$ are continuous for every fixed $t$ and $\alpha$, and, for every $\alpha\in\frak A$,  the DS $S_\alpha(t)$ possesses an attractor $\Cal A_\alpha$ with respect to the bornology of all subsets of $\Phi$. Then,
\par
1) The family $\{\Cal A_\alpha\}_{\alpha\in\frak A}$ is uniformly (w.r.t $\alpha\in\frak A$) continuous (i.e. \eqref{7.metr-cont} holds uniformly with respect to $\alpha_0$) if and only if
\begin{equation}\label{7.un}
\lim_{t\to\infty}\sup_{\alpha\in\frak A}\dist_\Phi(S_\alpha(t)\Phi,\Cal A_\alpha)=0.
\end{equation}
\par
2) If, in addition, the following stronger version of \eqref{7.un} (exponential attraction) holds:
\begin{equation}\label{7.exp1}
\dist_\Phi(S_\alpha(t)\Phi,\Cal A_\alpha)\le Qe^{-\lambda t},
\end{equation}
where positive constants $Q$ and $\lambda$ are independent of $t$ and $\alpha$, then the family of attractors $\Cal A_\alpha$ is uniformly H\"older continuous:
\begin{equation}\label{7.hol}
\dist_\Phi^{sym}(\Cal A_{\alpha_1},\Cal A_{\alpha_2})\le (Q+C)\|\alpha_1-\alpha_2\|^{\frac{\lambda}{K+\lambda}}_{\frak A}.
\end{equation}
\end{proposition}
\begin{proof} We consider here only the second part of the statement, the first one is more straightforward, see e.g. \cites{Hal88,BV92,LK04,HOR15}. Indeed, let $u_{\alpha_1}\in\Cal A_{\alpha_1}$. Then, by the invariance of the attractor, for any $T>0$, there exists $v_{\alpha_1}\in\Cal A_{\alpha_1}$ such that $S_{\alpha_1}(T)v_{\alpha_1}=u_{\alpha_1}$. Let us consider the point $\bar u_{\alpha_1}:=S_{\alpha_2}(T)v_{\alpha_1}$. Then, due to \eqref{7.exp1}, we have
$$
\dist_\Phi(\bar u_{\alpha_1},\Cal A_{\alpha_2})\le Qe^{-\lambda T}
$$
and, due to \eqref{7.cont}, we know that
$$
\|u_{\alpha_1}-\bar u_{\alpha_1}\|_\Phi\le C\|\alpha_1-\alpha_2\|_{\frak A}e^{KT}.
$$
Therefore, using that $u_{\alpha_1}\in\Cal A_{\alpha_1}$ is arbitrary, we have
$$
\dist_\Phi(\Cal A_{\alpha_1},\Cal A_{\alpha_2})\le Qe^{-\lambda T}+C\|\alpha_1-\alpha_2\|_{\frak A}e^{KT}.
$$
Fixing finally $T=\frac1{\lambda+K}\ln\frac1{\|\alpha_1-\alpha_2\|_{\frak A}}$, we arrive at
$$
\dist_\Phi(\Cal A_{\alpha_1},\Cal A_{\alpha_2})\le (Q+C) \|\alpha_1-\alpha_2\|_{\frak A}^\kappa,\ \kappa:=\frac\lambda{K+\lambda}.
$$
Swaping  $\alpha_1$ and $\alpha_2$, we get the desired estimate \eqref{7.hol}.
\end{proof}
\begin{remark} Assumption \eqref{7.cont} may look too restrictive since it requires the uniformity with respect to all $u_0\in\Phi$, but it is actually necessary for $u_0\in\Cal A:=\cup_{\alpha\in\frak A}\Cal A_{\alpha}$ only. The unifrom attraction property is also necessary for the set $\Cal A$ only, so the assumption that the bornology on $\Phi$ consists of all subsets of $\Phi$ is not restrictive at all. We may also formulate the natural non-uniform analogue of Proposition \ref{Prop7.low} for the upper and lower semicontinuity of attractors $\Cal A_\alpha$ at a single point $\alpha=\alpha_0$ as well as its analogues for non-autonomous case, see \cite{Hal88} for more details.
\par
Crucial for us is that the continuity of attractors $\Cal A_\alpha$ is determined by the rate of attraction and even the qualitative bounds for the closeness of the perturbed and non-perturbed attractors are automatically obtained if the rate of attraction is under  control. Unfortunately, the problem of getting this control looks unsolvable on the level of usual (global) attractors in a more or less general situation. This makes the attractor somehow unobservable, namely, no matter how long you observe the system and how precise your measurments/simulations, you cannot guarantee that the reconstructed approximate  attractor is close to the precise one and this is an extra source of motivation to consider exponential attractors. Note also that the H\"older continuity \eqref{7.hol} is the best what we can expect and it cannot be improved till Lipschitz continuity even when ideal objects like Morse-Smale or uniformly hyperbolic attractors are considered, see \cite{KH95} and  references therein.
\end{remark}
\begin{remark} We mention also an interesting result, which shows that the lower continuity of attractors $\Cal A_\alpha$ holds for "typical" values of the parameter $\alpha\in\frak A$, see \cite{HOR15} and \cite{WC22}. This result tells us that, if we are given a family of DSs $S_\alpha(t)$ (acting on a separable, complete  and bounded metric space $\Phi$), which depends continuosly on both $t$ and $\alpha\in\frak A$, where $\frak A$ is a compact metric space, then the corresponding attractors $\Cal A_\alpha$ (if exist) are both upper and lower semicontinuous for every $\alpha$ belonging to the residual set in $\frak A$ (i.e. for every $\alpha\in\frak A$ except of a countable union of nowhere dense sets). Moreover, if we are given a Borel probability measure $\mu$ on $\frak A$, then, for every $\eb>0$, there is a closed set $\frak A_\eb$ satisfying $\mu(\frak A\setminus\frak A_\eb)\le\eb$, such that the family $\Cal A_\alpha$ of attractors is uniformly continuous on $\frak A_\eb$ in the sense of the symmetric Hausdorff distance. Due to the first part of Proposition \ref{Prop7.low}, this, in turn, gives us the uniform attraction rate to the attractors $\Cal A_\alpha$ with respect to $\alpha\in\frak A_\eb$.
\par
Indeed, let us consider the space $B\Phi$, which consists of all closed non-empty subsets of $\Phi$ with the symmetric Hausdorff distance as a metric. Then, as known, $B\Phi$ is a complete separable metric space and the continuity of the family $\Cal A_\alpha$ can be interpreted as a continuity of the map $f:\frak A\to B\Phi$ given by $f(\alpha):=\Cal A_\alpha$. The idea of the proof consists of approximating the function $f(\alpha)$ by $f_n(\alpha):=S_\alpha(n)\Phi$. It is not difficult to see that the continuity assumptions posed on $S_\alpha(t)$ imply that the functions $f_n$ are continuous and the existence of attractors $\Cal A_\alpha$ ensures that $f_n(\alpha)\to f(\alpha)$ point-wise. Thus, by the Bair theorem, the set of discontinuities of $f(\alpha)$ is a countable union of nowhere dense sets. The second statement is an immediate corollary of the Egoroff theorem.
\par
It is worth mentioning that these results can not replace the exponential attractors since they do not give any reasonable way to find the set $\frak A_\eb$ explicitly or compute it as well as  any way to control the uniform rate of attraction to $\Cal A_\alpha$ for $\alpha\in\frak A_\eb$ in terms of physical parameters.
\end{remark}
The result of Proposition \ref{Prop7.low} can be partially extended to exponential attractors.
\begin{proposition}\label{Prop7.exp-bad} Let the assumptions of Proposition \ref{Prop7.low} hold and let the DS $S_\alpha(t)$ possess exponential attractors $\Cal M_\alpha$, $\alpha\in\frak A$, which satisfy the attraction property \eqref{7.exp1} uniformly with respect to $\alpha\in\frak A$. Then, the folowing estimate holds:
\begin{equation}\label{7.hol-exp}
\max\{\dist_\Phi(\Cal A_{\alpha_1},\Cal M_{\alpha_2}),\dist_\Phi(\Cal A_{\alpha_2},\Cal M_{\alpha_1})\}\le (Q+C)\|\alpha_1-\alpha_2\|^{\frac{\lambda}{K+\lambda}}_{\frak A}.
\end{equation}
\end{proposition}
The proof of this fact repeats word by word the arguments given in the proof Proposition \ref{Prop7.low} and for this reason is omitted.
\begin{remark} The robustness of exponential attractors with respect to perturbations has been formulated in the original work \cite{EFNT94} exactly in the form of \eqref{7.hol-exp}. The reason why we have to replace $\Cal M_\alpha$ by $\Cal A_\alpha$ in \eqref{7.hol-exp} is that the exponential attractor is only semi-invariant, so starting from $u_\alpha\in \Cal M_\alpha$, we may fail to find $v_\alpha\in\Cal M_\alpha$ (or in the set satisfying the uniform attraction property) such that $S_\alpha(T)v_\alpha=u_\alpha$. In  fact, we will be able to do so if $u_\alpha\in S_\alpha(t)\Cal M_\alpha$ for $t$ is large enough, so the following shifted version of H\"older continuity holds:
\begin{equation}\label{7.hol-gen}
\max\{\dist_\Phi(S_{\alpha_1}(t)\Cal M_{\alpha_1},\Cal M_{\alpha_2}),\dist_\Phi(\Cal S_{\alpha_2}(t)\Cal A_{\alpha_2},\Cal M_{\alpha_1})\}\le (Q+C)\|\alpha_1-\alpha_2\|^{\frac{\lambda}{K+\lambda}}_{\frak A}. \ \ t\ge T,
\end{equation}
where $T=T(\alpha_1,\alpha_2):=\frac1{\lambda+K}\ln\frac1{\|\alpha_1-\alpha_2\|_{\frak A}}$, see \cite{EFNT94}. Remarkable that this shifted H\"older continuity holds for {\it any} choice of exponential attractors $\Cal M_\alpha$ satisfying the uniform attraction property. In contrast to this, if we want to have the full analogue of estimate \eqref{7.hol} for exponential attractors, we need to construct them in a special way taking care about the closeness of the sets $\Cal M_{\alpha_1}\setminus S_{\alpha_1}(t)\Cal M_{\alpha_1}$ and $\Cal M_{\alpha_2}\setminus S_{\alpha_2}(t)\Cal M_{\alpha_2}$ for $t\le T$, see \cites{PRSE05,MZ08} and Theorem \ref{Th7.2lad} below.
\end{remark}
We state below the  transitivity of exponential attraction which is one of the key tools in the theory of exponential attractors.
\begin{proposition}\label{Prop7.trans} Let $\Phi$ be a metric space and $S(t)$ be a DS on it which is Lipschitz continuous:
\begin{equation}\label{7.l-cont}
d(S(t)u_1,S(t)u_2)\le Ce^{Kt}d(u_1,u_2),\ \ u_1,u_2\in\Phi, \ t\ge0,
\end{equation}
where the constants $C$ and $K$ are independent of $t$ and $u_i\in\Phi$. Assume also that there are 3 subsets $\Cal M_i\subset\Phi$, $i=1,2,3$ such that
$$
\dist_\Phi(S(t)\Cal M_1,\Cal M_2)\le C_1e^{-\lambda_1 t},\ \ \dist_\Phi(S(t)\Cal M_2,\Cal M_3)\le C_2e^{-\lambda_2 t}.
$$
Then,
\begin{equation}\label{7.trans-exp}
\dist_\Phi(S(t)\Cal M_1,\Cal M_3)\le C'e^{-\lambda' t},
\end{equation}
where $C'=CC_1+C_2$ and $\lambda'=\frac{\lambda_1\lambda_2}{K+\lambda_1+\lambda_2}$.
\end{proposition}
The proof of estimate \eqref{7.trans-exp} is based on the arguments which are similar to the derivation of \eqref{7.hol} and can be found in \cite{FMGZ}.

\subsection{Exponential attractors via squeezing properties} We now turn to theorems which guarantee the existence of exponential attractors with nice properties. These theorems are based on various forms of the squeezing property and  are very close to the ones given in subsection \ref{s5.sq}. Namely, the iterative $\eb$-nets which are constructed in order to estimate the fractal dimension of an attractor are exactly the "metastable" states which we need to add to the attractor in order to get the exponential attraction. We demonstrate the main idea on the analogue of Theorem \ref{Th5.lad} which was proved in \cite{EMZ00} (see also \cite{PRSE05} and \cite{MZ08} for more details).
\begin{theorem}\label{Th7.lad1} Let $\Phi$ and $\Phi_1$ be two Banach spaces and the embedding $\Phi_1\subset\Phi$ be compact. Assume also that we are given a bounded subset $\Cal B$ of $\Phi_1$ and a map $S:\Cal B\to\Cal B$ satisfying the squeezing property
\begin{equation}\label{7.repeat}
\|S(u_1)-S(u_2)\|_{\Phi_1}\le L\|u_1-u_2\|_\Phi,\ \ u_1,u_2\in\Cal B.
\end{equation}
Then the discrete DS $S(n):=S^n$, $n\in\Bbb N$, possesses an exponential attractor $\Cal M$ with the following properties:
\par
1) $\Cal M$ is compact in $\Phi_1$ and is semi-invariant: $S\Cal M\subset\Cal M$;
\par
2) Its fractal dimension in $\Phi_1$ is finite and satisfies the estimate
\begin{equation}\label{7.frac-est}
\dim_f(\Cal M,\Phi_1)\le \Bbb H_{\frac1{4L}}(\Phi_1\hookrightarrow\Phi);
\end{equation}
\par
3) $\Cal M$ attracts exponentially the set $\Cal B$ and
\begin{equation}\label{7.exp2}
\dist_{\Phi_1}(S(n)\Cal B,\Cal M)\le R_02^{-n},\ n\in \Bbb N,
\end{equation}
where $R_0$ is such that $\Cal B\subset B_{R_0}(0,\Phi)$.
\end{theorem}
\begin{proof}[Sketch of the proof] Indeed, arguing as in the proof of Theorem \ref{Th5.lad}, for every $n\in\Bbb N$, we construct an $R_02^{-n}$ net in $V_n\subset S(n)\Cal B$ such that $V_n\subset S(n)\Cal B$ and $\#V_n\le N^{n}$ where $N$ is the same as in the proof of Theorem \ref{Th5.lad}. To keep the semi-invariance, we define $E_1=V_1$ and $E_n=V_n\cup S(E_{n-1})$ for $n=2,3,\cdots$. Then, $E_n\subset S(n)\Cal B$, $\#E_n\le CN^{n+1}$ and $S(E_n)\subset E_{n+1}$. Let us finally define
\begin{equation}
\Cal M'=\cup_{n\in\Bbb N}E_n,\ \ \Cal M:=[\Cal M']_{\Phi_1}.
\end{equation}
Then the semi-invariance as well as exponential attraction \eqref{7.exp2} are immediately satisfied and we only need to check estimate \eqref{7.frac-est} for the fractal dimension. Let $\eb_n=R_02^{-n}$. Then all of the sets $E_k$ for $k\ge n$ are subsets of $S(n)\Cal B$ and, therefore, the $\eb_n$-balls centered at points of $V_n$ cover them. So, the set $\cup_{k=1}^nE_n$ is an $\eb_n$-net of $\Cal M$ and, therefore,
\begin{multline*}
\limsup_{n\to\infty}\frac{\Bbb H_{\eb_n}(\Cal M,\Phi_1)}{\log_2\frac1{\eb_n}}\le \lim_{n\to\infty}\frac{\log_2 C+\log_2(\sum_{k=1}^n\#E_n)}{n-\log_2R_0}\le\\\le \lim_{n\to\infty}\frac{2\log_2C+(n+2)\log_2N}{n-\log_2R_0}=\log_2N,
\end{multline*}
which gives the desired estimate \eqref{7.frac-est} and finishes the proof of the theorem.
\end{proof}
Let us now assume that we are given a continuous DS $S(t)$ in $\Phi$ such that $S=S(1)$ satisfies all of the assumptions of Theorem \ref{Th7.lad1}. Then, we first can construct a discrete exponential attractor $\Cal M_d$ (i.e. the attractor for the semigroup $S(n)$, $n\in\Bbb N$) and after that extend it to the continuous one via
\begin{equation}\label{7.ext}
\Cal M:=[\cup_{t\in[1,2]}S(t)\Cal M_d]_{\Phi}.
\end{equation}
Indeed, all the properties of an exponential attractor, except of the control of fractal dimension,  automatically follow from the analogous properties of $\Cal M_d$, but to control the fractal dimension in $\Phi$, we need an extra H\"older continuity assumption:
\begin{equation}\label{7.t-hol}
\|S(t_1)u_1-S(t_2)u_2\|_\Phi\le C\(|t_1-t_2|+\|u_1-u_2\|_\Phi\)^\alpha,\
\ t_1,t_2\in[1,2],\ u_1,u_2\in\Cal B,
\end{equation}
for some positive $C$ and $\alpha$. Then passing from discrete to  continuous exponential attractor can increase the fractal dimension in $\Phi$ by the additive factor $\alpha^{-1}$ at most. Moreover, using that $S(1)\Cal M$ is also an exponential attractor if $\Cal M$ is, we may get the finiteness of the fractal dimension in $\Phi_1$ as well (due to \eqref{7.repeat}, the fractal dimension of $S(t)\Cal M$ in $\Phi_1$ is controlled by the fractal dimension of $\Cal M$ in $\Phi$).

\begin{remark} In applications $\Cal B$ is usually a compact absorbing or an exponentially attracting set of the semigroup $S(t)$. Then, due to the transitivity of exponential attraction, we conclude that the constructed attractor $\Cal M$ attracts not only the set $\Cal B$, but all bounded sets of the phase space $\Phi$. Note also that for semigroups $S(t)$ generated by PDEs, the uniform H\"older continuity in time stated in \eqref{7.t-hol} typically holds only if $\Cal B$ is more smooth than the initial phase space $\Phi$. This is not a problem when the parabolic PDEs are considered since, due to the instantaneous smoothing property, we may find a compact (and more regular) {\it absorbing} set $\Cal B$. However, in more general cases (e.g. for damped wave equations), we only have an asymptotic smoothing property, so $\Cal B$ must be a compact exponentially {\it attracting} set and the usage of the transitivity of exponential attraction becomes unavoidable, see \cites{FMGZ,MZ08} and the references therein.
\end{remark}
At the next step, we consider, following  \cite{PRSE05}, the analogue of Theorem \ref{Th5.2lad}, which also includes the H\"older continuity of exponential attractors with respect to perturbations. To this end, we need the following definition.
\begin{definition} Let $\Phi_1\subset\Phi$ be two Banach spaces such that the embedding is compact and let a bounded set $\Cal B$ of $\Phi_1$, $\eb>0$, $\kappa\in[0,1)$ and $L>0$ be given. Then, a map $S:\Cal O_\eb(\Cal B)\to\Cal B$ belongs to the class $\Bbb S_{\eb,\kappa,L}(\Cal B)$ if the squeezing property \eqref{5.sq2} is satisfied for all $u_1,u_2\in \Cal O_\eb(\Cal B)$. The distance between two maps $S_1,S_2\in\Bbb S_{\eb,\kappa,\delta}(\Cal B)$ is defined as follows:
$$
\|S_1-S_2\|_{\Bbb S}:=\sup_{u\in\Cal O_\eb(\Cal B)}\|S_1(u)-S_2(u)\|_{\Phi_1}.
$$
\end{definition}
\begin{theorem}\label{Th7.2lad} Let $S\in\Bbb S_{\eb,\delta,L}(\Cal B)$ and let $S(n)=S^n$ be a discrete DS on $\Cal B$ associated with this map. Then this semigroup possesses an exponential attractor $\Cal M_S\subset \Cal B$ which satisfies the following properties:
\par
1) $\Cal M_S$ is a compact semi-invariant set in $\Cal B$ whose fractal dimension satisfies the analogue of  estimate \eqref{5.est1};
\par
2) The following exponential attraction property holds:
\begin{equation}
\dist_{\Phi_1}(S(n)\Cal B,\Cal M_S)\le \eb\(\frac{1+\kappa}2\)^n;
\end{equation}
3) If $S_1,S_2\in\Bbb S_{\eb,\kappa,L}(\Cal B)$, then
\begin{equation}\label{7.hhol}
\dist_{\Phi_1}^{sym}(\Cal M_{S_1},\Cal M_{S_2})\le C\|S_1-S_1\|_{\Bbb S}^{\theta}
\end{equation}
for some positive $C$ and $\theta$ depending only on $\eb,\kappa, L$ and the spaces $\Phi$ and $\Phi_1$.
\end{theorem}
\begin{proof}[Idea of the proof] The attractors $\Cal M_S$ can be constructed using the $\eb\(\frac{\kappa+1}2\)^n$-nets $E_n=E_n(S)$ similarly to Theorem \ref{Th7.lad1} (see also Theorem \ref{Th5.2lad}), so we only need to explain how to get the H\"older continuity \eqref{7.hhol}. We actually need to estimate the distance between $E_n(S_1)$ and $\Cal M_{S_2}$ for all $n\in\Bbb N$. For big values of $n$, we use \eqref{7.hol-gen} and get the desired estimate without any extra care. In contrast to this, for relatively small values of $n$, we need an extra assumption that the sets $E_n$ are constructed in such a way that
\begin{equation}
\dist_{\Phi_1}^{symm}(E_n(S_1),E_n(S_2))\le K^n \|S_1-S_2\|_{\Bbb S}
\end{equation}
for some $K$ which is independent of $n$,  $S_1$ and $S_2$. The details can be found in \cite{PRSE05}.
\end{proof}
\begin{remark} Recall that, similarly to inertial manifolds, exponential attractors are non-unique, so the problem of choosing  of an "optimal" exponential attractor becomes crucial for both theory and applications. The theorem stated above gives us a single valued H\"older continuous branch of the function $S\to\Cal M_S$ for wide class of nonlinear maps $S\in\Bbb S_{\eb,\kappa,L}(\Cal B)$.
\end{remark}

We now turn to the non-autonomous DS and start with the uniform (deterministic) case. We say that a discrete cocycle $S_\xi(n):\Cal O_\eb(\Cal B)\to\Cal B$, $n\in\Bbb N$ over the DS $T(n):\Psi\to\Psi$ belongs to the class $\Bbb S_{\eb,\kappa,L}(\Cal B)$ if $S_\xi(1)\in\Bbb S_{\eb,\kappa,L}(\Cal B)$ for all $\xi\in\Psi$. Then the non-autonomous analogue of Theorem \ref{Th7.2lad} reads.

\begin{theorem}\label{Th7.3lad} Let the cocycle $S_\xi(n)$ belong to the class $\Bbb S_{\eb,\kappa,L}(\Cal B)$ and let $U_\xi(m,n)$, $m\ge n$, $\xi\in\Psi$, be the corresponding dynamical processes on $\Cal B$. Then, there exists a family $\Cal M_S(\xi)$, $\xi\in\Psi$, of compact sets in $\Cal B$ (the  non-autonomous exponential attractor) which satisfies the following properties:
\par
1) Their fractal dimensions are finite and uniformly bounded: $\dim_f(\Cal M_S(\xi),\Phi_1)\le C$, $\xi\in\Psi$.
\par
2) They are semi-invariant: $S_\xi(1)\Cal M_S(\xi)\subset \Cal M_S(T(1)\xi)$, $\xi\in\Psi$, and the following uniform exponential attraction property holds:
\begin{equation}\label{7.un-exp}
\dist_{\Phi_1}(S_\xi(n)\Cal B,\Cal M_S(T(n)\xi))\le Qe^{-\alpha n},\ \ \xi\in\Psi,\ \ n\in\Bbb N,
\end{equation}
for some positive constants $Q$ and $\alpha$ depending only on $\Phi,\Phi_1$, $\Cal B$, $\eb$, $\kappa$ and $L$.
\par
3) For any two cocycles $S_\xi(n)$ and $\hat S_\xi(n)$ over the same DS $T(n):\Psi\to\Psi$ belonging to the class $\Bbb S_{\eb,\kappa,L}(\Cal B)$, the following uniform H\"older continuity holds:
\begin{equation}\label{7.n-hol}
\dist_{\Phi_1}^{sym}(\Cal M_{S}(\xi),\Cal M_{\hat S}(\xi))\le C\sup_{n\in\Bbb N}\{e^{-\beta n}\|S_{T(-n)\xi}(1)-\hat S_{T(-n)\xi}(1)\|_{\Bbb S}^\theta\},
\end{equation}
where positive constants  $C$, $\beta$ and $\theta$ depend only on $\Phi,\Phi_1$, $\Cal B$, $\eb$, $\kappa$ and $L$.
\end{theorem}
 The proof of this result is almost identical  to the proof of Theorem \ref{Th7.2lad} and is given in \cite{PRSE05}.
 \begin{remark} As usual, passage from discrete to continuous time require some uniform H\"older regularity in time for the considered cocycle (the analogue of estimate \eqref{7.t-hol}). In this case, the desired exponential attractor $\Cal M_S(\xi)$ can be defined via
 \begin{equation}\label{7.ext1}
 \Cal M_S(\xi):=[\cup_{t\in[1,2]}S_{\xi}(t)\Cal M_S^d(T(-t)\xi)]_{\Phi_1}.
 \end{equation}
 This construction extends all properties of the discrete exponential attractor stated in Theorem \ref{Th7.3lad} to the case of continuous time, see \cite{PRSE05} for the details. In particular if the cocycle $S_\xi(t)$ is autonomous, periodic or almost-periodic in time, the same will be true for the constructed non-autonomous exponential attractors. It also can be checked that, under some further natural assumptions, the function $t\to\Cal M_S(T(t)\xi)$ will be H\"older continuous in time.
 \par
 As we already mentioned, the assumption of H\"older continuity in time may be rather restrictive since it usually requires extra regularity in space (although in most part of applications it is not a big problem due to the transitivity of an exponential attraction), so it would be interesting to relax it. The attempts to do so are related with the usage of more straightforward and naive (than \eqref{7.ext1}) extension, namely, the continuous attractor is defined via $\Cal M_S(t):=\Cal M_S(T(t)\xi)$ for  $t=n\in\Bbb Z$  and, for $t=n+s$, $n\in\Bbb Z$, $s\in[0,1)$, we set $\Cal M_S(t):=S_{T(n)\xi}(s)\Cal M_S^d(n)$. In this case, we indeed will have semi-invariance, exponential attraction and uniform bounds for the fractal dimension without H\"older continuity in time, but the obtained object will be no more consistent with the autonomous case (where it will give an artificial time-periodic attractor), it will be not continuous in time (artificial jumps in integer points), etc. Since such an object hardly can be considered as a satisfactory version of a non-autonomous exponential attractor, the problem of removing/relaxing the H\"older continuity in time remains open.
 \end{remark}
\begin{remark} We emphasize that the {\it non-autonomous} exponential attractor $\Cal M_S(\xi)$, $\xi\in\Psi$, constructed in Theorem \ref{Th7.3lad} {\it is not} just a pullback attractor, it attracts also {\it forward} in time with {\it a uniform} exponential rate. This demonstrate one of the main advantages of exponential attractors for non-autonomous equations, namely, they allow us to settle the problem with forward attraction which looks unsolvable on the level of pullback attractors, see Example \ref{Ex4.bad} and, in contrast to uniform attractors, the constructed object remains finite-dimensional. For this reason, the denomination "pullback" exponential attractors, which is used by some authors (see e.g., \cite{CarS13}), looks confusing for us and we prefer to refer to them as {\it non-autonomous} exponential attractors.  Also, in contrast to uniform attractors, such exponential attractors do not violate the causal principle. Indeed, estimate \eqref {7.n-hol} shows us that the exponential attractor $\Cal M_S(\xi)$  depends on $S_{T(-n)\xi}(1)$, $n\in\Bbb N$, only and does not depend on the future ($S_{T(n)\xi}(1)$ with $n\ge0$). Moreover the impact of the past to the present attractor $\Cal M_S(\xi)$ decays exponentially fast with respect to the time passed (in a complete agreement with our intuition). Note that this property is also violated on the level of pullback attractors.
\par
The uniform analogues of exponential attractors which are independent of time and  where the finite-dimensionality is replaced by the proper estimate for the Kolmogorov $\eb$-entropy  also appear in the literature, see \cites{EMZ03,YKSZ23} and  references therein. Such constructions are usually based on the straightforward generalizations of Theorem \ref{Th5.un-lad}, so we will not give more details here. An alternative possibility, where the infinite-dimensional exponential attractors may appear, is the theory of dissipative PDEs in {\it unbounded} domains, where the global attractor is usually infinite-dimensional and we need to use Kolmogorov's entropy in order to control the size of the attractor, see \cites{EMZ04,MZ08} and references therein.
\end{remark}
We now turn to the  non-uniform (random) case where we have an ergodic measure $\mu$ for the underlying DS $T(t):\Psi\to\Psi$ and the set $\Cal B=\Cal B(\xi)\in\Phi_1$ as well as the squeezing factor $L=L(\xi)$ in \eqref{5.sq3} are  random variables. Namely, analogously to Theorem \ref{Th5.3lad}, we say that the discrete measurable cocycle $S_\xi(n)$ belongs to the class $\Bbb S_{\eb,\kappa,L}(\Cal B)$ if $S_\xi(1):\Cal O_\eb(\Cal B(\xi))\to \Cal B(T(1)\xi)$, $\xi\in\Psi$, and satisfy the squeezing property \eqref{5.sq3} for all $u_1,u_2\in\Cal O_\eb(\Cal B(\xi))$. Then the following result holds.

\begin{theorem}\label{Th7.5lad} Let $T(n):\Psi\to\Psi$, $n\in\Bbb Z$, be a DS on the Polish space $\Psi$ which possesses a Borel ergodic probability measure $\mu$. Let the cocycle $S_\xi(n)$ belong to the class $\Bbb S_{\eb,\kappa,L}(\Cal B)$ for some deterministic constants $\eb>0$, $\kappa\in[0,1)$, random constant $L=L(\xi)$,   $\Cal B(\xi)$ be a bounded random set in $\Phi_1$ and let the function $\xi\to \|\Cal B(\xi)\|_{\Phi_1}$ be tempered. We also assume that \eqref{5.sob-ent} is satisfied and $\Bbb E(L^\theta)<\infty$. Then there exists a compact random set $\Cal M_S(\xi)\subset \Cal B(\xi)$ such that
\par
1) The sets $\Cal M_S(\xi)$ are semi-invariant: $S_\xi(n)\Cal M_S(\xi)\subset \Cal M_S(T(n)\xi)$ for all $n\in\Bbb N$ and almost all $\xi\in\Psi$ and their fractal dimensions are bounded for almost all $\xi\in\Psi$ by a deterministic constant.
\par
2) The uniform attraction property
\begin{equation}\label{7.exp-un}
\dist_{\Phi_1}(S_\xi(n)\Cal B(\xi),\Cal M_S(T(n)\xi))\le Ce^{-\beta n}, \ n\in\Bbb N
\end{equation}
holds for almost every $\xi\in\Psi$ for some deterministic positive constants $C$ and $\beta$ which are independent of $\xi$ and $n$.
\par
5) Let $\hat S_\xi(n)$ be another cocycle belonging to the class $\Bbb S_{\eb,\kappa,L}(\Cal B)$ such that
\begin{equation}\label{7.delta}
\|S_\xi(1)-\hat S_\xi(1)\|_{\Bbb S}\le K(\xi)\delta
\end{equation}
for some deterministic $\delta$ and random $K(\xi)$ such that $\Bbb E(K^\theta)<\infty$. Then, there exists a random variable $P(\xi)$ which is finite almost everywhere and is independent of $\delta$ and the concrete choice of $S_\xi$ and $\hat S_\xi$, and a deterministic constant $\gamma>0$ such that
\begin{equation}\label{7.bad-H}
\dist_{\Phi_1}^{sym}(\Cal M_S(\xi),\Cal M_{\hat S}(\xi))\le P(\xi)\delta^\gamma.
\end{equation}
\end{theorem}
The proof of this theorem in the particular case $\kappa=0$ is given in \cite{ShZ13}. The general case is completely analogous and so is omitted.

\begin{remark} We note that, similarly to the deterministic case, the obtained random attractor {\it is not} just a pullback exponential attractor, but also a forward attractor and actually the exponential rate of attraction is {\it uniform}. Moreover, we have a {\it deterministic} rate of attraction to it, which can be controlled in terms of physical parameters of the considered system (as was pointed out in the original paper \cite{ShZ13}, where exponential attractors for random DS have been introduced). The attempts to relax this uniform attraction rate and to allow the constants $C$ and $\beta$ in \eqref{7.exp-un} to be random were also made later (see \cite{Zhou17} and references therein), but this did not lead to any essential simplifications of the assumptions on the cocycle $S_\xi(n)$, which may compensate the drawbacks related with the loss of the uniform/deterministic  attraction property. For instance, the random attractor $\Cal A(\xi)$ constructed in Example \ref{Ex4.sode} becomes "pullback exponential" if we allow the constant $C=C(\xi)$ in the exponential attraction property to be random and the control over the rate of attraction as well as H\"older continuity with respect to  $\eb\to0$ will be lost.
\par
As in the deterministic case, the passage from discrete to continuous exponential attractors require some H\"older continuity in time and the random version of transitivity of exponential attraction, see \cite{ShZ13}.
\par
We finally mention that,   in the H\"older continuity estimate \eqref{7.bad-H}, we may guarantee only that the random variable $P(\xi)$ is finite almost-everywhere. In order to verify that it has finite moments of some order, we need to control the rate of convergence in the Birkhoff ergodic theorem. This problem is rather delicate and requires not only the finiteness of the moments of any order for   $L(\xi)$ and $K(\xi)$, but also some kind of exponential mixing (in applications to stochastic equations we need the corresponding Ornstein-Uhlenbeck process to be exponentially mixing). We will not give more details here and send the interested reader to \cites{GAS11,Hipp, pri11,Kach, KuSh12,ShZ13} and also to references therein.
\end{remark}
We conclude this subsection by the model example of a reaction-diffusion system perturbed by white noise.
\begin{example}\label{Ex7.sde} Let $\Omega\subset\R^d$ be a bounded domain of $\R^d$ with a smooth boundary. Let us consider the following reaction-diffusion system in $\Omega$:
\begin{equation}\label{7.srds}
\Dt u=a\Dx u-f(u)+\delta \Dt \eta(t),\ \ u=(u_1,\cdots,u_n), \ \ u\big|_{t=0}=u_0,\ \ u\big|_{\partial\Omega}=0.
\end{equation}
It is assumed that $a$ is a constant diffusion matrix satisfying the condition $a+a^*>0$, the nonlinearity $f\in C^2(\R^n,\R^n)$ satisfies the dissipativity conditions of the form
\begin{equation}
1. \ f(u).u\ge -C+c|u|^{p+1},\ \ 2. \ f'(u)\ge-K,\ \ 3.\ |f'(u)|\le C(1+|u|^{p-1})
\end{equation}
for some positive constants $C$, $c$ and $K$ and the exponent $p$ satisfying $0\le p\le\frac{d+2}{d-2}$ for $d\ge 2$. Let $\{\lambda_i\}_{i=1}^\infty$ be the eigenvalues of the minus Laplacian in $\Omega$ enumerated in the non-decreasing order and $\{e_i\}_{i=1}^\infty$ be the corresponding eigenvectors. We assume that the two-sided Wiener process $\eta(t)$ in a Hilbert space $\Phi:=L^2(\Omega)$ is given by
\begin{equation}\label{7.Wiener-H}
\eta(t):=\sum_{i=1}^\infty \beta_i\eta_i(t)e_i,
\end{equation}
where $\{\eta_i(t)\}_{i=1}^\infty$ are the standard independent scalar Wiener processes and the deterministic vectors $\beta_i\in\R^n$ satisfy
$$
\sum_{i=1}^\infty\lambda_i^3|\beta_i|^2<\infty.
$$
Finally, $\delta\in[0,1]$ is a given parameter. Note that $\delta=0$ corresponds to the standard deterministic reaction-diffusion system and positive $\delta\ll1$ gives its stochastic perturbation.
\par
Following the standard scheme (see e.g. \cites{Da92, KuSh12} and references therein), we associate with equation \eqref{7.srds} a solution cocycle $S^\delta_\xi(t):\Phi\to\Phi$ over the standard DS on the canonical probability space $(\Psi,\Cal F,\mu)$. Namely, analogously to Example \ref{Ex4.sode}, we
consider the space $C_0(\R)$ endowed with the locally compact topology, the scalar Wiener measure $\bar \mu$ and the DS $(\theta(h)\bar\xi)(t):=\bar\xi(t+h)-\bar\xi(h)$. Then, $\Psi:=C_0(\R)^{\Bbb N}$ endowed with the Tikhonov topology, $T(h):\Psi\to\Psi$ and $\mu$ are the Cartesian products of semigroup $\theta(h)$ and the measures $\bar\mu$ respectively. It is well-known that the obtained measure $\mu$ is ergodic, see e.g. \cite{GAS11}. The solution cocycle $S^\delta_\xi(t):\Phi\to\Phi$ is constructed similarly to Example \ref{Ex4.sode}, subtracting the corresponding Ornstein-Uhlenbeck process,  see \cite{ShZ13} for more details.
\par
Moreover, it is straightforward to verify (again similarly to Example \ref{Ex4.sode}) that this cocycle possesses a tempered absorbing set $\Cal B_\xi$ in $\Phi_1:=H^1_0(\Omega)$ with respect to the bornology of tempered random sets. In addition, this absorbing set is uniform with respect to the bornology of bounded deterministic subsets of $\Phi$. Thus, it is enough to construct the random exponential attractors $\Cal M^\delta(\xi)$ for the set $\Cal B(\xi)$ only, so we may use the result of Theorem \ref{Th7.5lad} to construct the desired attractors. We also note that the trajectories of the Wiener process are H\"older continuous in time for almost all $\xi\in\Psi$, so the passage from discrete to continuous attractors does not cause any problems and we only need to check the assumptions of Theorem \ref{Th7.5lad} which is actually done in \cite{ShZ13}. Thus, we have the following result proved in \cite{ShZ13}.
\begin{theorem}\label{Th7.srds} Let the assumptions stated above hold. Then the solution cocycle $S^\delta_\xi(t):\Phi\to\Phi$, $\delta\in[0,1]$, associated with problem \eqref{7.srds} possesses a random family $\Cal M^\delta(\xi)$, $\delta\in[0,1]$, of exponential attractors which are
\par
1) Semi-invariant and have finite fractal dimensions in $\Phi_1$. These dimensions are bounded by a deterministic constant independent of $\delta$ for almost all $\xi$.
\par
2) Possess a uniform exponential attraction property for deterministic bounded sets of $\Phi$, i.e. there are positive constant $\alpha$ and monotone function $Q$ such that for every bounded subset $B$ of $\Phi$, we have
$$
\dist_{\Phi}(S_\xi^\delta(t)B,\Cal M^\delta(T(t)\xi))\le Q(\|B\|_{\Phi})e^{-\alpha t}, \ t\ge0
$$
for almost all $\xi$.
\par
3. There exists an almost everywhere  finite random variable $P(\xi)$ and a positive deterministic constant $\gamma$ which are independent of $\delta$ such that
\begin{equation}
\dist_{\Phi}^{sym}(\Cal M^{\delta_1}(\xi),\Cal M^{\delta_2}(\xi))\le P(\xi)|\delta_1-\delta_2|^\beta.
\end{equation}
Moreover, in the case $\delta=0$, the corresponding exponential attractor $\Cal M^0(\xi)$ is independent of $\xi$ and is the standard exponential attractor for the limit autonomous and deterministic reaction-diffusion system.
\end{theorem}
\end{example}
\begin{remark} We see that the properly constructed exponential attractors remain robust with respect to white noise perturbations as well. This demonstrates an advantage of exponential attractors in comparison to global ones, where  the limit deterministic attractor is not robust with respect to noise even in the ideal situation of regular attractors considered below, see Example \ref{Ex4.sode} (see also \cite{CF98a} and \cite{CDLM17} for the discussion related with stochastic bifurcation theory). We also mention that the conditions posed on the stochastic reaction-diffusion system \eqref{7.srds} are far from being optimal and can be essentially relaxed, which is however beyond  the scope of this survey.
\end{remark}

\subsection{Regular attractors} To conclude this section, we briefly consider the class of DS whose global attractors have an exponential rate of attraction, so they may be simultaneously considered as exponential ones. This usually happens when the considered DS possesses a global Lyapunov function and all of the equilibria are hyperbolic. Then any complete trajectory on the attractor is a heteroclinic orbit between these equilibria and the attractor is a finite collection of smooth finite-dimensional unstable submanifolds of the equilibria. Following \cite{BV83} such attractors are called {\it regular} and this is probably the only more or less general class of attractors, where we are able to understand their structure, see \cites{BV92,CL09,Hal04,HR89} for more details. In our exposition, we will mainly follow \cite{VZCh13}.
\par
We restrict ourselves  to considering only the case of discrete time (passing from discrete to continuous time is straightforward, see \cite{VZCh13} for the details), so we assume that we are given a map $S:\Phi\to\Phi$ in a Banach space $\Phi$ and construct a discrete semigroup via $S(n):=S^n$ in $\Phi$. We pose the following assumptions on the map $S$:
\par
{\it Assumption A:} The map $S\in C^1(\Phi,\Phi)$ and its Frechet derivative $S'(u)$ is uniformly continuous on bounded sets of $\Phi$; the map $S$ is injective and $\ker S'(u)=\{0\}$ for all $u\in\Phi$.
\par
{\it Assumption B:} The set of equilibria $\Cal R_0$ of the map $S$ is finite: $\Cal R_0:=\{u_1,\cdots,u_N\}$ and every of the equilibria is hyperbolic, the latter means that the spectrum of $S'(u_i)$ does not intersect with the unit circle for any $u_i\in\Cal R_0$.
\par
{\it Assumption C:} The semigroup $S(n):\Phi\to\Phi$ possesses a global attractor $\Cal A$ in $\Phi$ endowed with the standard bornology of bounded subsets of $\Phi$.
\par
{\it Assumption D:} Any trajectory $u(n)=S(n)u_0$ stabilizes to one of the equilibria from $\Cal R_0$ and there are no homoclinic structures, i.e. if $u_1,\cdots u_k\in l^\infty(\Bbb Z)$ are complete bounded trajectories of $S(n)$ such that
$$
\lim_{n\to-\infty}\|u_i(n)-v_i\|_\Phi=0,\ \ \lim_{n\to+\infty}\|u_i(n)-v_{i+1}\|_\Phi=0,\ \ i=1,\cdots k
$$
for some equilibria $v_1,\cdots,v_{k+1}\in\Cal R_0$, then necessarily all of $v_i$ are {\it different}.
\par
In applications, the key Assumption  D usually follows from the existence of a global Lyapunov function. We recall that a continuous function $L:\Phi\to\R$ is called a global Lyapunov functional if the function $n\to L(S(n)u_0)$ is non-increasing along the trajectories and the equality $L(S u_0)=L(u_0)$ implies that $u_0\in\Cal R_0$. Together with Assumption C and the finiteness of the set $\Cal R_0$, this implies not only the validity of D, but also the fact that any complete bounded trajectory is a heteroclinic orbit between two equilibria, see \cites{BV92,VZCh13} for more details. In turn, this gives the following description of the attractor $\Cal A$:
\begin{equation}\label{7.man}
\Cal A=\cup_{i=1}^N\Cal M^+(u_i),
\end{equation}
where $\Cal M^+(u_i)$ is an unstable set of the equilibrium $u_i$, i.e. the set of all initial data $u_0$ for which there exists a complete trajectory $u(n)$ converging to $u_i$ as $n\to-\infty$.
\par
We now use the hyperbolicity Assumption B to verify that the sets $\Cal M^+(u_i)$ are actually finite-dimensional manifolds in $\Phi$ with the dimensions equal to the  instability index $\ind(u_i)$ which is the algebraic number of eigenvalues of $S'(u_0)$ outside of the unit circle (it is easy to verify that this numbers are all finite due to the existence of a compact global attractor). The proof of this fact can be done similarly to our verification of the existence of inertial manifolds, so we will not present it here. Note only that we first construct the manifolds {\it locally} in a small neighbourhood of the equilibria and then extend them to the global submanifolds using the injectivity assumption of A and the absence of homolcinic structures. This gives us the fact that all $\Cal M^+(u_i)$ are finite-dimensional submanifolds of $\Phi$ which are diffeomorphic to $\R^{\ind(u_i)}$, see \cite{VZCh13} for the details.
\par
Finally, we recall that, similarly to inertial manifolds, the manifolds $\Cal M^+(u_i)$ possess an exponential tracking property, i.e. any trajectory of $S(n)$ is attracted exponentially fast to some trajectory on $\Cal M^+(u_i)$ until it remains in a small neighbourhood of $u_i$. Moreover, due to Assumption D, any trajectory of $S(n)$ spends a finite time $T_\delta$ outside of the $\delta$-neighbourhood of $\Cal R_0$ and $T_\delta$ is uniform with respect to all trajectories starting from a bounded set. These two facts give us the exponential attraction to $\Cal A$ with the rate of convergence controlled by the number of equilibria and their hyperbolicity constants, see \cite{VZCh13} for details. We summarize the obtained results in the following theorem.

\begin{theorem}\label{Th7.reg} Let the map $S:\Phi\to\Phi$ satisfy Assumptions A--D. Then the attractor $\Cal A$ of the associated semigroup $S(n):\Phi\to\Phi$ possesses the  description \eqref{7.man}, where $\Cal M^+(u_i)$ are the $\ind(u_i)$-dimensional unstable  submanifolds of the equilibria $u_i$. Moreover, any trajectory  belonging to the attractor is a heteroclinic orbit between two equilibria and the rate of attraction to $\Cal A$ is exponential, i.e., there exists a positive constant $\alpha$ and a monotone function $Q$ such that, for every bounded subset $B\subset\Phi$,
\begin{equation}\label{7.exp-reg-bad}
\dist_\Phi(S(n)B,\Cal A)\le Q(\|B\|_\Phi)e^{-\alpha n},\ \ n\in\Bbb N.
\end{equation}
\end{theorem}
\begin{remark} Note that, despite the exponential attraction rate \eqref{7.exp-reg-bad} for the regular attractor, it still makes sense to construct an exponential attractor even in the case where the global  attractor is regular. The problem here is that the constant $\alpha$ and the function $Q$ are still not controllable in terms of physical parameters of the considered DS and we still may lose the important intermediate dynamics, see Example \ref{Ex7.inter}. Note also that although Assumption B is in a sense generic due to the Sard theorem,  the number of equilibria and their hyperbolicity constants can be explicitly found/extimated in very exceptional cases only, so even the attraction constant $\alpha$ in the exponential attraction rate is usually "not observable".
\end{remark}
We are now turn to the perturbation theory. To this end we assume that there is a family of cocycles $S_{\delta,\xi}(n):\Phi\to\Phi$ for some DS $T(n):\Psi\to\Psi$ depending on a parameter $\delta\in[0,1]$
(for simplicity, we assume that the set $\Psi$ and $T(n)$ are independent of $\delta$ although this is not essential). We assume that, for $\delta=0$, the corresponding maps $S_{\delta,\xi}(n)$ are independent of $\xi$ and $S_{0}(1)$ satisfies Assumptions A-D. Thus, at the limit $\delta=0$, we have a regular attractor. To specify the perturbation, we need two more assumptions:
\par
{\it Assumption E.} The maps $S_{\delta,\xi}(1)\in C^1(\Phi,\Phi)$ and the corresponding Frechet derivatives $S'_{\delta,\xi}(1)$ are uniformly continuous on  bounded sets of $\Phi$ (also uniformly with respect to $\xi$ and $\delta$). Moreover the maps $S_{\delta,\xi}(1)$ are injective for all $\xi$ and $\delta$ and $\ker{S'_{\delta,\xi}(1)}=\{0\}$ for all $\delta$, $\xi$ at any point of $\Phi$. Finally, we assume that
\begin{equation}
\||S_{\delta,\xi}(1)(v)-S_0(1)(v)\|_{\Phi}+\|S'_{\delta,\xi}(1)(v)-S'_0(1)(v)\|_{\Cal L(\Phi,\Phi)}\le C\delta,
\end{equation}
where the constant $C$ depends on the norm $\|v\|_\Phi$ only (and is independent of $\xi$ and $\delta$).
\par
{\it Assumption F.} The family of cocycles $S_{\delta,\xi}(n)$ possesses a compact uniformly attracting set $\Cal B\in\Phi$ with respect to the standard bornology of bounded sets in $\Phi$ (which is also uniform with respect to $\delta$).
\par
The perturbation theory of regular attractors is based on two relatively simple observations:
\par
 1) Since the cocycle $S_{\delta,\xi}(n)$ is close to the limit semigroup $S_0(n)$ and possesses a compact uniformly attracting set, the properties of the perturbed trajectories will be close to the properties of the non-perturbed ones. In particular, any perturbed trajectory will visit some $\eb$-neighbourhood of the equilibria set $\Cal R_0$ with $\eb=\eb(\delta)$ tending to zero as $\delta\to0$. Moreover, due to the absence of homoclinic structures for the limit DS, this trajectory is  unable to visit the proper neighbourhood of any $u_i\in\Cal R_0$ more than once and spends a uniformly bounded time outside of the $\eb$-neighbourhood of $\Cal R_0$, see \cite{VZCh13} for the details. This actually reduces the analysis to the {\it local} perturbation theory near the hyperbolic equilibria.
\par
2) Since the limit equilibria $u_i\in\Cal R_0$ are hyperbolic, the saddle structure survives under small non-autonomous perturbations. In particular, any equilibrium $u_i\in\Cal R_0$ generates an "equilibrium" $u_i(\xi)$ of the perturbed cocycle $S_{\delta,\xi}(n)$ if $\delta$ is small enough. This equilibrium is uniquely determined by the condition
\begin{equation}
\sup_{\xi\in\Psi}\|u_i(\xi)-u_i\|_\Phi\le C\delta.
\end{equation}
The constructed "equilibria" $u_i(\xi)$ remain hyperbolic and the corresponding unstable manifolds $\Cal M^+_\delta(\xi)$ are close to the unstable manifold $\Cal M^+_0$ of the limit problem as $\delta=0$ and possess the uniform exponential tracking property. In addition, due to the injectivity property these manifolds (which are initially defined for small neighbourhood of $u_i$ only) can be extended to the global unstable submanifolds which will be diffeomorphic to $\R^{\ind(u_i)}$. This allows us to establish the analogue of \eqref{7.man} for the non-autonomous (pullback) attractor $\Cal A_\delta(\xi)$ for the perturbed cocycle $S_{\delta,\xi}(n)$:
\begin{equation}\label{7.man-non}
\Cal A_\delta(\xi)=\cup_{i=1}^N\Cal M_\delta^+(\xi).
\end{equation}
Moreover, we also have the property that any complete bounded trajectory is a heteroclinic orbit between two "equilibria" $u_i(\xi)$ as well as uniform exponential attraction of bounded sets. Finally, arguing as in Proposition \ref{Prop7.low}, we establish the uniform H\"older continuity of the form
\begin{equation}\label{7.h-last}
\dist_\Phi(\Cal A_\delta(\xi),\Cal A_0)\le C\delta^\kappa
\end{equation}
for some positive $\kappa$ and $C$ which are independent of $\xi$ and $\delta$, see \cite{VZCh13} for the details. We summarize the results obtained there in the following theorem.

\begin{theorem}\label{Th7.last} Let the family of cocycles $S_{\delta,\xi}(n):\Phi\to\Phi$ over the DS $T(h):\Psi\to\Psi$ depending on the parameter $\delta\in[0,1]$ satisfy Assumptions A-F. Then there exists $\delta_0>0$ such that, for every $\delta\le\delta_0$, the non-autonomous (pullback) attractor $\Cal A_\delta(\xi)$ of the cocycle $S_{\delta,\xi}(n)$ is a finite union of the unstable manifolds $\Cal M^+_\delta(\xi)$ of the corresponding perturbed "equilibria" $u_i(\xi)$, $\xi\in\Psi$, of the limit hyperbolic equilibria $u_i\in\Cal R_0$ (i.e., \eqref{7.man-non} holds). Moreover, any complete bounded trajectory of the DP $U_{\delta,\xi}(m,n)$ associated with the cocycle $S_{\delta,\xi}(n)$ is a heteroclinic orbit between $u_i(T(n)\xi)$ and $u_j(T(n)\xi)$ for some $i\ne j$. The rate of attraction to the attractors $\Cal A_\delta(\xi)$ is uniform and exponential, i.e. there exists a positive constant $\alpha$ and a monotone function $Q$ such that, for every bounded set $B\subset\Phi$ and every $m\ge n$,
$$
\dist_\Phi(U_\xi(m,n)B, \Cal A_\delta(T(n)\xi))\le Q(\|B\|_\Phi)e^{-\alpha (m-n)},\ \ \delta\le \delta_0,\ \xi\in\Psi.
$$
Finally, the family of attractors $\Cal A_\delta(\xi)$ is H\"older continuous at $\delta=0$, i.e. \eqref{7.h-last} holds.
\end{theorem}
\begin{remark} We see that the object obtained as a non-autonomous perturbation of a regular attractor gives us an example of non-autonomous {\it exponential} attractor introduced before. In particular, it is not only pullback, but also forward in time exponential attractor. This brings us an evidence that the used definition of the non-autonomous exponential attractor is correct and natural at least in the deterministic (uniform) case. Unfortunately, a reasonably general  analogue of Theorem \ref{Th7.last} does not seem to exist, for the random (non-uniform) case, see Example \ref{Ex4.sode}. This is closely related with the well-known problem of developing the center (stable/unstable) manifolds theory for random/stochastic DS. The theory of random exponential attractors developed in \cite{ShZ13} and discussed above  may be one of possible ways to handle this problem.
\end{remark}

\section{Determining functionals}\label{s8}
In this section, we discuss an alternative approach to the justification of the finite-dimensional reduction in dissipative PDEs, which is related with the concept of {\it determining functionals}. This approach was introduced in \cite{FP67} (see also \cite{Lad72}) for the case of 2D Navier-Stokes equations and Fourier modes and has been extended later to many other classes of dissipative systems and various classes of determining functionals, see \cites{6,7,8,9,14,15} and references therein.  In our exposition, we mainly follow the recent paper \cite{KAZ22}.

\begin{definition}\label{Def8.det} Let $S(t):\Phi\to\Phi$ be a DS acting in a Banach space $\Phi$. Then a finite system of continuous functionals $\Cal F:=\{\Cal F_1,\cdots,\Cal F_N\}$, $\Cal F_i:\Phi\to\R$ is called asymptotically determining if for any two trajectories $u_1(t):=S(t)u_1$ and $u_2(t):=S(t)u_2$, the convergence
$$
\lim_{t\to\infty}(\Cal F_i(u_1(t))-\Cal F_i(u_2(t)))=0, \ \ i=1,\cdots, N
$$
implies that $\lim_{t\to\infty}\|u_1(t)-u_2(t)\|_{\Phi}=0$.
\end{definition}
Thus, if $\Cal F$ is an asymptotically determining, the behaviour of any trajectory $u(t)$ as $t\to\infty$ is determined by the behaviour of finitely many quantities $\Cal F_1(u(t)),\cdots,\Cal F_N(u(t))$. At first glance, this may look as a kind of finite-dimensional reduction, but a more detailed analysis shows that it is not the case in general since, in contrast to inertial forms constructed via Man\'e projection theorem or inertial manifolds, the quantities $\{\Cal F_i(u(t))\}_{i=1}^N$ are not obliged to satisfy a system of ODEs. Actually, in many cases, they satisfy the system of {\it delay} differential equations, but the phase space of such systems remain infinite-dimensional, so despite the widespread misunderstanding, determining  functionals {\it do not} give any finite-dimensional reduction and are actually responsible for the reduction to a system of ODEs with delay.  Nevertheless, such a reduction is also interesting from theoretical point of view and has non-trivial applications in many related areas, for instance, for establishing the
controllability of an initially infinite dimensional system by finitely many
modes (see e.g. \cite{AT14}), verifying the uniqueness of an invariant measure for
random/stochasitc PDEs (see e.g. \cite{KuSh12}), etc. We also mention more recent
but  promising applications of determining functionals to data assimilation problems where the values of functionals $\Cal F_i((u(t))$ are interpreted as
the results of observations and the theory of determining functionals allows
us to build new methods of restoring the trajectory $u(t)$ by the results of
observations, see \cites{AT14, AOT13, OT08} and references therein.
\par
In the case where the considered DS possesses an attractor (with respect to the standard bornology of bounded sets in $\Phi$), we may use an alternative non-equivalent but closely related concept of separating functionals.
\begin{definition}\label{Def8.sep} Let the DS $S(t):\Phi\to\Phi$ acting on a Banach space $\Phi$ possess an attractor $\Cal A$ with respect to the standard bornology. Then a system $\Cal F:=\{\Cal F_i\}_{i=1}^N$ is separating on the attractor if for any two complete bounded trajectories $u_1(t)$ and $u_2(t)$ belonging to the attractor, the equality
$$
\Cal F_i(u_1(t))=\Cal F_i(u_2(t)),\ \ t\in\R,\ \ i=1,\cdots,N
$$
implies that $u_1(t)\equiv u_2(t)$.
\end{definition}
It is not difficult to show that (under the mild extra assumption that $S(t)$ are continuous for every $t$), any separating system $\Cal F$ of continuous functionals is automatically asymptotically determining, see e.g. \cite{KAZ22}. The converse is not true in general, the corresponding counterexample is also given in \cite{KAZ22} based on Example \ref{Ex1.det}. This simple observation is rather useful since usually the attractor is more regular and the separation property is easier to verify.
\par
In order to illustrate the standard approach to determining functionals, we consider the model example of an abstract semilinear parabolic equation
\begin{equation}\label{8.sem}
\Dt u+Au=F(u),\ u\big|_{t=0}=u_0
\end{equation}
in a Hilbert space $\Phi$, where $A:D(A)\to\Phi$ is a self-adjoint linear positive operator with the compact inverse and $F:\Phi\to\Phi$ is globally Lipschitz in $\Phi$ with the Lipschitz constant $L$. Let
$\{\lambda_i\}_{i=1}^\infty$ be the eigenvalues of the operator $A$ enumerated in a non-decreasing order and $\{e_i\}_{i=1}^\infty$ be the corresponding eigenvectors.

\begin{theorem}\label{Th8.det1} Let $N\in\Bbb N$ be such that  $\lambda_{N+1}>L$. Then the system functionals $\Cal F_i(u)=(u,e_i)$, $i=1,\cdots,N$ (which are the first $N$ Fourier modes) is asymptotically determining for DS generated by equation \eqref{8.sem}.
\end{theorem}
\begin{proof} Indeed, let $u_1(t)$ and $u_2(t)$ be two solutions of equation \eqref{8.sem} such that
$P_N(u_1(t)-u_2(t))\to0$ as $t\to\infty$, where $P_N$ is a spectral orthoprojector related with the first $N$ Fourier modes and $Q_N=1-P_N$. Let $v(t):=P_N(u_1(t)-u_2(t))$ and $w(t):=Q_N(u_1(t)-u_2(t))$. Then, the last function solves the equation
\begin{equation}\label{8.good}
\Dt w+Aw=Q_N(F(u_1)-F(u_2)).
\end{equation}
Multiplying this equation by $w$ and using that $F$ is globally Lipschitz, we end up with
$$
\frac12\frac d{dt}\|w(t)\|^2_\Phi+\lambda_{N+1}\|w(t)\|^2_\Phi\le L\|w\|^2_\Phi+L\|w\|_\Phi\|v\|_\Phi.
$$
 Using the assumption $\lambda_{N+1}>L$, we arrive at
 $$
 \frac d{dt}\|w(t)\|_\Phi+\alpha\|w(t)\|_{\Phi}\le L\|v(t)\|_{\Phi}
 $$
 for some positive $\alpha$. Integrating this inequality and using that $\|v(t)\|_\Phi\to0$ as $t\to\infty$, we conclude that also $\|w(t)\|_{\Phi}\to0$ which finishes the proof of the theorem.
\end{proof}
\begin{remark} The result of the last theorem can be generalized in a straightforward way to more general classes of linear determining functionals. For instance, if a system $\Cal F:=\{\Cal F_1,\cdots,\Cal F_N\}$ of  linear continuous functionals satisfies the following inequality
$$
(L+\alpha)\|v\|^2_{\Phi}\le (Av,v)+C\sum_{n=1}^N|\Cal F_n(v)|, \ v\in D(A^{1/2}),
$$
for some positive $\alpha$ and $C$, then taking two trajectories $u_1(t)$ and $u_2(t)$ and $v(t):=u_1(t)-u_2(t)$, multiplying the equation for $v$ by $v$ in $\Phi$ and using the Lipschitz continuity of $F$ and the last assumption, we arrive at
\begin{equation}\label{8.det-lin}
\frac12\frac d{dt}\|v(t)\|^2_\Phi+\alpha\|v(t)\|^2_\Phi\le C\sum_{n=1}^N|\Cal F_n(v(t))|.
\end{equation}
Thus, the system $\Cal F$ is indeed asymptotically determining. This allowed one to construct determining nodes as well as much more general systems of linear determining functionals, see \cite{7} and references therein for more details.
\par
We also note that, due to \eqref{8.det-lin}, the complete trajectory $u(t)$, $t\in\R$, is determined in a unique way by the values $\xi_i(t):=\Cal F_i(u(t))$, i.e.
$$
u(t)=\frak F(\xi_1(t+\cdot),\cdots,\xi_N(t+\cdot)),\ \ \frak F:[C(-\infty,0;\R)]^N\to\Phi
$$
and we have the exponential decay of the delay kernel:
$$
\|\frak F(\xi)-\frak F(\bar\xi)\|_\Phi\le C\sup_{s\in\R_+}\{e^{\alpha s} \|\xi(-s)-\bar\xi(-s)\|_{\R^N}\}.
$$
Thus, returning back to the quantities $\xi_n(t):=\Cal F_n(u(t))$, we see that, in general, they do not satisfy a system of ODEs in $\R^N$, but  satisfy a system of {\it retarded} ODEs with an infinite delay and exponentially decaying delay kernel:
$$
\frac d{dt}\xi(t)=\Cal G(\xi(t+\cdot)),\ \ \xi(t)\in\R^N, \ \ \Cal G: [C(-\infty,0;\R)]^N\to\R^N.
$$
Thus, since the phase space for such a retarded system remains infinite-dimensional, determining functionals do not provide any finite-dimensional reduction, but reduce the initial PDE to a system of ODEs with delay (which is often referred as Lyapunov-Schmidt reduction).
\end{remark}
We now give a more precise look to the number of elements in the "optimal" system of determining functionals. To this end, we give the following definition.
\begin{definition}\label{Def8.dim} Let $S(t):\Phi\to\Phi$ be a DS in a Banach space $\Phi$. Then the determining  dimension $\dim_{det}(S(t),\Phi)$ is the minimal number $N$ of continuous functionals $\Cal F_1,\cdots,\Cal F_N:\Phi\to\R$ such that $\Cal F:=\{\Cal F_n\}_{n=1}^N$ is asymptotically determining.
\par
Note that we do not require the functionals $\Cal F_n$ to be {\it linear}. Although in applications determining systems often consist of linear functionals, the requirement of the linearity in a general theory looks artificial when the nonlinear DS is considered. In addition, the usage of nonlinear, e.g. quadratic, cubic, etc. functionals simplifies the theory and makes it more elegant.
\end{definition}
We recall that there exists a simple and natural lower bound for $\dim_{det}(S(t),\Phi)$ which is related with the embedding dimension of the set $\Cal R$ of equilibria points. Namely, by definition, $\dim_{emb}(\Cal R,\Phi)$ is the minimal number $M\in\Bbb N$ such that there exists a continuous injective map $\Cal F:\Cal R\to\R^M$. Then, obviously,
\begin{equation}\label{8.lower}
\dim_{det}(S(t),\Phi)\ge\dim_{emb}(\Cal R,\Phi).
\end{equation}
Indeed, any asymptotically determining system must, in particular, distinguish different equilibria. Mention also that $\dim_{emb}(\Cal R, \Phi)$ is a topological invariant and can be estimated using e.g. Lebesgue covering dimension.
\par
Another interesting observation is that, due to the Sard lemma, the set of equilibria $\Cal R$ is generically finite, see \cites{BV92, Rob11}, in particular, it will be so for equation \eqref{8.sem}. Therefore, in a generic situation, \eqref{8.lower} gives us $\dim_{det}(S(t),\Phi)\ge1$ for the lower bound. One of the most surprising results of the theory is that this estimate is sharp, i.e., there exists a dense set of smooth functionals such that each of them is asymptotically determining.

\begin{theorem}\label{Th8.det} Let the set $\Cal R$ of equilibria be finite. Then, the determining dimension of the solution semigroup $S(t):\Phi\to\Phi$ associated with equation \eqref{8.sem} is equal to one. Moreover, there is a dense/prevalent set of polynomial functionals on $\Phi$ such that each of them is a determining functional for $S(t)$.
\end{theorem}
\begin{proof}[Idea of the proof] The proof is based on the H\"older continuous infinite-dimensional version of the famous Takens delay embedding theorem, see \cites{SIC91,Tak81,Rob11} and references therein. To apply this theorem, we need to control the size of the set $\Cal R$ of equilibria (which is available due to our assumptions) as well as the size of the sets of periodic orbits of period
$\tau,2\tau,\cdots, k\tau$, where $k$ is large enough. This control is attained using the fact that if $\tau\ll1$, all of such periodic orbits must be equilibria, see \cite{Rob11} for the details. On the other hand, the fractal dimension of the attractor $\Cal A$ is finite and we have the one-to-one H\"older continuous Man\'e projection of this attractor to the finite-dimensional set $\bar{\Cal A}\subset\R^N$ with the projected DS $\bar S(t)$ on it.
\par
Thus, due to the proper version of the Takens delay embedding theorem, see e.g. \cite{Rob11}, there exists a prevalent set of continuous functionals $\Cal F:\Phi\to\R$ (polynomials of sufficiently big degree are enough) such that the map
$$
\Bbb F_k:\Cal A\to \R^k,\ \ \Bbb F_k(u_0):=(\Cal F(u_0),\Cal F(S(\tau)u_0),\Cal F(S((k-1)\tau)u_0))
$$
is one-to-one on the attractor $\Cal A$ if $k$ is large enough and $\tau>0$ is small enough. Thus, each of such functionals is separating trajectories on the attractor and, therefore, is asymptotically determining, see \cites{KAZ22, Rob11} for more details.
\end{proof}
\begin{remark} The result of Theorem \ref{Th8.det} also gives a reduction to a delayed ODE. Namely, as shown in \cite{KAZ22}, there exists a continuous function $\Theta:\R^k\to \R$ such that the observable $\xi(t):=\Cal F(u(t))$, where $u(t)$ is any complete trajectory on the attractor, solves the scalar ODE with the {\it finite} delay:
\begin{equation}\label{8.scalar}
\frac d{dt}\xi(t)=\Theta(\xi(t-\tau),\cdots,\xi(t-k\tau))
\end{equation}
and if the values of $\xi(t-\tau),\xi(t-2\tau),\cdots,\xi(t-k\tau)$ are known for some $t\in\R$, the corresponding point $u(t)\in\Cal A$ can be restored using the attractor reconstruction procedure provided by the Takens delay embedding theorem.
\end{remark}
The situation when the set $\Cal R$ is not-finite looks similar.
\begin{proposition}\label{Prop8.det} Let $S(t):\Phi\to\Phi$ be a DS associated with equation \eqref{8.sem} and let $\Cal R$ be its set of equilibria. Then, the determining dimension of $S(t)$ satisfies:
\begin{equation}\label{8.gen}
\dim_{emb}(\Cal R,\Phi)\le \dim_{det}(S(t),\Phi)\le \dim_{emb}(\Cal R,\Phi)+1.
\end{equation}
\end{proposition}
The proof of this proposition is analogous to the proof of Theorem \ref{Th8.det} and is given in \cite{KAZ22}. We expect that the left inequality \eqref{8.gen} is actually an equality, but up to the moment we did not check this. Moreover, the analogue of delay ODE \eqref{8.scalar} holds for this case as well, see \cite{KAZ22}.
\begin{remark} We see that the determining functionals are indeed responsible for the reduction of the initial PDE to ODEs {\it with delay} and are not related with any kind of {\it finite-dimensional} reduction. In addition, the minimal number of determining functionals (determining dimension) is related with the "size" of the set of equilibria $\Cal R$ of the considered DS only (and is not related with any dynamical properties of this DS). Although these results are formulated for semilinear parabolic system \eqref{8.sem} only, their close analogues remain true for much wider class of equations including Navier-Stokes system, damped wave equations, etc.
\end{remark}
We conclude this section by several illustrating examples. We start with the most studied case of one spatial dimension.
\begin{example}\label{Ex8.dir} Let us consider the following 1D semilinear heat equation
\begin{equation}\label{8.heat}
\Dt u=\nu\partial_x^2u-f(u)+g,\ \ x\in[0,\pi], \ \ \nu>0
\end{equation}
endowed with Dirichlet boundary conditions
$$
u\big|_{x=0}=u\big|_{x=\pi}=0.
$$
Assume also that $f\in C^1(\R,\R)$ satisfies some dissipativity conditions, say $f(u)u\ge-C$. Then, equation \eqref{8.heat} generates a dissipative DS in the phase space $\Phi=L^2(0,\pi)$ and this DS possesses a global attractor $\Cal A$ which is bounded at least in $C^2([0,\pi])$, see e.g. \cite{BV92},
\par
Moreover, the equilibria $\Cal R$ of this problem satisfy the second order ODE
$$
\nu u''(x)-f(u(x))+g=0,\ \ u(0)=u(\pi)=0.
$$
Thus, the map $u\to u'(0)$ gives a homeomorphic (and even smooth) embedding of $\Cal R$ to $\R$, so $\dim_{emb}(\Cal R,\Phi)=1$. Thus, we expect that $\dim_{det}(S(t),\Phi)=1$ (or at most $2$ according to Proposition \ref{Prop8.det}). The possible explicit form of the determining functional is well-known here: $\Cal F(u):= u\big|_{x=x_0}$, where $x_0>0$ is small enough, see \cite{kukavica}. Indeed, let $u_1(t),u_2(t)\in\Cal A$ be two complete trajectories of \eqref{8.heat} belonging to the attractor such that $u_1(t,x_0)\equiv u_2(t,x_0)$. Then the function $v(t)=u_1(t)-u_2(t)$ solves
\begin{equation}\label{8.dif}
\Dt v=\nu\partial_x^2v-l(t)v,\  v\big|_{x=0}=v\big|_{x=x_0}=0,
\end{equation}
where $l(t):=\int_0^1f'(su_1(t)+(1-s)u_2(t))\,ds$. Since the attractor $\Cal A$ is bounded in $C[-\pi,\pi]$, we know that $|l(t)|_{C[0,\pi]}\le L$ is also globally bounded. Multiplying equation \eqref{8.dif} by $v$, integrating in $x\in[0,x_0]$ and using the fact that the first eigenvalue of $-\partial_x^2$ with Dirichlet boundary conditions is $(\frac\pi{x_0})^2$, we get
$$
\frac12\frac d{dt}\|v(t)\|^2_{L^2}+\nu\(\frac\pi{x_0}\)^2\|v(t)\|^2_{L^2}-L\|v(t)\|^2_{L^2}\le0
$$
Fixing $x_0>0$ to be small enough that $\nu\(\frac\pi{x_0}\)^2>L$, applying the Gronwall inequality and using that $\|v(t)\|_{L^2}$ remains bounded as $t\to-\infty$, we conclude that $v(t)\equiv0$ for all $t\in\R$ and $x\in[0,x_0]$. Thus, the trajectories $u_1(t,x)$ and $u_2(t,x)$ coincide for all $t$ and all $x\in[0,x_0]$. Using now the arguments related with logarithmic convexity or Carleman type estimates which work for much more general class of equations, see \cite{RL}), we conclude that the trajectories $u_1$ and $u_2$ coincide. Thus, $\Cal F$ is separating on the attractor and therefore is asymptotically determining.
Being pedantic, we need to note that the functional $\Cal F(u)=u\big|_{x=x_0}$ is not defined on the phase space $H$, but on its proper subspace $C[0,\pi]$ (this is a typical situation for determining nodes, see \cite{OT08}). However, we have an instantaneous $\Phi\to C[0,\pi]$ smoothing property, so if we start from $u_0\in \Phi$, the value $\Cal F(u(t))$ will be defined for all $t>0$, so we just ignore this small inconsistency.
\end{example}
\begin{example}\label{Ex8.per} Let us consider the same equation \eqref{8.heat} on $[0,\pi]$, but endowed with {\it periodic} boundary conditions. In this case we do not have the condition $v(0)=0$, so the set $\Cal R$ of equilibria is naturally embedded in $\R^2$, not in $\R^1$, by the map $u\to (u\big|_{x=0},u'\big|_{x=0})$. Thus, we cannot expect that the determining dimension is one. Moreover, at least in the case when $g=const$, equation \eqref{8.heat} possesses a spatial shift as a symmetry and, therefore, any nontrivial equilibrium generates the whole circle of equilibria. Since a circle cannot be homeomorphically embedded in $\R^1$, the determining dimension must be at least $2$. We claim that it is indeed $2$ and the determining functionals can be taken in the form:
\begin{equation}\label{8.2-det}
\Cal F_1(u):=u\big|_{x=0},\ \ \Cal F_2(u)=u\big|_{x=x_0}.
\end{equation}
Indeed, arguing exactly as in the previous example, we see that the system $\Cal F=\{\Cal F_1,\Cal F_2\}$ is asymptotically determining if $x_0>0$ is small enough.
\end{example}
The next natural example shows that the determining dimension may be finite and small even if the corresponding global attractor is infinite-dimensional.
\begin{example}\label{Ex8.inf} Let us consider the semilinear heat equation \eqref{8.heat} on the whole line $x\in\R$. The natural phase space for this problem is the  {\it uniformly-local} space
\begin{equation}\label{8.ul}
\Phi=L^2_b(\R):=\big\{u\in L^2_{loc}(\R),\ \|u\|_{L^2_b}:=\sup_{x\in\R}\|u\|_{L^2(x,x+1)}<\infty\big\}.
\end{equation}
It is known that, under natural dissipativity assumption on $f\in C^1(\R)$ (e.g., $f(u)u\ge-C+\alpha u^2$ with $\alpha>0$), this equation generates a dissipative DS $S(t)$ in $\Phi$ for every $g\in L^2_b(\R)$. Moreover, this DS possesses the {\it locally compact} global attractor $\Cal A$ that  is a bounded in $\Phi$ and compact in $L^2_{loc}(\R)$ strictly invariant set which attracts bounded in $\Phi$ sets in the topology of $L^2_{loc}(\R)$, see \cite{MZ08} and references therein. Note that, in contrast to the case of bounded domains, the compactness and attraction property in $\Phi$ fail in general in the case of unbounded domains. It is also known that, at least in the case where equation \eqref{8.heat} possesses a spatially homogeneous exponentially unstable equilibrium (e.g., in the case where $g=0$ and $f(u)=u^3-u$), the fractal dimension of $\Cal A$ is infinite (actually it contains submanifolds of any finite dimension), see \cite{MZ08}.
\par
Nevertheless, the system of linear functionals \eqref{8.2-det} remains determining for this equations by {\it exactly} the same reasons as in Examples \ref{Ex8.dir} and \ref{Ex8.per}. Thus,
$$
\dim_{det}(S(t),\Phi)=2.
$$
One  functional is not determining in general by the reasons explained in Example \ref{Ex8.per}.
\end{example}
We now give an example of a non-dissipative and even {\it conservative}  system with determining dimension one.
\begin{example}\label{Ex8.wave} Let us consider the following 1D wave equation:
\begin{equation}\label{8.wave}
\partial^2_tu=\partial_x^2u,\ \ x\in(0,\pi),\ \ u\big|_{x=0}=u\big|_{x=\pi}=0,\ \ \xi_u\big|_{t=0}=\xi_0,
\end{equation}
where $\xi_u(t):=\{u(t),\Dt u(t)\}$. It is well-known that problem \eqref{8.wave} is well-posed in the energy phase space $E:=H^1_0(0,\pi)\times L^2(0,\pi)$ and the energy identity holds
\begin{equation}
\|\Dt u(t)\|_{L^2}^2+\|\partial_x u(t)\|^2_{L^2}=const.
\end{equation}
Moreover, the solution $u(t)$ can be found explicitly in terms of $\sin$-Fourier series:
\begin{equation}\label{8.sol}
u(t)=\sum_{n=1}^\infty (A_n\cos(nt)+B_n\sin(nt))\sin (nx),
\end{equation}
where $A_n=\frac2{\pi}(u(0),\sin (nx))$, $B_n=\frac2{\pi}(u'(0),\sin(nx))$. Crucial for us is that the function \eqref{8.sol} is an {\it almost-periodic} function of time with values in $H^1_0$. Let us consider now a linear functional on $\Phi=L^2(0,\pi)$, i.e,
\begin{equation}
\Cal Fu=(l(x),u)=\sum_{n=1}^\infty l_nu_n,\ \
\end{equation}
where $l_n$ and $u_n$ are the Fourier coefficients of $l\in \Phi$ and $u$ respectively. Then
\begin{equation}
\Cal Fu(t)=\sum_{n=1}^\infty l_nA_n\cos(nt)+l_nB_n\sin(nt)
\end{equation}
is a scalar almost periodic function. Since the Fourier coefficients of an almost-periodic function are uniquely determined by the function, we have
$$
\Cal Fu(t)\equiv0\ \ \text{ if and only if }\ \ l_nA_n=l_nB_n=0,
$$
see \cite{LZ} for details. Thus, if we take a generic function $l$ (for which $l_n\ne0$ for all $n\in\Bbb N$, $\Cal Fu(t)\equiv0$ will imply that $A_n=B_n=0$ and, therefore, $u(t)\equiv0$. Thus, $\Cal F$ is separating on the set of complete trajectories. It remains to note that, since any trajectory of \eqref{8.wave} is almost-periodic in $E$, the $\omega$-limit set of any trajectory exists and is compact in $E$.  Then,  as not difficult to check (see \cite{KAZ22}), $\Cal F$ is also asymptotically determining, so the determining dimension of this system is one.
\end{example}
\begin{remark} The principal difference between the case of one spactial dimension and the multidimensional case is that, in 1D case, any equilibrium $u_0\in\Cal R$ of such a PDE solves a system of ODEs, so the dimension of $\Cal R$ is restricted by the order of this system. This allows us in many cases to get sharp estimates for the determining dimension and even compute it explicitly. In particular,  Examples \ref{Ex8.dir} and \ref{Ex8.per} clearly show that it is independent of physical parameters of the DS considered as well as of the dimensions of the attractor. In contrast to this, in the multi-dimensional case, $u_0\in\Cal R$ usually solves an elliptic PDE and we do not have any good formulas for the dimension of $\Cal R$. The best what we can do in general is to use the obvious estimate
$$
\dim_{emb}(\Cal R,\Phi)\le\dim_{emb}(\Cal A,\Phi)\le 2\dim_f(\Cal A,\Phi)+1.
$$
Since the fractal dimension of the attractor $\Cal A$ usually depends on physical parameters (e.g. on the Grashoff number if the Navier-Stokes system is considered), this may produce an illusion that the number of determining functionals is also related with these parameters. Nevertheless, as Proposition \ref{Prop8.det} shows, it is still determined by the size of the equilibria set and is not related with the complexity of the dynamics on it. These arguments show also that, in contrast to 1D case, in multi-dimensional case, we really need to use some "generic" assumptions in order to kill pathological equilibria and get a reasonable result.
\end{remark}

\section{Appendix. Function Spaces}\label{a1}

The aim of this appendix is to introduce and discuss various classes of function spaces which are used throughout of the survey.  We start with the Lebesgue and Sobolev spaces.
\begin{definition}\label{Def9.Sob} Let $\Omega$ be a domain of $\R^d$ with a sufficiently smooth boundary. We denote by $L^p(\Omega)$, $1\le p\le\infty$, the Lebesgue space of functions whose $p$th power is Lebesgue integrable. For any $n\in\Bbb N$, we denote by $W^{n,p}(\Omega)$ the Sobolev space of distributions whose derivatives up to order $n$ inclusively belong to the space $L^p(\Omega)$. For non-integer positive $s=n+\alpha$, $n\in\Bbb Z_+$, $\alpha\in(0,1)$, the space $W^{s,p}(\Omega)$ is defined as a Besov space $B^s_{p,p}(\Omega)$ via the following norm
\begin{equation}\label{9.bes}
\|u\|_{W^{s,p}}^p:=\|u\|_{W^{n.p}(\Omega)}^p+\sum_{|\beta|=n}\int\int_{\Omega\times\Omega}\frac{|D^{\beta}u(x)-D^\beta u(y)|^p}{|x-y|^{d+p\alpha}}\,dx\,dy.
\end{equation}
We also denote by $W^{s,p}_0(\Omega)$ the closure of $C_0^\infty(\Omega)$ in the metric of $W^{s,p}(\Omega)$ and, for negative values of $s$, we define $W^{s,p}(\Omega)$ by the duality
$$
W^{s,p}(\Omega):=[W^{-s,q}_0(\Omega)]^*,\ \ \frac1p+\frac1q=1,
$$
see \cite{Tri78} and references therein for more details. We also use the notation $H^s(\Omega)$ for the spaces $W^{s,p}(\Omega)$ with $p=2$.
\end{definition}
Let now $V$ be a Banach (or more general locally convex) space and let $V^*$ be its dual space (the space of linear continuous functionals on $V$). Then the weak topology on $V$ is defined by the following system of seminorms on $V$: $p_l(x):=|lx|$, $l\in V^*$. On the level of sequences this means that $x_n\rightharpoondown x$ in $V$ if and only if $lx_n\to lx$ for all $l\in V^*$.  The weak-star topology on a dual space $V^*$ is defined analogously, see e.g. \cites{RR,Ru91} and references therein.
\par
The key result, which is widely used in the theory of attractors, is the Banach-Alaoglu theorem which claims that the closed unit ball in $V^*$ is {\it compact} in the weak-star topology (this result can be extended to locally convex or even linear topological spaces if we replace the unit ball by the {\it polar} $U^0\in V^*$ of any bounded set $U\subset V$, see \cite{RR} for details). We recall that a set $U$ is bounded in a linear topological space if it is absorbed by any neighbourhood of  zero and the polar $U^0$ is defined by
$$
U^0:=\{l\in V^*,\sup_{x\in U} |lx|\le 1\}.
$$
The analogous results hold for the weak topology if and only the space $V$ is {\it reflexive} ($V=V^{**}$). Most important for us is the fact that the Sobolev spaces $W^{s,p}(\Omega)$ are reflexive if and only if $1<p<\infty$, see \cite{Tri78}. We also recall that compactness and sequential compactness are different (unrelated) concepts in non-metrizable topological spaces, so the Banach-Alaoglu theorem (which is based on the Tikhonov compactness theorem) does not give a {\it sequential} weak-star compactness of the unit ball in a dual space. In order to get a sequential compactness, we need either to assume that either the space $V$ is separable (then the unit ball in a dual space is metrizable in a weak-star topology) or that the space $V$ is reflexive (then weak compactness and weak sequential compactness coincide due to the Eberlein-Smulian theory), see \cite{RR} for more details. We also recall that $L^\infty(\Omega)$ is a dual space to $L^1(\Omega)$, so there is a natural choice of the weak-star topology in $L^\infty(\Omega)$ which gives the compactness and sequential compactness of a unit ball. In contrast to this, the spaces $C(\bar\Omega)$ or $L^1(\Omega)$ are not duals to any Banach spaces and there is no reasonable topology in them which gives the compactness of their unit balls.
\par
 In order to consider evolutionary equations, we often use the spaces of functions $u(t)$ with values in some Banach space $V$. Here usually $t\in\R$ or its subset is interpreted as time and $V$ is the properly chosen Sobolev space.  For any $[a,b]\subset\R$ and any $1\le p\le\infty$, we define $L^p(a,b;V)$ as a space of Bochner measurable functions such that
 $$
 \|u\|_{L^p(a,b;V)}^p:=\int_a^b\|u(t)\|_V^p\,dt<\infty.
 $$
 The space of Bochner measurable {\it locally} integrable in power $p$ functions will
 be denoted by $L^p_{loc}(a,b;V)$. Sometimes, when we need to emphasize the properties of the functions near the finite endpoints $a$ or $b$, we will write, say, $L^p_{loc}([a,b],V)$ or $L^p_{loc}((a,b),V)$. The time Sobolev and Besov spaces $W^{s,p}(a,b;V)$ are defined analogously to \eqref{9.bes}, see \cite{LR02} and references therein.
\par
The uniformly local space $L^p_b(a,b;V)$ is defined as a subset of $L^p_{loc}(a,b;V)$ for which the following norm is finite:
$$
\|u\|_{L^p_b(a,b;V)}:=\sup_{s\in\R,[s,s+1]\subset[a,b]}\|u\|_{L^p(s,s+1;V)}
$$
and for the case $b-a<1$, we just set $\|u\|_{L^p_b(a,b;V)}:=\|u\|_{L^p(a,b;V)}$. The spaces $W^{s,p}_b(a,b;V)$ are defined analogously.
\par
We also need some more properties of the locally convex (Frechet) space $L^p_{loc}(\R,V)$. We assume that the space $V$ is reflexive and separable. Then, $[L^p(a,b;V)]^*=L^q(a,b;V^*)$ for $\frac1p+\frac1q=1$, $1\le p<\infty$. In particular, these spaces are reflexive and separable for any $1<p<\infty$ and, therefore, the Banach-Alaoglu theorem gives us the weak sequential compactness of the unit ball in $L^p(a,b;V)$. This, in turn, gives us the weak sequential compactness of any bounded weakly closed subset of $L^p_{loc}(\R,V)$. In particular, if this bounded set is convex, then the closeness in a strong topology is enough. For instance, the unit ball in $L^p_b(\R,V)$ is weakly sequentially compact in $L^p_{loc}(\R,V)$. We use this fact, in particular, for defining the hulls of translation bounded external forces.
\par
Recall that the weak topology in $L^p_{loc}(\R,V)$ can be obtained as a projective limit of spaces $L^p(-n,n;V)$ (endowed with weak topology) as $n\to\infty$. In particular, $u_n\rightharpoondown u$ in $L^p_{loc}(\R,V)$ if and only if for any $n\in\Bbb N$ and any $l\in L^q(-n,n;V^*)$, we have the convergence
$$
\int_{-n}^n\left<l(t),u_n(t)\right>\,dt\to   \int_{-n}^n\left<l(t),u(t)\right>\,dt.
$$

\end{document}